\documentclass[10pt]{amsart}
\usepackage{amsthm,amssymb,amscd,amsmath,amsfonts,amssymb,amscd,stmaryrd,hyperref,mathrsfs,yhmath} 
\usepackage[shortlabels]{enumitem}
\usepackage[all]{xy}
\usepackage{hyperref} 					
\usepackage{texmac} 					

\newtheorem{MainThm}{Theorem}

\theoremstyle{definition}
\newtheorem{defn}{Definition}[subsection]
\newtheorem{thm}[defn]{Theorem}
\newtheorem{cor}[defn]{Corollary}
\newtheorem{prop}[defn]{Proposition}
\newtheorem{lem}[defn]{Lemma}
\newtheorem{ex}[defn]{Example}

\newtheorem{rmk}[defn]{Remark}
\newtheorem{rmks}[defn]{Remarks}
\newtheorem{constr}[defn]{Construction}

\newtheorem{setup}[defn]{Hypothesis}
\newtheorem{notn}[defn]{Notation}

\calsymbols{c}{A,B,C,D,E,F,G,H,I,J,K,L,M,N,O,P,Q,R,S,T,U,V,W,X,Y,Z}
\bbsymbols{b}{A,B,C,D,E,F,G,H,I,J,K,L,M,N,O,P,Q,R,S,T,U,V,W,X,Y,Z}
\scrsymbols{s}{A,B,C,L,Q,R,S,T,U,}
\fraksymbols{f}{g,m,h,t,n,b,}
\bsymbols{b}{g,h,x,y,A,B,C,D,E,F,G,H,I,J,K,L,M,N,O,P,Q,R,S,T,U,V,W,X,Y,Z}
\bbsymbols{}{B,R,C,Z,N,P,Q,F,G,H,U,W,X,Y,T}

\operators{Der,gr,Spec,Proj,MaxSpec,ad,Loc,Stab,Aut,Hom,End,im,coker,Lie,Spf,rig,Sp,Ad,aug,ac,tors,Sym,mod,Homeo,op,qc,Res,alg,rk,Sch,id,GL,Ab}

\DeclareMathOperator{\Frech}{Frech}
\DeclareMathAlphabet{\mathpzc}{OT1}{pzc}{m}{it}

\newcommand{\h}[1]{\widehat{#1}}

\newcommand{\htf}[1]{\overline{\h{#1}}}
\newcommand{\hK}[1]{\h{#1_K}}
\newcommand{\hsULK}{\hK{\sU(\cL)}}
\newcommand{\dnl}{\cD^\lambda_n}
\newcommand{\hdnlK}{\h{\cD^\lambda_{n,K}}}
\newcommand{\hdnK}{\h{\cD_{n,K}}}

\newcommand{\hnK}[1]{\h{#1_{n,K}}}
\newcommand{\w}[1]{\wideparen{#1}}

\newcommand{\wDGG}{\w\cD(\bG, G)^{\bG}}
\newcommand{\fr}[1]{\mathfrak{{#1}}}

\newcommand{\ts}[1]{\texorpdfstring{$#1$}{}}

\newcommand{\be}{\begin{enumerate}[{(}a{)}]}
\newcommand{\ee}{\end{enumerate}}
\newcommand{\qmb}[1]{\quad\mbox{#1}\quad}
\newcommand{\csp}{\hspace{-0.1cm}}
\newcommand{\hsp}{\hspace{0.1cm}}
\newcommand{\Star}{\hsp \w{\rtimes} \hsp }
\newcommand{\kAlg}{k\hspace{-0.1cm}-\hspace{-0.1cm}\textbf{Alg}}
\newcommand{\Emod}[2]{#1\csp-\csp #2\csp -\csp \mod}
\newcommand{\congs}{\stackrel{\cong}{\longrightarrow}}
\newcommand{\zAut}{\mathpzc{Aut}}
\newcommand{\zLie}{\mathpzc{Lie}}
\newcommand{\Ex}[1]{\widetilde{#1}}
\newcommand{\Uf}[1]{U(\fr{#1})}
\newcommand{\wUg}[1]{\w{U}(\fr{g},{#1})}
\newcommand{\tocong}{\stackrel{\cong}{\longrightarrow}}
\newcommand{\curtdash}{\hspace{-0.1cm}-\hspace{-0.1cm}}
  \let\leq=\leqslant
  \let\geq=\geqslant

\begin{document}

\title{Equivariant $\cD$-modules on rigid analytic spaces}
\author{Konstantin Ardakov}
\address{Mathematical Institute\\University of Oxford\\Oxford OX2 6GG}
\thanks{The first author was supported by EPSRC grant EP/L005190/1.}
\begin{abstract} We define coadmissible equivariant $\mathcal{D}$-modules on smooth rigid analytic spaces and relate them to admissible locally analytic representations of semisimple $p$-adic Lie groups.
\end{abstract}
\maketitle
\tableofcontents
\section{Introduction}

Let $K$ be a non-archimedean field of mixed characteristic $(0,p)$, and let $L$ be a finite extension of $\Qp$ contained in $K$. In a series of papers including \cite{ST1,ST2,ST3,ST}, Schneider and Teitelbaum developed the theory of \emph{locally analytic representations} of an $L$-analytic group $G$ in locally convex (usually infinite dimensional) topological vector spaces over $K$. These kinds of representations of $G$ arise naturally in several places in number theory, for example in the theory of automorphic forms \cite{EmertonInterpolation} and in the $p$-adic local Langlands program \cite{BreuilICM, ColmezFunctor, EmertonLFuncs}. Many of these representations enjoy an important finiteness condition called \emph{admissibility}, introduced in \cite{ST}; almost by definition, the category of admissible locally analytic representations of $G$ is anti-equivalent to the category of \emph{coadmissible modules} over the locally $L$-analytic distribution algebra $D(G,K)$ of $G$.

We would like to better understand these representations using the theory of $\cD$-modules. The Lie algebra $\fr{g}$ of $G$ acts naturally on every $D(G,K)$-module $V$; when $\fr{g}$ is semisimple, it is then possible to localise the $\fr{g}$-module $V$ to the flag variety of the corresponding algebraic group in the sense of Beilinson and Bernstein \cite{BB} without losing too much information about the $\fr{g}$-action on $V$ in the process: for example, localisation yields an equivalence of categories between the category of $\fr{g}$-modules with trivial infinitesimal central character and the category of $\cD$-modules on the flag variety that are quasi-coherent as $\cO$-modules in the Zariski topology.

In the setting of infinite-dimensional representations of complex Lie groups, this (purely algebraic) method of localisation was used in a spectacular way by Beilinson-Bernstein \cite{BB} and Brylinski-Kashiwara \cite{BryKash} to prove the Kazhdan-Lusztig conjectures \cite{KLConj}. However, applying the algebraic localisation functor $\cD \otimes_{U(\fr{g})} -$ directly to a coadmissible $D(G,K)$-module $V$ only remembers information about the underlying $\fr{g}$-action on $V$ and completely forgets about the $G$-action. Slightly less naively, it is possible to remember the $G$-action on the localised $\cD$-module in the form of a $G$-equivariant structure; however this still forgets the natural topology carried by $V$ as well as the coadmissibility of $V$, and it is not reasonable to expect to be able to give a purely local characterisation of the essential image of coadmissible $D(G,K)$-modules under this functor. 

To address these topological issues, consider the closure $\w{U(\fr{g}_K)}$ in $D(G,K)$ of the enveloping algebra $U(\fr{g}_K)$, where $\fr{g}_K := \fr{g} \otimes_L K$. This is the Arens-Michael envelope of $U(\fr{g}_K)$ considered by Schmidt in \cite{SchmidtStable,SchmidtVerma}.  It was observed already in \cite{ST2} that this closure consists precisely of non-commutative formal power series $\sum_{\alpha \in \N^d} \lambda_\alpha \mathbf{x}^\alpha$, where $\{x_1,\ldots,x_d\}$ is any $K$-basis of $\fr{g}_K$ and $\{\lambda_{\alpha} : \alpha \in \N^d\} \subset K$, that converge everywhere in the sense that no matter how large the real number $R > 0$, we still have $|\lambda_\alpha| R^{|\alpha|} \to 0$ as $\alpha \to \infty$. In a more invariant formulation, $\w{U(\fr{g}_K)}$ is naturally in bijection with the $K$-vector space $\cO(\fr{g}_K^{\ast,\rig})$ of rigid analytic functions converging everywhere on the rigid analytification $\fr{g}_K^{\ast, \rig}$ of the affine variety $\fr{g}_K^\ast = \Spec S(\fr{g})$. Thus, $\w{U(\fr{g}_K)}$ is a ``rigid analytic quantisation" of $\fr{g}_K^{\ast, \rig}$ in the same way that $U(\fr{g}_K)$ is an ``algebraic quantisation" of $\fr{g}_K^\ast$. 

In a series of recent papers \cite{ArdICM,DCapOne,DCapTwo,ABB}, we have introduced a sheaf of rings $\w{\cD}$ of infinite order differential operators on any smooth rigid analytic space $\bX$. This sheaf is morally a ``rigid analytic quantisation" of the cotangent bundle $T^\ast \bX$, in the same way that the sheaf of finite order differential operators $\cD_X$ on any complex smooth algebraic variety $X$ is an algebraic quantisation of $T^\ast X$. We also introduced the abelian category $\cC_{\bX}$ of \emph{coadmissible $\w{\cD}$-modules on $\bX$} and showed that it behaves in several ways analogously to the category of coherent $\cD$-modules on complex smooth algebraic varieties.

The purpose of this paper is to introduce the category $\cC_{\bX/G}$ of \emph{coadmissible $G$-equivariant $\cD$-modules} on any smooth rigid analytic space $\bX$ equipped with a continuous action\footnote{see Definition \ref{CtsAct}} of a $p$-adic Lie group $G$. As the notation suggests, we would like to think of an object in this category as a sheaf on the quotient space $\bX / G$. Since $\bX$ and $G$ live in different categories --- one is a set equipped with a Grothendieck topology,  the other is a topological group --- a little care is required to make sense of this. The category of abelian sheaves $\Ab_{\bX}$ on the rigid analytic space $\bX$ is equivalent to the category of abelian sheaves $\Ab_X$ on an honest topological space $X$, namely the Huber space $X:=\cP(\bX)$ --- see, for example, \cite{Huber, SchVdPut}. By the functoriality of this construction, the action of $G$ on $\bX$ induces a continuous action of $G$ on $X$, and we then have at our disposal the topological space $X/G$ consisting of the $G$-orbits of $X$, equipped with the quotient topology, as well as the abelian category $\Ab_{X/G}$ of abelian sheaves on $X/G$. Following Grothendieck \cite{Tohoku}, we regard the category $G$-$\Ab_X$ of $G$-equivariant abelian sheaves on $X$ to be a better-behaved enhancement of $\Ab_{X/G}$: there is always the forgetful functor $G$-$\Ab_X \to \Ab_{X/G}$. Our category $\cC_{\bX/G}$ is a certain subcategory of $G$-$\Ab_{\bX}$.

To define $\cC_{\bX/G}$, we proceed in several steps, broadly analogously to \cite{DCapOne}. First, we consider the case where the pair $(\bX,G)$ is \emph{small}; roughly speaking \footnote{see Definition \ref{DefnOfSmall}}, this means that the variety $\bX$ is affinoid and the group $G$ is compact. Whenever $(\bX,G)$ is small, in $\S \ref{MainConstr}$ we introduce a $K$-Fr\'echet algebra 
\[ \w\cD(\bX,G)\]
which is a particular $K$-Fr\'echet space completion of the abstract skew-group algebra $\cD(\bX) \rtimes G$. This construction is partly motivated by the fact that the locally analytic distribution algebra $D(G,K)$ can be viewed as a particular $K$-Fr\'echet space completion of the abstract skew-group algebra $U(\fr{g}_K) \rtimes G$. In the case where the group $G$ is trivial,  $\w\cD(\bX,G)$ reduces to the algebra $\w\cD(\bX)$ from \cite{DCapOne} and in the case where $\bX$ is a one-point space, $\w\cD(\bX,G)$ reduces to the algebra $D^\infty(G,K)$ of locally constant distributions on $G$ from \cite{ST1}.

\begin{MainThm}\label{FrStIntro} $\w\cD(\bX,G)$ is Fr\'echet-Stein whenever $(\bX,G)$ is small.
\end{MainThm}

When the ground field $K$ is discretely valued and either $\bX$ or $G$ is trivial, this is a special case of results in \cite{DCapOne} or \cite{ST}, respectively. However, we give a completely new proof that is powerful enough to deal with the case where the valuation of $K$ is not necessarily discrete, using deep results of Raynaud and Gruson \cite{GruRay}. 

After Theorem \ref{FrStIntro} and the constructions of \cite{ST}, we have at our disposal the abelian category $\cC_{\w\cD(\bX,G)}$ of coadmissible $\w\cD(\bX,G)$-modules. Still when $(\bX,G)$ is small\footnote{in fact, we define this localisation functor in a greater generality that will prove useful later}, we construct in $\S \ref{CompActSect}, \ref{LocXSect}$  a \emph{localisation functor}
\[\Loc : \cC_{\w\cD(\bX,G)} \longrightarrow \Emod{G}{\cD_{\bX}}\]
from coadmissible $\w\cD(\bX,G)$-modules to $G$-equivariant $\cD$-modules on $\bX$. This functor is not full in general: for example, if $\bX$ is a one-point space, the localisation functor is then simply the restriction functor along the inclusion $K[G] \hookrightarrow D^\infty(G,K)$.  In order to address this issue, we observe that the sections $\Loc(M)(\bU)$ of our localised sheaf $\Loc(M)$ over sufficiently small affinoid subdomains $\bU$ of $\bX$ carry a natural Fr\'echet-space topology, being (by construction) coadmissible modules over the Fr\'echet-Stein algebra $\w\cD(\bU,H)$ for any compact open subgroup $H$ of $G$ stabilising $\bU$. This leads us to consider the category $\Frech(G-\cD_\bX)$ of $G$-equivariant \emph{locally Fr\'echet} $\cD$-modules for any pair $(\bX,G)$, which, roughly speaking, allows us to keep track of these Fr\'echet-space topologies on local sections over sufficiently small affinoid subdomains of $\bX$. Our localisation functor takes values in $\Frech(G-\cD_{\bX})$, and for general $(\bX,G)$ we define $\cC_{\bX/G}$ to be the full subcategory of $\Frech(G-\cD_{\bX})$ whose objects are locally isomorphic to a sheaf of the form $\Loc(M)$ --- see Definition \ref{CoadmFrechEqDmod} and Remarks \ref{TopRmk} for more details. Here is a summary of its properties.

\begin{MainThm}\label{CXGIntro} Let $G$ be a $p$-adic Lie group acting continuously on the smooth rigid analytic space $\bX$, let $H$ be an open subgroup of $G$ and let $\bY$ be a $G$-stable admissible open subset of $\bX$.
\be \item There is a faithful forgetful functor $\cC_{\bX / G} \to \cC_{\bX/H}$.
\item There is a forgetful functor $\cC_{\bX/G} \to \cC_{\bY/G}$.
\item Whenever $(\bX,G)$ is small, the localisation functor
\[\Loc : \cC_{\w\cD(\bX,G)} \longrightarrow \cC_{\bX/G}\]
is an equivalence of categories, with quasi-inverse $\Gamma(\bX,-)$.
\item $\cC_{\bX/G}$ is an abelian category.
\ee\end{MainThm}

We can now state our main result: an equivariant Beilinson-Bernstein localisation theorem for admissible locally analytic representations. Recall that a $\fr{g}_K$-module $V$ is said to have \emph{trivial infinitesimal central character} if the central elements in $U(\fr{g}_K)$ that annihilate the trivial representation also annihilate $V$.

\begin{MainThm}\label{EqEq} Let $L$ be a finite extension of $\Qp$ contained in $K$, let $\G$ be an affine algebraic $L$-group such that $\G_K := \G \otimes_L K$ is connected and split semisimple, and let $G$ be an open subgroup of $\G(L)$. Then there is an equivalence of categories
\[
\left\{ 
				\begin{array}{c} 
					admissible \hsp locally \hsp L\curtdash analytic\\ 
					K\curtdash representations \hsp of \hsp G\\
					with \hsp trivial \hsp infinitesimal\\
					central \hsp character
				\end{array}
\right\} \cong \left\{
				\begin{array}{c}
				 coadmissible\\
				 G \curtdash equivariant\\
				 \cD \curtdash modules \hsp on \hsp \bX
				\end{array}
\right\}^{\op}\]
where $\bX := (\G_K/\B)^{\rig}$ is the rigid analytic flag variety associated with $\G_K$.
\end{MainThm}

In fact, in Theorem \ref{BBEquiv} below we establish a more general version of Theorem \ref{EqEq} where the hypothesis on the group $G$ is significantly relaxed. This involves defining a  generalisation $\w{U}(\fr{g}_K, G)$ of the locally analytic distribution algebra $D(G,K)$ that is a certain completion of the usual skew-group algebra $U(\fr{g}_K) \rtimes G$.

In order to keep the length of this paper manageable, we will postpone applications of the main results of this paper to the theory of locally analytic representation to a later publication. However, one invariant that emerges from our theory is the \emph{support} of the coadmissible $G$-equivariant $\cD$-module that is the localisation of some coadmissible $D(G,K)$-module. This is \emph{a priori} a $G$-stable subset of the Huber space $\cP( (\G_K/\B)^{\rig} )$, and evidently two coadmissible $D(G,K)$-modules cannot be isomorphic if their localisations have distinct support. In our forthcoming work \cite{ArdEqKash} we will construct examples of objects in $\cC_{\bX / G}$ having prescribed support. We also hope that the theory developed in this paper will eventually shed some light on the locally analytic representations that appear in the $p$-adic local Langlands program.

Several works on locally analytic Beilinson-Bernstein localisation have already appeared. This paper can be viewed as a natural continuation and enhancement of our earlier work \cite{AW13} with Simon Wadsley. That paper contains a localisation equivalence for certain Banach-algebra completions of $U(\fr{g}_K)$ and is concerned with formal completions of flag varieties defined over a discrete valuation ring along the special fibre. In \cite{SchmidtBB}, Tobias Schmidt explained how to use the constructions of \cite{AW13} to give a localisation equivalence for $D(G,K)$. However, because the sheaves appearing in \cite{AW13,SchmidtBB} only live on the smooth formal model of the rigid analytic flag variety, they cannot immediately see the support of the localised $\cD$-module on the rigid analytic flag variety itself. In a series of papers \cite{PSS1, PSS2, PSS3}, Deepam Patel, Tobias Schmidt and Matthias Strauch address this (as well as many other) issues and localise coadmissible $D(G,K)$-modules onto the rigid analytic flag variety by regarding it as an inverse limit of all possible $G$-equivariant formal models. We expect that their construction can be shown to agree with ours; however they do not give a local characterisation of the essential image of their localisation functor. Close to the end of the preparation of this paper, we discovered that in her Chicago PhD thesis \cite{FanTianqi}, Tianqi Fan develops yet another version of locally analytic localisation. Whilst formally different and perhaps lighter on the details, her approach is much closer to ours than those of \cite{SchmidtBB, PSS2, PSS3}. 

Let us now make some comments on several technical aspects of this paper, that may in part justify its length. As we have mentioned already, we do not assume that our ground field is discretely valued. A large part of $\S \ref{LevelwiseSect}$ is concerned with extending the results from \cite{DCapOne} to this more general setting. Just like \cite{PSS2, PSS3}, we do not make any restrictions on the prime $p$ in contrast to our earlier work \cite{AW13}; since we rely on the main constructions of \cite{AW13} we extend the results of that paper to any $(p,K)$ in $\S \ref{hDaffFlagVar}$. Although every coadmissible $G$-equivariant $\cD$-module in our sense is also naturally a $\w\cD$-module, it will not be a \emph{coadmissible} $\w\cD$-module unless the group $G$ is finite. Because of this, we chose to work with equivariant $\cD$-modules instead of equivariant $\w\cD$-modules for simplicity. We work with general, not necessarily compact, $p$-adic Lie groups $G$: also, the algebraic group $\G$ appearing in Theorem \ref{EqEq} and the $p$-adic Lie group $G$ appearing in Theorem \ref{BBEquiv} do not need to be split semisimple. In order to keep the length of this paper down, we do not mention equivariant twisted $\cD$-modules or non-trivial infinitesimal central characters.

Throughout this paper, $K$ will denote a field equipped with a complete non-archimedean norm $|\cdot|$, $\cR := \{ \lambda \in K : |\lambda| \leq 1\}$ denotes the unit ball inside $K$ and $\pi \in \cR$ is a fixed non-zero non-unit element.

\section{Algebraic background}

\subsection{Enveloping algebras of Lie-Rinehart algebras}\label{AutLRsect}

Let $R$ be a commutative ring, let $A$ be a commutative $R$-algebra and let $L$ be an $(R,A)$-Lie algebra \cite[\S 2.1]{DCapOne}. Recall that $L$ is a left $A$-module and $A$ is a left $L$-module via the anchor map $\sigma_L$ of $L$. We denote these actions by 
\[ (a,v) \mapsto a\cdot v \qmb{and} (v,a) \mapsto v\cdot a = \sigma_L(v)(a) \qmb{for all} a\in A, v\in L.\]
We recall the details of the following important construction.
\begin{constr}\label{ConstrUL} There is an associative algebra $U(L)$ canonically associated with $L$, called the \emph{enveloping algebra} of $L$. It can be constructed as the quotient of the free associative $R$-algebra $\langle A, L\rangle$ generated by the ring $A$ and the abelian group $L$, by the two-sided ideal $J$ generated by
\begin{equation}\label{IdealJ}\{\overline{v} \hspace{.1cm} \overline{w} - \overline{w}\hspace{.1cm} \overline{v} - \overline{[v,w]}, \quad \overline{v}\hspace{.1cm} \overline{a} - \overline{a} \hspace{.1cm} \overline{v} - \overline{v\cdot a}, \quad \overline{a\cdot v}- \overline{a}\hspace{.1cm} \overline{v} \quad : \quad v,w \in L, a\in A\}.\end{equation}
Here $\overline{a}$ and $\overline{v}$ denote the images of $a \in A$ and $v \in L$ in $\langle A,L \rangle$, respectively. 

The algebra  $\langle A, L\rangle$ carries a natural positive filtration $F_\bullet \langle A,L\rangle$ given by
\[ F_j \langle A,L\rangle = \sum_{i\leq j} \overline{L}^i\]
with the convention that $\overline{L}^0 = \overline{A}$. We will denote the quotient filtration in $U(L) = \langle A,L\rangle / J$ by $F_\bullet U(L)$, and the natural maps $A \to U(L)$ and $L \to U(L)$ by $i_A$ and $i_L$, respectively.
\end{constr}

We recall universal property of $U(L)$ from \cite[\S 2.1]{DCapOne}.

\begin{lem}\label{ULUP} Let $S$ be an associative $R$-algebra, and let $j_A \colon A \to S$, $j_L : L \to S$ be homomorphisms of $R$-algebras and $R$-Lie algebras, respectively. Suppose that
\[j_L(av) = j_A(a)j_L(v) \qmb{and} [j_L(v), j_A(a)] = j_A(v \cdot a) \qmb{for all} a \in A, v \in L.\]
Then there exists a unique $R$-algebra homomorphism $\varphi \colon U(L) \to S$ making the following diagram commute:
\[ \xymatrix{ A \oplus L \ar[d]_{i_A \oplus i_L} \ar[dr]^{j_A \oplus j_L} & \\ U(L) \ar[r]_{\varphi} & S. }\]
\end{lem}

\begin{defn}\label{PhiMorph} Let $\varphi : A \to A'$ be an $R$-algebra homomorphism, let $L$ be an $(R,A)$-Lie algebra and let $L'$ be an $(R,A')$-Lie algebra. Then $\tilde{\varphi} : L \to L'$ is a \emph{$\varphi$-morphism} if
\begin{itemize} \item $\tilde{\varphi}$ is a homomorphism of $R$-Lie algebras,
\item $\tilde{\varphi}(a \cdot v) = \varphi(a) \cdot \tilde{\varphi}(v)$,
\item $\tilde{\varphi}(v) \cdot \varphi(a) =\varphi(v \cdot a)$, for all $a \in A, v\in L$.
\end{itemize}
\end{defn}

\begin{ex}\label{AutOfDer} Suppose that $\varphi : A \to A'$ is an $R$-algebra isomorphism, let $L = \Der_R(A)$ and $L' = \Der_R(A')$, viewed as an $(R,A)$-Lie algebra (respectively, $(R,A')$-Lie algebra) with the identity anchor map. Then 
\[ \dot{\varphi} : L \to L', \quad\quad v \mapsto \varphi \circ v \circ \varphi^{-1}\]
is a $\varphi$-morphism.
\end{ex}

It is straightforward to verify that if $A \stackrel{\varphi}{\longrightarrow} A' \stackrel{\psi}{\longrightarrow} A''$ are $R$-algebra homomorphisms, $\tilde{\varphi} : L \to L'$ is a $\varphi$-morphism, and $\tilde{\psi} : L' \to L''$ is a $\psi$-morphism, then $\tilde{\psi} \tilde{\varphi} : L \to L''$ is a $\psi\varphi$-morphism. In this way, pairs $(A,L)$ consisting of a commutative $R$-algebra $A$ and an $(R,A)$-Lie algebra $L$ naturally form a category $\sL\sR_R$, where a morphism $(A, L) \to (A',L')$ is a pair $(\varphi, \tilde{\varphi})$ where $\varphi : A \to A'$ is an $R$-algebra homomorphism and $\tilde{\varphi} : L \to L'$ is a $\varphi$-morphism. In particular, we have at our disposal the automorphism group
\[ \Aut(A, L) \]
for every $(R,A)$-Lie algebra $L$.

\begin{defn}\label{GactsLR} Let $G$ be a group. An \emph{action} of $G$ on the $(R,A)$-algebra $L$ is a group homomorphism $G \to \Aut(A, L)$.\end{defn}

\begin{ex}\label{GactDerEx} Let $A$ be a commutative $R$-algebra and let $L = \Der_R(A)$. Suppose that $\rho : G \to \Aut_R(A)$ is an action of the group $G$ on $A$ by $R$-algebra automorphisms, and for every $g \in G$, let $\dot{\rho}(g) := \dot{\rho(g)} :  L \to L$ be the $\rho(g)$-morphism defined in Example \ref{AutOfDer}. Then $(\rho, \dot{\rho}) : G \to \Aut (A,L)$ is a $G$-action on $L$. \end{ex}

\begin{lem}\label{Ufunctor} Let $\varphi : A \to A'$ be a ring homomorphism, let $L$ be an $(R,A)$-Lie algebra and let $L'$ be an $(R,A')$-Lie algebra. Then every $\varphi$-morphism $\tilde{\varphi} : L \to L'$ extends uniquely to a filtration-preserving $R$-algebra homomorphism $U(\varphi, \tilde{\varphi}) : U(L) \to U(L')$ which makes the following diagram of $R$-modules commute:
\[ \xymatrix{ A \oplus L \ar[rr]^{\varphi \oplus \tilde{\varphi}} \ar[d] _{i_A \oplus i_L}&& A' \oplus L'\ar[d]^{i_{A'} \oplus i_{L'}} \\ U(L) \ar[rr]_{U(\varphi, \tilde{\varphi})} && U(L').}\]
\end{lem}
\begin{proof} This follows immediately from Lemma \ref{ULUP}.
\end{proof}
Thus $U(-)$ is a functor from $\sL\sR_R$ to the category of positively filtered associative $R$-algebras. We denote the group of filtration-preserving $R$-algebra automorphisms of $U(L)$ by $\Aut U(L)$. 

\begin{rmk} Recall that the $(R,A)$-algebra $L$ is said to be \emph{smooth} if it is finitely generated and projective as an $A$-module. It can be shown that if $L$ is smooth, then the group homomorphism $\Aut(A,L) \to \Aut U(L)$ is injective, and in fact there is a semi-direct product decomposition
\[ \Aut U(L) \cong \Der_A(L,A) \rtimes \Aut(A,L).\]
\end{rmk}

\begin{cor}\label{ExtendToUL} Every $G$-action on $L$ extends to a $G$-action on $U(L)$ by filtration-preserving $R$-algebra automorphisms, and we may form the skew-group ring 
\[U(L) \rtimes G.\]
\end{cor}


\subsection{Trivialisations of skew-group rings}
\label{SkewGroupRingSect}
Let $S$ be a ring and let $G$ be a group acting on $S$ by ring automorphisms from the left. We denote the result of the group action of $g \in G$ on $s \in S$ by $g\cdot s$. Then we may form the \emph{skew-group ring}
\[ S \rtimes G.\]
This is a free left $S$-module with basis $G$, and its multiplication is determined by 
\begin{equation}\label{SkewGpRngMult} (sg) \cdot (s'g') = (s (g \cdot s') )(gg') \qmb{for all} s,s' \in S, \quad g,g' \in G.\end{equation}
See, for example, \cite[Chapter 1]{Pass} for more details about this construction. Note that the basic relation
\[ g \hsp s \hsp g^{-1} = g \cdot s\qmb{for all} g \in G, s \in S\]
holds in the ring $S \rtimes G$. 

\begin{defn}\label{Triv}  Let $S \rtimes G$ be a skew-group ring. A \emph{trivialisation} is a group homomorphism $\beta : G \to S^\times$ such that for all $g \in G$, the conjugation action of $\beta(g) \in S^\times$ on $S$ coincides with the action of $g \in G$ on $S$.
\end{defn}

\begin{lem}\label{Untwist} If $\beta : G \to S^\times$ is a trivialisation, then there is a ring isomorphism
\[ \tilde{\beta} : S[G] \congs S \rtimes G\]
given by $\tilde{\beta}(s) = s$ for all $s \in S$ and $\tilde{\beta}(g) = \beta(g)^{-1}g$ for all $g \in G$.
\end{lem}
\begin{proof} Let $g, h\in G$. Since $\beta$ is a group homomorphism,
\[\tilde{\beta}(g)\tilde{\beta}(h) = \left(\beta(g)^{-1}g\right)\beta(h)^{-1}h= \beta(h)^{-1} \left(\beta(g)^{-1} g\right) h = \tilde{\beta}(gh)\]
because $\beta(g)^{-1}g \in S \rtimes G$ commutes with all elements of $S$ by assumption. Thus $\tilde{\beta}$ is a well-defined ring homomorphism. Because $\tilde{\beta}$ is left $S$-linear by definition, and because it sends the left $S$-module basis $G$ for $S[G]$ to an $S$-module basis for $S \rtimes G$, it is a bijection.
\end{proof}

In general, skew-group rings do not admit trivialisations: a skew-group ring may well be simple as a ring --- see, for example, \cite[Proposition 8.12]{MCR} --- whereas the group ring $S[G]$ always has the augmentation ideal which is non-trivial whenever $G$ has more than one element. However it may happen that the sub-skew-group rings $S \rtimes N$ for sufficiently small normal subgroups $N$ of $G$ \emph{do} admit trivialisations. 

\begin{defn}\label{DefnTriv} Let $N$ be a normal subgroup of $G$, and suppose that $\beta : N \to S^\times$ is a trivialisation of the sub-skew-group ring $S \rtimes N$.
\be \item We define \[S \rtimes_N^\beta G := \frac{ S \rtimes G }{ (S \rtimes G) \cdot (\tilde{\beta}(N) - 1)}.\]
\item We say that $\beta$ is \emph{$G$-equivariant} if 
\[\beta({}^gn) = g \cdot \beta(n) \qmb{for all} g \in G \qmb{and} n \in N.\]
\ee Here ${}^gn := gng^{-1}$ denotes the conjugation action of $G$ on $N$.\end{defn}
Note that $S \rtimes_N^\beta G$ is \emph{a priori} only a left $S \rtimes G$-module. 

\begin{lem}\label{RingSGN} Suppose that $N$ is a normal subgroup of $G$, and that $\beta : N \to S^\times$ is a $G$-equivariant trivialisation. Then 
\be \item $S \rtimes_N^\beta G$ is an associative ring.
\item $S \rtimes_N^\beta G$ is isomorphic to a crossed product $S \ast (G/N)$.
\ee \end{lem}
\begin{proof} Let $\overline{G} := G/N$. Then $S \rtimes G$ is isomorphic to some crossed product of $S \rtimes N$ with $\overline{G}$ by \cite[Lemma 1.3]{Pass}. It follows from Lemma \ref{Untwist} that the image $I := (S \rtimes N) \cdot (\tilde{\beta}(N) - 1)$ of the augmentation ideal of $S[N]$ in $S \rtimes N$ under $\tilde{\beta}$ is a two-sided ideal in $S\rtimes N$. This ideal is stable under conjugation by $G$ inside $S \rtimes G$ because the trivialisation $\beta$ is $G$-equivariant:
\[g \tilde{\beta}(n) g^{-1} = \left(g \cdot \beta(n)\right)^{-1} {}^g n = \beta({}^gn)^{-1} \hsp {}^gn = \tilde{\beta}({}^gn) \qmb{for all} g \in G, n\in N.\]
Now $S \rtimes_N^\beta G$ is the factor ring of $S \rtimes G = (S \rtimes N) \ast \overline{G}$ by the two-sided ideal $I \ast \overline{G}$, and $(S \rtimes G) / (I \ast \overline{G}) \cong ((S \rtimes N) / I) \ast \overline{G}$ by \cite[Lemma 1.4(ii)]{Pass}. Finally, $(S \rtimes N) / I \cong S[N] / S[N](N-1) \cong S$ by Lemma \ref{Untwist}.
\end{proof}

It follows that the natural map $S \to S \rtimes_N^\beta G$ is always injective, and we will always identify $S$ with its image in $S \rtimes_N^\beta G$. Letting
\begin{equation}\label{DefnGamma} \gamma : G \longrightarrow (S \rtimes_N^\beta G)^\times\end{equation}
be the group homomorphism which sends $g \in G$ to the image of $g \in S \rtimes G$ in $S\rtimes_N^\beta G$, we see that every element of $S \rtimes_N^\beta G$ can be written as a sum of elements of the form $s \gamma(g)$ for some $s \in S$ and $g \in G$. Note that we also have the following relations in $S \rtimes_N^\beta G$:
\begin{equation}\label{RelSNG} \gamma(g) \hsp s\hsp \gamma(g)^{-1} = g \cdot s \qmb{for all} s \in S, g \in G,\end{equation}
\begin{equation}\label{EqualityInSNG} s \hsp \gamma(g) = s' \hsp \gamma(g') \quad\Leftrightarrow \quad g'g^{-1} \in N \qmb{and} s = s' \beta(g'g^{-1}).\end{equation}

\begin{lem}\label{GammaTriv} Let $N \leq H$ be normal subgroups of $G$, and let $\beta : N \to S^\times$ be a $G$-equivariant trivialisation.
\be 
\item There is a $G$-action on $S \rtimes_N^\beta H$ which satisfies
\[g \cdot (s \hsp \gamma(h))  = (g \cdot s) \gamma( {}^g h) \qmb{for all} g \in G, s \in S, h \in H.\]
\item The map $\gamma : H \to (S \rtimes_N^\beta H)^\times$ is a $G$-equivariant trivialisation. \ee
\end{lem}
\begin{proof} (a) Note that $\beta$ is also an $H$-equivariant trivialisation, so we may form the crossed product $S \rtimes_N^\beta H$. The skew-group ring $S \rtimes G$ contains $S \rtimes H$ as a subring, and conjugation by $G$ inside $S \rtimes G$ preserves $S \rtimes H$ setwise because $H$ is a normal subgroup of $G$. Because $\beta$ is $G$-equivariant, the ideal of $S \rtimes H$ generated by $\tilde{\beta}(N)-1$ is stable under this conjugation action of $G$, so it descends to a $G$-action on $S \rtimes_N^\beta H$ by ring automorphisms. This action is given explicitly in the statement of the Lemma.

(b) Using $(\ref{RelSNG})$ together with the fact that $\gamma$ is a group homomorphism, we see that for any $s \in S$ and $h, h' \in H$ we have
\[\gamma(h) \hsp s \gamma(h') \hsp  \gamma(h)^{-1} = \gamma(h) s \gamma(h)^{-1} \gamma({}^hh') = (h \cdot s) \gamma({}^hh') = h \cdot (s \gamma(h')).\]
Hence $\gamma$ is a trivialisation of the $H$-action on $S \rtimes_N^\beta H$. It is $G$-equivariant because $\gamma( {}^gh) = \gamma(g) \gamma(h) \gamma(g)^{-1} = g \cdot \gamma(h)$, by $(\ref{RelSNG})$.
\end{proof}

Using Lemma \ref{GammaTriv} and Definition \ref{DefnTriv}, we form the iterated crossed product
\[ (S \rtimes^\beta_N H) \rtimes^\gamma_H G\]
whenever $N \leq H$ are normal subgroups of $G$, $\beta : N \to S^\times$ is a $G$-equivariant trivialisation, and $\gamma : H \to (S \rtimes_N^\beta H)^\times$ is given by $(\ref{DefnGamma})$.

\begin{prop}\label{SHGassoc} Let $N \leq H$ be normal subgroups of $G$ and let $\beta : N \to S^\times$ be a $G$-equivariant trivialisation. Then there is a natural ring isomorphism
\[ \Phi : (S \rtimes_N^\beta H) \rtimes_H^\gamma G \congs S \rtimes_N^\beta G.\]
\end{prop}
\begin{proof} It follows from Lemma \ref{RingSGN}(b) that there is a natural inclusion $S \rtimes_N^\beta H \hookrightarrow S \rtimes_N^\beta G$, and we will identify $S \rtimes_N^\beta H$ with its image in $S \rtimes_N^\beta G$ under this inclusion. Using $(\ref{RelSNG})$ together with the universal property of the skew-group ring $(S\rtimes_N^\beta H)\rtimes G$, we obtain a ring homomorphism
\[\begin{array}{lcccl} \varphi : & (S \rtimes_N^\beta H) \rtimes G & \longrightarrow & S \rtimes_N^\beta G & \\ & u \hsp g & \mapsto & u \gamma(g) &\qmb{for all} u \in S \rtimes_N^\beta H, \quad g \in G. \end{array}\]
For any $h \in H$, $\varphi$ sends the element $\gamma(h)h^{-1}$ of $( S \rtimes_N^\beta H) \rtimes G$ to $1$, so $\varphi$ descends to a ring homomorphism 
\[\begin{array}{lcccl} \Phi : & (S \rtimes_N^\beta H) \rtimes^{\gamma}_H G & \longrightarrow & S \rtimes_N^\beta G & \\
& s \hsp \gamma(h) \hsp \theta( g ) & \mapsto & s \hsp \gamma(hg) & \qmb{for all} s \in S, h\in H, g \in G, \end{array}\]
where $\theta : G \to \left((S \rtimes^\beta_NH) \rtimes^{\gamma}_HG\right)^\times$ is the group homomorphism which sends $g \in G$ to the image of $g \in (S \rtimes^\beta_N H) \rtimes G$ in $(S \rtimes^\beta_NH) \rtimes^{\gamma}_HG$.  

In the other direction, there is a ring homomorphism 
\[\begin{array}{lcccl} \psi : & S \rtimes G& \longrightarrow & (S \rtimes^\beta_NH) \rtimes^{\gamma}_HG& \\

& s \hsp g &\mapsto & s \hsp \theta(g) & \qmb{for all} s \in S, g \in G.\end{array}\]
For any $n \in N$, $\psi$ sends the element  $\beta(n) n^{-1} \in S \rtimes G$ to $1$, so it descends to a ring homomorphism
\[\begin{array}{lcccl} \Psi : & S \rtimes_N^\beta G & \longrightarrow & (S \rtimes^\beta_NH) \rtimes^{\gamma}_HG & \\

& s \hsp \gamma(g) & \mapsto &  s \hsp \theta(g) & \qmb{for all} s\in S, g\in G.\end{array}\]
It is now straightforward to verify that $\Phi$ and $\Psi$ are mutually inverse.
\end{proof}

Finally, we record a useful result on the functoriality of our construction $S \rtimes^{\beta}_N G$.

\begin{lem}\label{CPfunc} Let $f : S \to S'$ be a ring homomorphism, and let $\tau : G \to G'$ be a group homomorphism. Suppose that $G$ acts on $S$ and $G'$ acts on $S'$. Let $N, N'$ be normal subgroups of $G, G'$ respectively, and let $\beta : N \to S^\times$ and $\beta' : N' \to S'^\times$ be equivariant trivialisations. Suppose that:
\begin{itemize}
\item  $\tau(N) \subseteq N'$, 
\item $ f( g \cdot s ) = \tau(g) \cdot f(s)$ for all $g \in G, s \in S$, and
\item $f^\times \circ \beta = \beta' \circ \tau_{|N}$.
\end{itemize}
Then $f$ and $\tau$ extend to a ring homomorphism
\[  f \rtimes \tau \hsp:\hsp S \rtimes_N^\beta G \longrightarrow S' \rtimes_{N'}^{\beta'} G'\]
which is an isomorphism whenever $f$ and $\tau$ are bijective, and $\tau(N) = N'$.
\end{lem} 
\subsection{Equivariant sheaves on \ts{G}-topological spaces}
Let $X$ be a set equipped with a Grothendieck topology in the sense of \cite[Definition 9.1.1/1]{BGR}. Note that we do not assume at the outset that there is a final object in the category of admissible open subsets of $X$, as $X$ is not itself required to be admissible open in a $G$-topology. 

Let $\Homeo(X)$ be the group of continuous bijections from $X$ to itself. We say that a group $G$ \emph{acts on $X$} if there is given a group homomorphism $\rho : G \to \Homeo(X)$. If this action is understood, we write $gU$ to denote the image of an admissible open subset $U$ of $X$ under the action of $g \in G$. For every $g \in G$, there is an auto-equivalence $\rho(g)_\ast$ of the category of sheaves on $X$, with inverse $\rho(g)^\ast = \rho(g^{-1})_\ast$. To simplify the notation, we will simply denote these auto-equivalences by $g_\ast$ and $g^\ast$, respectively. Thus
\[ (g_\ast \cF)(U) = \cF(g^{-1}U)\qmb{and} (g^\ast \cF)(U) = \cF(gU)\]
for all admissible open subsets $U$ of $X$ and all $g\in G$. 

Let $R$ be a commutative base ring. We review some definitions from \cite[\S 5.1]{Tohoku}.

\begin{defn}\label{DefnEquivSheaf} Let $G$ act on $X$, and let $\cF$ be a presheaf of $R$-modules on $X$.
\be \item  An \emph{$R$-linear equivariant structure} on $\cF$ is a set $\{g^{\cF}: g \in G\}$, where 
\[ g^{\cF} : \cF \to g^\ast \cF\]
is a morphism of presheaves of $R$-modules for each $g \in G$,  such that 
\begin{equation}\label{Cocycl} (gh)^{\cF} = h^\ast(g^\cF) \circ h^\cF \qmb{for all} g,h \in G.\end{equation}
\item An \emph{$R$-linear $G$-equivariant presheaf} is a pair $(\cF, \{g^{\cF}\}_{g \in G})$, where $\cF$ is a presheaf of $R$-modules on $X$, and  $\{g^{\cF}\}_{g \in G}$ is an $R$-linear equivariant structure on $\cF$. 
\item A \emph{morphism} of $R$-linear $G$-equivariant presheaves 
\[\varphi : (\cF, \{g^{\cF}\}) \to (\cF', \{g^{\cF'}\})\] 
is a morphism of presheaves of $R$-modules $\varphi : \cF \to \cF'$ such that 
\[ g^\ast(\varphi) \circ g^{\cF} = g^{\cF'} \circ \varphi \qmb{for all} g \in G.\]
\ee\end{defn}
We will frequently use this abuse of notation, and simply write $\varphi(x)$ to mean $\varphi(U)(x)$ if $x$ is a section of $\cF$ over the admissible open subset $U$ of $X$. Note that with this abuse of notation, the cocycle condition $(\ref{Cocycl})$ becomes simply
\begin{equation}\label{EasyCocyc} g^{\cF}(h^{\cF}(x)) = (gh)^{\cF}(x) \qmb{for all} x \in \cF, g,h \in G.\end{equation}
When the base ring $R$ and the $R$-linear equivariant structure on a sheaf $\cF$ of $R$-modules is understood, we will simply say that $\cF$ is a \emph{$G$-equivariant sheaf}, and omit the equivariant structure from the notation. 

\begin{defn} Let $G$ act on $X$, and let $\cA$ be a sheaf of $R$-algebras on $X$. We say that $\cA$ is a \emph{$G$-equivariant sheaf of $R$-algebras} if there is given an $R$-linear $G$-equivariant structure $\{g^{\cA} : g \in G\}$ such that each $g^{\cA} : \cA \to g^\ast \cA$ is a morphism of sheaves of $R$-algebras.
\end{defn}

\begin{rmk}\label{sAEquiv}\hspace{2em} 
\be \item If $U$ is a $G$-stable admissible open subset of $X$, then there is a natural $G$-action on $\cA(U)$ by $R$-algebra automorphisms, given by 
\[g \cdot a = g^{\cA}(a) \qmb{for all} g \in G, a \in \cA(U).\]
\item If $V \subset U$ are $G$-stable, then the restriction map $\cA(U) \to \cA(V)$ is $G$-equivariant.
\ee\end{rmk}

\begin{defn}\label{GAmodDefn} Let $\cA$ be a $G$-equivariant sheaf of $R$-algebras on $X$. 
\be \item A \emph{$G$-equivariant sheaf of $\cA$-modules} on $X$, or a \emph{$G$-$\cA$-module}, is an $R$-linear $G$-equivariant sheaf $\cM$ on $X$, such that $\cM$ is a sheaf of left $\cA$-modules and $g^{\cM}(a \cdot m) = g^{\cA}(a) \cdot g^{\cM}(m)$ for all $g \in G$, $a \in \cA$ and $m \in \cM$. 
\item A \emph{morphism} of $G$-$\cA$-modules is a morphism of sheaves of $\cA$-modules, which is simultaneously a morphism of $R$-linear $G$-equivariant sheaves.
\item We denote the category of $G$-$\cA$-modules by $\Emod{G}{\cA}$. \ee\end{defn}

By Remark \ref{sAEquiv}(a), we have at our disposal the skew-group ring 
\[ \cA(X) \rtimes G.\]
We have the following important fact.
\begin{prop}\label{EqGlSec} If $X$ is an admissible open in the $G$-topology, then $\Gamma(X,-)$ is a functor from $G$-$\cA$-modules to $\cA(X) \rtimes G$-modules.
\end{prop}
\begin{proof} Let $\cM$ be a $G$-$\cA$-module on $X$, and define
\[ ag \bullet m = a \cdot g^\cM(m) \qmb{for all} a \in \cA(X), g\in G \qmb{and} m \in \cM(X).\]
Then $g \bullet (a \cdot m) = g^{\cM}(a \cdot m) = g^{\cA}(a) \cdot g^{\cM}(m) = (g^{\cA} (a) g) \bullet m$ by Definition \ref{GAmodDefn}(a), and similarly we see that $(gh) \bullet m = g \bullet (h \bullet m)$ for all $g,h \in G$ and $m \in \cM(X)$. In this way, $\cM(X)$  naturally becomes a $\cA(X) \rtimes G$-module via $\bullet$, and it is straightforward to verify that if $\varphi : \cM \to \cN$ is a morphism of $G$-$\cA$-modules, then $\varphi(X) : \cM(X) \to \cN(X)$ is $\cA(X) \rtimes G$-linear.
\end{proof}

\section{Equivariant differential operators on rigid analytic spaces}\label{EqDModSect}
\subsection{Automorphisms of admissible formal schemes and rigid spaces} \label{CtsGactions}

We begin by recalling some notations and definitions from \cite[\S 1]{BL1}.

\begin{notn} $\cR$ will denote a valuation ring of Krull dimension $1$, complete and separated with respect to the $(\pi)$-adic topology, where $\pi$ is a fixed non-zero element of the maximal ideal of $\cR$. We will always denote the field of fractions of $\cR$ by $K$.
\end{notn}

\begin{defn}\label{AdmRAlg}\be \item An $\cR$-algebra $\cA$ is said to be \emph{topologically of finite presentation} if it is isomorphic to a quotient of the algebra of restricted formal power series in finitely many variables over $\cR$ by a finitely generated ideal:
\[ \cA = \cR \langle x_1,\ldots, x_n \rangle / \mathfrak{a}. \]
\item The algebra $\cA$ is said to be \emph{admissible} if it is topologically of finite presentation, and flat as an $\cR$-module. 
\item A formal $\cR$-scheme $\cX$ is said to be \emph{admissible} if it is locally isomorphic to an affine formal scheme $\Spf \cA$ for some admissible $\cR$-algebra $\cA$. 
\ee\end{defn}
\begin{defn} Let $\cX$ be an admissible formal scheme.
\be \item $\cG(\cX) := \Aut_{\cR}(\cX, \cO_{\cX})$ denotes the group of $\cR$-linear automorphisms of $\cX$.
\item For every $n\geq 0$, let $\cR_n := \cR / \pi^n \cR$ and $\cX_n := \cX \times_{\Spf \cR} \Spf \cR_n$. 
\item The \emph{$n$-th congruence subgroup} of $\cG(\cX)$ is 
\[ \cG_{\pi^n}(\cX) := \ker \left( \cG(\cX) \to \Aut_{\cR_n}(\cX_n, \cO_{\cX_n}) \right).\]
\ee\end{defn}
We keep $\pi$ in the notation because $\cG_{\pi^n}(\cX)$ depends not only on $n$ but also on the choice of $\pi$. These congruence subgroups form a descending chain 
\[ \cG(\cX) = \cG_{1}(\cX) \supset \cG_{\pi}(\cX) \supset \cG_{\pi^2}(\cX) \supset \cdots \]
of normal subgroups of $\cG(\cX)$ whose intersection is trivial. Note that if $\cX$ is affine, then it follows from the discussion following the proof of \cite[Chapitre I, Proposition 10.2.2]{EGAInew} that for an automorphism $\varphi \in \cG(\cX)$,
\[ \varphi \in\cG_{\pi^n}(\cX) \qmb{ if and only if } (\varphi^\sharp(\cX) - 1)(\cO(\cX))\subseteq \pi^n \cO(\cX).\] 

\begin{lem}\label{AutSheaf} $\cG_{\pi^n}$ is a sheaf of groups on $\cX$ for all $n \geq 1$.
\end{lem}
\begin{proof} By replacing $\pi$ by $\pi^n$ if necessary, we may assume that $n = 1$. Note that if $(\varphi, \varphi^\sharp) \in \cG_\pi(\cX)$ then $\varphi = 1_{|\cX|}$ because $|\cX| = |\cX_1|$. Thus $\cG_\pi(\cX)$ stabilises every open formal subscheme $\cU$ of $\cX$ set-theoretically, giving a natural group homomorphism $\cG_\pi(\cX) \to \cG_\pi(\cU)$. Similarly we have natural restriction maps $\cG_\pi(\cU) \to \cG_\pi(\cV)$ whenever $\cV \subseteq \cU$ are open formal subschemes of $\cX$. Thus $\cG_\pi(\cX)$ is a presheaf on $\cX$, and the verification of the sheaf axioms is straightforward.
\end{proof}

Now let $\bX$ be a quasi-compact and quasi-separated (qcqs) rigid analytic variety over $K$. We will view $(\bX,\cO_\bX)$ as a $G$-ringed topological space over $K$, and study its group $\Aut_K(\bX, \cO_\bX)$ of $K$-linear automorphisms.
 
 By Raynaud's Theorem \cite[Theorem 4.1]{BL1}, we can find a \emph{formal model} for $\bX$: this is a quasi-compact admissible formal scheme $\cX$ such that $\bX = \cX_{\rig}$ is the \emph{generic fibre} of $\cX$. By statement (b) of the proof of \cite[Theorem 4.1]{BL1}, the generic fibre functor is faithful. It therefore induces an injection
\[ \rig : \cG(\cX) \hookrightarrow \Aut_K(\bX, \cO_{\bX})\]
whose image we denote by $\cG(\cX)_{\rig}$. Our next goal is to establish the following
\begin{thm}\label{AutRigidTop} Let $\bX$ be a qcqs rigid analytic variety over $K$.
\be \item For any formal model $\cX$ of $\bX$, the congruence subgroups 
\[\{ \cG_{\pi^r}(\cX)_{\rig} : r \geq 0\}\]
 form a filter base for a Hausdorff topology $\cT_{\cX}$ on $\Aut_K(\bX,\cO_\bX)$, which is compatible with the group structure on $\Aut_K(\bX, \cO_\bX)$.
\item The topology $\cT_{\cX}$ does not depend on the choice of $\cX$.
\ee
\end{thm}

Let $\cX$ be a quasi-compact admissible formal scheme, let $\cI$ be a coherent open ideal of $\cO_{\cX}$ and recall the \emph{admissible formal blow-up} $\theta: \cY \to \cX$ of $\cI$ on $\cX$ from \cite[\S 2]{BL1}. The group $\cG(\cX)$ acts on the set of coherent open ideals of $\cO_{\cX}$ by pullback, and we denote the stabiliser of $\cI$ in $\cG(\cX)$ under this action by $\Stab_{\cG(\cX)}(\cI)$. 

For every $\varphi \in \Stab_{\cG(\cX)}(\cI)$, the functoriality of admissible formal blow-ups induces a morphism of formal $\cR$-schemes $\eta_{\cI}(\varphi) : \cY \to \cY$ 
such that the diagram
\[\xymatrix{ \cY \ar_{\theta}[d] \ar^{\eta_{\cI}(\varphi)}[r] & \cY \ar^{\theta}[d] \\ \cX \ar_\varphi[r] & \cX.}\]
is commutative. This defines a natural group homomorphism
\[ \eta_{\cI} : \Stab_{\cG(\cX)}(\cI) \to \cG(\cY)\]
which is injective because the $\rig$ functor is faithful, and because 
\[\eta_{\cI}(\varphi)_{\rig} = \theta_{\rig}^{-1} \circ \varphi_{\rig} \circ \theta_{\rig}\]
by construction. It turns out that $\Stab_{\cG(\cX)}(\cI)$ is an open subgroup of $\cG(\cX)$, and that the map $\eta_{\cI}$ is continuous. More precisely, we have the following

\begin{lem}\label{BlowUpAut} Let $\cI$ be a coherent open ideal of $\cO_{\cX}$ and suppose that $\pi^a \in \cI(\cX)$ for some $a \geq 0$. Then
\be \item $\cG_{\pi^a}(\cX) \subseteq \Stab_{\cG(\cX)}(\cI)$, and
\item $\eta_{\cI}( \cG_{\pi^{n+a}}(\cX)) \subseteq \cG_{\pi^n}(\cY)$ for all $n \geq 0$.
\ee\end{lem}
\begin{proof} If $a = 0$ then $\cI = \cO_{\cX}$ and $\cY = \cX$, so we may assume that $a \geq 1$. Now, because the construction of admissible formal blow-ups is local on $\cX$, using Lemma \ref{AutSheaf} we may assume that $\cX = \Spf \cA$ is affine. Write $\cA := \cO(\cX)$ and $I := \cI(\cX)$.

(a)  Let $\varphi \in \cG_{\pi^a}(\cX)$ and let $f := \varphi^\sharp(\cX)$ so that $(f - 1) \cdot \cA \subseteq \pi^a \cA$. Since $\pi^a \cA \subseteq I$ by assumption, we see that $(f-1) \cdot I \subseteq I$. Hence $f \cdot I = I$ as required.

(b) Let $\varphi \in \cG_{\pi^{n+a}}(\cX)$ so that $(f-1)\cdot \cA \subseteq \pi^{n+a}\cA$. Hence
\[ (f - 1) \cdot I \subseteq (f-1) \cdot \cA \subseteq \pi^{n+a} \cA \subseteq \pi^n I\]
which implies that the induced action of $f$ on the Rees algebra $\cA' = \bigoplus_{m \geq 0} I^m t^m$ satisfies $(f-1) \cdot \cA' \subseteq \pi^n \cA'$. Let $\Aut_{\cR,\gr}(\cA')$ denote the group of graded $\cR$-algebra automorphisms of $\cA'$ of degree zero, and let $\cA'_n := \cA' \otimes_{\cR} \cR_n$. By the functoriality of $\Proj$ and $\pi$-adic completion, there is a natural commutative square
\[\xymatrix{ \Aut_{\cR, \gr} (\cA') \ar[d]\ar[r] & \Aut_{\cR_n, \gr}(\cA'_n) \ar[d] \\ \Aut_{\cR} ( \h{\Proj(\cA')} ) \ar[r] & \Aut_{\cR}( \h{\Proj(\cA'_n)} ) }\]
where $\h{\Proj(\cA')}$ denotes the $\pi$-adic completion of $\Proj(\cA')$. Now $\cY = \h{\Proj(\cA')}$ by \cite[Proposition 2.1(a)]{BL1} and $\cY_n = \h{\Proj(\cA'_n)}$, and it follows that $\eta_{\cI}(\varphi) \in \cG_{\pi^n}(\cY)$.
\end{proof}

Thus, automorphisms of $\cX$ that are sufficiently close to the identity automorphism lift to automorphisms of the blow-up $\cY$. Conversely, we will now see that automorphisms of $\cY$ that are sufficiently close to the identity automorphism descend to automorphisms of $\cX$. 

\begin{lem}\label{BlowDownAut} Let $\tau : \cY \to \cX$ be an admissible formal blow-up, and suppose $b\geq 0$ is such that $\pi^b \tau_\ast \cO_{\cY} \subseteq \tau^\sharp(\cO_{\cX})$. Then for all $m \geq 1$ and all $\varphi \in \cG_{\pi^{m + b}}(\cY)$ there exists a unique $\zeta(\varphi) \in \cG_{\pi^m}(\cX)$ such that the following diagram commutes:
\[\xymatrix{ \cY \ar[r]^{\varphi}\ar[d]_{\tau} & \cY \ar[d]^{\tau} \\ \cX \ar[r]_{\zeta(\varphi)} & \cX .}\]
\end{lem}
\begin{proof}
Suppose first that $\cX$ is affine, and let $\cA := \cO(\cX)$ and $\cA' := \cO(\cY)$. Because $\tau_{\rig} : \cY_{\rig} \to \cX_{\rig}$ is an isomorphism, Tate's Acyclicity Theorem \cite[Theorem 8.2.1/1]{BGR} implies that the map $\tau^\sharp(\cX) : \cA \to \cA'$ is injective. We may therefore identify $\cA$ with its image in $\cA'$ under $\tau^\sharp(\cX)$, so that
\[ \pi^b \cA' \subseteq \cA \subseteq \cA'.\]
Now if $\varphi \in \cG_{\pi^{m+b}}(\cY)$ then $f := \varphi^\sharp(\cY)$ satisfies $(f-1)(\cA') \subseteq \pi^{m+b} \cA'$. Therefore
\[ (f-1)(\cA) \subseteq \pi^{m+b} \cA' \subseteq \pi^m \cA\]
which in particular shows that $f$ stabilises $\cA$ inside $\cA'$. Let 
\[\zeta(\varphi) := \Spf(f_{|\cA}): \cX \to \cX\] 
be the endomorphism of the formal $\cR$-scheme $\cX$ induced by $f_{|\cA}$. It is an isomorphism with inverse $\zeta(\varphi^{-1})$, and in fact $\zeta(\varphi) \in \cG_{\pi^m}(\cX)$. 

Note that the homeomorphisms of $|\cY|$ and $|\cX|$ defined by $\varphi$ and $\zeta(\varphi)$ respectively are trivial because $m  + b \geq m \geq 1$ by assumption. Since $\tau^\sharp(\cX) \circ f_{|\cA} = f \circ \tau^\sharp(\cX)$ by construction, it follows that $\tau \circ \varphi = \zeta(\varphi) \circ \tau$ by comparing the restrictions of these morphisms to arbitrary affine open formal subschemes of $\cY$.

Returning to the general case, choose an affine cover $\{\cX_i\}$ of $\cX$, define $\cY_i := \cX_i \times_{\cX} \cY = \tau^{-1}(\cX_i)$ and let $\varphi \in \cG_{\pi^{m+b}}(\cY)$ as above. Since $m + b\geq 1$, $\varphi_i := \varphi_{|\cY_i} \in \cG_{\pi^{m+b}}(\cY_i)$ by Lemma \ref{AutSheaf}, so by the above we obtain $\zeta(\varphi_i) \in \cG_{\pi^m}(\cX_i)$ such that $\tau_{|\cY_i} \circ \varphi_i = \zeta(\varphi_i)\circ\tau_{|\cY_i}$ for all $i$. Let $\cX_{ij} := \cX_i \cap \cX_j$ and $\cY_{ij} = \tau^{-1}(\cX_{ij})$; then
\[ \zeta(\varphi_i)_{|\cX_{ij}} \circ \tau_{|\cY_{ij}} = \tau_{|\cY_{ij}} \circ \varphi_{i|\cY_{ij}} = (\tau \circ \varphi)_{|\cY_{ij}} =\tau_{|\cY_{ij}} \circ \varphi_{j|\cY_{ij}} = \zeta(\varphi_j)_{|\cX_{ij}} \circ \tau_{|\cY_{ij}}\]
for all $i,j$. Because $\tau_{\rig}$ is an isomorphism, and because the $\rig$ functor is faithful, $\zeta(\varphi_i)_{|\cX_{ij}}=\zeta(\varphi_j)_{|\cX_{ij}}$ for all $i,j$. Therefore these local automorphisms $\zeta(\varphi_i)$ patch to some global automorphism $\zeta(\varphi) \in \cG_{\pi^m}(\cX)$ by Lemma \ref{AutSheaf}, which satisfies $\tau \circ \varphi = \zeta(\varphi) \circ \tau$ by construction.
\end{proof}
The faithfulness of the $\rig$ functor therefore induces a group homomorphism 
\[ \zeta : \cG_{\pi^{m+b}}(\cY) \to \cG_{\pi^m}(\cX) \qmb{for all} m\geq 0\]
such that $\tau\circ\varphi = \zeta(\varphi) \circ \tau$ for all $\varphi \in \cG_{\pi^{m+b}}(\cY)$. 

\begin{proof}[Proof of Theorem \ref{AutRigidTop}] (a) The set of congruence subgroups $\{\cG_{\pi^r}(\cX)_{\rig} : r \geq 0\}$ is a filter base for some topology $\cT_{\cX}$ on $\Aut_K(\bX, \cO_\bX)$ in the sense of \cite[Chapter I, \S 6.3, Definition 3]{BourGenTop}. By the discussion in \cite[Chapter III, \S 1.2]{BourGenTop}, to show that this topology is compatible with the group structure on $\Aut_K(\bX, \cO_\bX)$, it is necessary and sufficient to show that for all $\varphi \in \Aut_K(\bX, \cO_\bX)$ and  any $r \geq 0$ there is $s \geq 0$ such that
\[ \cG_{\pi^s}(\cX)_{\rig} \subseteq \varphi \hspace{.1cm} \cG_{\pi^r}(\cX)_{\rig}\hspace{.1cm}\varphi^{-1} . \]
Now, it follows from the proof of \cite[Theorem 4.1]{BL1} that we can find a diagram
\[ \cX \stackrel{\tau}{\leftarrow} \cY \stackrel{\theta}{\rightarrow} \cZ \stackrel{\sigma}{\rightarrow} \cX\]
where $\sigma$ is an isomorphism, $\theta, \tau$ are admissible formal blow-ups, and
\[ \varphi = \sigma_{\rig} \circ \theta_{\rig} \circ \tau_{\rig}^{-1}. \]
It follows from \cite[Proposition 3.5.1(i), (ii)]{Abbes} that the cokernel of the morphism of sheaves $\tau^\sharp : \cO_{\cX} \to \tau_\ast \cO_{\cY}$ is a $\pi$-torsion coherent $\cO_{\cX}$-module. Since $\cX$ is quasi-compact by assumption,  $\coker \tau^\sharp$ is \emph{bounded} $\pi$-torsion, so there exists some $b \geq 0$ such that 
\[ \pi^b \tau_\ast \cO_{\cY} \subseteq \tau^\sharp(\cO_{\cX}) \subseteq \tau_\ast \cO_{\cY}.\]
Thus the hypothesis of Lemma \ref{BlowDownAut} is satisfied, and we obtain the inclusion
\[\tau_{\rig} \cG_{\pi^{r+b}}(\cY)_{\rig} \tau_{\rig}^{-1} \subseteq \cG_{\pi^r}(\cX)_{\rig}\]
by applying Lemma \ref{BlowDownAut}. Next, let $\cI$ be the coherent open ideal of $\cO_{\cZ}$ blown up by $\theta : \cY \to \cZ$ and choose $a \geq 0$ such that $\pi^a \in \cI(\cZ)$.  Then
\[\theta_{\rig}^{-1} \cG_{\pi^{r+a+b}}(\cZ)_{\rig} \theta_{\rig} \subseteq \cG_{\pi^{r+b}}(\cY)_{\rig}  \]
by Lemma \ref{BlowUpAut}(b). Putting these inclusions together gives
\[ \begin{array}{llll} \cG_{\pi^{r+a+b}}(\cX)_{\rig} &=&\sigma_{\rig} \cG_{\pi^{r+a+b}}(\cZ)_{\rig} \sigma^{-1}_{\rig} & \subseteq \\ 
& \subseteq & \sigma_{\rig} \theta_{\rig} \cG_{\pi^{r+b}}(\cY)_{\rig} \theta_{\rig}^{-1} \sigma^{-1}_{\rig} & \subseteq \\
& \subseteq & \sigma_{\rig} \theta_{\rig} \tau_{\rig}^{-1} \cG_{\pi^r}(\cX)_{\rig} \tau_{\rig} \theta_{\rig}^{-1} \sigma_{\rig}^{-1} & = \\
&=& \varphi \hspace{.1cm} \cG_{\pi^r}(\cX)_{\rig}\hspace{.1cm}\varphi^{-1}. & \end{array} \]
(b) Let $\cX$ and $\cX'$ be two formal models of $\bX$. Again by the proof of \cite[Theorem 4.1]{BL1}, we can find a diagram
\[ \cX \stackrel{\tau}{\leftarrow} \cY \stackrel{\theta}{\rightarrow} \cZ \stackrel{\sigma}{\rightarrow} \cX'\]
where $\sigma$ is an isomorphism, $\theta, \tau$ are admissible formal blow-ups, and 
\[ \varphi = \sigma_{\rig} \circ \theta_{\rig} \circ \tau_{\rig}^{-1} : \cX_{\rig} \to \cX'_{\rig}\]
is the identity map $\bX \stackrel{1_\bX}{\longrightarrow} \bX$. The same argument as in the proof of part (a) implies that there are integers $a,b\geq 0$ such that for all $r \geq 0$
\[ \cG_{\pi^{r+a+b}}(\cX')_{\rig} \subseteq \varphi \hspace{.1cm} \cG_{\pi^r}(\cX)_{\rig}\hspace{.1cm}\varphi^{-1}  = \cG_{\pi^r}(\cX)_{\rig}.\]
Thus every $\cT_{\cX}$-open subset of $\Aut_K(\bX, \cO_\bX)$ is also $\cT_{\cX'}$-open, and by symmetry every $\cT_{\cX'}$-open subset is also $\cT_{\cX}$-open. 
\end{proof}

\begin{defn}\label{CtsAct} Let $G$ be a topological group, and let $\bX$ be a rigid analytic variety (not necessarily qcqs). We say that \emph{$G$ acts continuously on $\bX$} if there is given a group homomorphism $\rho : G \to \Aut_K(\bX, \cO_\bX)$ such that for every qcqs admissible open subset $\bU$ of $\bX$,
\be \item the stabiliser $G_\bU$ of $\bU$ in $G$ is open in $G$,
\item the induced group homomorphism $\rho_\bU : G_\bU \to \Aut_K(\bU, \cO_\bU)$ is continuous with respect to the subspace topology on $G_\bU$ and the topology on $\Aut_K(\bU,\cO_\bU)$ constructed in Theorem \ref{AutRigidTop}.
\ee\end{defn}

This notion enjoys the following closure properties.

\begin{lem}\label{CtsActLem} Let $G$ be a topological group acting continuously on the rigid analytic variety $\bX$. 
\be \item The restriction of the $G$-action to each subgroup of $G$ is continuous.
\item $G$ acts continuously on every $G$-stable admissible open subspace of $\bX$.
\item Suppose that $\tilde{G}$ is another topological group containing $G$ as an open subgroup. Then any extension of the $G$-action on $\bX$ to a $\tilde{G}$-action on $\bX$ is continuous.
\ee\end{lem}
\begin{proof} This is straightforward. \end{proof}

In the case where the rigid analytic variety $\bX$ is already quasi-compact, a continuous $G$-action on $\bX$ may be defined ``globally''. More precisely, we have the following 

\begin{prop}\label{OpenStab} Let $\bX$ be a qcqs rigid analytic variety, and let $G$ be a topological group. Suppose that $\rho : G \to \Aut_K(\bX, \cO_\bX)$ is a continuous group homomorphism. Then $G$ acts continuously on $\bX$.
\end{prop}
\begin{proof} Let $\bU$ be a quasi-compact admissible open subset of $\bX$. By \cite[Theorem 4.1 and Lemma 4.4]{BL1}, we can find a formal model $\cX$ for $\bX$ and an open formal subscheme $\cX'$ of $\cX$ such that $\cX'_{\rig} = \bU$, and we may use the $\cT_{\cX}$-topology on $\Aut_K(\bX, \cO_\bX)$ by Theorem \ref{AutRigidTop}. Now $\cG_\pi(\cX)$ stabilises $\cX'$, so $\cG_\pi(\cX)_{\rig}$ stabilises $\bU$, which implies that the stabiliser $\Stab(\bU)$ of $\bU$ in $\Aut_K(\bX,\cO_\bX)$ is open. Since $\rho : G \to \Aut_K(\bX,\cO_\bX)$ is continuous, it follows that $G_\bU = \rho^{-1}(\Stab(\bU))$ is open in $G$.  Moreover, since the restriction map $\Stab(\bU) \to \Aut_K(\bU, \cO_\bU)$ sends $\cG_{\pi^r}(\cX)_{\rig}$ to the open subset $\cG_{\pi^r}(\cX')_{\rig}$ of $\Aut_K(\bU, \cO_\bU)$ for all $r \geq 1$, this map is continuous. Hence $\rho_\bU : G_\bU \to \Aut_K(\bU, \cO_\bU)$ is also continuous.
\end{proof}

We will next exhibit a large class of examples of such continuous group actions. 
\begin{defn}\label{CongSubTop}
Let $\G$ be an $\cR$-group scheme. We equip its group of $\cR$-points $\G(\cR)$ with the topology in which the congruence subgroups
\[ \G_{\pi^n}(\cR) := \ker (\G(\cR) \to \G(\cR/\pi^n\cR)) \]
form a filter base; we call this the \emph{congruence-subgroup topology} on $\G(\cR)$. 
\end{defn}
This topology is Hausdorff because the natural map $\G(\cR) \longrightarrow \prod_{n=0}^\infty \G(\cR_n)$ is injective, as $\cR \hookrightarrow \prod_{n=0}^\infty\cR_n$ and as $\G$ is a left exact functor on $\cR$-algebras.

\begin{prop}\label{GpOfPtsActsCtsly} Let $\G$ be an $\cR$-group scheme, acting on a flat $\cR$-scheme $\X$ of finite presentation. Let $G$ be a topological group and let $\sigma : G \to \G(\cR)$ be a continuous group homomorphism.
\be \item The formal completion $\cX$ of $\X$ along its special fibre is a quasi-compact admissible formal scheme.
\item $G$ acts continuously on the rigid generic fibre $\bX := \cX_{\rig}$ of $\cX$.
\ee
\end{prop}
\begin{proof} (a) Because $\X$ is of finite presentation over $\Spec(\cR)$, $\cX$ is quasi-compact. We may hence assume that $\X = \Spec(A)$ is affine where $A$ is a finitely presented $\cR$-algebra which is flat as an $\cR$-module. Now $\cX = \Spf \h{A}$ where $\h{A}$ is the $\pi$-adic completion of $A$. This $\cR$-algebra is topologically of finite presentation by \cite[Corollaire 1.10.6]{Abbes}, and it is flat as an $\cR$-module by \cite[Corollaire 1.12.4]{Abbes} because $A/\pi^n A$ is flat as an $\cR_n$-module for all $n\geq 0$. Thus $\h{A}$ is an admissible $\cR$-algebra and $\cX$ is an admissible formal scheme.

(b) By \cite[\S I.2.6]{Jantzen}, the action $\G \times \X \to \X$ induces a morphism of $\cR$-group functors $\G \to \zAut(\X)$. Now by \cite[\S I.1.5]{Jantzen}, $\zAut(\X)(\cR) = \Aut_{\cR}(\X, \cO_{\X})$  and $\zAut(\X)(\cR_n) = \Aut_{\cR_n}(\X_n, \cO_{\X_n})$ where $\X_n := \X \times_{\Spec(\cR)} \Spec(\cR_n)$, whence a commutative diagram
\[\xymatrix{ \G(\cR) \ar[r]\ar[d] & \Aut_{\cR}(\X, \cO_{\X}) \ar[d]\\ \G(\cR_n) \ar[r] & \Aut_{\cR_n}(\X_n, \cO_{\X_n})}\]
for any $n\geq 0$. The functoriality of formal completion now induces a group homomorphism $\rho : \G(\cR) \to \Aut_{\cR}(\cX, \cO_{\cX}) = \cG(\cX)$ which sends $\G_n(\cR)$ into $\cG_{\pi^n}(\cX)$. Applying the $\rig$ functor gives a group homomorphism $\rho_{\rig} : \G(\cR) \to \Aut_K(\bX, \cO_\bX)$ which sends $\G_{\pi^n}(\cR)$ into $\cG_{\pi^n}(\cX)_{\rig}$ and is therefore continuous by Theorem \ref{AutRigidTop}. Thus, $\rho_{\rig} \circ \sigma : G \to \Aut_K(\bX,\cO_\bX)$ is a continuous group homomorphism, and it follows from part (a) that $\bX$ is qcqs. Hence $G$ acts continuously on $\bX$ by Proposition \ref{OpenStab}.
\end{proof}

\subsection{Actions of compact \ts{p}-adic Lie groups on \ts{K}-affinoid algebras}
\label{GActionSection}
Let $G$ be a compact $p$-adic Lie group. By a theorem of Lazard \cite[Corollary 8.34]{DDMS}, $G$ contains at least one open subgroup $N$ which is \emph{uniform pro-$p$}; the definition and basic properties of uniform pro-$p$ groups can be found in \cite[Chapter 4]{DDMS}. In fact, we have the more precise 
\begin{lem}\label{SmallU} Let $H$ be an open subgroup of the compact $p$-adic Lie group $G$. Then $H$ contains an open uniform pro-$p$ subgroup $N$ which is normal in $G$.
\end{lem}
\begin{proof} By \cite[Corollary 8.33]{DDMS}, $G$ contains an open subgroup $J$ which is a pro-$p$ group of finite rank. Now $H \cap J$ is an open subgroup of the profinite group $G$, so it contains an open normal subgroup $L$ of $G$ by \cite[Proposition 1.2(ii)]{DDMS}. This group $L$ is still pro-$p$ of finite rank, and therefore contains an open normal characteristic subgroup $N$ which is uniform pro-$p$, by \cite[Corollary 4.3]{DDMS}. Being a characteristic subgroup of the normal subgroup $L$, $N$ is normal in $G$.
\end{proof}
Recall \cite[\S 4.5]{DDMS} that every uniform pro-$p$ group $N$ has a $\Zp$-Lie algebra $L_N = (N, +, [,])$ functorially associated to it. The $\Zp$-module and Lie bracket structures on $L_N$ are extracted from the structure of $N$ as a uniform pro-$p$ group by the formulas
\[ \begin{array}{ccl}
 x + y &=& \lim\limits_{n \to \infty} (x^{p^n} y^{p^n})^{1/p^n} \\ 
\lambda \cdot x &=& \lim\limits_{n\to\infty} x^{\lambda_n} \\
 
 [x,y] &=& \lim\limits_{n\to\infty} (x^{-p^n} y^{-p^n} x^{p^n} y^{p^n} )^{1/p^{2n}}\end{array}\]
for any $x,y \in N$ and any choice of a sequence $(\lambda_n)_{n=0}^\infty \subseteq \Z$ converging to $\lambda \in \Zp$.

The $\Zp$-Lie algebra $L_N$ is a \emph{powerful}: $L_N$ is a finitely generated free $\Zp$-module and $[L_N, L_N] \subseteq p^\epsilon L_N$ (here $\epsilon := 1$ if $p$ is odd and $\epsilon :=2$ if $p$ is even). In fact, it follows from \cite[Theorem 9.10]{DDMS} that the functor $N \mapsto L_N$ is an equivalence between the category of uniform pro-$p$ groups and the category of powerful $\Zp$-Lie algebras.

Next, we recall the definition of the $\Qp$-Lie algebra of $G$ following \cite[\S 9.5]{DDMS}.
\begin{defn} Let $G$ be a compact $p$-adic Lie group. 
\be \item $N \leq_o^u G$ means that $N$ is an open uniform pro-$p$ subgroup of $G$.
\item $N \triangleleft_o G$ means that $N$ is an open \emph{normal} subgroup of $G$.
\item $N\triangleleft_o^u G$ means that $N$ is an open normal uniform pro-$p$ subgroup of $G$. 
\item The \emph{Lie algebra} of $G$ is defined to be $\Lie(G) := \lim\limits_{\stackrel{\longleftarrow}{N\leq_o^u G}} \Qp \otimes_{\Zp} L_N.$
\ee\end{defn}
The inverse limit is taken over the set of all open uniform pro-$p$ subgroups of $G$, which becomes a directed set under reverse inclusion. Whenever $H \leq N$ are two members of this set, then $H$ contains $N^{p^m}$ for sufficiently large $m \geq 0$, which implies that $p^m L_N \leq L_H \leq L_N$. Therefore each transition map
\[ \Qp \otimes_{\Zp} L_H \to \Qp \otimes_{\Zp} L_N\]
appearing in this inverse limit is actually an isomorphism, and $\Lie(G)$ is a finite dimensional Lie algebra over $\Qp$.

Now let $A$ be a $K$-affinoid algebra. Extending the terminology from \cite[\S 3.1]{DCapOne}, we say that an admissible $\cR$-algebra $\cA$ is an \emph{affine formal model} in $A$ if there is an isomorphism $A \cong \cA \otimes_{\cR} K$. Equivalently, the admissible formal scheme $\Spf \cA$ is an affine formal model in the affinoid variety $\Sp A$. 

\begin{lem}\label{ProdFormModels} Let $\cA$, $\cB$ be two affine formal models in the $K$-affinoid algebra $A$. Then $\cA \cdot \cB$ is another affine formal model in $A$ which is finitely generated as an $\cA$-module and as a $\cB$-module.
\end{lem}
\begin{proof} This is \cite[Lemma 3.1]{DCapOne}.\end{proof}

Let $\cG(\cA)$ denote the group of $\cR$-algebra automorphisms of $\cA$, let 
\[\cG_{\pi^n}(\cA) := \ker (\cG(\cA) \to \Aut(\cA \otimes_{\cR} \cR_n))\] 
be its $n$th congruence subgroup, and let $\Aut_K(A)$ denote the group of $K$-algebra automorphisms of $A$. There is a commutative diagram
\[ \xymatrix{ \cG(\Spf(\cA)) \ar[r]\ar[d]_{\rig} & \cG(\cA) \ar[d]^{\rig} \\ \Aut_K(\Sp A, \cO_{\Sp A}) \ar[r] & \Aut_K(A) } \]
where the horizontal arrows are isomorphisms given by 
\[(\varphi, \varphi^\sharp) \mapsto \Gamma(\Spf(\cA), (\varphi^{\sharp})^{-1}) \qmb{and} (\varphi, \varphi^\sharp) \mapsto \Gamma(\Sp A, (\varphi^{\sharp})^{-1}).\]
The rightmost arrow $\rig : \cG(\cA) \to \Aut_K(A)$ is given by $\rig(\varphi)(a \otimes \lambda) = \varphi(a) \otimes \lambda$ for all $a \in \cA$ and $\lambda \in K$. It is injective, and we will identify $\cG(\cA)$ with its image in $\Aut_K(A)$ under this map. 

For any $K$-affinoid algebra $A$, we equip $\Aut_K(A)$ with the topology constructed in Theorem \ref{AutRigidTop}: for any affine formal model $\cA$ in $A$, the congruence subgroups $\{\cG_{\pi^n}(\cA) : n\geq 0\}$ form a filter base for this topology.

\begin{lem}\label{GstableFM} Let $G$ be a compact topological group, and let $\rho : G \to \Aut_K(A)$ be a continuous group homomorphism. Then every affine formal model $\cA$ in $A$ is contained in a $G$-stable affine formal model.
\end{lem}
\begin{proof} Because the map $\rho : G \to \Aut_K(A)$ is continuous, the preimage of $\cG(\cA)$ in $G$ is open. But this preimage is just the stabiliser of $\cA$ in $G$. Because $G$ is compact, the $G$-orbit of $\cA$ in $A$ is finite, $\cA_1,\ldots,\cA_n$ say. Now $\cA_1\cdot\cdots\cdot \cA_n$ is another affine formal model in $A$ by Lemma \ref{ProdFormModels} which is evidently $G$-stable.
\end{proof}

\textbf{We will assume from now on until the end of this paper that our ground field $K$ is of mixed characteristic $(0,p)$.} 

Let $\cS$ be a $p$-adically complete and flat $\cR$-algebra. Then we may view it as the unit ball in the $K$-Banach algebra $\cS \otimes_{\cR} K$, and
\[\log : 1 + p^\epsilon \cS \to p^\epsilon \cS \qmb{and} \exp : p^\epsilon \cS \to 1 + p^\epsilon \cS\]
are well-defined mutually inverse bijections by  \cite[Corollary 6.25]{DDMS}.

\begin{prop}\label{GactsOnX}  Let $G$ be a compact $p$-adic Lie group, and let $\rho : G \to \Aut_K(A)$ be a continuous group homomorphism. Then there exists a canonical $\Qp$-Lie algebra homomorphism
\[d\rho : \Lie(G) \to \Der_K(A)\]
which is $G$-equivariant with respect to the adjoint action of $G$ on $\Lie(G)$ and the natural action of $G$ on $\Der_K(A)$ given in Example \ref{GactDerEx}. \end{prop}
\begin{proof} Choose a $G$-stable affine formal model $\cA$ in $A$ using Lemma \ref{GstableFM}, so that $\rho(G) \subseteq \cG(\cA)$.  The $\cR$-algebra $\cE := \End_{\cR}(\cA)$ is $p$-adically complete and flat, because the same is true of $\cA$. Because the characteristic of $K$ is zero by assumption, we therefore have at our disposal the bijections $\log : 1 + p^\epsilon \cE \to p^\epsilon \cE$ and $\exp : p^\epsilon \cE \to 1 + p^\epsilon \cE$. In particular, for every $\varphi \in \cG_{p^\epsilon}(\cA)$, the logarithm series $\log \varphi := - \sum_{r =1}^\infty \frac{(1-\varphi)^r}{r}$ converges inside $\cE$. Because $\varphi$ is an $\cR$-algebra automorphism, a well-known formal computation shows that in fact $\log \varphi$ is an $\cR$-linear \emph{derivation}: see, for example, \cite[proof of Theorem 4]{Praagman}. The image of $\log \varphi$ is contained in $p^\epsilon \cA$, so we see that $\log \varphi \in p^\epsilon \Der_{\cR}(\cA)$ for every $\varphi \in \cG_{p^\epsilon}(\cA)$. 

Now let $H := \rho^{-1} (\cG_{p^\epsilon}(\cA))$, an open normal subgroup of $G$. For every $N \leq_o^u G$ contained in $H$, the maps $\rho_{|N} : N \to \cG_{p^\epsilon}(\cA)$ and $\log : \cG_{p^\epsilon}(\cA) \to p^\epsilon \Der_{\cR}(\cA)$ are continuous. Thus it follows from the proof of \cite[Lemma 7.12]{DDMS} that 
\[ \log \circ \rho_{|N} : L_N \to p^\epsilon \Der_{\cR}(\cA)\]
is a $\Zp$-Lie algebra homomorphism. As the open uniform pro-$p$ subgroup $N$ shrinks, these maps assemble to produce the required $\Qp$-Lie algebra homomorphism
\[ d\rho : \Lie(G) \to \Der_K(A).\]
It is easy to check that $\log (g \varphi g^{-1}) = g (\log \varphi )g^{-1}$ for any $g \in \cG(\cA)$ and $\varphi \in \cG_{p^\epsilon}(\cA)$. Applying this with $g = \rho(b)$ and $\varphi = \rho(a)$ for any $a \in N$ and $g \in G$ shows that 
\[ \log \rho( gag^{-1} ) = \rho(g) \log \rho(a) \rho(g)^{-1}.\]
Letting $\Ad(g)$ and $\Ad(\rho(g))$ denote conjugation by $g$ in $G$ and by $\rho(g)$ in $\End_K(A)$, respectively, we see that the diagram
\[ \xymatrix{ N \ar[rr]^{\log \circ \rho}\ar[d]_{\Ad(g)} && p^\epsilon \Der_{\cR}(\cA) \ar[d]^{\Ad(\rho(g))} \ar[r] & \Der_K(A) \ar[d]^{\Ad(\rho(g))}  \\ g N g^{-1} \ar[rr]_{\log \circ \rho} && p^\epsilon \Der_{\cR}(\cA) \ar[r] & \Der_K(A) }\]
is commutative, and the $G$-equivariance of $d\rho$ follows.

Finally, if $\cA$ and $\cA'$ are two different $G$-stable affine formal models in $A$, then $\cG_{p^\epsilon}(\cA) \cap \cG_{p^\epsilon}(\cA')$ is open in both $\cG_{p^\epsilon}(\cA)$ and $\cG_{p^\epsilon}(\cA')$ by Theorem \ref{AutRigidTop}. Choose some $N \leq_o^u G$ contained in $\rho^{-1}(\cG_{p^\epsilon}(\cA) \cap \cG_{p^\epsilon}(\cA'))$. Then $\log \rho_{|N}$ induces the same $K$-linear derivation of $A$ when computed using either $\cG_{p^\epsilon}(\cA)$ or $\cG_{p^\epsilon}(\cA')$. Thus $d \rho : \Lie(G) \to \Der_K(A)$ does not depend on the choice of $\cA$.
\end{proof}

\begin{defn}\label{TypesOfLattices} Let $L := \Der_K(A)$, and let $\cL$ be an $\cA$-submodule of $L$.
\be \item $\cL$ is an \emph{$\cA$-lattice} in $L$ if it is a finitely presented as an $\cA$-module, and spans $L$ as a $K$-vector space. 
\item $\cL$ is an \emph{$\cA$-Lie lattice} if  $[\cL, \cL] \subseteq \cL$ and $\cL(\cA) \subseteq \cA$.
\item The $\cA$-Lie lattice is \emph{smooth} if it is projective as an $\cA$-module. 
\item The $\cA$-Lie lattice is \emph{free} if it is free of finite rank as an $\cA$-module. 
\ee\end{defn}

Recall from Example \ref{GactDerEx} that the $G$-action $\rho$ on the $K$-algebra $A$ induces in a functorial manner a $G$-action $(\rho, \dot{\rho})$ on the $(K,A)$-Lie algebra $L = \Der_K(A)$ in the sense of Definition \ref{GactsLR}. 

\begin{defn} Let $\cA$ be an affine formal model in $A$ and let $\cL$ be an $\cA$-lattice in $L$. We say that $\cL$ is \emph{$G$-stable} if the affine formal model $\cA$ is $G$-stable, and $\cL$ itself is invariant under the natural action of $G$ on $L$.
\end{defn}

If $\cL$ is a $G$-stable $\cA$-Lie lattice in $L$, then $g \mapsto (\rho(g)_{|\cA},\dot{\rho}(g)_{|\cL})$ is an action of $G$ on the $(\cR, \cA)$-Lie algebra $\cL$. The existence of $G$-stable Lie lattices follows from

\begin{lem}\label{StabLopen} Let $\cA$ be a $G$-stable affine formal model in $A$. 
\be \item $\cJ := \Der_{\cR}(\cA)$ is a $G$-stable $\cA$-Lie lattice in $L$.
\item The stabiliser in $G$ of any other $\cA$-lattice $\cL$ in $L$ is open.
\ee \end{lem}
\begin{proof} (a) Note that $\cJ$ is an $(\cR, \cA)$-Lie algebra. Because the $\cA$-module of continuous K\"ahler differentials $\Omega^1_{\cA/\cR}$ is coherent by \cite[\S 2.15.1]{Abbes}, its $\cA$-linear dual $\cJ$ is the kernel of a morphism between two free $\cA$-modules. Because $\cA$ is a coherent ring by \cite[Proposition 1.10.3(i)]{Abbes}, we see that $\cJ$ is a coherent, and in particular finitely presented, $\cA$-module. There is a natural inclusion $\cJ \hookrightarrow L$ which realises $\cJ$ as the stabiliser of $\cA$ in $L$.  Now if $S$ denotes a finite generating set for $\cA$ as a topological $\cR$-algebra and if $v \in L$, we can find $n \geq 0$ such that $\pi^n v(s) \in \cA$ for all $s \in S$. It follows that $\pi^n v \in \cJ$ and hence $\cJ$ is an $\cA$-lattice in $L$. Since $\cA$ is $G$-stable by assumption, its stabiliser $\cJ$ in $L$ is automatically $G$-stable. 

(b) Because any two $\cA$-lattices in $L$ contain $\pi$-power multiples of each other, we can choose $n,m \geq 0$ such that $\pi^n \cJ \subseteq \pi^m \cL \subseteq \cJ$.  Because the $G$-action on $A$ is continuous, the subgroup $H := \{ g \in G : (g - 1) \cdot \cA \subseteq \pi^n \cA \}$ is open in $G$. Now if $g \in H$, $v \in \cJ$ and $a \in \cA$, then 
\[ (g \cdot v - v)(a) = g v ( g^{-1} a) - v(a)  = (g - 1) \cdot v(g^{-1} a) + v( (g^{-1} - 1)\cdot a) \in \pi^n \cA,\]
which implies that $(g - 1) \cdot \cJ \subseteq \pi^n \cJ$ for all $g \in H$. Therefore
\[(g - 1) \cdot \pi^m \cL \subseteq (g-1) \cdot \cJ \subseteq \pi^n \cJ \subseteq \pi^m \cL \qmb{for all} g \in H,\] 
which shows that the $\cA$-lattice $\cL$ is $H$-stable. 
\end{proof}

\begin{defn} Let $\cL$ be an $\cA$-Lie lattice in $L$. 
\be \item The \emph{$\pi$-adic completion} of $U(\cL)$ is $\h{U(\cL)} := \invlim U(\cL) / \pi^a U(\cL).$
\item We denote the $\cR$-torsion submodule of $\h{U(\cL)}$ by $\h{U(\cL)}_{\tors}$. 
\item We denote the $\cR$-torsion-free part of $\h{U(\cL)}$ by $\htf{U(\cL)} := \h{U(\cL)} / \h{U(\cL)}_{\tors}.$
\item We define $\hK{U(\cL)} := \h{U(\cL)} \otimes_{\cR} K.$
\ee\end{defn}

It is clear that $\htf{U(\cL)}$ is isomorphic to the image of $\h{U(\cL)}$ under the natural map $\h{U(\cL)} \to \hK{U(\cL)}$, and that the $\cR$-algebras
\[ U(\cL), \quad \h{U(\cL)}, \quad \htf{U(\cL)}, \qmb{and} \hK{U(\cL)}\]
carry natural $G$-actions by functoriality, whenever the $\cA$-Lie lattice $\cL$ happens to be $G$-stable. Note that $\hK{U(\cL)}$ is a $K$-Banach algebra with unit ball isomorphic to $\htf{U(\cL)}$.

\begin{lem}\label{PsiL} Let $\cL$ be an $\cA$-Lie lattice in $L$, let $\cE := \End_{\cR}(\cA)$, and let 
\[\iota :=  \htf{i_{\cA} \oplus i_{\cL}} : \cA \oplus \cL \to \cU := \htf{U(\cL)}\] 
denote the natural map. 
\be \item There is a unique $\cR$-algebra homomorphism
\[ \psi_{\cL} : \cU \to \cE\]
such that $\psi_{\cL}(\iota(a)) = \ell(a)$ and $\psi_{\cL}(\iota(v)) = v$ for all $a \in \cA$ and $v \in \cL$.
\item The restriction of $\psi_{\cL}$ to $\iota(\cA \oplus \cL)$ is injective.
\item The restriction of $\psi_{\cL}^\times$ to $\exp(\iota(p^\epsilon \cL))$ is injective, with image $\exp(p^\epsilon \cL)$.
\item If $\cA$ and $\cL$ are $G$-stable, then 
\[ \psi_{\cL}(g \cdot s) = \rho(g) \hsp \psi_{\cL}(s) \hsp \rho(g)^{-1}\]
for all $g \in G$ and $s \in \cU$.
\ee \end{lem}
\begin{proof} (a) Let $\ell : \cA \to \cE$ be defined by $\ell(a)(b) = ab$ for all $a,b\in \cA$, and let $j : \cL \hookrightarrow \cE$ be the natural inclusion. Then $\ell$ is an $\cR$-algebra homomorphism, $j$ is an $\cR$-Lie algebra homomorphism, $j( av ) = \ell(a) j(v)$ for all $a\in \cA, v \in \cL$, and \[ [ j(v), \ell(a)](b) = v(ab) - av(b) = v(a)b = \ell(v(a))(b) \qmb{for all} a,b\in\cA.\]
Thus $[j(v), \ell(a)] = \ell(v \cdot a)$ for all $v \in \cL$ and $a \in \cA$, so by Lemma \ref{ULUP}, there is an $\cR$-algebra homomorphism $\psi : U(\cL) \to \cE$ such that $\psi(\ell(a)) = a$ and $\psi(j(v)) = v$ for all $a \in \cA, v\in \cL$. 

Since $\cA$ is $\pi$-adically complete and $\cR$-flat, the same is true for $\cE$. Hence $\psi$ extends to an $\cR$-algebra homomorphism $\psi_{\cL} : \htf{U(\cL)} \to \cE$ with the required properties. This homomorphism is unique because the $\cR$-subalgebra of $\cU$ generated by $\iota(\cA \oplus \cL)$ is dense in $\cU$, and any $\cR$-algebra homomorphism between two $\pi$-adically complete $\cR$-algebras is automatically continuous.

(b) This follows immediately from part (a), since $\ell : \cA \to \cE$ is injective.

(c) Apply part (b), together with the fact that $\exp$ and $\log$ are bijections.

(d) Let $g \in G$ and $a\in \cA$. Then
\[ \rho(g) \hsp \ell(a) \hsp \rho(g)^{-1} = \ell(g \cdot a)\]
because $(\rho(g)\hsp\ell(a)\rho(g)^{-1})(b) = g \cdot ( a \hsp(g^{-1} \cdot b)) = (g\cdot a) b = \ell(g \cdot a)(b)$ for all $b \in \cA$, since $\rho(g)$ is an $\cR$-algebra automorphism of $\cA$. Now define $\alpha : \cU \to \cE$ and $\alpha' : \cU \to \cE$ by 
\[ \alpha(s) = \psi_{\cL}(g \cdot s) \qmb{and} \alpha'(s) = \rho(g) \hsp \psi_{\cL}(s) \hsp \rho(g)^{-1}\]
for all $s \in \cU$. Then 
\[ \begin{array}{lll} 
\alpha( \iota(a) ) &=& \psi_{\cL}( g \cdot \iota(a) ) =  \psi_{\cL}( \iota(g \cdot a) ) = \ell(g \cdot a) = \\
 &=& \rho(g) \hsp \ell(a) \hsp \rho(g)^{-1} = \rho(g) \hsp \psi_{\cL}(\iota(a)) \hsp \rho(g)^{-1} = \alpha'(\iota(a))\end{array}\]
 for any $a \in \cA$, and similarly
\[ \begin{array}{lll}
\alpha(\iota(v)) &=& \psi_{\cL}( g \cdot \iota(v) ) = \psi_{\cL}( \iota(g \cdot v) ) = g \cdot v = \\
&=& \rho(g) \hsp v \hsp \rho(g)^{-1} = \rho(g) \hsp \psi_{\cL}(\iota(v)) \hsp \rho(g)^{-1} = \alpha'(\iota(v)).\end{array}\]
Thus the two $\cR$-algebra homomorphisms $\alpha$ and $\alpha'$ from $\cU$ to $\cE$ agree on $\iota(\cA \oplus \cL)$, and hence $\alpha = \alpha'$ by the argument in part (a). \end{proof}

\begin{defn}\label{BetaDefn}Let $\cA$ be a $G$-stable affine formal model in $A$ and let $\cL$ be a $G$-stable $\cA$-Lie lattice in $L = \Der_K(A)$. We define
\[ G_{\cL} := \rho^{-1} \left( \exp( p^\epsilon \cL ) \right) \qmb{and} \beta_{\cL} := (\psi_{\cL}^\times)^{-1} \circ \rho : G_{\cL} \to \cU^\times.\]
\end{defn}
It may be helpful to visualise these maps as follows:
\[\xymatrix{ & G_{\cL} \ar[dl]_{\rho_{|G_\cL}}\ar[r] \ar@{..>}[d]^{\beta_\cL} \ar[d] & G \ar[d]^\rho  \\ \exp(p^\epsilon \cL) \ar[r] & \cU^\times \ar[r]_{\psi_\cL^\times} & \cE^\times. }\]
\begin{thm}\label{Beta} Let $\cA$ be a $G$-stable affine formal model in $A$ and let $\cL$ be a $G$-stable $\cA$-Lie lattice in $L$. 
\be\item $G_{\cL}$ is an open normal subgroup of $G$.
\item $\beta_{\cL}$ is a $G$-equivariant trivialisation of the $G_{\cL}$-actions on $\htf{U(\cL)}$ and $\hK{U(\cL)}$.
\ee\end{thm}
\begin{proof} (a) Since $\cE$ is a $K$-Banach algebra, the Campbell-Baker-Hausdorff series
\[ \Phi(X,Y) = X + Y + \frac{1}{2} [X,Y] + \frac{1}{12}[X,[X,Y]] + \frac{1}{12}[Y,[Y,X]] + \cdots.\]
does converge at $(X,Y) = (p^\epsilon u, p^\epsilon v)$ for any $u,v \in \cL$, and 
\[ \exp(p^\epsilon u) \exp(p^\epsilon v) = \exp( \Phi( p^\epsilon u, p^\epsilon v) )\]
by \cite[Proposition 6.27]{DDMS}. Moreover, since $[\cL, \cL] \subseteq \cL$ and since $\cL$ is $\pi$-adically complete, we see that $\Phi(p^\epsilon u, p^\epsilon v) \in p^\epsilon \cL$. Thus $\exp( p^\epsilon \cL )$ is a subgroup of $\cE^\times$.

Let $g \in G_{\cL}$, so that $\rho(g) = \exp(u)$ for some $u \in p^\epsilon \cL$. If $x \in G$ then $x \cdot u = \dot{\rho}(x)(u) = \rho(x) u \rho(x)^{-1}$ by definition, so
\[ \rho(xgx^{-1}) = \rho(x) \exp(u) \rho(x)^{-1} = \exp(x \cdot u) \in \exp(p^\epsilon \cL)\]
because $\cL$ is $G$-stable.  Thus $xgx^{-1} \in G_{\cL}$ and $G_{\cL}$ is normal in $G$. To see that it is open, choose $N \leq_o^u G$ contained in $\rho^{-1}(\cG_{p^\epsilon}(\cA))$ as in the proof of Proposition \ref{GactsOnX}. Then $\log \rho(N)$ is a finitely generated $\Zp$-submodule of $\Der_K(A)$, and because $\cL$ is an $\cA$-lattice in $\Der_K(A)$ we can find some $n \geq 0$ such that $p^n \log \rho(N) \subseteq p^\epsilon \cL$. Hence $N^{p^n} \leq G_{\cL}$, so $G_{\cL}$ is open in $G$.

(b) $\beta_{\cL}$ is a well-defined group homomorphism by Lemma \ref{PsiL}(c). Now
\[ \psi_{\cL}(\beta_{\cL}(g) \hsp s \hsp \beta_{\cL}(g)^{-1}) = \rho(g) \hsp \psi_{\cL}(s) \hsp \rho(g)^{-1} = \psi_{\cL}(g \cdot s)\]
for any $g \in G_{\cL}$ and $s \in \cU$ by Lemma \ref{PsiL}(d), so  Lemma \ref{PsiL}(b) implies that
\[ \beta_{\cL}(g) \hsp s \hsp \beta_{\cL}(g)^{-1} = g \cdot s\]
for all $g \in G_{\cL}$ and $s \in \iota(\cA \oplus \cL)$. But $\iota(\cA \oplus \cL)$ generates $\cU$ as a topological $\cR$-algebra, so this equation actually holds for all $s \in \cU$. Thus $\beta_{\cL}$ is a trivialisation of the $G_{\cL}$-action on $\cU$. Finally, applying Lemma \ref{PsiL}(d) again gives
\[ \psi_{\cL}( \beta_{\cL}( {}^xg ) ) = \rho(xgx^{-1}) = \rho(x) \hsp \psi_{\cL}( \beta_{\cL}(g)) \hsp \rho(x)^{-1} = \psi_{\cL}( x \cdot \beta_{\cL}(g) )\]
for all $x \in G$ and $g \in G_{\cL}$. Note that $\exp : p^\epsilon \cU \to \cU^\times$ is $G$-equivariant, and that $\beta_{\cL}(g) \in \exp( \iota( p^\epsilon \cL ) )$ by construction. Hence $x \cdot \beta_{\cL}(g) \in \exp( \iota( p^\epsilon \cL ) )$, so $\beta_{\cL}({}^xg) = x \cdot \beta_{\cL}(g)$ by Lemma \ref{PsiL}(c).

We have shown that $\beta_{\cL}$ is a trivialisation of the $G_{\cL}$-action on $\cU$. It follows easily that when regarded as a map $G_{\cL} \to \hK{U(\cL)}^\times = (\cU \otimes_{\cR}K)^\times$, it is also a trivialisation of the $G_{\cL}$-action on $\hK{U(\cL)}$.
\end{proof}

\begin{defn}\label{DefnTrivPair} Let $\cA$ be a $G$-stable affine formal model in $A$. We say that $(\cL,N)$ is an \emph{$\cA$-trivialising pair} if $\cL$ is a $G$-stable $\cA$-Lie lattice in $\Der_K(A)$ and $N$ is an open normal subgroup of $G$ contained in $G_{\cL}$. We denote the set of all $\cA$-trivialising pairs by $\cI(\cA,\rho, G)$ or simply by $\cI(G)$ if the other parameters are understood.
\end{defn}

Recall that $\cD(A) := U(\Der_K(A))$ denotes the algebra of differential operators on $A$ of finite order, and let $\varphi : A \to A'$ be an \'etale morphism of $K$-affinoid algebras. By \cite[Lemma 2.4]{DCapOne} there is a unique map $\tilde{\varphi} : \Der_K(A) \to \Der_K(A')$ which is a $\varphi$-morphism in the sense of Definition \ref{PhiMorph}, so Lemma \ref{Ufunctor} induces a $K$-algebra homomorphism $U(\varphi, \tilde{\varphi}) : \cD(A) \to \cD(A')$ extending $\varphi$ and $\tilde{\varphi}$. 

\begin{lem}\label{LogRhoLemma} Let $N$ and $N'$ be compact $p$-adic Lie groups, acting continuously on $K$-affinoid algebras $A$ and $A'$, respectively, let $\varphi : A \to A'$ be an \'etale morphism and let $\tau : N \to N'$ be a group homomorphism such that $\varphi(n \cdot a) = \tau(n) \cdot \varphi(a)$ for all $n \in N$, $a \in A$.  Let $\cA$ be an $N$-stable affine formal model in $A$ and let $\cA'$ be a $N'$-stable affine formal model in $A'$ such that $\varphi(\cA) \subseteq \cA'$. Suppose that 
\[\rho(N) \subseteq \exp(p^\epsilon \Der_{\cR}(\cA)) \qmb{and} \rho'(N') \subseteq \exp(p^\epsilon \Der_{\cR}(\cA')). \]
Then $\tilde{\varphi} \circ \log \circ \rho = \log \circ \rho' \circ \tau.$
\end{lem}
\begin{proof}For each $m \geq 1$, let $\ell_m(t) := -\sum_{r=1}^m \frac{(1 - t)^r}{r} \in K[t]$ be the $m$-th partial sum in the logarithm series $\log(t)$ appearing in the proof of Proposition \ref{GactsOnX}. Fix $n \in N$ and $a \in A$. Because $\varphi : A \to A'$ is continuous by \cite[Theorem 6.1.3/1]{BGR} and $\ell_m(\rho(n))(a) \to \log(\rho(n))(a)$ in $A$, $\varphi( \ell_m(\rho(n)) (a) ) \to \varphi( \log(\rho(n)) (a) )$ in $A'$. Hence
\[ \tilde{\varphi}\left(\log(\rho(n))\right)(\varphi(a)) = \varphi( \log(\rho(n))(a) = \lim_{m\to\infty} \varphi( \ell_m(\rho(n))(a) ).\]
Now the $N$-equivariance of $\varphi$ shows that $\varphi( \ell_m(\rho(n))(a) ) = \ell_m(\rho'(n'))(\varphi(a))$. Hence
\[ \tilde{\varphi}\left(\log(\rho(n))\right)(\varphi(a)) = \lim_{m\to\infty} \ell_m(\rho'(n'))(\varphi(a)) = \log(\rho'(n')) ( \varphi(a) ).\]
Because $\varphi : A \to A'$ is \'etale, it follows that the derivations $\tilde{\varphi}\left(\log(\rho(n))\right)$ and $\log(\rho'(n'))$ of $A'$ are equal. 
\end{proof}

By Theorem \ref{Beta} and Definition \ref{DefnTriv}, we have the $\cR$-algebras
\[\htf{U(\cL)} \rtimes_N G := \htf{U(\cL)} \rtimes_N^{\beta_{\cL|N}} G \qmb{and} \hK{U(\cL)} \rtimes_N G := \hK{U(\cL)} \rtimes_N^{\beta_{\cL|N}} G\] 
at our disposal whenever $(\cL,N)$ is an $\cA$-trivialising pair. We finish $\S$ \ref{GActionSection} by discussing the functoriality of the second construction in a rather general setting. 

\begin{prop}\label{hUGfunc}Let $G$ and $G'$ be compact $p$-adic Lie groups, acting continuously on $K$-affinoid algebras $A$ and $A'$, respectively, and suppose that $\tau : G \to G'$ is a group homomorphism such that $\varphi(g \cdot a) = \tau(g) \cdot \varphi(a)$ for all $g \in G$, $a \in A$.  Let $\cA$ be a $G$-stable affine formal model in $A$ and let $\cA'$ be a $G'$-stable affine formal model in $A'$. Let $(\cL, N)$ be an $\cA$-trivialising pair, let $(\cL',N')$ be an $\cA'$-trivialising pair, and suppose that
\[\varphi(\cA) \subseteq \cA', \quad \tilde{\varphi}(\cL) \subseteq \cL' \qmb{and} \tau(N) \subseteq \tau(N').\]
Then there is a unique continuous $K$-algebra homomorphism 
\[\hK{\theta} \rtimes \tau : \hK{U(\cL)} \rtimes_N G \longrightarrow \hK{U(\cL')} \rtimes_{N'} G'\]
which makes the following diagram commute:
\[ \xymatrix{ \cD(A) \rtimes G \ar[rr]^{U(\varphi, \tilde{\varphi}) \rtimes \tau} \ar[d] && \cD(A') \rtimes G' \ar[d] \\ \hK{U(\cL)} \rtimes_N G \ar[rr]_{\hK{\theta} \rtimes \tau} && \hK{U(\cL')} \rtimes_{N'} G'.}\]
\end{prop}
\begin{proof} Note that $\tilde{\varphi}_{|\cL} : \cL \to \cL'$ is a $\varphi_{|\cA}$-morphism in the sense of Definition \ref{PhiMorph}, so Lemma \ref{Ufunctor} induces an $\cR$-algebra homomorphism
\[ \theta := U(\varphi_{|\cA}, \tilde{\varphi}_{|\cL}) : U(\cL) \to U(\cL')\]
extending $\varphi_{|\cA}$ and $\tilde{\varphi}_{|\cL}$.  Let $\rho : G \to \Aut_K(A)$ and $\rho' : G \to \Aut_K(A')$ be the given actions, and write $g' := \tau(g)$ for any $g \in G$. The diagrams
\[ \xymatrix{ \cA \ar[rr]^{\rho(g)_{|\cA}}\ar[d]_{\varphi_{|\cA}} && \cA \ar[d]^{\varphi_{|\cA}} \\ \cA' \ar[rr]_{\rho'(g')_{|\cA'}} && \cA' } \qmb{and} \xymatrix{ \cL \ar[rr]^{\dot{\rho}(g)_{|\cL}} \ar[d]_{\tilde{\varphi}_{|\cL}} && \cL \ar[d]^{\tilde{\varphi}_{|\cL}} \\ \cL' \ar[rr]_{\dot{\rho}'(g')_{|\cL'}} && \cL' }\]
are commutative for any $g \in G$, so $\theta$ is $G$-equivariant with respect to the natural $G$-action on $U(\cL)$ and the $\tau$-twisted $G$-action on $U(\cL')$:
\begin{equation}\label{ThetaEq} \theta\left( g \cdot x \right) = g' \cdot \theta(x) \qmb{for all} g \in G, x \in U(\cL).\end{equation}
Hence its $\pi$-adic completion $\htf{\theta} : \htf{U(\cL)} \to \htf{U(\cL')}$ is also $G$-equivariant. Next, 
\begin{equation}\label{ExpIotaLogRho} \beta_{\cL} = (\psi_{\cL}^\times)^{-1} \circ \rho = \exp \circ \iota \circ \log \circ \rho\end{equation}
from Definition \ref{BetaDefn}, so consider the diagram
\[ \xymatrix{ N \ar[rr]^{\log \circ \rho} \ar[d]_{\tau} && p^\epsilon \cL \ar[rr]^{\exp \circ \iota}\ar[d]^{\tilde{\varphi}} && \htf{U(\cL)}^\times \ar[d]^{\htf{\theta}^\times} \\ N' \ar[rr]_{\log \circ \rho'}  && p^\epsilon \cL' \ar[rr]_{\exp\circ\iota'} && \htf{U(\cL')}^\times .} \]
The first square commutes by Lemma \ref{LogRhoLemma}, whereas the second square commutes by the definition of $\htf{\theta}$. Hence
\begin{equation}\label{BetaEq}\htf{\theta}^\times \circ \beta_{\cL|N} = \beta_{\cL'|N'} \circ \tau_{|N}.\end{equation}
By $(\ref{ThetaEq})$ and $(\ref{BetaEq})$, we may apply Lemma \ref{CPfunc} to $\hK{\theta} : \hK{U(\cL)} \to \hK{U(\cL')}$ and $\tau : G \to G'$ to obtain the $K$-algebra homomorphism
\[\hK{\theta} \rtimes \tau : \hK{U(\cL)} \rtimes_N G \longrightarrow \hK{U(\cL')} \rtimes_{N'} G'\]
which makes the diagram in the statement of the Lemma commute. Any other continuous map $\hK{U(\cL)}\rtimes_N G \to \hK{U(\cL')} \rtimes_{N'} G'$ making the diagram commute agrees with $\hK{\theta} \rtimes \tau$ on the dense image of $\cD(A) \rtimes G$ in $\hK{U(\cL)} \rtimes_N G$, and therefore must be equal to $\hK{\theta} \rtimes \tau$.
\end{proof}

\subsection{The completed skew-group ring \ts{\w\cD(\bX,G)}}
\label{MainConstr}
We continue to assume throughout this subsection that:
\begin{itemize} 
\item $\bX$ is a $K$-affinoid variety,
\item $G$ is a compact $p$-adic Lie group acting continuously on $\bX$, and
\item $\cA$ is a $G$-stable affine formal model in $A := \cO(\bX)$,
\end{itemize}
Recall that this means that we are given a group homomorphism 
\[\rho : G \to \Aut_K(\bX, \cO_\bX)\] 
satisfying certain continuity conditions spelt out in Definition \ref{CtsAct}. Note that the set of $\cA$-trivialising pairs $\cI(G)$ becomes directed when ordered by component-wise reverse inclusion:
\[ (\cL_1, N_1) \leq (\cL_2, N_2) \qmb{if and only if} \cL_1 \supseteq \cL_2 \qmb{and} N_1 \supseteq N_2.\]
Whenever $(\cL_1, N_1) \leq (\cL_2, N_2)$, Lemma \ref{CPfunc} and Lemma \ref{Ufunctor} induce canonical connecting homomorphisms
\[\begin{array}{rcl} \htf{U(\cL_2)} \rtimes_{N_2} G &\longrightarrow& \htf{U(\cL_1)} \rtimes_{N_1}G \qmb{and} \\

\hK{U(\cL_2)} \rtimes_{N_2}G &\longrightarrow& \hK{U(\cL_1)} \rtimes_{N_1}G\end{array}.\]
Using these connecting maps, we can now give our first central definition.

\begin{defn}\label{StarDefn}  We define the \emph{completed skew-group algebra}
\[ \w\cD(\bX,G)_{\cA} := \invlim\limits_{(\cL, N)\in \cI(\cA, \rho, G) } \hK{U(\cL)} \rtimes_N G\]
and the \emph{integral completed skew-group ring}
\[ \cA \Star G := \invlim\limits_{(\cL, N)\in \cI(\cA, \rho, G) } \htf{U(\cL)} \rtimes_N G.\]
\end{defn}

\begin{rmk}\label{GammaMap}\hspace{2em}
\be \item It can be shown that the canonical map $\cA \to \invlim\limits \h{U(\pi^n \cL)}$ is an isomorphism whenever  $\cL$ is a smooth $(\cR, \cA)$-Lie algebra. This justifies the notation $\cA \Star G$: the $(K,\cO(\bX))$-algebra $\cT(\bX)$, as well as finite non-empty products of elements in $\cT(\bX)$, disappear when we pass to the limit in the integral completed skew-group ring, and only $\cA$ and $G$ remain.
\item We will shortly see that in fact $\w\cD(\bX,G)_{\cA}$ does not depend on the choice of $\cA$, up to canonical isomorphism. 
\item There is a canonical group homomorphism 
\[ \gamma : G \to  \w\cD(\bX,G)_{\cA}^\times\]
and a canonical $K$-algebra homomorphism
\[ i : \cD(\bX) \to \w\cD(\bX,G)_{\cA}.\]
These extend to a canonical $K$-algebra homomorphism
\[ [i \rtimes \gamma]_{\cA} : \cD(\bX) \rtimes G \longrightarrow \w\cD(\bX,G)_{\cA}.\]
\ee\end{rmk}

\begin{defn}\label{GoodChain} Let $(N_\bullet) := N_0 \geq N_1 \geq N_2 \geq \cdots$ be a chain of open normal subgroups of $G$, such that $\bigcap_{n=0}^\infty N_n = \{1\}$, and let $\cL$ be a $G$-stable $\cA$-Lie lattice in $\Der_K(A)$. 

We say $(N_\bullet)$ is a \emph{good chain for $\cL$} if $(\pi^n \cL, N_n) \in \cI(\cA, \rho, G)$ for all $n \geq 0$.
\end{defn}

\begin{lem}\label{StdPres} For every good chain $N_\bullet$ for $\cL$, there is a canonical isomorphism
\[ \w\cD(\bX,G)_{\cA} \cong \invlim \hK{U(\pi^n \cL)} \rtimes_{N_n} G\]
of $K$-algebras.\end{lem}
\begin{proof} Let $(N_\bullet)$ be a good chain for $\cL$, and let $(\cL', N')$ be some other member of $\cI(\cA, \rho, G)$. Since $N'$ is open, its complement is closed and therefore compact. Because $\bigcap_{n=1}^\infty N_n$ is trivial, $G \backslash N' \subseteq \bigcup_{n=0}^\infty G \backslash N_n$ is an open covering, so by compactness, $G \backslash N' \subseteq G \backslash N_r$ for some $r\geq 0$. In other words,  $N'$ contains $N_r$. Also $\cL'$ contains $\pi^s \cL$ for some $s \geq 0$. Taking $n = \max\{r,s\}$ we see that $(\pi^n \cL, N_n) \geq (\cL', N)$. Thus $\{(\pi^n \cL, N_n) : n \geq 0\}$ is cofinal inside $\cI(\cA, \rho, G)$, and the result follows.
\end{proof}

\begin{rmk} By Definition \ref{DefnTrivPair} and Theorem \ref{Beta}(a), $(\cL, G_{\cL})$ is always an $\cA$-trivialising pair, so if the action of $G$ is \emph{faithful} in the sense that $\ker \rho$ is trivial, then $\{(\pi^n \cL, G_{\pi^n \cL}) : n \geq 0\}$ is a good chain for $\cL$. Lemma \ref{StdPres} now shows that in this case $\w\cD(\bX,G)$ can be defined in a slightly less elaborate way as follows:
\[ \w\cD(\bX,G)_{\cA} = \invlim \hK{U(\pi^n\cL)} \underset{G_{\pi^n \cL}}{\rtimes} G\]
for any choice of $G$-stable $\cA$-Lie lattice $\cL$. However, this does not give the correct definition when the action is not faithful, and keeping track of all $\cA$-trivialising pairs will afford some extra flexibility.
\end{rmk}

\begin{lem}\label{FastChain} Let $H_0, H_1, H_2,\ldots$ be open subgroups of $G$. Then there is a chain $N_0 \geq N_1 \geq N_2 \geq \cdots$ of open normal subgroups such that $\bigcap_{n=0}^\infty N_n$ is trivial, and $N_n \leq H_n$ for all $n \geq 0$.
\end{lem}
\begin{proof} Choose any chain $J_0 \geq J_1 \geq J_2 \geq \cdots$ of open subgroups of $G$ such that $\bigcap_{n=0}^\infty J_n$ is trivial. Choose any open normal subgroup $N_0$ of $G$ contained in $H_0$. Assuming inductively that $N_{n-1}$ has been chosen for $n \geq 1$, choose an open normal subgroup $N_n$ of $G$ contained in $H_n \cap J_n \cap N_{n-1}$. Since $N_n \leq J_n$ for all $n \geq 0$, $\bigcap_{n=0}^\infty N_n$ is trivial.
\end{proof}

\begin{cor}\label{ChainCap} Let $A_1,\ldots, A_m$ be a finite collection of $K$-affinoid algebras, let $\rho_i : G \to \Aut_K(A_i)$ be continuous group actions, let $\cA_i$ be a $G$-stable affine formal model in $A_i$ and let $\cL_i$ be a $G$-stable $\cA_i$-Lie lattice in $\Der_K(A_i)$ for each $i$. Then there is a chain $(N_\bullet)$ which is good for \emph{each} $\cL_i$.
\end{cor}
\begin{proof} Let $H_n := \bigcap_{i=1}^m G_{\pi^n \cL_i}$ for each $n \geq 0$. This is an open subgroup of $G$ by Theorem \ref{Beta}(a). Using Lemma \ref{FastChain},  choose a descending chain $N_0 \geq N_1 \geq N_2 \geq \cdots $ of open normal subgroups of $G$, intersecting trivially, such that $N_n \leq H_n$ for all $n \geq 0$. Then $(\pi^n \cL_i, N_n) \in \cI(\cA_i, \rho_i, G)$ for each $i$ and each $n \geq 0$, and $(N_\bullet)$ is good for each $\cL_i$.
\end{proof}

\begin{prop}\label{DXGwelldef} $\w\cD(\bX,G)_{\cA}$ is independent of the choice of $\cA$. \end{prop}
\begin{proof} Let $\cA, \cB$ be two $G$-stable affine formal models in $A$. Choose a $G$-stable $\cA$-Lie lattice $\cL$ and a $G$-stable $\cB$-Lie lattice $\cJ$ in $L$ using Lemma \ref{StabLopen}(a). By Lemma \ref{ProdFormModels}, we can find an integer $r$ such that $\pi^r \cdot \cA \subseteq \cB$. By replacing $\cL$ by a $\pi$-power multiple, we will assume that $\cL \subseteq \cJ$. 

Let $x_1,\ldots,x_d$ generate $\cL$ as an $\cA$-module. The universal property of $U(-)$ induces a $G$-equivariant $\cR$-algebra homomorphism $\theta_0 : U(\cL) \to \hK{U(\cJ)}$. Now $U(\cL)$ is generated as an $\cA$-module by finite products of the $x_i$, and $\theta_0$ sends all these elements to $\htf{U(\cJ)}$. Because $\cA \subseteq \pi^{-r} \cB$, we see that the image of $\theta_0$ is contained in $\pi^{-r}\htf{U(\cJ)}$. Hence $\theta_0$ extends to a $G$-equivariant $K$-algebra homomorphism
\[\theta_0 : \hK{U(\cL)} \to \hK{U(\cJ)}.\]
Applying the same argument to $\pi^n \cL \subseteq \pi^n \cJ$ for each $n\geq 0$, we obtain a compatible sequence of $G$-equivariant $K$-algebra homomorphisms
\[\theta_n : \hK{U(\pi^{n}\cL)} \to \hK{U(\pi^n \cJ)}.\]
Now, choose a chain $(N_\bullet)$ which is good for both $\cL$ and $\cJ$ using Corollary \ref{ChainCap}. Then $\theta_n^\times \circ \beta_{\pi^n \cL, N_n} = \beta_{\pi^n \cJ, N_n}$ for all $n \geq 0$, so Lemma \ref{CPfunc} gives a compatible sequence of $K$-algebra homomorphisms
\[ \theta_n \rtimes 1_G : \hK{U(\pi^n \cL)} \rtimes_{N_n} G \longrightarrow \hK{U(\pi^n \cJ)} \rtimes_{N_n} G.\]
Passing to the limit and applying Lemma \ref{StdPres}, we obtain a commutative diagram
\[  \xymatrix{ \w\cD(\bX,G)_{\cA} \ar[rrr]^{\theta_{\cA,\cB}}\ar[d]_{\cong} &&& \w\cD(\bX,G)_{\cB}\ar[d]^{\cong} \\ \invlim \hK{U(\pi^n \cL)} \rtimes_{N_n} G \ar[rrr]_{\invlim \theta_n \rtimes 1_G} &&& \invlim \hK{U(\pi^n \cJ)} \rtimes_{N_n} G. }\]
By construction, $\theta_{\cA, \cB} \circ [i \rtimes \gamma]_{\cA} = [i \rtimes \gamma]_{\cB}$, so $\theta_{\cB,\cA}\circ\theta_{\cA,\cB}$ is the identity map on the dense image of $[i \rtimes \gamma]_{\cA}$ inside $\w\cD(\bX,G)_{\cA}$. Because $\theta_{\cA, \cB}$ and $\theta_{\cB,\cA}$ are continuous, it follows that they are mutually inverse isomorphisms.
\end{proof}

We will henceforth denote $\w\cD(\bX,G)_{\cA}$ simply by $\w\cD(\bX,G)$.

\begin{cor}\label{DXGFrechet} $\w\cD(\bX,G)$ is a $K$-Fr\'echet algebra.
\end{cor}
\begin{proof} We can find a $G$-stable $\cA$-Lie lattice $\cL$ in $\cT(\bX)$ and a good chain $(N_\bullet)$ for $\cL$. By Lemma \ref{StdPres}, there is a $K$-algebra isomorphism $\w\cD(\bX,G) \cong \invlim \hK{U(\pi^n\cL)} \rtimes_{N_n} G$. Each $\hK{U(\pi^n \cL)} \rtimes_{N_n} G$ is naturally a $K$-Banach algebra whose unit ball is given by $\h{U(\pi^n\cL)} \rtimes_{N_n} G$.  In this way, $\w\cD(\bX,G)$ is isomorphic to a countable inverse limit of $K$-Banach algebras and therefore carries a $K$-Fr\'echet algebra structure. It follows from Proposition \ref{DXGwelldef} that this is independent of the choice of $\cA$, $\cL$ and $(N_\bullet)$.
\end{proof}

\begin{prop}\label{IHIGcofinal} Let $H$ be an open subgroup of $G$. Then $\cI(H) \cap \cI(G)$ is cofinal in both $\cI(G)$ and $\cI(H)$.
\end{prop}
\begin{proof}  It follows from Definition \ref{BetaDefn} that $H_{\cL} = G_{\cL} \cap H.$ Let $(\cL,N) \in \cI(G)$, and choose some $U \triangleleft_o G$ contained in the open subgroup $N \cap H$. Then $U \leq N \cap H \leq G_{\cL} \cap H = H_{\cL}$, so $(\cL, U) \in \cI(H) \cap \cI(G)$.  Now let $(\cL,N) \in \cI(H)$. Choose some $U \triangleleft_o G$ contained in $N$. Then $U \leq N \leq H_{\cL} \leq G_{\cL}$, so $(\cL,U) \in \cI(H) \cap \cI(G)$.\end{proof}

\begin{cor}\label{Stacky} Let $H$ be an open normal subgroup of $G$. Then there are natural isomorphisms
\[ \w\cD(\bX,G) \congs \w\cD(\bX,H) \rtimes_H G \]
and
\[ \cA \Star G \congs \left( \cA \Star H \right) \rtimes_H G.\]
\end{cor}
\begin{proof} For every $(\cL, J) \in \cI(G) \cap \cI(H)$, Proposition \ref{SHGassoc} induces isomorphisms
\[ \hK{U(\cL)}  \rtimes_J G \cong \left( \hK{U(\cL)} \rtimes_J H\right) \rtimes_H G \qmb{and} \h{U(\cL)}  \rtimes_J G \cong \left( \h{U(\cL)} \rtimes_J H\right) \rtimes_H G\]
which are functorial in $(\cL, J)$. Since $(\cL, J) \in \cI(G) \cap \cI(H)$ is cofinal in both $\cI(G)$ and $\cI(H)$ by Proposition \ref{IHIGcofinal}, the result follows by passing to the inverse limit.
\end{proof}

We finish $\S$ \ref{MainConstr} by discussing the functoriality of our construction $\w\cD(\bX,G)$.   Recall the natural $K$-algebra homomorphism
\[ i \rtimes \gamma := [i \rtimes \gamma]_{\cA} : \cD(\bX) \rtimes G \longrightarrow \w\cD(\bX,G)\]
from Remark \ref{GammaMap}(b).

\begin{thm}\label{wUGfunc}
Let $\varphi : A  \to A'$ be an \'etale morphism of $K$-affinoid algebras, let $\bX = \Sp(A), \bX' = \Sp(A')$ and let $G, G'$ be compact $p$-adic Lie groups, acting continuously on $A$ and $A'$, respectively. Suppose that $\tau : G \to G'$ is a group homomorphism such that 
\[\varphi(g \cdot a) = \tau(g) \cdot \varphi(a) \qmb{for all} g \in G, a \in A.\] 
Then there is a unique continuous $K$-algebra homomorphism 
\[\w{\varphi \rtimes \tau} : \w\cD(\bX,G) \longrightarrow \w\cD(\bX',G')\]
which makes the following diagram commute:
\[ \xymatrix{ \cD(\bX) \rtimes G \ar[rr]^{U(\varphi, \tilde{\varphi}) \rtimes \tau} \ar[d]_{i \rtimes \gamma} && \cD(\bX') \rtimes G' \ar[d]^{i' \rtimes \gamma'} \\ \w\cD(\bX,G) \ar[rr]_{\w{\varphi \rtimes \tau}} && \w\cD(\bX',G').}\]
\end{thm}
\begin{proof}  Let $\cA \subset A$ and $\cA' \subset A'$ be $G$-stable (respectively, $G'$-stable) affine formal models. Then $\varphi(\cA)\cdot \cA'$ is another affine formal model in $A'$ containing $\varphi(\cA')$, so by Lemma \ref{GstableFM} we may find a $G'$-stable affine formal model $\cA''$ containing $\varphi(\cA)\cdot \cA'$. Replacing $\cA'$ with $\cA''$ we will assume that $\varphi(\cA) \subset \cA'$. 

Choose a $G$-stable $\cA$-Lie lattice in $\Der_K(A)$, and a $G'$-stable $\cA'$-Lie lattice  $\cL'$ in $\Der_K(A')$ using Lemma \ref{StabLopen}(a). Because $\cL$ is a finitely generated $\cA$-module and $\tilde{\varphi}(a v)  = \varphi(a) \tilde{\varphi}(v)$ for any $a \in A$ and $v \in \Der_K(A)$, we see that $\pi^m \tilde{\varphi}(\cL) \subseteq \cL'$ for some $m \geq 0$. By rescaling $\cL$, we can assume that $m = 0$, so that $\varphi(\cL) \subseteq \cL'$. 

Choose a good chain $(N_\bullet)$ in $G$ for $\cL$, and a good chain $(N'_{\bullet})$ in $G'$ for $\cL'$ using Lemma \ref{ChainCap}. It follows from \cite[Corollary 8.34 and Corollary 1.21(i)]{DDMS} that the group homomorphism $\tau$ is automatically continuous, so $\tau^{-1}(N'_n)$ is open in $G$ for each $n \geq 0$. Applying Lemma \ref{FastChain} to the open subgroups $N_n \cap \tau^{-1}(N'_n)$, we may assume that $\tau(N_n) \leq N'_n$ for each $n \geq 0$. Now Proposition \ref{hUGfunc} produces a compatible sequence of commutative diagrams
\[\xymatrix{ \cD(\bX) \rtimes G \ar[rr]^{U(\varphi, \tilde{\varphi}) \rtimes \tau} \ar[d] && \cD(\bX') \rtimes G' \ar[d] \\ \hK{U(\pi^n\cL)} \rtimes_{N_n} G \ar[rr]_{\h{\theta_{n,K}} \rtimes \tau} && \hK{U(\pi^n\cL')} \rtimes_{N_n'} G'.}\]
Passing to the limit and applying Lemma \ref{StdPres} produces the required map
\[ \w{\varphi \rtimes \tau} : \w\cD(\bX,G) \longrightarrow \w\cD(\bX',G')\]
which fits into the following commutative diagram:
\[ \xymatrix@C=20pt{  
\cD(\bX) \rtimes G \ar[rrr]^{U(\varphi, \tilde{\varphi}) \rtimes \tau} \ar[dr]_{i \rtimes \gamma} \ar[dd]&&& \cD(\bX') \rtimes G' \ar[dl]^{i' \rtimes \gamma'} \ar[dd] \\
& \w\cD(\bX,G) \ar[r]^{\w{\varphi\rtimes\tau}} \ar[dl]_{\cong} & \w\cD(\bX',G') \ar[dr]^{\cong} & \\
\invlim \hK{U(\pi^n \cL)} \rtimes_{N_n} G \ar[rrr]_{\invlim \h{\theta_{n,K}} \rtimes \tau} &&& \invlim \hK{U(\pi^n \cL')} \rtimes_{N'_n} G'
}\]
Finally, if $\psi : \w\cD(\bX,G) \longrightarrow \w\cD(\bX',G')$ is any other continuous $K$-algebra map such that $\psi \circ (\iota \rtimes \gamma) = (i' \rtimes \gamma') \circ U(\varphi, \tilde{\varphi})$, then $\psi$ agrees with $\w{\varphi \rtimes \tau}$ on the dense image of $\cD(\bX) \rtimes G$ in $\w\cD(\bX,G)$, and therefore must be equal to $\w{\varphi \rtimes \tau}$.
\end{proof}

\subsection{Compatible actions}\label{CompActSect}
Until the end of  $\S \ref{EqDModSect}$, we will assume that:
\begin{itemize}
\item $\bX$ is a  \emph{smooth} rigid analytic variety,
\item $G$ is a $p$-adic Lie group,
\item $G$ acts continuously on $\bX$ in the sense of Definition \ref{CtsAct}.
\end{itemize}
In particular, we do \emph{not} assume that $\bX$ is affinoid, nor that $G$ is compact.

\begin{defn} $\bX_w/G$ denotes the set of $G$-stable affinoid subdomains of $\bX$.
\end{defn}

\begin{lem}\label{wDGpresheaf} If $G$ is compact, then $\w{\cD}(-,G)$ is a presheaf of $K$-Fr\'echet algebras on $\bX_w / G$. \end{lem}
\begin{proof} Note that $\w\cD(\bU,G)$ is a $K$-Fr\'echet algebra for each $\bU \in \bX_w/G$ by Corollary \ref{DXGFrechet}. Given $\bV \subset \bU$ in $\bX_w / G$, the restriction map $\varphi : \cO(\bU) \to \cO(\bV)$ is \'etale, and it is $G$-equivariant by Remark \ref{sAEquiv}(b). So by Theorem \ref{wUGfunc}, there is a unique continuous $K$-algebra homomorphism $\tau^\bU_\bV := \w{\varphi \rtimes 1_G}: \w\cD(\bU,G) \to \w\cD(\bV,G)$ extending $\bU(\varphi, \tilde{\varphi}) \rtimes 1_G : \cD(\bU) \rtimes G \to \cD(\bV) \rtimes G$. If $\bW \subset \bV$ is another object of $\bX_w/ G$, the functoriality of $\cD(-)$ ensures that $\tau^\bV_\bW \circ \tau^\bU_\bV$ and $\tau^\bU_\bW$ both extend the restriction map $\cD(\bU) \to \cD(\bW)$, and are therefore equal by the uniqueness part of  Theorem \ref{wUGfunc}.\end{proof}

\begin{lem}\label{gUHmaps} Let $H$ be a compact open subgroup of $G$ and let $\bU \in \bX_w/H$. For every $g \in G$, there is a continuous $K$-algebra isomorphism
\[ \w{g}_{\bU,H} : \w\cD(\bU,H) \longrightarrow \w\cD(g\bU, gHg^{-1})\]
such that the diagram 
\[\xymatrix{ \w\cD(\bU,H)\ar[rrrr]^{\w{h}_{\bU,H}} \ar[rrd]_{\w{gh}_{\bU,H}}&&&& \w\cD(h\bU, hHh^{-1}) \ar[dll]^{\w{g}_{h\bU,hHh^{-1}}} \\ && \w\cD(gh\bU, ghHh^{-1}g^{-1}) &&}\]
is commutative for all $g,h \in G$.
\end{lem} \begin{proof} Fix $g \in G$. The structure sheaf $\cO$ on $\bX$, and the sheaf $\cD$ of finite order differential operators on $\bX$ are naturally $G$-equivariant: if  $g_\bV := g^{\cO}(\bV) : \cO(\bV) \to \cO(g\bV)$ defines the $G$-equivariant structure on $\cO$, then 
\[g^{\cD}(\bV) := U(g_\bV, \widetilde{g_\bV}) : \cD(\bV) \to \cD(g\bV)\]
defines the $G$-equivariant structure on $\cD$. Let 
\[\Ad_g : G \to G\]
be the map $x \mapsto gxg^{-1}$; then $\Ad_g(H) = gHg^{-1}$, and using $(\ref{EasyCocyc})$, we see that
\[ g_\bV( x \cdot a ) = g^{\cO}( x^\cO(a)) = (gxg^{-1})^{\cO}(g^{\cO}(a)) = \Ad_g(x) \cdot g_\bV(a)\]
for every $x \in H$ and $a \in \cO(\bV)$. Theorem \ref{wUGfunc} now induces the required continuous $K$-algebra isomorphism
\[\w{g}_{\bV,H} := \w{g_\bV \rtimes \Ad_g} : \w\cD(\bV,H) \to \w\cD(g\bV, gHg^{-1}) \]
which uniquely extends $g^{\cD}(\bV) = U(g_\bV, \widetilde{g_\bV}) : \cD(\bV) \to \cD(g\bV)$. It is easily checked, using the $G$-equivariance of $\cD$, that the diagram
\[\xymatrix{ \cD(\bU) \rtimes H\ar[rrrr]^{h^{\cD}(\bU) \rtimes \Ad_h} \ar[rrd]_{(gh)^{\cD}(\bU) \rtimes \Ad_{gh}}&&&& \cD(h \bU) \rtimes hHh^{-1} \ar[dll]^{g^{\cD}(h\bU) \rtimes \Ad_g} \\ && \cD(gh\bU) \rtimes ghHh^{-1}g^{-1} &&}\]
is commutative. The result follows from the uniqueness part of Theorem \ref{wUGfunc}.
\end{proof}

\begin{defn}\label{DefnOfSmall} We say that $(\bU,H)$ is \emph{small} if:
\be \item $\bU$ is an affinoid subdomain of $\bX$,
\item $H$ is a compact open subgroup of $G_{\bU}$,
\item $\cT(\bU)$ has an $H$-stable free $\cA$-Lie lattice $\cL$ for some $H$-stable affine formal model $\cA$ in $\cO(\bU)$. 
\ee 
If $\bU$ is an affinoid subdomain of $\bX$, we say that $H$ is \emph{$\bU$-small} if $(\bU,H)$ is small.
\end{defn}
We refer the reader to Definition \ref{TypesOfLattices} for the meaning of part (c).

\begin{lem}\label{SubsOfSmall} Suppose that $(\bX,G)$ is small. Then $(\bU, H)$ is small for every $\bU \in \bX_w$ and every compact open subgroup $H$ of $G_{\bU}$.
\end{lem}
\begin{proof} It is clear that $(\bX, J)$ is small for every compact open subgroup $J$ of $G$. Because the stabiliser $G_{\bU}$ is open in $G$ by Definition \ref{CtsAct}(a), by replacing $G$ by $H$ we may therefore assume that $H = G$ and $\bU \in \bX_w/G$.  Choose a $G$-stable affine formal model $\cA$ in $\cO(\bX)$ and a $G$-stable free $\cA$-Lie lattice $\cL$ in $\cT(\bX)$. By \cite[Lemma 7.6(b)]{DCapOne} we may replace $\cL$ by a $\pi$-power multiple, and ensure that $\bU$ is also $\cL$-admissible. Now, if $\cB$ is an $\cL$-stable affine formal model in $\cO(\bU)$, then we saw in the proof of Lemma \ref{GstableFM} that the $G$-orbit of $\cB$ is finite, $\cB_1,\ldots, \cB_n$ say. Then $\cC := \cB_1 \cdot \cdots \cB_n$ is again an affine formal model in $\cO(\bU)$ by Lemma \ref{ProdFormModels} which is both $G$-stable and $\cL$-stable. It is now easy to see that $\cC \otimes_{\cA} \cL$ is a $G$-stable free $\cC$-Lie lattice in $\cT(\bU)$.
\end{proof}

\begin{defn}\hspace{2em}
\be \item Let $\bX_w(\cT)$ denote the set of affinoid subdomains $\bU$ of $\bX$ such that $\cT(\bU)$ admits a free $\cA$-Lie lattice for some affine formal model $\cA$ in $\cO(\bU)$.
\item Let $\bX_w(\cT)/G \subset \bX_w/G$ denote the set of $G$-stable affinoid subdomains $\bU$ of $\bX$ such that $(\bU,G)$ is small. 
\ee \end{defn}

\begin{lem}\label{SmallPairsExist} For every $\bU \in \bX_w(\cT)$, there is a $\bU$-small subgroup $H$.
\end{lem}
\begin{proof} Choose an affine formal model $\cA$ in $\cO(\bU)$ and a free $\cA$-Lie lattice $\cL$ in $\cT(U)$. The stabiliser of $\cA$ in $G_{\bU}$ is open because $G_{\bU}$ acts continuously on $\bU$, so we can find some compact open subgroup $J$ of $G_{\bU}$ such that $\cA$ is $J$-stable. Now by Lemma \ref{StabLopen}(b) we can find an open subgroup $H$ of $J$ which also stabilises $\cL$. \end{proof}

Recall the definition of \emph{two-sided Fr\'echet-Stein algebras} from \cite[\S 6.4]{DCapOne}.

\begin{thm}\label{FrSt} Suppose that $(\bX, G)$ is small. Then $\w\cD(\bX,G)$ is a two-sided Fr\'echet-Stein algebra. 
\end{thm}
We postpone the proof until $\S \ref{NoethFlat}$ --- it can be found immediately after Theorem \ref{HKULflat}. We will shortly see that when $(\bX,G)$ is small, it is possible to localise a coadmissible $\w\cD(\bX,G)$-module to a $G$-equivariant sheaf of $\cD$-modules defined on \emph{every} affinoid subdomain of $\bX$. In fact, we will give a construction of this localisation functor in a more general, axiomatic, setting. 

Recall the canonical map $\gamma^G : G \to \w\cD(\bU,G)^\times$ from Remark \ref{GammaMap}(c).

\begin{defn}\label{AactsCompatibly} Let $G$ be a $p$-adic Lie group, acting continuously on a smooth rigid analytic variety $\bX$, and let $A$ be a $K$-algebra. We say that \emph{$A$ acts on $\bX$ compatibly with $G$} if there are given
\begin{itemize}[leftmargin=7mm]
\item a group homomorphism $\eta : G \to A^\times$, 
\item a Fr\'echet-Stein subalgebra $A_H$ of $A$ for every compact open subgroup $H$ of $G$, 
\item a continuous homomorphism $\varphi^H : A_H \to \w\cD(-,H)$ of presheaves of $K$-Fr\'echet algebras on $\bX_w/H$, for every compact open subgroup $H$ of $G$
\end{itemize}
such that for every pair $H \leq N$ of compact open subgroups of $G$:
\be
\item $A_H \leq A_N$, $\eta(H) \subseteq A_H^\times$ and the canonical map 
\[A_H \underset{K[H]}{\otimes}{} K[N] \longrightarrow A_N\]
is a bijection,
\item the following diagram of presheaves on $\bX_w/N$ is commutative:
\[\xymatrix{ A_H \ar[rr]^{\varphi^H}\ar[d] && \w\cD(-,H) \ar[d] \\ A_N \ar[rr]_{\varphi^N} && \w\cD(-,N), }\]
\item for every $g \in G$, the conjugation-by-$\eta(g)$ map 
\[\Ad_{\eta(g)} : A \to A\]
sends $A_H$ into $A_{gHg^{-1}}$, and for every $\bU \in \bX_w/H$, the diagram
\[\xymatrix{A_H \ar[rr]^{\Ad_{\eta(g)}} \ar[d]_{\varphi^H(\bU)} && A_{gHg^{-1}} \ar[d]^{\varphi^{gHg^{-1}}(g\bU)} \\ \w\cD(\bU,H) \ar[rr]_{\w{g}_{\bU,H}} && \w\cD(g\bU, gHg^{-1}) }\]
is commutative, and
\item  $\varphi^H \circ \eta_{|H} = \gamma^H$.
\ee
\end{defn}

Here is our first example of a compatible action.

\begin{prop}\label{XaffCompat} Suppose that $(\bX,G)$ is small. Then $\w\cD(\bX,G)$ acts on $\bX$ compatibly with $G$.
\end{prop}
\begin{proof} As $\bX$ is a $G$-stable affinoid variety, we may define the group homomorphism
\[ \eta := \gamma^G: G \to \w\cD(\bX,G)^\times.\]
For every compact open subgroup $H$ of $G$, we set 
\[ A_H := \w\cD(\bX,H)\]
which is a Fr\'echet-Stein subalgebra of $A = \w\cD(\bX,G)$ by Lemma \ref{SubsOfSmall} and Theorem \ref{FrSt}. For every $\bU \in \bX_w/H$, we let
\[\varphi^H(\bU) : A_H \longrightarrow \w\cD(\bU,H)\]
be the restriction map $\tau^\bX_\bU$ in the presheaf $\w\cD(-,H)$ on $\bX_w/H$ from Lemma \ref{wDGpresheaf}. Viewing $A_H$ as a constant sheaf on $\bX_w/H$, we see that $\varphi^H : A_H \to \w\cD(-,H)$ is a continuous morphism of presheaves, again by Lemma \ref{wDGpresheaf}. 

We now check that axioms (a)-(d) of Definition \ref{AactsCompatibly} are verified for these data.

(a) Let $(\cL,J) \in \cI(N) \cap \cI(H)$. Since $\hK{U(\cL)} \rtimes_JN$ is a crossed product of $\hK{U(\cL)}$ with $N/J$ by Lemma \ref{RingSGN}(b), we see that the canonical map
\[ \hK{U(\cL)} \rtimes_J H \underset{K[H]}{\otimes}{} K[N] \longrightarrow \hK{U(\cL)} \rtimes_J N\]
is a bijection. Now consider the following commutative diagram:
\[\xymatrix{  \w\cD(\bU,H) \underset{K[H]}{\otimes}{} K[N] \ar[d]\ar[rr]  && \w\cD(\bU,N) \ar[d] \\ 
\invlim \left( \hK{U(\cL)} \rtimes_J H  \underset{K[H]}{\otimes}{}K[N] \right) \ar[rr] && \invlim \hK{U(\cL)} \rtimes_J N. }\]
The bottom horizontal arrow is a bijection, being the inverse limit over all $(\cL,J) \in \cI(N) \cap \cI(H)$ of the maps considered above. Since $J$ has finite index in $H$ and inverse limits commute with finite direct sums, using Lemma \ref{IHIGcofinal} we see that the left vertical arrow is a bijection. The right vertical arrow is also a bijection, again by Lemma \ref{IHIGcofinal}. So the top horizontal arrow is bijective, as required.

(b) Theorem \ref{wUGfunc} induces a commutative diagram of $K$-algebra maps
\[\xymatrix{ A_H = \w\cD(\bX,H) \ar[r]\ar[d]\ar[d] & \w\cD(\bU,H) \ar[d] \\ A_N = \w\cD(\bX,N) \ar[r] & \w\cD(\bU,N).}\]

(c) Because $g^{\cD} : \cD \to g^\ast \cD$ is a morphism of presheaves, the diagram
\[\xymatrix{ \cD(\bX) \rtimes G \ar[rr]^{g^{\cD}(\bX) \rtimes \Ad_g} && \cD(\bX) \rtimes G \\ \cD(\bX) \rtimes H \ar[rr]^{g^{\cD}(\bX) \rtimes \Ad_g} \ar[u]\ar[d] && \cD(\bX, gHg^{-1}) \ar[u]\ar[d] \\ \cD(\bU) \rtimes H \ar[rr]_{g^{\cD}(\bU) \rtimes \Ad_g} && \cD(g\bU, gHg^{-1})
}\]
is commutative. Hence $\Ad_{\eta(g)} = \w{g_\bX \rtimes \Ad_g}$ sends $\w\cD(\bX,H)$ into $\w\cD(\bX, gHg^{-1})$, and 
\[\w{g}_{\bU,H} \circ \varphi^H(\bU) = \varphi^{gHg^{-1}}(g\bU) \circ \Ad_{\eta(g)}.\]
by the uniqueness part of Theorem \ref{wUGfunc}.

(d) This follows directly from the definitions.
\end{proof}

We can now record some useful consequences of Definition \ref{AactsCompatibly}.

\begin{lem}\label{ANAH} Suppose that $A$ acts on $\bX$ compatibly with $G$ and let $H \leq N$ be compact open subgroups of $G$. Then
\be \item $A_N$ is a finitely presented left $A_H$-module, and
\item the multiplication map of $\w\cD(-,H)$---$A_N$-bimodules on $\bX_w/N$
\[ \w\cD(-,H) \underset{A_H}{\otimes}{} A_N \longrightarrow \w\cD(-,N)\]
is an isomorphism.
\ee\end{lem}
\begin{proof} (a) This follows from Definition \ref{AactsCompatibly}(a) because $H$ has finite index in $N$. 

(b) Fix $\bU \in \bX_w/N$ and consider the following commutative diagram:
\[ \xymatrix{  \w\cD(\bU,H) \underset{A_H}{\otimes}{} A_N \ar[rr] && \w\cD(\bU,N) \\
\w\cD(\bU,H) \underset{A_H}{\otimes}{} ( A_H \underset{K[H]}{\otimes}{} K[N]) \ar[u] \ar[rr] && \w\cD(\bU,H) \underset{K[H]}{\otimes}{} K[N]. \ar[u] }\]
The top horizontal arrow is induced by the commutative diagram in Definition \ref{AactsCompatibly}(b), and is therefore a morphism of $\w\cD(\bU,H)$--$A_N$-bimodules. The vertical arrow on the left is an bijection by Definition \ref{AactsCompatibly}(a), whereas the vertical arrow on the right is a bijection by Proposition \ref{XaffCompat}. The result follows.
\end{proof}

Following \cite[\S 6]{ST}, we make the following
\begin{defn}\label{NonCptCoadm} Suppose that $A$ acts on $\bX$ compatibly with $G$. We say that the $A$-module $M$ is \emph{coadmissible} if it is coadmissible as an $A_H$-module for some compact open subgroup $H$ of $G$. 
\end{defn}

It follows from Lemma \ref{ANAH}(a) that if the $A$-module $M$ is coadmissible, then it is coadmissible as an $A_H$-module for \emph{every} compact open subgroup $H$ of $G$. We will write $\cC_A$ to denote the full subcategory of $A$-modules consisting of the coadmissible $A$-modules.

\subsection{The localisation functor \ts{\Loc^A_\bX}} 
\label{LocXSect}
Throughout $\S \ref{LocXSect}$, we will assume that:
\begin{itemize}
\item $G$ is a $p$-adic Lie group, not necessarily compact,
\item $A$ acts on $\bX$ compatibly with $G$, and
\item $M$ is a coadmissible $A$-module.
\end{itemize}
Suppose that $(\bU, H)$ is small. Then $M$ is a coadmissible $A_H$-module, and because ${\varphi^H(\bU) : A_H \to \w\cD(\bU,H)}$ is a continuous homomorphism between two Fr\'echet-Stein algebras by Theorem \ref{FrSt}, $\w\cD(\bU,H)$ is a $\w\cD(\bU,H)$-coadmissible $\w\cD(\bU,H)-A_H$-bimodule in the sense of \cite[Definition 7.3]{DCapOne}. It follows from \cite[Lemma 7.3]{DCapOne} that in this situation we may form the coadmissible $\w\cD(\bU,H)$-module 
\[\w\cD(\bU,H) \underset{A_H}{\w\otimes} M.\]

\begin{defn}\label{LocFuncDefn} Whenever $(\bU,H)$ is small, we define
\begin{equation}\label{LUH}M(\bU,H) := \w\cD(\bU,H) \underset{A_H}{\w\otimes} M.\end{equation}
\end{defn}
 Recall that because we are assuming throughout $\S \ref{LocXSect}$ that $G$ is acting continuously on $\bX$ in the sense of Definition \ref{CtsAct}, the stabiliser $G_\bU$ in $G$ of every affinoid subdomain $\bU$ of $\bX$ is an \emph{open} subgroup of $G$. We have the following basic functorialities of this construction. 

\begin{prop}\label{BiFunc} Let $H$ be a compact open subgroup of $G$ and let $\bU \in \bX_w(\cT)$.
\be \item $M(\bU,-)$ is a covariant functor on the $\bU$-small subgroups of $G$.
\item $M(-,H)$ is a contravariant functor on $\bX_w(\cT)/H$.
\item  Let $H \leq N$ be compact open subgroups of $G$ and let $\bV \subseteq \bU$ be members of $\bX_w(\cT)/N$. Then the natural diagram of $\w\cD(\bU,H)$-modules 
\begin{equation}\label{MUHbifunc} \xymatrix{ M(\bU,H) \ar[r]\ar[d] & M(\bU,N) \ar[d] \\ M(\bV,H) \ar[r] & M(\bV,N) }\end{equation}
obtained from parts (a) and (b) is commutative.
\ee\end{prop}
\begin{proof}(a) Let $J \leq N$ be $\bU$-small subgroups of $G$. By Definition \ref{AactsCompatibly}(b), there is a commutative diagram of $K$-algebra homomorphisms
\[\xymatrix{ A_J \ar[rr]^{\varphi^J(\bU)}\ar[d] && \w\cD(\bU,J) \ar[d] \\ A_N \ar[rr]_{\varphi^N(\bU)} && \w\cD(\bU,N), }\]
Hence the map $\w\cD(\bU,J) \times M \to M(\bU,N)$ given by $(a,m) \mapsto a \w\otimes m$ is $A_J$-balanced and left $\w\cD(\bU,J)$-linear. Since $\w\cD(\bU,N)$ is a finitely presented $\w\cD(\bU,J)$-module by Proposition \ref{XaffCompat},  $M(\bU,N)$ is a coadmissible $\w\cD(\bU,J)$-module by \cite[Lemma 3.8]{ST}, so this map extends uniquely to a $\w\cD(\bU,J)$-linear map 
\[ M(\bU,J) = \w\cD(\bU,J)\underset{A_J}{\w\otimes}M \longrightarrow \w\cD(\bU,N) \underset{A_N}{\w\otimes}M = M(\bU,N)\]
by the universal property of $\w\otimes$ given in \cite[Lemma 7.3]{DCapOne}. The uniqueness of this map makes it easy to see that the triangle
\[ \xymatrix{ M(\bU,J) \ar[r]\ar[dr] & M(\bU,N) \ar[d] \\ & M(\bU,N') }\]
commutes whenever $N'$ is a third open subgroup of $G_\bU$ containing $N$.

(b) Let $\bV \subseteq \bU$ be $H$-stable affinoid subdomains of $\bX$. By Definition \ref{AactsCompatibly}, we have another commutative diagram of $K$-algebra homomorphisms
\[ \xymatrix{ A_H \ar[rr]^{\varphi^H(\bU)}\ar[drr]_{\varphi^H(\bV)} && \w\cD(\bU,H) \ar[d] \\ && \w\cD(\bV,H). }\]
Now the universal property of $\w\otimes$ gives a $\w\cD(\bV,H)$-linear, $\w\cD(\bU,H)$-balanced map
\[ \iota : \w\cD(\bV,H) \hsp \times \hsp M(\bU,H) \longrightarrow \w\cD(\bV,H) \underset{\w\cD(\bU,H)}{\w\otimes} M(\bU,H),\]
whereas \cite[Corollary 7.4]{DCapOne} gives a $\w\cD(\bV,H)$-linear isomorphism
\[\theta : \w\cD(\bV,H) \underset{\w\cD(\bU,H)}{\w\otimes} M(\bU,H) \longrightarrow M(\bV,H).\]
Therefore we obtain a $\w\cD(\bU,H)$-linear map $\theta \circ \iota(1,-) : M(\bU,H) \to M(\bV,H)$ which fits into the commutative diagram
\[ \xymatrix{ M(\bU,H) \ar[r]\ar[dr] & M(\bV,H) \ar[d] \\ & M(\bW,H) }\]
whenever $\bW$ is a third $H$-stable affinoid subdomain of $\bX$ contained in $\bV$. 

(c) Theorem \ref{wUGfunc} gives a commutative diagram of Fr\'echet-Stein algebras and continuous $K$-algebra homomorphisms
\[\xymatrix{ \w\cD(\bU,H) \ar[r]\ar[d]\ar[d] & \w\cD(\bU,N) \ar[d] \\ \w\cD(\bV,H) \ar[r] & \w\cD(\bV,N).}\]
It is now straightforward to verify that diagram $(\ref{MUHbifunc})$ is commutative.
\end{proof}

\begin{defn}\label{PreLoc} For every $\bU \in \bX_w(\cT)$, define
\[ \cP^A_\bX(M)(\bU) := \invlim M(\bU,H)\]
where in the inverse limit, $H$ runs over all the $\bU$-small subgroups of $G$.
\end{defn}

It is clear that $\cP^A_\bX(M)(\bU)$ is a $\cD(\bU)$-module.

\begin{lem}\label{DefnOfRestMaps} For every $\bV,\bU \in \bX_w(\cT)$ such that $\bV \subset \bU$, there is a $\cD(\bU)$-linear restriction map $\tau^\bU_\bV : \cP^A_\bX(M)(\bU) \longrightarrow \cP^A_\bX(M)(\bV)$.
\end{lem}
\begin{proof} Let $N$ be a $\bV$-small subgroup of $G$. Choose a $\bU$-small subgroup $H$ of $N_{\bU}$ using Lemma \ref{SmallPairsExist}. There is a natural $\cD(\bU)$-linear map
\[ \tau^\bU_{\bV,N} : \cP^A_\bX(M)(\bU) \to M(\bV,N)\]
which factors through the maps $M(\bU, H) \to M(\bV,H) \to M(\bV,N)$ given by Proposition \ref{BiFunc}(a,b), and does not depend on the choice of $H$. If $N' \leq N$ is another $\bV$-small subgroup and $H' := N'_{\bU} \cap H$, then $H'$ is a $\bU$-small open subgroup of $N'_{\bU}$ by Lemma \ref{SubsOfSmall}, and the diagram
\[ \xymatrix{ M(\bU,H) \ar[rr] && M(\bV, N_{\bU}) \ar[rr] && M(\bV,N) \\ M(\bU, H') \ar[u]\ar[rr] && M(\bV, N'_{\bU}) \ar[rr]\ar[u] && M(\bV,N') \ar[u] }\]
commutes by Proposition \ref{BiFunc}(c,a). Hence the triangle
\[ \xymatrix{ \cP^A_\bX(M)(\bU) \ar[rr]^{\tau^\bU_{\bV,N}} \ar[rrd]_{\tau^\bU_{\bV,N'}} && M(\bV,N) \\ && M(\bV,N') \ar[u] }\]
is commutative, and it induces the required $\cD(\bU)$-linear map
\[ \tau^\bU_\bV: \cP^A_\bX(M)(\bU) \longrightarrow \cP^A_\bX(M)(\bV) = \invlim M(\bV,N)\]
by applying the universal property of inverse limit.
\end{proof}

In fact, each arrow in the inverse system defining $\cP^A_\bX(M)$ is an isomorphism. This follows from our next result.

\begin{prop}\label{RestrictFurther} Let $H \leq N$ be compact open subgroups of $G$, and let $\bU \in \bX_w(\cT)/N$. Then for every coadmissible $A$-module $M$, the natural map
\[ M(\bU,H) \longrightarrow M(\bU,N)\]
from Proposition \ref{BiFunc}(a) is an \emph{isomorphism} of coadmissible $\w\cD(\bU,H)$-modules.
\end{prop}
\begin{proof} By Lemma \ref{ANAH}(b), the natural map
\[\alpha : \w\cD(\bU,H) \otimes_{A_H} A_N \longrightarrow \w\cD(\bU,N)\]
is an isomorphism of $\w\cD(\bU,H)$-$A_N$-bimodules.  On the other hand, there is an obvious isomorphism $\beta : A_N \underset{A_N}{\w\otimes} M \longrightarrow M$ of coadmissible left $A_H$-modules, and $A_N$ is a coadmissible $A_H-A_N$-bimodule in the sense of \cite[Definition 7.3]{DCapOne}. These maps combine to produce a commutative diagram
\[\xymatrix{  \w\cD(\bU,H) \underset{A_H}{\w\otimes} \left(A_N \underset{A_N}{\w\otimes} M\right) \ar[rr]^(0.6){1 \w\otimes \beta}\ar[d]_{\cong} &&  \w\cD(\bU,H) \underset{A_H}{\w\otimes} M \ar[d] \\
\left(\w\cD(\bU,H) \underset{A_H}{\w\otimes} A_N\right) \underset{A_N}{\w\otimes} M \ar[rr]_(0.6){\alpha \w\otimes 1} && \w\cD(\bU,N) \underset{A_N}{\w\otimes} M
}\]
where the vertical map on the left is the canonical associativity isomorphism given by \cite[Proposition 7.4]{DCapOne}. The result follows, because $\alpha$ and $\beta$ are isomorphisms.
\end{proof}

\begin{cor}\label{ResCor} Whenever $(\bU,H)$ is small, the canonical map $\cP^A_\bX(M)(\bU) \to M(\bU,H)$ is a bijection.
\end{cor}

Next, we study the $G$-equivariant functoriality of $M(-,-)$.
\begin{prop}\label{GEquivFunc} Suppose that $(\bU,H)$ is small and let $g \in G$.
\be
\item There is a $K$-linear map
\[ g^M_{\bU,H} : M(\bU,H) \longrightarrow M(g\bU, gHg^{-1})\]
such that for every $a \in \w\cD(\bU,H)$ and every $m \in M(\bU,H)$, we have 
\[ g^M_{\bU,H}(a \cdot m) = \w{g}_{\bU,H}(a) \cdot g^M_{\bU,H}(m).\]
\item Whenever $N$ is an open subgroup of $H$ and $\bV$ is an $N$-stable affinoid subdomain of $\bU$, the diagram
\begin{equation}\label{MUHequiv}
 \xymatrix{  M(\bU,H) \ar[rr]^{g^M_{\bU,H}} && M(g\bU, gHg^{-1}) \\
  M(\bU,N) \ar[rr]^{g^M_{\bU,N}}\ar[d]\ar[u] && M(g\bU, gNg^{-1}) \ar[d]\ar[u] \\ 
  M(\bV,N) \ar[rr]_{g^M_{\bV,N}} && M(g\bV, gNg^{-1})  } 
 \end{equation}
where the vertical arrows are given by Proposition \ref{BiFunc}, is commutative.
\ee\end{prop}
\begin{proof}(a) We will regard the coadmissible left $\w\cD(g\bU, gHg^{-1})$-module
\[ M(g\bU, gHg^{-1}) = \w\cD(g\bU, gHg^{-1}) \underset{A_{gHg^{-1}}}{\w\otimes} M\]
as a coadmissible left $\w\cD(\bU,H)$-module via the $K$-algebra isomorphism $\w{g}_{\bU,H}$. Now, consider the map
\[\begin{array}{cccc} \psi^M : & \w\cD(\bU,H) \times M & \longrightarrow & M(g\bU, gHg^{-1})  \\
& (a,m) &\mapsto& \w{g}_{\bU,H}(a) \hsp \w\otimes \hsp \eta(g)\cdot m \end{array}\]
It is evidently left $\w\cD(\bU,H)$-linear; we will show that $\psi^M$ is $A_H$-balanced. By Definition \ref{AactsCompatibly}(c), we know that
\[\w{g}_{\bU,H} \circ \varphi^H(\bU) = \varphi^{gHg^{-1}}(g\bU) \circ \Ad_{\eta(g)}.\]
Let $a \in \w\cD(\bU,H), b\in A_H$ and $m \in M$. Then
\[ \begin{array}{lll} \psi^M( a \cdot b, m ) &=& \w{g}_{\bU,H}(a \hsp \varphi^H(b)) \hsp \w\otimes \hsp \eta(g)\cdot m = \\
&=& \w{g}_{\bU,H}(a) \hsp \w{g}_{\bU,H}(\varphi^H(b)) \hsp \w\otimes \hsp \eta(g) \cdot m = \\
&=& \w{g}_{\bU,H}(a) \hsp \varphi^{gHg^{-1}}(\Ad_{\eta(g)}(b)) \hsp \w\otimes \hsp \eta(g) \cdot m = \\
&=& \w{g}_{\bU,H}(a) \hsp \w\otimes \hsp \Ad_{\eta(g)}(b)\cdot (\eta(g)\cdot m) = \\
&=& \w{g}_{\bU,H}(a) \hsp \w\otimes \hsp \eta(g) \cdot (bm) = \\
&=& \psi^M(a, bm), \end{array}\]
so $\psi^M$ is $A_H$-balanced, as claimed. Therefore, by the universal property of $\w\otimes$, $\psi^M$ extends uniquely to a left $\w\cD(\bU,H)$-linear homomorphism
\begin{equation}\label{DefnOfgM}\begin{array}{cccc} g_{\bU,H}^M : M(\bU,H) = & \w\cD(\bU,H) \underset{A_H}{\w\otimes} M & \longrightarrow & M(g\bU, gHg^{-1}) \\
 & \hspace{1.05cm} a \hsp \w\otimes \hsp m &\mapsto & \w{g}_{\bU,H}(a) \hsp \w\otimes \hsp \eta(g)\cdot m \end{array}\end{equation}

Since the left $\w\cD(\bU,H)$-module structure on $M(g\bU,gHg^{-1})$ is given via $\w{g}_{\bU,H}$, we see that $g^M_{\bU,H}(a \cdot m) = \w{g}_{\bU,H}(a) \cdot g^M_{\bU,H}(m)$ for all $a \in \w\cD(\bU,H), m \in M(\bU,H)$ as claimed.

(b) This is straightforward.
\end{proof}
Whenever $\bU_1, \ldots, \bU_m$ is a finite collection of affinoid subdomains of $\bX$, we will use $G_{\bU_1,\ldots,\bU_m}$ to denote the intersections of their stabilisers in $G$:
\[G_{\bU_1,\ldots,\bU_m} := G_{\bU_1} \cap \cdots \cap G_{\bU_m}.\]
Note that this is an open subgroup of $G$ because $G$ acts continuously on $\bX$.

\begin{thm} \label{MUGU}
Let $M$ be a coadmissible $A$-module. Then $\cP^A_\bX(M)$, equipped with the restriction maps $\tau^\bU_\bV$ from Lemma \ref{DefnOfRestMaps}, becomes a $G$-equivariant presheaf of $\cD$-modules on $\bX_w(\cT)$.
\end{thm}
\begin{proof} Let $\bW \subset \bV \subset \bU$ be members of $\bX_w(\cT)$, and write $\cM := \cP^A_\bX(M)$.  Using Lemma \ref{SmallPairsExist}, choose a $\bU$-small compact open subgroup $H$ of $G_{\bU,\bV,\bW}$, and consider the following diagram:
\[ \xymatrix@C=20pt{
& & & M(\bW,H) & & & \\
& & & \cM(\bW)\ar[u] & & & \\
& & \cM(\bU) \ar[ru]^{\tau^\bU_\bW}\ar[rr]_{\tau^\bU_\bV}\ar[dll] && \cM(\bV) \ar[ul]_{\tau^\bV_\bW}\ar[drr] & & \\
M(\bU,H) \ar@/^2pc/[rrruuu]\ar[rrrrrr] &&&&&& M(\bV,H). \ar@/_2pc/[uuulll]
}\]
The large outer triangle commutes by Proposition \ref{BiFunc}(a), and the three arrows connecting the inner triangle with the outer triangle are isomorphisms by Corollary \ref{ResCor}. Hence the inner triangle commutes, and $\cM$ is a presheaf. 

Next, fix $g \in G$, and for every $\bU \in \bX_w(\cT)$ define
\[g^{\cM}(\bU) : \cM(\bU) \to \cM(g\bU)\]
to be the inverse limit of the maps $g^M_{\bU,H} : M(\bU,H) \to M(g\bU, gHg^{-1})$ constructed in Proposition \ref{GEquivFunc}(a). Note that $g^{\cM}(a \cdot m) = g^{\cD}(a) \cdot g^{\cM}(a)$ for all $a \in \cD$ and all $m \in \cM$ by Proposition \ref{GEquivFunc}(a,b). Now let $\bV \subset \bU$ be another affinoid subdomain of $\bX$, choose a $\bU$-small compact open subgroup $H$ of $G_{\bU,\bV}$, and consider the following diagram:
\[\xymatrix{ M(\bU,H) \ar[rrrr]^{g^M_{\bU,H}}\ar[ddd]&&&& M(g\bU, gHg^{-1}) \ar[ddd] \\
& \cM(\bU) \ar[rr]^{g^{\cM}(\bU)} \ar[ul]\ar[d]_{\tau^\bU_\bV} && \cM(g\bU) \ar[d]^{\tau^{g\bU}_{g\bV}}\ar[ur] &\\
& \cM(\bV) \ar[rr]_{g^{\cM}(\bV)}\ar[dl] && \cM(g\bV) \ar[rd]& \\
M(\bV,H) \ar[rrrr]_{g^M_{\bV,H}}&&&& M(g\bV,gHg^{-1}).}\]
In this diagram, the four trapezia commute by definition of the maps $\tau^\bU_\bV$ and $g^{\cM}(\bU)$. The outer square commutes by Proposition \ref{GEquivFunc}(b). Because the diagonal arrows in this diagram are isomorphisms by Corollary \ref{ResCor}, the inner square also commutes. Therefore $g^{\cM} : \cM \to g^\ast \cM$ is a morphism of presheaves on $\bX_w(\cT)$.

Finally, it follows from Lemma \ref{gUHmaps} that for every compact open subgroup $H$ of $G_\bU$ and for every $g,h \in G$, we have
\[ \w{g}_{h\bU, hHh^{-1}} \circ \w{h}_{\bU,H} = \w{gh}_{\bU,H}.\]
Since $M$ is an $A$-module, inspecting $(\ref{DefnOfgM})$ we see that
\[\begin{array}{rcl} g^{\cM}\left(h^{\cM}(a \hsp \w\otimes\hsp  m)\right) &=& g^{\cM}( \w{h}(a) \hsp \w\otimes \hsp \eta(h) \cdot m) =\\
&=& \w{g}(\w{h}(a)) \hsp \w\otimes \hsp \eta(g)\cdot (\eta(h)\cdot m) = \\
&=& \w{gh}(a) \hsp \w\otimes \hsp \eta(gh)\cdot m = (gh)^{\cM}(a) \end{array}\]
for all $a \hsp \w\otimes \hsp m \in \cM$. Thus $\{g^{\cM} : g \in G\}$ is a $G$-equivariant structure on $\cM$.\end{proof}

We have not yet used part (d) of Definition \ref{AactsCompatibly} in our exposition. We will do so crucially in the proof of our next result, which essentially states that our functors $\cP^A_\bX(-)$ enjoy a certain transitivity property.

\begin{prop}\label{LocTrans} Suppose that $(\bU,J)$ is small, write $B := \w\cD(\bU,J)$ and let $N := B \w\otimes_{A_J} M$. Then $N$ is a coadmissible $B$-module, and there is a natural isomorphism
\[ \cP_\bX^A(M)_{|\bU_w}\congs \cP_\bU^B(N) \]
of $J$-equivariant presheaves of $\cD$-modules on $\bU_w$.
\end{prop}
\begin{proof} Note that $\bU_w(\cT) = \bU_w$ because $(\bU,J)$ is small. The algebra $B$ acts on $\bU$ compatibly with $J$ by Proposition \ref{XaffCompat}; here $B_H = \w\cD(\bU,H)$ for any compact open subgroup $H$ of $J$. Let $\bV$ be an affinoid subdomain of $\bU$, and let $H$ be a compact open subgroup of $J_\bV$. Unravelling the definitions, we see that
\[M(\bV,H) = \w\cD(\bV,H) \underset{A_H}{\w\otimes} M \qmb{and} N(\bV,H) = \w\cD(\bV,H) \underset{\w\cD(\bU,H)}{\w\otimes} \left( \w\cD(\bU,J) \underset{A_J}{\w\otimes} M \right).\]
Now consider the following diagram:
\[\xymatrix{
M(\bV,H) \ar[r]\ar@{.>}[d]_{\lambda_{\bV,H}} &\left(\w\cD(\bV,H) \underset{\w\cD(\bU,H)}{\w\otimes} \w\cD(\bU,H)\right) \underset{A_H}{\w\otimes} \left(A_J \underset{A_J}{\w\otimes} M\right)  \ar[d] \\
N(\bV,H)  & \w\cD(\bV,H) \underset{\w\cD(\bU,H)}{\w\otimes} \left( (\w\cD(\bU,H) \underset{A_H}{\w\otimes} A_J) \underset{A_J}{\w\otimes} M\right).  \ar[l]
}\]
The top horizontal map is given by $a \hsp \w\otimes \hsp m \mapsto (a \hsp \w\otimes  \hsp 1 ) \hsp \w\otimes \hsp ( 1 \hsp \w\otimes \hsp m)$; this is clearly an isomorphism. The bottom horizontal map is induced by the isomorphism
\[\w\cD(\bU,H) \underset{A_H}{\w\otimes} A_J \congs \w\cD(\bU,J)\]
from Lemma \ref{ANAH}(b), therefore it is an isomorphism by functoriality. The vertical map on the right is an isomorphism obtained by applying the associativity isomorphism \cite[Proposition 7.4]{DCapOne} three times. Hence there is a left $\w\cD(\bV,H)$-linear isomorphism $\lambda_{\bV,H} : M(\bV,H) \longrightarrow N(\bV,H)$ that makes the diagram commute. This isomorphism is given by
\begin{equation}\label{LambdaFormula} \lambda_{\bV,H}(a \hsp \w\otimes \hsp m) = a \hsp\w\otimes \hsp(1 \hsp\w\otimes \hsp m) \quad\quad a \in \w\cD(\bV,H), m \in M.\end{equation}
Let $\cM = \cP^A_\bX(M)$ and let $\cN = \cP^B_\bU(N)$. Passing to the inverse limit over all $H$ as above, we obtain a $\cD(\bU)$-linear map 
\[\lambda(\bV) := \invlim \lambda(\bV,H) : \cM(\bV) \longrightarrow \cN(\bV)\]
which makes the following square commute:
\[\xymatrix{ \cM(\bV) \ar[rr]^{\lambda(\bV)}\ar[d] && \cN(\bV) \ar[d] \\ 
M(\bV,H) \ar[rr]_{\lambda_{\bV,H}} && N(\bV,H). } \]
Next, let $\bW$ be an affinoid subdomain of $\bV$, choose a compact open subgroup $H$ of $G_{\bU,\bV,\bW}$, and consider the following diagram:
\begin{equation}\label{LambdaPresheafDiag}\xymatrix{ M(\bV,H) \ar[rrrr]^{\lambda_{\bV,H}}\ar[ddd] &&&& N(\bV,H) \ar[ddd] \\
& \cM(\bV) \ar[rr]^{\lambda(\bV)} \ar[d]_{\tau^\bV_\bW(\cM)} \ar[ul] && \cN(\bV) \ar[d]^{\tau^\bV_\bW(\cN)} \ar[ur]&\\
& \cM(\bW) \ar[rr]_{\lambda(\bW)} \ar[dl] && \cN(\bW)\ar[dr] & \\
M(\bW,H) \ar[rrrr]_{\lambda_{\bW,H}}&&&& N(\bW,H).}\end{equation}
where the vertical arrows in the inner square come from Lemma \ref{DefnOfRestMaps}. The four trapezia in this diagram commute by definition, and the outer square commutes because of the formula $(\ref{LambdaFormula})$. Since the diagonal maps are isomorphisms by Corollary \ref{ResCor}, it follows that the inner square is also commutative. Thus
\[ \lambda : \cM_{|\bU_w} \longrightarrow \cN\]
is an isomorphism of presheaves of $\cD$-modules on $\bU_w$. 

Next, we check that $\lambda$ is $J$-equivariant. To see this, let $g \in J$, let $H$ be a compact open subgroup of $J_\bV$ and consider the following diagram:
\[\xymatrix@C=5pt{ 
\w\cD(\bV,H) \underset{A_H}{\w\otimes} M \ar[rrrr]^(.4){\lambda_{\bV,H}} \ar[ddd]_{g^M_{\bV,H}} &&&& \w\cD(\bV,H)\underset{\w\cD(\bU,H)}{\w\otimes} \left( \w\cD(\bU, J)\underset{A_J}{\w\otimes} M\right) \ar[ddd]^{g^N_{\bV,H}} \\
& \cM(\bV) \ar[rr]^{\lambda(\bV)}\ar[d]_{g^{\cM}(\bV)}\ar[ul] && \cN(\bV) \ar[d]^{g^{\cN}(\bV)} \ar[ur]&\\
& \cM(g\bV) \ar[rr]_{\lambda(g\bV)}\ar[dl] && \cN(g\bV)\ar[dr]&\\
\w\cD(g\bV,{}^gH) \underset{A_{{}^gH}}{\w\otimes} M \ar[rrrr]_(.4){\lambda_{g\bV, {}^gH}} &&&& \w\cD(g\bV,{}^gH)\underset{\w\cD(\bU,{}^gH)}{\w\otimes} \left( \w\cD(\bU, J)\underset{A_J}{\w\otimes} M\right). 
}\]
The map $\eta_B : J \to B^\times$ is $\gamma^J : J \to \w\cD(\bU,J)^\times$, by the construction given in the proof of Proposition \ref{XaffCompat}. Hence, for any $a \in \w\cD(\bV,H)$ and any $m \in M$, we have
\[\begin{array}{lll} g^N_{\bV,H}(\lambda_{\bV,H}(a \hsp \w\otimes \hsp m)) &=& g^N_{\bV,H}(a \hsp \w\otimes \hsp( 1 \hsp \w\otimes \hsp m)) = \\
&=& \w{g}_{\bV,H}(a) \hsp\w\otimes\hsp \eta_B(g) \cdot ( 1 \hsp\w\otimes\hsp m) = \\
&=& \w{g}_{\bV,H}(a) \hsp\w\otimes\hsp (\gamma^J(g) \hsp \w\otimes\hsp m) = \\
&=& \w{g}_{\bV,H}(a) \hsp\w\otimes\hsp (\varphi^J( \eta(g) ) \hsp \w\otimes \hsp m) = \\
&=& \w{g}_{\bV,H}(a) \hsp\w\otimes\hsp (1 \hsp\w\otimes \hsp \eta(g) \cdot m) = \\
&=& \lambda_{g\bV, {}^gH}\left( \w{g}_{\bV,H}(a) \hsp\w\otimes\hsp \eta(g) \cdot m\right) = \\
&=& \lambda_{g\bV, {}^gH} ( g^M_{\bV,H}(a \hsp\w\otimes\hsp m) )
\end{array}\]
where on the fourth line we used part (d) of Definition \ref{AactsCompatibly}. Thus the outer square of the above diagram commutes. The diagonal arrows in this diagram are isomorphisms by Corollary \ref{ResCor}, so the inner square is commutative and ${\lambda : \cM_{|\bU_w} \longrightarrow \cN}$ is $J$-equivariant.  \end{proof}

\begin{lem}\label{RestToHonP} Suppose that $(\bX,G)$ is small, and that $M$ is a coadmissible $\w\cD(\bX,G)$-module. Let $H$ be an open subgroup of $G$ and let $N$ be the restriction of $M$ to $\w\cD(\bX,H)$. Then there is a natural isomorphism
\[\cP^{\w\cD(\bX,H)}_\bX(N) \congs \cP^{\w\cD(\bX,G)}_\bX(M)\]
of $H$-equivariant presheaves of $\cD$-modules on $\bX_w$.
\end{lem}
\begin{proof} If $\bU \in \bX_w$ and $J$ is any open subgroup of $H_\bU$, the identity map is a $\w\cD(\bU,J)$-linear isomorphism $M(\bU,J) \to N(\bU,J)$. Passing to the inverse limit over all such $J$ gives us an isomorphism of presheaves $\cP^{\w\cD(\bX,G)}_\bX(M) \stackrel{\cong}{\rightarrow} \cP^{\w\cD(\bX,H)}_\bX(N)$ on $\bX_w$, which is readily checked to be $H$-equivariant and $\cD$-linear.
\end{proof}

We will now state the main result of $\S$ \ref{LocXSect}.

\begin{thm}\label{DGsheaf} Let $\bU \in \bX_w(\cT)$. Then $\cP^A_\bX(M)_{|\bU_w}$ is a sheaf on $\bU_w$ with vanishing higher \v{C}ech cohomology, for every coadmissible $A$-module $M$.
\end{thm}

We postpone the proof until $\S \ref{TateFlatKiehl}$. Because $\bX$ is smooth, the tangent sheaf $\cT$ is locally free. This means that $\bX_w(\cT)$ forms a \emph{basis} for $\bX$, and Theorem \ref{DGsheaf} implies that $\cP^A_\bX(M)$ is a sheaf on $\bX_w(\cT)$ in the sense of \cite[\S 9.1]{DCapOne}. We can now apply a general result (see \cite[Theorem 9.1]{DCapOne}) to deduce that $\cP^A_\bX(M)$ extends uniquely to a sheaf on $\bX_{\rig}$, the strong $G$-topology of $\bX$. 

\begin{defn}\label{DefnOfLocM} We define $\Loc_\bX^A(M)$ be the unique sheaf on $\bX_{\rig}$ whose restriction to $\bX_w(\cT)$ is the presheaf $\cP^A_\bX(M)$.
\end{defn}

It is straightforward to see that $\Loc_\bX^A(M)$ is in fact a $G$-equivariant sheaf of $\cD$-modules on $\bX_{\rig}$. 

\subsection{Coadmissible equivariant \ts{\cD}-modules}\label{CoadEqDmodSec} We continue to assume that $\bX$ is a smooth rigid analytic variety, and that $G$ is a not necessarily compact $p$-adic Lie group acting continuously on $\bX$. In $\S\ref{CoadEqDmodSec}$, we will give a `purely local' definition of coadmissible $G$-equivariant $\cD$-modules on $\bX$, by gluing together the categories of coadmissible $\w\cD(\bU,H)$-modules, as $\bU$ varies over all affinoid subdomains in $\bX_w(\cT)$ and $H$ varies over all possible $\bU$-small subgroups of $G$. To do this correctly, it turns out that we need the local sections of morphisms between the $G$-equivariant $\cD$-modules of interest to be \emph{continuous}. This, in turn, necessitates keeping track of certain topologies on certain local sections of these $G$-equivariant $\cD$-modules --- see Remark \ref{TopRmk}(d) below. The following framework will turn out to be convenient for our purposes.

\begin{defn}\label{FrechDmod}  Recall that $\bX_w(\cT)$ is the set of affinoid subdomains $\bU$ of $\bX$ such that $\cT(\bU)$ admits a free $\cA$-Lie lattice for some affine formal model $\cA$ in $\cO(\bU)$.
\be\item We say that a $G$-equivariant $\cD$-module $\cM$ on $\bX_{\rig}$ is \emph{locally Fr\'echet} if
\begin{itemize} \item $\cM(\bU)$ is equipped with a Fr\'echet topology for every $\bU \in \bX_w(\cT)$,
\item the maps $g^{\cM}(\bU) : \cM(\bU) \to \cM(g\bU)$ are continuous, whenever $\bU \in \bX_w(\cT)$ and $g \in G$.
\end{itemize}
\item A \emph{morphism of $G$-equivariant locally Fr\'echet $\cD$-modules} is a morphism ${f : \cM \to \cN}$ of $G$-equivariant $\cD$-modules, such that $f(\bU) : \cM(\bU) \to \cN(\bU)$ is continuous for every $\bU \in \bX_w(\cT)$. We call such morphisms \emph{continuous}.
\ee\end{defn}

\begin{notn}\label{FrechNotn} We denote the category whose objects are $G$-equivariant locally Fr\'echet $\cD$-modules on $\bX$ and whose morphisms are continuous maps by 
\[\Frech(G-\cD_\bX).\]
We denote the forgetful functor to $G$-equivariant $\cD$-modules on $\bX$ by 
\[ \Phi : \Frech(G-\cD_\bX) \to \Emod{G}{\cD}.\]
\end{notn}

\begin{rmks}\label{TopRmk}\hspace{2em}
\be \item Note that we are not equipping $\cM(\bU)$ with a Fr\'echet topology for \emph{every} admissible open subset $\bU$ of $\bX$: the choice of $\bX_w(\cT)$ is the most economical because Theorem \ref{DGsheaf} tells us that we know what the restrictions of our sheaf $\Loc^A_\bX(M)$ to $\bU \in \bX_w(\cT)$ look like. Instead of $\bX_w(\cT)$, we could also have chosen to work with quasi-compact admissible open subsets. Because the category of Fr\'echet spaces is stable only under \emph{countable} and not \emph{arbitrary} projective limits, which are needed to compute local sections of admissible open subsets that may be too large to admit countable admissible affinoid coverings, and it is certainly not stable under arbitrary filtered inductive limits, which arise in the sheafification process. Of course, instead of Fr\'echet spaces, we could have chosen to work with a category of topological vector spaces that is large enough to be stable under every such limit, such as locally convex vector spaces, or perhaps instead with the category of bornological vector spaces. Since we are not particularly interested in the topology of the local sections of our  coadmissible $G$-equivariant $\cD$-modules over arbitrary admissible open subsets in this paper, to keep matters as simple as possible, we will not consider any topology on these spaces of local sections.

\item The category $\Frech(G - \cD_\bX)$ is additive, and admits kernels. Other axioms of abelian categories, even such basic ones as the existence of cokernels, do not follow from the weak axiomatic framework of Definition \ref{FrechDmod}: one problem is that the $\cD(\bU)$-action on $\cM(\bU)$ is not assumed to be separately continuous. However, we will see that the category of \emph{coadmissible} $G$-equivariant $\cD$-modules \emph{is} abelian.

\item The real technical reason why we need to keep track of these topologies is as follows. In Definition \ref{CoadmFrechEqDmod} below, we define the category $\cC_{\bX /G}$ of coadmissible $G$-equivariant $\cD$-modules on $\bX$. One of its main properties should be that $\Gamma(\bU,-)$ takes values in coadmissible $\w\cD(\bU,G_\bU)$-modules for every $\bU \in \bX_w(\cT)$. Now even assuming we can show that $\cM(\bU)$ is a coadmissible $\w\cD(\bU,G_\bU)$-module for every $\cM \in \cC_{\bX /G}$, a morphism $\varphi : \cM \to \cN$ of $G$-equivariant $\cD$-modules will at best \emph{a priori} only give rise to a $\cD(\bU) \rtimes G_\bU$-linear map $\cM(\bU) \to \cN(\bU)$, which need \emph{not} be $\w\cD(\bU,G_\bU)$-linear, in general. However, because $\cD(\bU)\rtimes G_\bU$ is dense in $\w\cD(\bU,G_\bU)$, a $\cD(\bU) \rtimes G_\bU$-linear morphism between two coadmissible $\w\cD(\bU,G_\bU)$-modules is $\w\cD(\bU,G_\bU)$-linear if and only if it is continuous with respect to the canonical topologies on $\cM(\bU)$ and $\cN(\bU)$. 

\item We do \emph{not} require the restriction maps in, nor the $\cD$-module structure on, a $G$-equivariant locally Fr\'echet $\cD$-module $\cM$ to be continuous. This is done to keep the exposition as simple as possible: in fact our coadmissible $G$-equivariant $\cD$-modules \emph{do} have continuous $\cD$-module structure, but these are consequences of the far more stringent condition of being \emph{coadmissible}. 
\ee 
\end{rmks}

\begin{rmk}\label{GeneralEqRestToSubgps} If $\cM$ is a $G$-equivariant $\cD$-module on $\bX$ (respectively, $\bX_w$ or $\bX_w(\cT)$) and $H$ is any closed subgroup of $G$, then we can always regard $\cM$ as an $H$-equivariant $\cD$-module on $\bX$ (respectively, $\bX_w$ or $\bX_w(\cT)$). In this way we obtain the restriction functor 
\[ \Res^G_H : \Frech(G-\cD_\bX) \longrightarrow \Frech(H-\cD_\bX)\]
and its obvious analogues on $\bX_w$ and $\bX_w(\cT)$.
\end{rmk}

\begin{lem}\label{AutoCts}  Let $A$ be a Fr\'echet-Stein algebra. Then any $A$-linear map $M \to N$ between two co-admissible $A$-modules is automatically continuous.
\end{lem}
\begin{proof}This follows from \cite[Proposition 2.1(iii) and Corollary 3.3]{ST} --- see the remarks appearing immediately after the proof of \cite[Lemma 3.6]{ST}.
\end{proof}

\begin{prop}\label{LocFunctor}
 Suppose that $A$ acts on $\bX$ compatibly with $G$. Then $\Loc^A_\bX$ is a functor from coadmissible $A$-modules to $G$-equivariant locally Fr\'echet $\cD$-modules on $\bX$.
\end{prop}
\begin{proof} Fix the coadmissible $A$-module $M$, and recall the $G$-equivariant presheaf $\cM := \cP^A_\bX(M)$ of $\cD$-modules on $\bX_w(\cT)$ from Theorem \ref{MUGU}. Let $\bU \in \bX_w(\cT)$, and choose a $\bU$-small subgroup $J$ of $G$. There is a canonical isomorphism $\cM(\bU) \cong M(\bU,J) = \w\cD(\bU,J) \w\otimes_{A_J}M$ by Corollary \ref{ResCor}, so $\cM(\bU)$ carries a canonical Fr\'echet topology, being a coadmissible $\w\cD(\bU,J)$-module. This topology is independent of the choice of $J$. Now, let $g \in G$ and consider the map $g^M_{\bU,J} : M(\bU,J) \to M(g\bU, gJg^{-1})$.  By Proposition \ref{GEquivFunc}(a), this map satisfies $g^M_{\bU,J} ( a \cdot m ) = \w{g}_{\bU,J}(a) \cdot (g \cdot m)$ for all $a \in \w\cD(\bU,J)$ and $m \in M$, so when we regard $\cM(g\bU)$ as a coadmissible $\w\cD(\bU,J)$-module via the continuous ring isomorphism $\w{g}_{\bU,J}$ from Lemma \ref{gUHmaps}, $g^M_{\bU,J}$ becomes a $\w\cD(\bU,J)$-linear map between two coadmissible $\w\cD(\bU,J)$-modules. Hence $g^\cM(\bU)$ is continuous by Lemma \ref{AutoCts}. The restriction of the functor $\Loc^A_\bX(M)$ to $\bX_w(\cT)$ is  $\cM$ by Theorem \ref{DGsheaf}, so $\Loc^A_\bX(M)$ is a $G$-equivariant locally Fr\'echet $\cD$-module on $\bX$. 

Now if $f : M \to N$ is an $A$-linear map between two coadmissible $A$-modules, then for any $\bV \in \bX_w(\cT)$ and any $\bV$-small subgroup $H$ of $G$, the functoriality of $\w\otimes$ induces a $\w\cD(\bV,H)$-linear map $1 \w\otimes f: M(\bV,H) \to N(\bV,H)$, which is continuous by Lemma \ref{AutoCts}.  In this way we obtain a $G$-equivariant morphism of presheaves $\cP^A_\bX(M) \to \cP^A_\bX(N)$ on $\bX_w(\cT)$ whose local sections are continuous, and after applying \cite[Theorem 9.1]{DCapOne}, we get a morphism $\Loc^A_\bX(f) : \Loc^A_\bX(M) \to \Loc^A_\bX(N)$ of $G$-equivariant locally Fr\'echet $\cD_\bX$-modules. It is straightforward to verify that $\Loc^A_\bX(g\circ f) = \Loc^A_\bX(g) \circ \Loc^A_\bX(f)$ whenever $g : N \to N'$ is another $A$-linear map to a coadmissible $A$-module $N$.
\end{proof}

We now come to another central definition of this paper.

\begin{defn}\label{CoadmFrechEqDmod} Let $\cM$ be a $G$-equivariant locally Fr\'echet $\cD$-module on $\bX$.
\be \item Let $\cU$ be an $\bX_w(\cT)$-covering. We say that $\cM$ is \emph{$\cU$-coadmissible} if for each $\bU \in \cU$ there is a $\bU$-small subgroup $H$ of $G$, a coadmissible $\w\cD(\bU,H)$-module $M$, and an isomorphism 
\[\Loc^{\w\cD(\bU,H)}_\bU(M) \congs \cM_{|\bU_{\rig}}\] 
of $H$-equivariant locally Fr\'echet $\cD$-modules on $\bU$. 
\item We say that $\cM$ is \emph{coadmissible} if it is $\cU$-coadmissible for some $\bX_w(\cT)$-covering $\cU$.
\item We denote the full subcategory of $\Frech(G-\cD_\bX)$ consisting of coadmissible $G$-equivariant locally Fr\'echet $\cD$-modules by 
\[\cC_{\bX /G}.\]
\ee\end{defn}

We will abbreviate the term ``coadmissible $G$-equivariant locally Fr\'echet $\cD$-module" to just ``coadmissible $G$-equivariant $\cD$-module". As one may expect, we have the following canonical source of examples of such modules.

\begin{prop}\label{LocFunCoadm} Suppose that $A$ acts on $\bX$ compatibly with $G$. Then the functor $\Loc^A_\bX$ from coadmissible $A$-modules to $\Frech(G-\cD_\bX)$ takes values in $\cC_{\bX / G}$.
\end{prop}
\begin{proof} Choose any $\bX_w(\cT)$-covering $\cU$ of $\bX$, fix $\bU \in \cU$ and choose a $\bU$-small subgroup $J$ of $G$ using Lemma \ref{SmallPairsExist}. Let $M$ be a coadmissible $A$-module and let $N$ be the coadmissible $\w\cD(\bU,J)$-module $\w\cD(\bX,J) \w\otimes_{A_J}M$. By \cite[Theorem 9.1]{DCapOne}, the isomorphism $\cP^A_\bX(M)_{|\bU_w} \congs \cP^{\w\cD(\bX,J)}_\bU(N)$ of $J$-equivariant presheaves of $\cD$-modules on $\bU_w$ from Proposition \ref{LocTrans} extends uniquely to an isomorphism $\Loc^A_\bX(M)_{|\bU_w} \congs \Loc^{\w\cD(\bX,J)}_J(N)$ of $J$-equivariant $\cD$-modules on $\bU$. This isomorphism is continuous by Lemma \ref{AutoCts}, so we see that $\Loc^A_\bX(M)$ is $\cU$-coadmissible. 
\end{proof}

Coadmissible $G$-equivariant $\cD$-modules behave well under refinements.

\begin{lem}\label{C/Grefinement} Let $\cM$ be a $G$-equivariant locally Fr\'echet $\cD$-module on $\bX$, let $\cU$ be an $\bX_w(\cT)$-covering of $\bX$ and let $\cV$ be $\bX_w(\cT)$-refinement of $\cU$. Suppose that $\cM$ is $\cU$-coadmissible. Then $\cM$ is also $\cV$-coadmissible.\end{lem}
\begin{proof} Let $\bV \in \cV$. Since $\cV$ is a refinement of $\cU$, we can find some $\bU \in \cU$ which contains $\bV$. Since $\cM$ is coadmissible, there is a $\bU$-small subgroup $H$ of $G$, a coadmissible $B := \w\cD(\bU,H)$-module $M$ and an isomorphism $\Loc^B_\bU(M) \congs \cM_{|\bU_{\rig}}$ of $H$-equivariant locally Fr\'echet $\cD$-modules on $\bU$. Choose an open subgroup $J$ of $H_\bV$. This restricts to a continuous isomorphism  $\cP^B_\bU(M)_{|\bV_w} \congs \cM_{|\bV_w}$ of $J$-equivariant $\cD$-modules on $\bV_w$. On the other hand, letting $C := \w\cD(\bV,J)$ and $N := C \w\otimes_B M$ and applying Proposition \ref{LocTrans} gives us an isomorphism $\cP^B_\bU(M)_{|\bV_w} \congs \cP^C_\bV(N)$ of $J$-equivariant $\cD$-modules on $\bV_w$. The local sections of this isomorphism are automatically continuous by Lemma \ref{AutoCts}, so by combining these two isomorphisms we obtain a continuous isomorphism $\cP^C_\bV(N) \congs \cM_{|\bV_w}$ of $J$-equivariant $\cD$-modules on $\bV_w$. Applying \cite[Theorem 9.1]{DCapOne} gives a continuous isomorphism $\Loc^{\w\cD(\bV,J)}_\bV(N) \congs \cM_{|\bV_{\rig}}$ of $J$-equivariant $\cD$-modules on $\bV$.
\end{proof}

Coadmissible $G$-equivariant $\cD$-modules behave well under restriction.

\begin{prop}\label{FrechRestr} Let $\cM \in \cC_{\bX /G}$.
\be \item Let $H$ be an open subgroup of $G$. Then $\cM \in \cC_{\bX /H}$.
\item Let $\bY$ be an admissible open subset of $\bX$. Then $\cM_{|\bY} \in \cC_{\bY / G_\bY}$.
\ee
\end{prop}
\begin{proof} (a) The problem reduces in a straightforward way to the case where $(\bX,G)$ is small and $\cM = \Loc^{\w\cD(\bX,G)}_\bX(M)$ for some coadmissible $\w\cD(\bX,G)$-module $M$. Let $N$ be the restriction of $M$ to $\w\cD(\bX,H)$, a coadmissible $\w\cD(\bX,H)$-module by Lemma \ref{RestToHonP}. The isomorphism $\cP^{\w\cD(\bX,G)}_\bX(M) \stackrel{\cong}{\rightarrow} \cP^{\w\cD(\bX,H)}_\bX(N)$ on $\bX_w$ from Lemma \ref{RestToHonP} is continuous by Corollary \ref{ResCor} and Lemma \ref{AutoCts}, so its canonical extension to $\bX_{\rig}$ an isomorphism $\Loc^{\w\cD(\bX,G)}_\bX(M) \stackrel{\cong}{\rightarrow} \Loc^{\w\cD(\bX,H)}_\bX(N)$ of $H$-equivariant locally Fr\'echet $\cD$-modules on $\bX$. Hence $\cM$ is a coadmissible $H$-equivariant $\cD$-module on $\bX$.

(b) It follows from Remark \ref{GeneralEqRestToSubgps} that $\cM$ is a $G_\bY$-equivariant locally Fr\'echet $\cD$-module on $\bX$. Since $\bY$ is admissible open in $\bX$,  $\bY_w(\cT)$  is a subset of $\bX_w(\cT)$ so $\cM_{|\bY}$ is also a $G_\bY$-equivariant locally Fr\'echet  $\cD$-module on $\bY$, and it remains to see that $\cM_{|\bY}$ is coadmissible. 

Suppose that $\cM$ is $\cU$-coadmissible for some $\bX_w(\cT)$-covering $\cU$. For every $\bU \in \cU$, choose an admissible affinoid covering $\cV_\bU$ of $\bY \cap \bU$ and let $\cV = \cup_{\bU \in \cU} \cV_\bU$. Then $\cV$ is a $\bY_w(\cT)$-covering of $\bY$ which refines $\cU$. In this way we reduce the problem to the case where both $\bY$ and $\bX$ are affinoid, $(\bX,G)$ is small, and $\cM = \Loc^{\w\cD(\bX,G)}_\bX(M)$ for some coadmissible $\w\cD(\bX,G)$-module $M$. Let $H := G_\bY$, an open subgroup of $G$. It now follows from Proposition \ref{LocTrans} that
\begin{equation}\label{RestOfLocShvs}\cM_{|\bY} = \Loc_\bX^{\w\cD(\bX,G)}(M)_{|\bY} \cong \Loc^{\w\cD(\bY,H)}_\bY\left(\w\cD(\bY,H)\underset{\w\cD(\bX,H)}{\w\otimes} M\right)\end{equation}
as $H$-equivariant locally Fr\'echet  $\cD$-modules on $\bY$.
\end{proof}

\begin{thm}\label{LocEquiv} Suppose that $(\bX,G)$ is small. Then the localisation functor
\[\Loc^{\w\cD(\bX,G)}_\bX : \cC_{\w\cD(\bX,G)} \longrightarrow \cC_{\bX/G}\]
is an equivalence of categories.
\end{thm}
\begin{proof} Let $M, N$ be coadmissible $\w\cD(\bX,G)$-modules, and write $\Loc := \Loc^{\w\cD(\bX,G)}_\bX$. By Corollary \ref{ResCor} and Theorem \ref{DGsheaf}, we can identify $M(\bX,G)$ with $\Loc(M)(\bX)$, functorially in $M$. Therefore, for any $\w\cD(\bX,G)$-linear morphism $f : M \to N$ we have the commutative diagram
\[ \xymatrix{  M \ar[rrr]^f\ar[d] &&& N \ar[d]\\ 
M(\bX,G) \ar[rrr]_{\Loc(f)(\bX)} &&& N(\bX,G).
}\]
Because the map $N \to N(\bX,G)$ which sends $n$ to $1 \hsp\w\otimes\hsp n$ is an isomorphism, we see that $\Loc$ is faithful. 

Suppose next that $\alpha : \Loc(M) \to \Loc(N)$ is a morphism of $G$-equivariant locally Fr\'echet $\cD$-modules on $\bX$. Now $\alpha(\bX) : \Loc(M)(\bX) \to \Loc(N)(\bX)$ is $\cD(\bX) \rtimes G$-linear by Proposition \ref{EqGlSec}, and continuous by Definition \ref{FrechDmod}(b), so it must also be $\w\cD(\bX,G)$-linear because $M(\bX,G)$ and $N(\bX,G)$ are both coadmissible $\w\cD(\bX,G)$-modules. Letting $f : M \to N$ be defined by the commutative diagram
\[ \xymatrix{  M \ar[rrr]^f\ar[d]_{\cong} &&& N \ar[d]^{\cong}\\ 
M(\bX,G) \ar[rrr]_{\alpha(\bX)} &&& N(\bX,G)
}\]
we see that $f$ is $\w\cD(\bX,G)$-linear, and we claim that $\Loc(f) = \alpha$. To see this, let $\bU \in \bX_w$, recall that the map $M(\bX,G_\bU) \to M(\bX,G)$ is a bijection by Proposition \ref{RestrictFurther}, and consider the commutative diagram
\[\xymatrix{ M(\bU,G_\bU) \ar@<0.5ex>[rrr]^{\alpha(\bU)} \ar@<-0.5ex>[rrr]_{\Loc(f)(\bU)}&&& N(\bU,G_\bU) \\
&M(\bX,G_\bU) \ar[ul]\ar[dl]_{\cong}& N(\bX,G_\bU) \ar[ur]\ar[dr]^{\cong}& \\
M(\bX,G) \ar[uu] \ar@<0.5ex>[rrr]^{\alpha(\bX)}  \ar@<-0.5ex>[rrr]_{\Loc(f)(\bX)} &&& N(\bX,G) \ar[uu].}\]
By construction, $\alpha(\bX) = \Loc(f)(\bX)$, and the image of $M(\bX,G)$ in $M(\bU,G_\bU)$ under the vertical map in this diagram generates a dense $\w\cD(\bU,G_\bU)$-submodule of $M(\bU,G_\bU) = \w\cD(\bU,G_\bU) \underset{\w\cD(\bX,G)}{\w\otimes} M$. Since $\alpha(\bU)$ 
and $\Loc(f)(\bU)$ are continuous $\cD(\bU) \rtimes G_\bU$-linear maps that agree on this submodule, it follows that they are equal. Since $\bX_w$ is a basis, we see that $\Loc(f) = \alpha$, and therefore $\Loc$ is full.

The proof of the fact that $\Loc$ is essentially surjective is quite long. In order to not interrupt the flow of the paper, we have postponed it until $\S \ref{TateFlatKiehl}$ below  --- see Theorem \ref{Kiehl}.
\end{proof}

\subsection{\ts{\cC_{\bX/G}} is an abelian category}\label{CXGAbCat} We continue to assume that $\bX$ is a smooth rigid analytic variety, and that $G$ is a not necessarily compact $p$-adic Lie group acting continuously on $\bX$. In $\S \ref{CXGAbCat}$, we will prove that the category $\cC_{\bX/G}$ of coadmissible $G$-equivariant $\cD$-modules on $\bX$ is abelian.  Recall from \cite[Definition 7.5(a)]{DCapOne} that if $A \to B$ is a continuous homomorphism of Fr\'echet-Stein algebras, then $B$ is said to be a \emph{right c-flat $A$-module} if the functor $B \w\otimes_A -$ from coadmissible $A$-modules to coadmissible $B$-modules is exact.

\begin{thm}\label{DUHcflat} Suppose that $(\bX,H)$ is small and let $\bU \in \bX_w/H$. Then $\w\cD(\bU,H)$ is a c-flat right $\w\cD(\bX,H)$-module.
\end{thm}
We postpone the proof until $\S\ref{TateFlatKiehl}$. The category $\Emod{G}{\cD_\bX}$ is abelian, and we have the forgetful functor 
\[\Phi : \Frech(G-\cD_\bX) \longrightarrow \Emod{G}{\cD_\bX}.\]

\begin{cor}\label{LocExact} If $(\bX,G)$ is small then 
\[\Phi \circ \Loc^{\w\cD(\bX,G)}_\bX : \cC_{\w\cD(\bX,G)} \longrightarrow \Emod{G}{\cD_\bX}\]
is an exact functor.
\end{cor}
\begin{proof} Let $0 \to M_1 \to M_2 \to M_3 \to 0$ be a short exact sequence of coadmissible $\w\cD(\bX,G)$-modules, let $\bU \in \bX_w$ and choose an open subgroup $H$ of $G_\bU$. Then $(\bU,H)$ is small by Lemma \ref{SubsOfSmall}, so the sequence of $\w\cD(\bU,H)$-modules
\[ 0 \to M_1(\bU,H) \to M_2(\bU,H) \to M_3(\bU,H) \to 0 \]
is exact by Theorem \ref{DUHcflat} and \cite[Proposition 7.5(a)]{DCapOne}. Corollary \ref{ResCor} and Theorem \ref{DGsheaf} now imply that $\cP^{\w\cD(\bX,G)}_\bX$ is an exact functor from $\cC_{\w\cD(\bX,G)}$ to $G$-equivariant $\cD$-modules on $\bX_w$. Since the extension functor from sheaves on $\bX_w$ to sheaves on $\bX_{\rig}$ is an equivalence of categories by \cite[Theorem 9.1]{DCapOne}, we see that $\Phi \circ \Loc^{\w\cD(\bX,G)}_\bX$ is also exact.
\end{proof}

In an attempt to make this material more readable, if $\alpha : \cM \to \cN$ is a morphism in $\cC_{\bX/G}$ we will continue to write $\ker \alpha$ to mean the $G$-equivariant $\cD$-module that is the kernel of $\Phi(\alpha) : \Phi(\cM) \to \Phi(\cN)$, rather than the more precise, but confusing, $\ker \Phi(\alpha)$. We adopt a similar abuse of notation with the sheaves $\im \alpha$ and $\coker \alpha$.

\begin{lem}\label{KerCoadm}Let $\alpha : \cM \to \cN$ be a morphism in $\cC_{\bX/G}$. Then the $G$-equivariant $\cD$-module $\ker \alpha$ is coadmissible, and the canonical morphism $\ker \alpha \hookrightarrow \cM$ is continuous.
\end{lem}
\begin{proof} For any $\bU \in \bX_w(\cT)$, $(\ker \alpha)(\bU)$ is the kernel of the continuous map $\alpha(\bU) : \cM(\bU) \to \cN(\bU)$ between two Fr\'echet spaces. Therefore it is closed, and we equip it with the subspace Fr\'echet topology from $\cM(\bU)$. For any $g \in G$, the map $g^{\ker \alpha}(\bU) : (\ker \alpha)(\bU) \to (\ker \alpha)(g\bU)$ is the restriction of $g^{\cM}(\bU) : \cM(\bU) \to \cM(g\bU)$, which is continuous. So the $G$-equivariant $\cD$-module $\ker \alpha$ is locally Fr\'echet.

By choosing a common refinement and appealing to Lemma \ref{C/Grefinement}, we may assume that $\cM$ and $\cN$ are both $\cU$-coadmissible for some $\bX_w(\cT)$-covering $\cU$. Fix $\bU \in \cU$. By Definition \ref{CoadmFrechEqDmod} and Theorem \ref{LocEquiv}, we can find a $\bU$-small subgroup $H$ of $G$, a morphism $f : M \to N$ of coadmissible $\w\cD(\bU,H)$-modules and a commutative diagram of $H$-equivariant $\cD_\bU$-modules 
\[ \xymatrix{  \Loc(M) \ar[r]^{\Loc(f)}\ar[d]_\lambda^\cong & \Loc(N) \ar[d]_\mu^\cong \\
 \cM_{|\bU} \ar[r]_{\alpha_{|\bU}} &  \cN_{|\bU},  }\]
where $\Loc := \Loc^{\w\cD(\bU,H)}_\bU$ and the maps $\lambda, \mu$  are continuous. The $\w\cD(\bU,H)$-module $\ker f$ is coadmissible by \cite[Corollary 3.4(ii)]{ST}. Because $\Loc$ is exact by Corollary \ref{LocExact}, we obtain a commutative diagram of $H$-equivariant $\cD_\bU$-modules
\[ \xymatrix{ 0 \ar[r] & \Loc(\ker f) \ar[r]\ar@{.>}[d]_{\psi}^\cong & \Loc(M) \ar[r]^{\Loc(f)}\ar[d]_\lambda^\cong & \Loc(N) \ar[d]_\mu^\cong \\
0 \ar[r] & \ker \alpha  \ar[r] &  \cM_{|\bU} \ar[r]_{\alpha_{|\bU}} &  \cN_{|\bU}  }\]
with exact rows. Since $\Emod{H}{\cD_\bU}$ is an abelian category, the Five Lemma gives an isomorphism of $H$-equivariant $\cD_\bU$-modules $\psi$ completing the diagram. For any $\bV \in \bU_w$, $\psi(\bV)$ is the restriction of the continuous map $\lambda(\bV)$ to $\Loc(\ker f)(\bV)$, and is therefore continuous. Hence $\ker \alpha$ is $\cU$-coadmissible, and $\ker \alpha \hookrightarrow \cM$ is continuous by construction.
\end{proof}

\begin{lem}\label{VanishingHi} Let $\cM \in \cC_{\bX/G}$, $\bU \in \bX_w(\cT)$ and $i > 0$. Then $H^i(\bU, \cM) = 0$.
\end{lem}
\begin{proof} Choose a $\bU$-small subgroup $H$ of $G$. By Proposition \ref{FrechRestr}(b), $\cM_{|\bU} \in \cC_{\bU/H}$ and $H^i(\bU, \cM) = H^i(\bU, \cM_{|\bU})$. So we may assume that $(\bX,G)$ is small. But now $\cM \cong \Loc^{\w\cD(\bX,G)}_\bX(M)$ for some coadmissible $\w\cD(\bX,G)$-module $M$ by Theorem \ref{LocEquiv}, so $\check{H}^i(\cU, \cM_{|\bV}) = 0$ for any finite affinoid covering $\cU$ of any affinoid subdomain $\bV$ of $\bU$ by Proposition \ref{LocTrans} and Theorem \ref{DGsheaf}. Now apply \cite[\href{http://stacks.math.columbia.edu/tag/03F9}{Lemma 21.11.9}]{stacks-project}.
\end{proof}

\begin{lem}\label{CokerCoadm}Let $\alpha : \cM \to \cN$ be a morphism in $\cC_{\bX/G}$. Then the $G$-equivariant $\cD$-module $\coker \alpha$ is coadmissible, and the canonical morphism $\cN \twoheadrightarrow \coker \alpha$ is continuous.
\end{lem}
\begin{proof} We first handle the special case where $\ker \alpha = 0$, so that we have the short exact sequence $0 \to \cM \to \cN \to \coker \alpha \to 0$ in $\Emod{G}{\cD}$. For any $\bU \in \bX_w(\cT)$, the sequence $0 \to \cM(\bU) \stackrel{\alpha(\bU)}{\longrightarrow} \cN(\bU) \to (\coker \alpha)(\bU) \to 0$ is exact by Lemma \ref{VanishingHi}, so $(\coker \alpha)(\bU) \cong \coker \alpha(\bU)$. Now $\cM(\bU)$ and $\cN(\bU)$ are coadmissible $\w\cD(\bU,H)$-modules for any $\bU$-small subgroup $H$ of $G$, by Proposition \ref{FrechRestr} and Theorem \ref{Kiehl}. The map $\alpha(\bU)$ is $\cD(\bU) \rtimes H$-linear by Proposition \ref{EqGlSec} as well as continuous, hence it is also $\w\cD(\bU,H)$-linear. Since $\w\cD(\bU,H)$ is Fr\'echet-Stein by Theorem \ref{FrSt}, $(\coker \alpha)(\bU) \cong \coker \alpha(\bU)$ is a coadmissible $\w\cD(\bU,H)$-module by \cite[Corollary 3.4(ii)]{ST}, and we will therefore equip it with the canonical Fr\'echet topology. Let $g \in G$. Since $g^{\coker \alpha}(\bU) : (\coker \alpha)(\bU) \to (\coker \alpha)(g\bU)$ is induced by the $\w\cD(\bU,H)$-linear map $g^{\cM}(\bU) : \cM(\bU) \to \cM(g\bU)$ (when we view $\cM(g\bU)$ as a $\w\cD(\bU,H)$-module via the isomorphism $\w{g}_{\bU,H}$), it is also $\w\cD(\bU,H)$-linear and thus continuous. So the $G$-equivariant $\cD$-module $\coker \alpha$ is locally Fr\'echet.

By choosing a common refinement and appealing to Proposition \ref{C/Grefinement}, we may assume that $\cM$ and $\cN$ are both $\cU$-coadmissible for some $\bX_w(\cT)$-covering $\cU$. Fix $\bU \in \cU$. By Theorem \ref{LocEquiv} and Corollary \ref{LocExact}, we can find a $\bU$-small subgroup $H$ of $G$, a morphism of coadmissible $\w\cD(\bU,H)$-modules $f : M \to N$ and a diagram of $H$-equivariant $\cD_\bU$-modules 
\[ \xymatrix{ 0 \ar[r] & \Loc(M) \ar[r]^{\Loc(f)}\ar[d]_\lambda^\cong & \Loc(N) \ar[r]\ar[d]_\mu^\cong & \Loc(\coker f) \ar@{.>}[d]^\cong_\psi \ar[r] & 0 \\
0 \ar[r] & \cM_{|\bU} \ar[r]_{\alpha_{|\bU}} &  \cN_{|\bU} \ar[r] & \coker \alpha_{|\bU}  \ar[r] & 0 }\]
with exact rows and continuous $\lambda, \mu$, where $\Loc := \Loc^{\w\cD(\bU,H)}_\bU$. Let $\bV \in \bU_w$. By Lemma \ref{VanishingHi} and Theorem \ref{Kiehl}, applying $\Gamma(\bV,-)$ to this diagram keeps the rows exact and sends all objects to coadmissible $\w\cD(\bV,H_\bV)$-modules. Now $\lambda(\bV)$ and $\mu(\bV)$ are continuous $\cD(\bV) \rtimes H_\bV$-linear maps by Proposition \ref{EqGlSec}, hence $\w\cD(\bU,H_\bV)$-linear. Therefore $\psi(\bV)$ is also $\w\cD(\bV,H_\bV)$-linear by the exactness of the diagram, and therefore continuous  by Lemma \ref{AutoCts}. So $\psi$ is continuous, and it follows that $\coker \alpha$ is $\cU$-coadmissible. 

The map $\cN(\bU) \to (\coker \alpha)(\bU)$ is $\w\cD(\bU,G_\bU)$-linear by construction, and therefore continuous. It follows that the canonical map $\cN \twoheadrightarrow \coker \alpha$ is continuous.

Returning to the general case where $\ker \alpha$ is not necessarily zero, we have two short exact sequences of $G$-equivariant $\cD$-modules
\[ 0 \to \ker \alpha \to \cM \to \im \alpha \to 0 \qmb{and} 0 \to \im \alpha \to \cN \to \coker \alpha \to 0.\]
Now $\ker \alpha \hookrightarrow \cM$ is a morphism in $\cC_{\bX/G}$ by Lemma \ref{KerCoadm}, so $\cM \to \im \alpha$ is also a morphism in $\cC_{\bX/G}$ by the special case applied to the first sequence.

Finally, consider the canonical map $\im \alpha \hookrightarrow \cN$. For any $\bU \in \bX_w(\cT)$, the canonical topology on $(\im \alpha)(\bU)$ defined above agrees with the subspace topology induced from $\cN(\bU)$. Thus $\im \alpha \hookrightarrow \cN$ is continuous, and another application of the special case completes the proof.\end{proof}

\begin{lem}\label{UPkercoker} Let $A$ be a Fr\'echet-Stein algebra, and suppose given the following commutative diagram of $K$-vector spaces
\[\xymatrix{ & P \ar@{.>}[d]_j\ar[dr]^u & & & Q & \\
0 \ar[r] & \ker \alpha \ar[r]_i & M \ar[r]_\alpha & N \ar[ur]^v\ar[r]_q & \coker \alpha \ar@{.>}[u]_r \ar[r] & 0.}\]
Suppose that all objects in this diagram are coadmissible $A$-modules, that the maps $u, \alpha, v$ are $A$-linear and that the maps $i$ and $q$ are canonical. Then the dotted arrows $j$ and $r$ are continuous.
\end{lem}
\begin{proof} The maps $u$ and $i$ are continuous by Lemma \ref{AutoCts}. Equip $\im i$ with the subspace topology from $M$. Then $\im i$ is closed in $M$ by \cite[Corollary 3.4(ii) and Lemma 3.6]{ST}, so the map $i$ is a homeomorphism onto its image by the Open Mapping Theorem \cite[Proposition 8.6]{SchNFA}. Hence every open neighbourhood of zero in $\ker \alpha$ is of the form $i^{-1}(V)$ for some open neighbourhood $V$ of $0$ in $M$. Hence $j^{-1}( i^{-1}(V) ) = u^{-1}(V)$ is open in $P$ and $j$ is continuous.

Similarly, $q$ and $v$ are continuous by Lemma \ref{AutoCts}, and $q$ is open by the Open Mapping Theorem, so for every open neighbourhood $V$ of zero in $Q$, $r^{-1}(V) = q( q^{-1} r^{-1}(V)) = q(v^{-1}(V))$ is open in $\coker \alpha$ and $r$ is continuous.
\end{proof}
\begin{thm}\label{CXGisAbelian} The category $\cC_{\bX / G}$ is abelian.
\end{thm}
\begin{proof} Let $\alpha : \cM \to \cN$ be a morphism in $\cC_{\bX/G}$. Then by Lemma \ref{KerCoadm}, the canonical map $i : \ker \alpha \hookrightarrow \cM$ is a morphism in $\cC_{\bX/G}$. Suppose that $u : \cP \to \cM$ is another morphism in $\cC_{\bX/G}$ such that $\alpha \circ u =0$. Since $i$ is the kernel of $\alpha$ in $\Emod{G}{\cD_\bX}$, there is a unique morphism of $G$-equivariant $\cD$-modules $j : \cP \to \ker \alpha$ such that $u = i \circ j$. Now if $\bU \in \bX_w(\cT)$ and $H$ is a $\bU$-small subgroup of $G$,  then all objects in the commutative diagram
\[\xymatrix{ & \cP(\bU) \ar@{.>}[d]_{j(\bU)}\ar[dr]^{u(\bU)} & & \\
0 \ar[r] & (\ker \alpha)(\bU) \ar[r]_{i(\bU)} & \cM(\bU) \ar[r]_{\alpha(\bU)} & \cN(\bU) }\]
are coadmissible $\w\cD(\bU,H)$-modules by Theorem \ref{Kiehl}, and it follows from Lemma \ref{UPkercoker} that $j(\bU)$ is continuous. Hence $j : \cP \to \ker \alpha$ is actually a morphism in $\cC_{\bX/G}$, and if $j' : \cP \to \ker\alpha$ is another map such that $u = i \circ j'$ then $j = j'$ because the forgetful functor $\Phi$ is faithful. Thus $i$ is the kernel of $\alpha$ in $\cC_{\bX/G}$. 

We have shown that $\cC_{\bX/G}$ has kernels, and an entirely similar argument using Lemma \ref{VanishingHi} and Lemma \ref{CokerCoadm} shows that it also has cokernels. Next, we verify that every monomorphism $\alpha : \cM \to \cN$ in $\cC_{\bX/G}$ is the kernel of its cokernel $q : \cN \to \coker \alpha$. To this end, consider the commutative diagram in $\cC_{\bX/G}$
\[ \xymatrix{ 
0 \ar[r] & \cM \ar@{.>}[d]_j\ar[r]^\alpha & \cN \ar@{=}[d]\ar[r]^q & \coker \alpha \ar@{=}[d] \ar[r] & 0 \\
0 \ar[r] & \ker q \ar[r]      & \cN \ar[r]_q         & \coker \alpha  \ar[r] & 0 
}\]
where $j$ is induced by the universal property of $\ker q$, and the rows are exact in $\Emod{G}{\cD_\bX}$. The arrow $\Phi(j)$ is an isomorphism by the Five Lemma in the abelian category $\Emod{G}{\cD_\bX}$, so the local sections $j(\bU)$ are continuous bijections for all $\bU \in \bX_w(\cT)$. By the Open Mapping Theorem, their inverses are also continuous and we see that $j$ is an isomorphism in $\cC_{\bX/G}$. An entirely similar argument shows that every epimorphism in $\cC_{\bX/G}$ is the cokernel of its kernel. According to \cite[Definition 1.2.2]{Wei1995}, we have verified that $\cC_{\bX/G}$ is an abelian category.
\end{proof}

\begin{proof}[Proof of Theorem \ref{CXGIntro}]
Parts (a) and (b) follow from Proposition \ref{FrechRestr}. Part (c) is Theorem \ref{LocEquiv} and part (d) is Theorem \ref{CXGisAbelian}.
\end{proof}

\section{Levelwise localisation} \label{LevelwiseSect}

\subsection{Noetherianity and flatness over general base fields}\label{NoethFlat}
In our earlier paper \cite{DCapOne} we worked throughout with a base field $K$ which was \emph{discretely valued}. We will now recall and extend some of the results from that paper to the case where $K$ is an arbitrary field equipped with a complete, non-trivial, non-Archimedean valuation. Our proofs will be based on the following deep result.

\begin{thm}[Raynaud-Gruson, 1971]\label{RG346} Let $R$ be a commutative ring with finitely many weakly associated primes, and let $S$ be a commutative $R$-algebra of finite presentation. Then every finitely generated $S$-module which is flat as an $R$-module is finitely presented as an $S$-module.
\end{thm}
\begin{proof} The definition of \emph{weakly associated primes} can be found at \cite[\href{http://stacks.math.columbia.edu/tag/056K}{Definition 30.5.1}]{stacks-project}. This follows immediately from \cite[Theorem 3.4.6]{GruRay}. A more modern account of the proof can be found at \cite[\href{http://stacks.math.columbia.edu/tag/05IF}{Lemma 37.13.7}]{stacks-project}. \end{proof}

\begin{cor}\label{RGCor} Theorem \ref{RG346} applies whenever $R$ is a domain, or has finitely many prime ideals.
\end{cor}
\begin{proof}  In the case where $R$ is a domain, the only weakly associated prime of the regular $R$-module $R$ is the zero ideal, and the second case is trivial.
\end{proof}

We will work in the following setting.

\begin{setup}\label{NCadmRalg}\hspace{2em}
\begin{itemize}
\item $\cU$ is a $\pi$-adically complete and separated $\cR$-algebra,
\item $\cU$ is flat over $\cR$,
\item $\cU / \pi \cU$ is a commutative $\cR / \pi \cR$-algebra of finite presentation.
\end{itemize}
\end{setup}

Following \cite[Proposition 1.9.14]{Abbes}, we can now give a very useful consequence of Theorem \ref{RG346}.

\begin{thm}\label{KeyRGLemma} Let $\cM$ be a finitely generated $\cU$-module and let $\cN$ be a $\cU$-submodule of $\cM$ such that $\cM/\cN$ is $\cR$-torsionfree. Then $\cN$ is finitely generated.
\end{thm}
\begin{proof} Because $\cM / \cN \cong \cU^m / \cN'$ for some $\cU$-submodule $\cN'$ of $\cU^m$ for some $m \in \N$, we may assume that $\cM = \cU^m$ is a free $\cU$-module of rank $m$. The short exact sequence $0 \to \cN \to \cM \to \cM / \cN \to 0$ consists of torsion-free $\cR$-modules by assumption. Because $\cR$ is a valuation ring, every finitely generated ideal in $\cR$ is principal, so each term in this sequence is a flat $\cR$-module by \cite[Chapter I, \S2.4, Proposition 3(ii)]{BourCommAlg}. Writing $\overline{\cR} := \cR / \pi \cR$, \cite[\href{http://stacks.math.columbia.edu/tag/00HI}{Lemma 10.38.7}]{stacks-project} implies that
\[ 0 \to \cN \otimes_{\cR} \overline{\cR} \to \cM \otimes_{\cR} \overline{\cR} \to (\cM / \cN) \otimes_{\cR} \overline{\cR} \to 0\]
is also exact, and consists of flat $\overline{\cR}$-modules. Now $\cU / \pi \cU$ is also a finitely presented commutative $\cR / \pi \cR$-algebra by assumption. Note that $\cR / \pi \cR$ has exactly one prime ideal because $\cR$ is a valuation ring of height $1$. Because $(\cM / \cN) \otimes_{\cR}\overline{\cR}$ is a finitely generated module over $\cU / \pi \cU$, we may now invoke Corollary \ref{RGCor} to deduce that it is in fact finitely presented. 

The displayed exact sequence above, together with Schanuel's Lemma \cite[7.1.2]{MCR}, now implies that $\cN / \pi \cN \cong \cN \otimes_{\cR} \overline{\cR}$ is a finitely generated $\cU/\pi \cU$-module; this also follows from \cite[\href{http://stacks.math.columbia.edu/tag/0519}{Lemma 10.5.3(5)}]{stacks-project}. Hence $\gr \cN := \bigoplus_{n \geq 0} \pi^n \cN / \pi^{n+1} \cN$ is a finitely generated module over $\gr \cU := \bigoplus_{n\geq 0} \pi^n \cU / \pi^{n+1} \cU$. Finally, $\cU$ is $\pi$-adically complete and separated by construction, so $\cN$ is $\pi$-adically separated as it is contained in $\cU^m$. Hence the $\cU$-module $\cN$ is finitely generated  by \cite[Chapter I, Theorem 5.7]{LVO}.
\end{proof}

\begin{cor}\label{HTFabs} Suppose $\cU$ satisfies Hypothesis \ref{NCadmRalg}. Then $U := \cU \otimes_{\cR}K$ is Noetherian. 
\end{cor}
\begin{proof} By symmetry, it will suffice to show that $U$ is left Noetherian. Let $I$ be a left ideal of $U$; then $\cI := \cU \cap I$ is a left ideal in $\cU$ such that $\cU / \cI$ is $\cR$-torsionfree. Hence $\cI$ is finitely generated as a left ideal by Theorem \ref{KeyRGLemma}, so $I = K \cdot \cI$ is finitely generated as a left ideal in $U$.
\end{proof}
We begin by adapting several results from \cite{Abbes}, namely \cite[1.9.13, 1.8.27, 1.8.29]{Abbes} to our non-commutative setting.
\begin{prop}\label{ArtinRees} Let $\cM$ be a finitely generated $\cU$-module.
\be \item For every $\cU$-submodule $\cN \subset \cM$ there is $n_0 \in \N$ such that 
\[ \pi^{n+1} \cM \cap \cN = \pi \left(\pi^n \cM \cap \cN\right) \qmb{for all} n \geq n_0.\]
\item $\cM$ is $\pi$-adically separated,
\item $\cM$ is $\pi$-adically complete.
\ee\end{prop}
\begin{proof} (a) Let $\cN'/\cN$ be the $\pi$-torsion submodule of $\cM/\cN$. As $\cM/\cN'$ is $\cR$-torsionfree, Theorem \ref{KeyRGLemma} implies that $\cN'$ is finitely generated over $\cU$. Hence we can find $n_0 \in \N$ such that $\pi^{n_0} \cN'\subseteq \cN$. Now, $\pi^n \cM \cap \cN = (\pi^n\cM \cap \cN') \cap \cN = \pi^n\cN' \cap \cN = \pi^n \cN'$ whenever $n \geq n_0$, so
$\pi^{n+1}\cM \cap \cN = \pi^{n+1} \cN' = \pi(\pi^n \cN') = \pi(\pi^n\cM \cap \cN)$ for all $n \geq n_0$ as required.

(b) Let $x \in \bigcap_{n=0}^\infty \pi^n \cM$ and let $\cN := \cU\cdot x$. Then part (a) implies that for some $n \in \N$,  $\cN \subseteq \pi^{n+1}\cM \cap \cN = \pi(\pi^n \cM \cap \cN) = \pi \cN$, so we can find $a \in \pi\cU$ such that $x = ax$. But $\cU$ is $\pi$-adically complete, so $1 + \pi \cU \subset \cU^\times$ and therefore $x = 0$.

(c) The proof of \cite[Chapter III, \S 2.12, Corollary 1 to Proposition 16]{BourCommAlg} works.
\end{proof}

\begin{prop}\label{FlatPrepProp} Suppose that $\cV$ is another $\pi$-adically complete, separated and flat $\cR$-algebra which contains $\cU$. Let $y \in \cV$, and suppose that the map $\cU / \pi \cU \to \cV / \pi \cV$ extends to an $\cR/\pi\cR$-algebra isomorphism $(\cU/\pi\cU)[Y] \cong \cV/\pi\cV$ which sends $Y$ to $y + \pi \cV$. Let $C$ be the centraliser of $y$ in $U := \cU_K$. Then
\be 
\item $V := \cV_K$ is a flat $U$-module on both sides, and 
\item for every finitely generated $U$-module $M$ there is a natural isomorphism of $C\langle Y\rangle$-modules $\eta_M \colon M\langle Y\rangle \stackrel{\cong}{\longrightarrow} V \otimes_U M$.
\ee
\end{prop}
\begin{proof}
(a) We will deduce this from \cite[Chapter II, \S 1.2, Proposition 1]{LVO}. Equip $U$ and $V$ with the $\pi$-adic filtrations whose degree zero pieces are given by $\cU$ and $\cV$, respectively. We must show that good filtrations on $V$-modules are separated, that the $\pi$-adic filtration on $U$ has the left Artin-Rees property in the sense of \cite[Chapter II, \S 1.1, Definition]{LVO}, and that $\gr V$ is a flat $\gr U$-module.

If $M$ is a $V$-module with some good filtration $F_\bullet M$, then writing $\cM := F_0 M$ we necessarily have $F_n M = \pi^{-n} \cM$ for all $n \in \Z$, and the associated Rees module with respect to this filtration therefore has the form $\widetilde{M} = \bigoplus_{n\in\Z} T^n \pi^{-n} \cM = \cM[T/\pi, \pi/T]$. Now, $\widetilde{M}$ is a finitely generated $\widetilde{U} \cong \cU[T/\pi,\pi/T]$-module. Setting $T = \pi$ implies that the $\cU$-module $\cM$ is finitely generated. So the $\pi$-adic filtration on $\cM$, and hence $F_\bullet M$, is separated by Proposition \ref{ArtinRees}(b). 

To see that the $\pi$-adic filtration on $U$ has the left Artin-Rees property, we must show that for every finitely generated left ideal $I = U x_1 + \cdots + U x_s $ of $U$ there is an integer $c$ such that $F_n U \cap I  \subseteq \sum_{j=1}^s \pi^{n+c} \cU \cdot x_j$ for all $n \in \Z$. Writing $\cI := \cU x_1 + \cdots + \cU x_s$ and using the fact that $I$ is a $K$-vector space, we see that it will suffice to show that $\cU \cap I \subseteq \pi^c \cI$ for some $c \in \Z$. However, we can find a finite generating set $v_1,\ldots,v_m$ for $\cU \cap I$ as a left ideal in $\cU$ by Theorem \ref{KeyRGLemma}. Choose $d \in \N$ such that $\pi^d v_j \in \cI$ for all $j=1,\ldots, m$; then $\pi^d (\cU \cap I) = \pi^d \sum_{j=1}^m \cU v_j \subseteq \cI$ and hence $\cU \cap I \subseteq \pi^{-d} \cI$ as required.

Finally, $\overline{\cV} := \cV / \pi \cV$ is a free $\overline{\cU} := \cU / \pi \cU$-module on both sides by assumption.  Hence $\gr V \cong \overline{\cV}[s, s^{-1}]$ is flat over $\gr U \cong \overline{\cU}[s,s^{-1}]$. 

(b) Because $\cV / \pi \cV \cong (\cU / \pi \cU)[Y]$ by assumption, the $\cR$-algebra $\cV$ also satisfies Hypothesis \ref{NCadmRalg}. Choose a finitely generated $\cU$-submodule $\cM$ in $M$ which generates $M$ as a $K$-vector space. Then $V \otimes_U M = (\cV \otimes_{\cU} \cM) \otimes_{\cR} K$. The finitely generated $\cV$-module $\cV \otimes_{\cU} \cM$ is $\pi$-adically complete and separated by Proposition \ref{ArtinRees}(b,c). So, for any sequence of elements $m_j \in M$ tending to zero, the series $\sum_{j=0}^\infty Y^j \otimes m_j$ converges to a unique element $\eta_M\left(\sum_{j=0}^\infty Y^j m_j \right) \in V \otimes_U M$. This defines a $C$-linear morphism $\eta_M : M \langle Y \rangle \longrightarrow V \otimes_U M$ which is functorial in $M$. It is straightforward to see that $\eta_M$ is also $C\langle Y\rangle$-linear. Now, $\eta_U$ is an isomorphism by construction of $V$, and $M$ is a finitely presented $U$-module because $U$ is Noetherian by Corollary \ref{HTFabs}. We may view $\eta$ as a natural transformation between two right exact functors and use the Five Lemma to conclude that $\eta_M$ is always an isomorphism. 
\end{proof}

We now suppose further that
\begin{itemize}
\item $y \in U$ (possibly $y \notin \cU$) is such that $[y, \cU] \subseteq \pi \cU$.
\end{itemize}

Let $\delta : \cU \to \cU$ be the restriction of $\ad(y) : U \to U$ to $\cU$. This is an $\cR$-linear derivation, so we may form the skew-polynomial ring $\cU[Y ; \delta]$, its $\pi$-adic completion
\[ \cV := \widehat{ \cU[Y;\delta] }\]
and the \emph{skew-Tate algebra}
\[ V := \cV \otimes_{\cR} K = \hK{\cU[Y;\delta]}.\]

We can now generalise \cite[Theorem 4.5]{DCapOne}.

\begin{thm}\label{AbstractFlatness} Suppose that $\cU$ satisfies Hypothesis \ref{NCadmRalg} and that $y \in U$ is such that $[y,\cU] \subseteq \pi \cU$. Then $V / (Y-y)V$ is a flat $U$-module on both sides.
\end{thm}
\begin{proof} 
By \cite[Proposition 4.4]{DCapOne}, to show that $V / (Y-y)V$ is a flat right $U$-module, we have to show left-multiplication by $Y-y$ on $V$ is injective, that $V$ is a flat right $U$-module, and $V \otimes_U M$ is $(Y-y)$-torsionfree for all finitely generated $U$-modules $M$. Note that the first condition follows from the third condition by taking $M = U$. 

Our assumptions on $y$ ensure that $\cV = \h{\cU[Y;\delta]}$ satisfies the hypotheses of Proposition \ref{FlatPrepProp}. Hence, the second condition of \cite[Proposition 4.4]{DCapOne} holds by Proposition \ref{FlatPrepProp}(a). In view of Proposition \ref{FlatPrepProp}(b), it remains to show that $M\langle Y \rangle$ is $(Y - y)$-torsion-free for every finitely generated $U$-module $M$. 

Suppose now that the element $\sum_{j=0}^\infty Y^jm_j \in M \langle Y \rangle$ is killed by $Y-y$. Then $\lim\limits_{j\to\infty}m_j=0$, and setting $m_{-1} := 0$, we deduce the equations 
\[ym_j=m_{j-1} \qmb{for all} j \in \N\]
because $y$ commutes with $Y$ inside $V$. Consider the $U$-submodule $N$ of $M$ generated by the $m_j$. Because $M$ is Noetherian by Corollary \ref{HTFabs}, $N$ must be generated by $m_0, \ldots, m_d$ for some $d \geq 0$, say. 

Let $\cM$ be a finitely generated $\cU$-submodule of $M$ which generates $M$ as a $K$-vector space, and let $\cN := \sum_{i=0}^d \cU m_i$. Because $\cM / \cM \cap N$ is $\cR$-torsionfree, Theorem \ref{KeyRGLemma} implies that $\cM \cap N$ is a finitely generated $\cU$-submodule of $N$. Because $K \cdot \cN = K \cdot (\cM \cap N) = N$, the $\cU$-modules $\cM \cap N$ and $\cN$ contain $\pi$-power multiples of each other. So, as $\lim\limits_{j \to \infty} m_j = 0$, for all $n \in \N$ we can find $j_n \in\N$ such that $m_j \in \pi^n \cN$ for all $j \geq j_n$. Next, because $[y,\cU] \subseteq \cU$ by assumption and because
\[ y \sum_{j=0}^d s_jm_j = \sum_{j=0}^d [y,s_j] m_j + s_jm_{j-1} \in \cN \qmb{for all} s_0,\ldots, s_d \in \cU\]
we see that $y^i \cN \subseteq \cN$ for all $i \geq 0$. Therefore for any $j, n \in \N$ we have
\[ m_j = y^{j_n} m_{j + j_n} \in y^{j_n}\pi^n \cN \subseteq \pi^n\cN.\]
Hence $m_j \in \bigcap_{n=0}^\infty \pi^n \cN$ for all $j \in \N$. But $\bigcap_{n=0}^\infty \pi^n \cN = 0$ by Proposition \ref{ArtinRees}(b) since $\cN$ is finitely generated, so $m_j = 0$ for all $j \in \N$. Hence $\sum_{j=0}^\infty Y^jm_j = 0$ and $V / (Y-y)V$ is a flat right $U$-module as claimed. \end{proof}

\begin{lem}\label{ULbarFP} Let $\cA$ be an admissible $\cR$-algebra and let $\cL$ be a smooth $(\cR, \cA)$-Lie algebra such that $[\cL,\cL]\subseteq \pi \cL$ and $\cL \cdot \cA \subseteq \pi \cA$. Then Hypothesis \ref{NCadmRalg} holds for the $\pi$-adic completion $\h{U(\cL)}$.
\end{lem}
\begin{proof} It follows from \cite[Theorem 3.1]{Rinehart} that $U(\cL)$ is a flat $\cA$-module. Because $\cA$ is an admissible $\cR$-algebra, it is also flat as an $\cR$-module. Since $\cR$ is a valuation ring, it follows that $\cU := \h{U(\cL)}$ is also flat as an $\cR$-module, and it is $\pi$-adically complete and separated by construction. Hence $\cU / \pi \cU \cong U(\cL)/\pi U(\cL)$.

By assumption, $\cL / \pi \cL$ is a smooth $(\cR/\pi\cR, \cA / \pi \cA)$-Lie algebra with trivial Lie bracket and trivial anchor map.  Because $\cL$ is smooth, it follows from \cite[Proposition 2.3]{DCapOne} that the canonical map $\cR/\pi\cR$-algebra homomorphism $U(\cL)/ \pi U(\cL) \longrightarrow U(\cL / \pi \cL)$ is bijective. However $U(\cL/\pi\cL)$ is the symmetric algebra of the finitely generated projective $\cA/\pi\cA$-module $\cL/\pi\cL$ and is therefore a commutative $\cA/\pi\cA$-algebra of finite presentation. It is therefore a commutative $\cR/\pi\cR$-algebra of finite presentation, because the $\cR$-algebra $\cA$ is admissible.\end{proof}

\begin{cor}\label{HTF} Let $\cA$ be an admissible $\cR$-algebra and let $\cL$ be a smooth $(\cR, \cA)$-Lie algebra such that $[\cL,\cL]\subseteq \pi \cL$ and $\cL \cdot \cA \subseteq \pi \cA$. Then $\hK{U(\cL)}$ is Noetherian. 
\end{cor}
\begin{proof}  Apply Lemma \ref{ULbarFP} together with Corollary \ref{HTFabs}.\end{proof}

We can now extend \cite[Theorem 6.6]{DCapOne} to our non-Noetherian setting under a mild restrictions on the $(\cR, \cA)$-Lie algebra $\cL$.

\begin{thm}\label{HKULflat} Let $\cA$ be an admissible $\cR$-algebra and let $\cL$ be an $(\cR, \cA)$-Lie algebra which is free of finite rank as an $\cA$-module. Suppose that $[\cL,\cL]\subseteq \pi^2 \cL$ and $\cL \cdot \cA \subseteq \pi \cA$.  Then $\hK{U(\cL)}$ is a flat $\hK{U(\pi \cL)}$-module on both sides.
\end{thm}
\begin{proof} Choose an $\cA$-module basis $\{x_1,\ldots,x_d\}$ for $\cL$, and define
\[\cL_j := \cA x_1 \oplus \cdots \oplus \cA x_j \oplus \cA (\pi x_{j+1}) \cdots \oplus \cA (\pi x_d)\]
for each $j=1,\ldots,d$. This gives us a chain
\[ \pi\cL =: \cL_0 \subset \cL_1 \subset \cdots \subset \cL_d = \cL\]
of $(\cR,\cA)$-Lie algebras, which are free of finite rank as $\cA$-modules. Note that
\[ [\cL_j, \cL_j] \subseteq [\cL,\cL] \subseteq \pi^2\cL = \pi \cL_0 \subseteq \pi \cL_j\]
for all $j \geq 0$, so by Lemma \ref{ULbarFP}, Hypothesis \ref{NCadmRalg} holds for each $\h{U(\cL_j)}$.

It will be sufficient to show that $\hK{U(\cL_{j+1})}$ is a flat $U := \hK{U(\cL_j)}$-module on both sides for each $j = 0, \ldots, d-1$. Now, because
\[ [x_{j+1}, \cL_j] \subseteq [\cL,\cL] \subseteq \pi^2 \cL \subset \pi\cL_j\]
we see that the element $x_{j+1} \in U$ satisfies $[x_{j+1}, \cU] \subseteq \pi \cU$, where $\cU :=  \h{U(\cL_j)}$. We can now form the skew-Tate algebra $V = \hK{\cU[Y;\delta]}$ as above, and by Theorem \ref{AbstractFlatness}, it will be enough to see that 
\[ V / (Y-x_{j+1})V \cong \hK{U(\cL_{j+1})}\]
as a right $U$-module. Write $B := U(\cL_j)$ and $C := U(\cL_{j+1})$ and note that we may view $B$ and $C$ as $\cR$-subalgebras of $U(\cL)_K$ with $B \subseteq C$. Form the skew-polynomial ring $B[Y;\delta]$, where $\delta : B \to B$ is the restriction of $\ad(x_{j+1})$ on $C$ to $B$. There is a natural $\cR$-algebra homomorphism
\[\varphi : B[Y;\delta] \to C\]
which extends the embedding $B \hookrightarrow C$ and which sends $Y$ to $x_{j+1}$.  Clearly $\varphi$ is surjective, and $\pi Y - z \in \ker \varphi$, where $z := \pi x_{j+1} \in \cL_j \subset B$. We will show that 
\[\ker \varphi = (\pi Y - z)B[Y;\delta].\] 
Now $\pi Y - z$ is central in $B[Y;\delta]$, so
\[(\pi Y)^i = (z + (\pi Y -z))^i \equiv z^i \quad \mod \quad (\pi Y - z)B \qmb{for all} i \in \N.\]
Therefore for any $u := \sum_{i=0}^m Y^i u_i  \in B[Y;\delta]$ we have
\[ \pi^m u = \sum_{i=0}^m (\pi Y)^i \pi^{m-i} u_i   \equiv \sum_{i=0}^m z^i \pi^{m-i}  u_i \quad \mod \quad (\pi Y - z) B.\]
On the other hand, if $u \in \ker \varphi$ then $\varphi(u) =\sum_{i=0}^m  x_{j+1}^i u_i = 0$ implies that $\sum_{i=0}^m z^i \pi^{m-i} u_i  = 0$ in $B$.  Hence $\pi^m u =  (\pi Y - z)v$ for some $v \in B$. We will prove by induction on $m \in \N$ that $v \in \pi^m B$, the case $m = 0$ being trivial. Now if $m \geq 1$ then $\pi(vY - \pi^{m-1}u) = zv \in \pi B$. Since $z$ is an element of an $\cA$-module basis for $\cL_j$, it follows from \cite[Theorem 3.1]{Rinehart} that its image in $B / \pi B \cong U(\cL_j / \pi \cL_j)$ is not a zero-divisor. Hence $v = \pi v'$ for some $v' \in B$. Hence $\pi^m u =  (\pi Y - z)(\pi v')$ so $\pi^{m-1} u = (\pi Y - z)v' $ as $B$ is flat as an $\cR$-module. So $v' \in \pi^{m-1}B$ by induction, and hence $v = \pi v' \in \pi^m B$, completing the induction. Hence $\pi^m u = \pi^m (\pi Y - z)w$ for some $w \in B$ and using the flatness of $B$ as an $\cR$-module again we see that $u \in (\pi Y - z)B$. Thus $\ker \varphi = (\pi Y - z)B[Y;\delta]$ as claimed. 

Because $B[Y;\delta]$ is contained in $U(\cL)_K[Y;\ad(x_{j+1})]$ and because $Y - x_{j+1}/\pi$ is not a zero-divisor in this ring, we see that $\pi Y - z$ is also not a zero-divisor in $B[Y;\delta]$. Thus we have proved that
\[0 \to B[Y;\delta] \stackrel{\pi Y - z}{\longrightarrow}  B[Y;\delta] \stackrel{\varphi}{\longrightarrow} C \to 0\]
is an exact sequence of right $B[Y;\delta]$-modules. Because $C$ is flat as an $\cR$-module, this sequence remains exact after applying $\pi$-adic completion:
\[ 0 \to \h{B[Y;\delta]} \stackrel{\pi Y - z}{\longrightarrow}  \h{B[Y;\delta]} \stackrel{\varphi}{\longrightarrow} \h{C} \to 0\]
is exact. Finally, the inclusion $B[Y;\delta] \hookrightarrow \cU[Y;\delta]$ extends to an isomorphism $\h{B[Y;\delta]} \cong \h{\cU[Y;\delta]} = \cV$. Therefore 
\[V / (Y - x_{j+1})V = V / (\pi Y - z)V = (\cV / (\pi Y - z)\cV)_K \cong \h{C}_K = \hK{U(\cL_{j+1})}\]
as right $\hK{B} = \hK{U(\cL_j)}=U$-modules, as required.
\end{proof} 
\begin{proof}[Proof of Theorem \ref{FrSt} and Theorem \ref{FrStIntro}] \label{ProofFrST} Choose a $G$-stable affine formal model $\cA$ in $\cO(\bX)$ and a $G$-stable free $\cA$-Lie lattice $\cL$ in $\cT(\bX)$. By replacing $\cL$ by $\pi^2 \cL$ if necessary, we may assume that $[\cL,\cL] \subseteq \pi^2 \cL$, and that $\cL \cdot \cA \subseteq \pi \cA$. Then $\pi^n \cL$ also satisfies all of these conditions for each $n \geq 0$.

Choose a good chain $(N_\bullet)$ for $\cL$ using Corollary \ref{ChainCap}. Letting $U_n := \hK{U(\pi^n\cL)}$, Lemma \ref{StdPres} implies that $\w\cD(\bX,G) \cong \invlim U_n \rtimes_{N_n} G$. Now each $K$-Banach algebra $U_n$ is Noetherian by Corollary \ref{HTF}, so $U_n \rtimes_{N_n} G$ is also Noetherian, being a crossed product of $U_n$ with the finite group $G / N_n$ by Lemma \ref{RingSGN}(b). The image of $U_{n+1}$ in $U_n$ is dense because it contains the dense image of $\cD(\bX)$ in $U_n$, so the image of $U_{n+1} \rtimes_{N_{n+1}} G$ in $U_n \rtimes_{N_n} G$ is also dense because $G / N_{n+1}$ surjects onto $G/N_n$.

It remains to show that $U_n \rtimes_{N_n} G$ is a flat $U_{n+1} \rtimes_{N_{n+1}} G$-module on both sides.  To this end, consider the commutative diagram
\[ \xymatrix{ U_{n+1} \ar[r]\ar[d]\ar@{.>}[dr] & U_n \ar[d] \\ U_{n+1} \rtimes_{N_{n+1}} G \ar[r] & U_n \rtimes_{N_n} G .}\]
The top arrow is flat by Theorem \ref{HKULflat}, and the rightmost arrow is flat because $U_n \rtimes_{N_n} G$ is a free $U_n$-module of rank $|G : N_n|$ on both sides.  Hence the dotted diagonal arrow is flat, and we conclude using \cite[Lemma 2.2]{SchmidtBB} that the bottom arrow is flat. \end{proof}
\subsection{The sheaf \ts{\hsULK} on \ts{\bX_w(\cL)}}\label{SectionHSULK} In this subsection we extend the results of \cite[\S 3,4,5]{DCapOne} to our setting, where we no longer assume that $\cR$ is Noetherian. This requires imposing some mild new hypotheses on the Lie lattice $\cL$.

\begin{defn}\label{DefnLAdm} Let $\bX$ be a $K$-affinoid variety, and let $\cL$ be an $(\cR,\cA)$-Lie algebra for some affine formal model $\cA$ in $\cO(\bX)$. Let $\bY$ be an affinoid subdomain of $\bX$ and let $\sigma : \cO(\bX) \to \cO(\bY)$ be the pullback map on functions. 
\be \item We say that the affine formal model $\cB$ in $\cO(\bY)$ is \emph{$\cL$-stable} if $\sigma(\cA) \subset \cB$ and the action of $\cL$ on $\cA$ lifts to $\cB$. 
\item We say that $\bY$ is \emph{$\cL$-admissible} if it admits an $\cL$-stable affine formal model.
\item We will denote the full subcategory of $\bX_w$ consisting of the $\cL$-admissible affinoid subdomains by $\bX_w(\cL)$. 
\item We define an \emph{$\cL$-admissible covering} of an $\cL$-admissible affinoid subdomain of $\bX$ to be a finite covering by objects in $\bX_w(\cL)$.
\ee\end{defn}

\begin{defn} \label{HSULKpresheaf}
Suppose that $\cL$ is a smooth $(\cR, \cA)$-Lie algebra for some affine formal model $\cA$ in $\cO(\bX)$.  For any $\cL$-admissible affinoid subdomain $\bY$ of $\bX$ and any $\cL$-stable affine formal model $\cB$ in $\cO(\bY)$, we define
\[\cS(\bY) := \hsULK(\bY) := \widehat{ U(\cB \otimes_{\cA} \cL) } \otimes_{\cR} K.\]
\end{defn}
It is explained in \cite[\S 3.3]{DCapOne} that this does not depend on the choice of the affine formal model $\cB$, and defines a presheaf $\hsULK$ on $\bX_w(\cL)$. The proof of \cite[Proposition 3.3]{DCapOne} does not require $\cR$ to be Noetherian.

\begin{thm}\label{TateHSULK}
Let $\bX$ be a $K$-affinoid variety, and let $\cL$ be a smooth $(\cR,\cA)$-Lie algebra for some affine formal model $\cA$ in $\cO(\bX)$. Then every  $\cL$-admissible covering of $\bX$ is $\hsULK$-acyclic.
\end{thm}
\begin{proof} This is \cite[Theorem 3.5]{DCapOne}. The proof of this result does not require $\cR$ to be Noetherian.\end{proof}

We will next work towards adapting a basic result about flatness from \cite{DCapOne}, namely \cite[Theorem 4.5]{DCapOne} to our non-Noetherian setting. Granted Theorem \ref{KeyRGLemma}, the proof of this generalisation of \cite[Theorem 4.5]{DCapOne} carries over to case where $\cR$ is not Noetherian with minor modifications, but we will repeat it here nevertheless for the benefit of the reader.

Let $A$ be a $K$-affinoid algebra and fix $f \in A$. Let $\cA$ be an affine formal model for $A$, and choose $a \in \mathbb{N}$ such that $\pi^af \in \cA$. Define
\[ u_1 = \pi^a t - \pi^a f \qmb{and} u_2 := \pi^a f t - \pi^a \in \cA\langle t\rangle.\]
Let $\bX := \Sp(A)$ and let $C_i =  A\langle t\rangle/ u_i  A\langle t\rangle$ be the $K$-affinoid algebras corresponding to the Weierstrass and Laurent subdomains 
\[\bX_1 := \bX(f) = \Sp(C_1) \qmb{and} \bX_2 := \bX(1/f) = \Sp(C_2)\]
 of $\bX$, respectively.  Now let $\cL$ be an $(\cR,\cA)$-Lie algebra such that $\cL \cdot f \subset \cA$. Recall that by \cite[Proposition 4.2]{DCapOne}, it is possible to lift the action of $\cL$ on $\cA$ to an action of $\cL$ on $ \cA\langle t\rangle$ in two different ways. These are given by $\cR$-Lie algebra homomorphisms  \[\sigma_1, \sigma_2 : \cL \to \Der_\cR(\cA\langle t\rangle) \] 
such that $\sigma_i(x)(a) = x \cdot a$ for all $x \in \cL$, $a\in \cA$ and $i=1,2$, and
\[\sigma_1(x)(t) = x \cdot f \qmb{and} \sigma_2(x)(t) = -t^2 (x \cdot f) \qmb{for all} x \in \cL.\]
Now \cite[Lemma 2.2]{DCapOne} implies that 
\[\cL_i :=  \cA\langle t\rangle\otimes_{\cA} \cL\]
becomes an $(\cR,  \cA\langle t\rangle)$-Lie algebra with anchor map $1 \otimes \sigma_i$.

\begin{lem}\label{XiLAdm} Let $\cL$ be an $(\cR,\cA)$-Lie algebra and let $f \in A$ be a non-zero element such that $\cL \cdot f \subset \cA$. Then the affinoid subdomains $\bX_i$ of $\bX$ are $\cL$-admissible.
\end{lem}
\begin{proof} Remember that $C_i = A \langle t \rangle / u_i A \langle t \rangle$ by construction, so the image $\overline{\cC_i}$ of $\cC_i :=  \cA\langle t\rangle/ u_i  \cA\langle t\rangle$ in $C_i$ spans $C_i$ as a $K$-vector space. A direct calculation shows that
\[ \sigma_1(x)(u_1) = 0 \qmb{and} \sigma_2(x)(u_2) = -(x\cdot f) t u_2 \qmb{for all} x \in \cL.\]
It follows that $u_i  \cA\langle t\rangle$ is a $\sigma_i(\cL)$-stable ideal of $ \cA\langle t\rangle$, and therefore $\overline{\cC_i}$ is an $\cL$-stable $\cR$-subalgebra of $C_i$. On the other hand, $\overline{\cC_i}$ is $\cR$-torsionfree and a finitely generated module over the admissible $\cR$-algebra $\cA\langle t\rangle$. Therefore it is a finitely presented $\cA\langle t \rangle$-module by Theorem \ref{KeyRGLemma} applied to the zero $\cA\langle t \rangle$-Lie algebra. Thus $\overline{\cC_i}$ is itself an admissible $\cR$-algebra and is therefore an $\cL$-stable affine formal model in $C_i$. Hence $\bX_i$ is $\cL$-admissible as claimed.
\end{proof}

\begin{prop}\label{UlXiPres} Let $\cL$ be a smooth $(\cR,\cA)$-Lie algebra such that $[\cL,\cL]\subseteq \pi\cL$, and let $f\in A$ be a non-zero element such that $\cL \cdot f \subset \cA$.  There is a short exact sequence 
\[ 0 \to \hK{U(\cL_i)} \stackrel{u_i \cdot}{\longrightarrow} \hK{U(\cL_i)} \to \hsULK(\bX_i) \to 0\]
of right $\hK{U(\cL_i)}$-modules.
\end{prop}
\begin{proof}
Recall from the proof of Lemma \ref{XiLAdm} the $\cR$-algebra $\cC_i = \cA\langle t \rangle / u_i \cA\langle t \rangle$ and its image $\overline{\cC_i}$ in $A\langle t \rangle / u_i A \langle t \rangle$. Let $\cJ_i / u_i \cA\langle t \rangle$ be the kernel of the canonical surjection $\cA \langle t \rangle \twoheadrightarrow \overline{\cC_i}$; it follows from Theorem \ref{KeyRGLemma} that it is killed by a power of $\pi$. Then we have the following commutative diagram of $\cA\langle t \rangle$-modules 
\[ \xymatrix{0 \ar[r] &  \cA \langle t \rangle \ar[r]^{u_i\cdot}\ar[d] & \cA \langle t \rangle \ar[r]\ar@{=}[d] &  \cC_i   \ar[d]\ar[r] & 0 \\
0 \ar[r] &                  \cJ_i  \ar[r] & \cA \langle t \rangle \ar[r]         & \overline{\cC_i} \ar[r] & 0 }\]
with exact rows. Here, the first vertical arrow sends $1$ to $u_i \in \cJ_i$, so it is injective with bounded $\pi$-torsion cokernel, and the third vertical arrow is surjective with bounded $\pi$-torsion kernel. Tensoring this diagram on the right with the flat $\cA$-module $U(\cL)$ and appealing to \cite[Proposition 2.3]{DCapOne} produces a commutative diagram of right $U(\cL_i)$-modules 
\[ \xymatrix{0 \ar[r] & U(\cL_i) \ar[r]^{u_i\cdot}\ar[d] & U(\cL_i) \ar[r]\ar@{=}[d] &  U(\cC_i \otimes_{\cA} \cL)   \ar[d]\ar[r] & 0 \\
0 \ar[r] &                  \cJ_i \otimes_{\cA} U(\cL) \ar[r] & U(\cL_i) \ar[r]         & U(\overline{\cC_i}\otimes_{\cA}\cL) \ar[r] & 0 }\]
with exact rows. Again, the first vertical arrow is injective with bounded $\pi$-torsion cokernel, and the third vertical arrow is surjective with bounded $\pi$-torsion cokernel. Now apply the $\pi$-adic completion functor to this diagram, and invert $\pi$. Because $U(\overline{C_i} \otimes_{\cA} \cL) \cong \overline{\cC_i} \otimes_{\cA} U(\cL)$ is a flat $\cR$-module, the bottom row stays exact after $\pi$-adic completion, so we obtain a commutative diagram of $\hK{U(\cL_i)}$-modules
\[ \xymatrix{0 \ar[r] & \hK{U(\cL_i)} \ar[r]^{u_i\cdot}\ar[d] & \hK{U(\cL_i)} \ar[r]\ar@{=}[d] &  \hK{U(\cC_i \otimes_{\cA} \cL)} \ar[d]& \\
0 \ar[r] &                  \h{\cJ_i \otimes_{\cA} U(\cL)}_K \ar[r] & \hK{U(\cL_i)} \ar[r]         & \hK{U(\overline{\cC_i}\otimes_{\cA}\cL)} \ar[r] & 0 }\]
with exact rows. Since $\h{(-)}_K$ sends morphisms with bounded $\pi$-torsion kernel and cokernel to isomorphisms, the vertical arrows in this diagram are isomorphisms. This produces the required short exact sequence in the statement because $\hsULK(\bX_i)\cong \hK{U(\overline{\cC_i}\otimes_{\cA}\cL)}$ by construction.
\end{proof}

We can now extend \cite[Theorem 4.5]{DCapOne} to the non-Noetherian setting, under mild new restrictions on $\cL$. 

\begin{thm}\label{HorizFlat} Let $\bX$ be a $K$-affinoid variety and let $f \in \cO(\bX)$ be non-zero. Let $\cA$ be an affine formal model in $\cO(\bX)$ and let $\cL$ be a smooth $(\cR,\cA)$-Lie algebra such that $[\cL, \cL] \subseteq \pi \cL$ and $\cL \cdot f \subseteq \pi \cA$. Let $\bX_1 = \bX(f)$ and $\bX_2 = \bX(1/f)$. Then $\hsULK(\bX_i)$ is a flat left and right $\hsULK(\bX)$-module for $i=1$ and $i=2$.
\end{thm}
\begin{proof} Write $\cU := \h{U(\cL)}$ and $U := \hK{U(\cL)} = \hsULK(\bX)$. Suppose first that $i=1$. Since $U(\cL)$ is generated by $\cA + \cL$ as an $\cR$-algebra and $[f, \cA + \cL] \subseteq \cL \cdot f \subseteq \pi \cA$  we see that $[f, U(\cL)] \subseteq \pi U(\cL)$ and consequently $[f,\cU] \subseteq \pi \cU$. We can therefore form the skew-Tate algebra $V := \hK{\cU[Y;\delta]}$ as above, where $\delta : \cU \to \cU$ is the restriction of $\ad(f) : U \to U$ to $\cU$. We can now compute that inside $\hK{U(\cL_1)}$ we have
\[ [t,x] = -[x,t] = -x\cdot t = -\sigma_1(x)(t) = -x\cdot f = [f,x] \qmb{and} [t,a] = 0 =[f,a]\]
for all $ x\in \cL$ and all $a \in\cA$. Hence there is homomorphism of $K$-Banach algebras $V \to \hK{U(\cL_1)}$ which is the identity on $U$ and sends $Y \in V$ to $t \in \hK{U(\cL_1)}$. Using the universal property of $\cL_1$, we can construct a $K$-Banach algebra homomorphism $\hK{U(\cL_1)} \to V$ which is the identity on $U$ and sends $t$ to $Y$. Thus we may identify $V$ with $\hK{U(\cL_1)}$, and 
\[ \hsULK(\bX_1) \cong \hK{U(\cL_1)} / (t-f) \hK{U(\cL_1)} \cong V / (Y-f) V\]
as $U$-modules by Proposition \ref{UlXiPres}. Hence $\hsULK(\bX_1)$ is a flat $U$-module on both sides by Theorem \ref{AbstractFlatness}.  

Suppose now that $i = 2$, and let $\cV := \h{U(\cL_2)}$. Now the element $t$ satisfies
\[ [t, x] = - x\cdot t = -\sigma_2(x)(t) = t^2(x \cdot f) \in \pi t^2 \cA \qmb{and} [t,a] = 0\]
for all $x \in \cL$ and $a \in \cA$, because $\cL \cdot f \subseteq \pi \cA$. So, the hypotheses of Proposition \ref{FlatPrepProp} are satisfied with $y := t \in \cV$; hence $V := \cV_K = \hK{U(\cL_2)}$ is a flat $U$-module on both sides and there is an $A \langle t \rangle$-linear isomorphism $M \langle t \rangle \cong V \otimes_U M$ for every finitely generated $U$-module $M$, by Proposition \ref{FlatPrepProp}. Because $M \langle t \rangle$ is always $(ft - 1)$-torsionfree by \cite[Lemma 4.1(a)]{DCapOne}, it follows that $V \otimes_U M$ is $u_2 = \pi^a(ft-1)$-torsionfree. Hence $V / u_2 V$ is a flat $U$-module by \cite[Proposition 4.4]{DCapOne}, and the result follows from Proposition \ref{UlXiPres}. \end{proof}

Because $\cR$ is no longer assumed to be Noetherian, we have to modify \cite[Definitions 4.6 and 4.8]{DCapOne} slightly, as follows.

\begin{defn} \label{LaccessDefn} Let $\bX$ be a $K$-affinoid variety.
\be
\item Let $\bY \subseteq \bX$ be a rational subdomain. If $\bY = \bX$, we say that it is \emph{$\cL$-accessible in $0$ steps}. Inductively, if $n \geq 1$ then we say that it is \emph{$\cL$-accessible in $n$ steps} if  there exists a chain $\bY \subseteq \bZ \subseteq \bX$, such that
\begin{itemize}
\item $\bZ \subseteq \bX$ is $\cL$-accessible in $(n-1)$ steps,
\item $\bY = \bZ(f)$ or $\bZ(1/f)$ for some non-zero $f \in \cO(\bZ)$,
\item there is an $\cL$-stable affine formal model $\cC \subseteq \cO(\bZ)$ such that $\cL \cdot f \subseteq \pi \cC$.
\end{itemize}
\item A rational subdomain $\bY \subseteq \bX$ is said to be \emph{$\cL$-accessible} if it is $\cL$-accessible in $n$ steps for some $n \in \mathbb{N}$.
\item An affinoid subdomain $\bY$ of $\bX$ is said to be \emph{$\cL$-accessible} if it is $\cL$-admissible and there exists a finite covering $\bY = \bigcup_{j=1}^r \bX_j$ where each $\bX_j$ is an $\cL$-accessible rational subdomain of $\bX$.
\item A finite affinoid covering $\{\bX_j\}$ of $\bX$ is said to be \emph{$\cL$-accessible} if each $\bX_j$ is an $\cL$-accessible affinoid subdomain of $\bX$.
\ee
\end{defn}

\begin{thm} \label{HSULKflat} Let $\bX$ be a $K$-affinoid variety and let $\cA$ be an affine formal model in $\cO(\bX)$. Let $\cL$ be a smooth $(\cR, \cA)$-Lie algebra such that $[\cL,\cL]\subseteq \pi \cL$ and $\cL \cdot \cA \subseteq \pi \cA$.
\be 
\item Let $\bY \subseteq \bX$ be an $\cL$-accessible affinoid subdomain. \\ Then $\cS(\bY)$ is a flat $\cS(\bX)$-module on both sides.
\item Let $\cX = \{\bX_1, \ldots, \bX_m\}$ be an $\cL$-accessible covering of $\bX$. \\ Then $\bigoplus_{i=1}^m \cS(\bX_j)$
is a faithfully flat $\cS(\bX)$-module on both sides.
\ee\end{thm}
\begin{proof} This is \cite[Theorem 4.9]{DCapOne} in the case where $\cR$ is Noetherian. The proof of \cite[Proposition 4.6]{DCapOne} carries over, using Theorem \ref{HorizFlat} instead of \cite[Theorem 4.5]{DCapOne}, to give part (a)  when $\bY$ is an $\cL$-accessible rational subdomain. Because the Noetherianity of $\cR$ is not used anywhere else in the proof, the result follows.
\end{proof}

\begin{defn} A \emph{Laurent covering} of $\bX$ is a covering of the form 
\[\{\bX(f_1^{\alpha_1},\ldots,f_n^{\alpha_n}) : \alpha_i \in \{\pm 1\}\},\]
where $f_1,\ldots, f_n$ are elements of $\cO(\bX)$. \end{defn}
For example, for any $a,b \in \cO(\bX)$,  $\{\bX(a,b), \bX(a,1/b), \bX(1/a, b), \bX(1/a,1/b)\}$ is a Laurent covering of $\bX$. It follows immediately from the definitions that a Laurent covering $\{\bX(f_1^{\alpha_1},\ldots,f_n^{\alpha_n}) : \alpha_i \in \{\pm 1\}\}$  is $\cL$-accessible whenever 
\[ \cL \cdot f_i \in \pi \cA \qmb{for all} i=1,\ldots, n.\]

Recall from \cite[\S 5.1]{DCapOne} the \emph{localisation functor} $\Loc_{\cS}$ from finitely generated $\cS(\bX)$-modules to presheaves of $\cS$-modules on $\bX_{\ac}(\cL)$, given by
\[ \Loc_{\cS}(M)(\bY) = \cS(\bY) \underset{\cS(\bX)}{\otimes} M \qmb{for all} \bY \in \bX_{\ac}(\cL).\]

\begin{prop}\label{LocSMacyclic} Let $\bX$ be a $K$-affinoid variety and let $\cA$ be an affine formal model in $\cO(\bX)$. Let $\cL$ be a smooth $(\cR, \cA)$-Lie algebra such that $[\cL,\cL]\subseteq \pi \cL$ and $\cL \cdot \cA \subseteq \pi \cA$. For every finitely generated $\cS(\bX)$-module $M$,  the presheaf $\Loc_{\cS}(M)$  is a sheaf with vanishing higher \v{C}ech cohomology.
\end{prop} 
\begin{proof} This is \cite[Proposition 5.1]{DCapOne}, whose proof is valid even when $\cR$ is not Noetherian in view of Theorem \ref{TateHSULK} and Theorem \ref{HSULKflat}. The proof uses the flatness of $\cS(\bY)$ as a right $\cS(\bX)$-module.
\end{proof}

\begin{thm}\label{OldKiehl} Let $\bX$ be a $K$-affinoid variety. Let $\cL$ be a smooth $(\cR, \cA)$-Lie algebra such that $[\cL,\cL]\subseteq \pi \cL$ and $\cL \cdot \cA \subseteq \pi \cA$ for some affine formal model $\cA$ in $\cO(\bX)$. Let $\cU$ be an $\cL$-accessible Laurent covering of $\bX$. Then every $\cU$-coherent sheaf $\cM$ of left (respectively, right) $\cS$-modules on $\bX_{ac}(\cL)$ is isomorphic to $\Loc_{\cS}(M)$ for some finitely generated left (respectively, right) $\cS(\bX)$-module $M$.
\end{thm}
\begin{proof} This is \cite[Theorem 5.5]{DCapOne} in the case where $\cR$ is Noetherian. In the view of Corollary \ref{HTF}, Theorem \ref{TateHSULK}, Theorem \ref{HSULKflat} and Proposition \ref{LocSMacyclic}, the proof of \cite[Theorem 5.5]{DCapOne} carries over with obvious modifications to the general case.
\end{proof}

\subsection{The sheaf \ts{\hsULK \rtimes_N G} on \ts{\bX_w(\cL,G)}}
\label{SectionhsULKG}
In this subsection we extend the results of $\S \ref{SectionHSULK}$ to the equivariant setting. Throughout, we will assume that:
\begin{itemize} 
\item $\bX$ is a $K$-affinoid variety, 
\item $G$ is a compact $p$-adic Lie group acting continuously on $\bX$,
\item $\cA$ is a $G$-stable affine formal model in $A := \cO(\bX)$,
\item $(\cL, N)$ is an $\cA$-trivialising pair.
\end{itemize}

\begin{defn} \noindent
\be \item Let $\bX_w / G$ denote the set of $G$-stable affinoid subdomains $\bY$ of $\bX$.
\item Let $\bX_w(\cL,G) := \bX_w(\cL) \cap (\bX_w/G)$ denote the set of $G$-stable affinoid subdomains of $\bX$ which are also $\cL$-admissible.
\ee\end{defn}

We refer the reader to Definitions \ref{DefnTrivPair} and \ref{DefnLAdm} for the meaning of the terms \emph{$\cA$-trivialising pair} and \emph{$\cL$-admissible}. We view $\bX_w/G$ and $\bX_w(\cL,G)$ as partially ordered sets under reverse inclusion, and consequently as categories whose morphism sets have at most one member.

\begin{lem}\label{XwG} $\bX_w / G$ is stable under finite intersections.
\end{lem}
\begin{proof}  Given $\bU, \bV \in \bX_w / G$, let $\cA$ and $\cA'$ be $G$-stable affine formal models in $\cO(\bU)$ and $\cO(\bV)$, respectively. Because $\bX$ is affinoid, it is separated. Thus it follows from \cite[Definition 9.6.1/1]{BGR} that $\bU \cap \bV$ is again an affinoid subdomain of $\bX$, and that the natural map  $\cO(\bU) \h{\otimes}_K \cO(\bV) \to \cO(\bU\cap \bV)$ is surjective. 

Now $\bU \cap \bV$ is clearly $G$-stable; let $\cC$ be the complete $\cR$-subalgebra of $\cO(\bU \cap \bV)$ generated by the images of $\cA$ and $\cA'$ in $\cO(\bU\cap \bV)$. Then $\cC$ is a $G$-stable affine formal model in $\cO(\bU \cap \bV)$, because  $\cO(\bU) \h{\otimes}_K \cO(\bV) \to \cO(\bU\cap \bV)$ is surjective. 
\end{proof}

\begin{cor}\label{XLwG} $\bX_w(\cL,G)$ is also stable under finite intersections.
\end{cor}
\begin{proof} This follows from \cite[Lemma 3.2]{DCapOne} and Lemma \ref{XwG}.
\end{proof}

It follows that $\bX_w / G$ forms a $G$-topology on $\bX$ in the sense of \cite[Definition 9.1.1/1]{BGR} if we declare $\bX_w / G$-admissible coverings to be finite coverings by objects in $\bX_w / G$. 

\begin{rmk}\label{GTopCovs}In what follows, we will define several other similar-looking collections $\sC$ of subsets of $\bX$ that will turn out to be stable under finite intersections. In every such case, we will regard $\sC$ as a $G$-topology on $\bX$ by declaring that an admissible covering of an object $\bY \in \sC$ is a finite set-theoretic covering of $\bY$ by other objects in $\sC$. \end{rmk}

For example, $\bX_w(\cL,G)$ is a $G$-topology on $\bX$ by Corollary \ref{XLwG}. 

\begin{lem}\label{GandLstableFM} If $\bY \in \bX_w(\cL,G)$, then every $\cL$-stable affine formal model $\cB$ in $\cO(\bY)$ is contained in an affine formal model $\cB'$ which is both $\cL$-stable and $G$-stable.
\end{lem}
\begin{proof} We saw in the proof of Lemma \ref{GstableFM} that the $G$-orbit of $\cB$ is finite, $\cB_1,\ldots, \cB_n$ say. Now $\cB' := \cB_1 \cdot \cB_2 \cdot \cdots \cdot \cB_n$ is both $\cL$-stable and $G$-stable.
\end{proof}

\begin{prop}\label{UPofAfSub} Let $\bY \in \bX_w(\cL,G)$ and let $\cB$ be a $G$-stable and $\cL$-stable affine formal model in $\cO(\bY)$. Then 
\be \item $\cL' := \cB \otimes_{\cA} \cL$ is a $G$-stable $\cB$-Lie lattice in $\cT(\bY)$, and 
\item $G_{\cL} \leq G_{\cL'}.$
\ee\end{prop}
\begin{proof} (a) It follows from \cite[Lemma 2.2]{DCapOne} that $\cL'$ is a $\cB$-Lie lattice in $\cT(\bY)$. Now, $G$ acts diagonally on $\cO(\bY) \otimes_{\cO(\bX)} \cT(\bX)$ and the natural isomorphism $\cO(\bY) \otimes_{\cO(\bX)} \cT(\bX) \to \cT(\bY)$ is $G$-equivariant. It follows that $\cL' = \cB \otimes_{\cA} \cL$ is $G$-stable. 

(b) Let $\rho_\bX : G \to \Aut_K \cO(\bX)$ and $\rho_\bY : G \to \Aut_K \cO(\bY)$ be the continuous $G$-actions on $\cO(\bX)$ and $\cO(\bY)$, respectively; then the restriction map $\cO(\bX) \to \cO(\bY)$ is $G$-equivariant by Remark \ref{sAEquiv}(b). Let $g \in G_{\cL}$, so that $\rho_\bX(g) \in \cG_{p^\epsilon}(\cA)$ and $u := \log \rho_\bX(g) \in p^\epsilon \cL$ by Definition \ref{BetaDefn}.

Now $v := 1 \otimes u \in p^\epsilon \cB \otimes_{\cA} \cL = p^\epsilon \cL'$ is a $K$-linear derivation of $\cO(\bY)$, so $\exp(v)$ is a $K$-linear automorphism of $\cO(\bY)$. Since the restriction map $\varphi : \cO(\bX) \to \cO(\bY)$ is continuous and since $v^n(\varphi(a)) = \varphi(u^n(a))$ for all $a \in \cO(\bX)$ and $n \geq 0$, we see that 
\[ \exp(v)(\varphi(a)) = \sum_{n=0}^\infty \frac{1}{n!} v^n(\varphi(a)) = \sum_{n=0}^\infty \frac{1}{n!} \varphi(u^n(a)) = \varphi( \exp(u)(a) ) = \varphi( \rho_\bX(g)(a) )\]
for all $a \in \cO(\bX)$. Therefore, the following diagram is commutative:
\[ \xymatrix{ & \cO(\bY) & \\\cO(\bY) \ar@<1ex>[ur]^{\exp(v)} \ar[ur]_{\rho_\bY(g)}  && \cO(\bX).\ar[ll]^{\varphi} \ar[ul]_{\varphi \circ \rho_\bX(g)}  }\]
Because $\bY$ is an affinoid subdomain of $\bX$, its universal property \cite[\S 7.2.2]{BGR} implies that $\rho_\bY(g) = \exp(v)$. Since $v(\cB) \subseteq p^\epsilon \cB$, it follows that $\rho_\bY(g) \in \cG_{p^\epsilon}(\cB)$ and that $\log \rho_\bY(g) = v \in p^\epsilon \cL'$. Hence $g \in G_{\cL'}$ as required.
\end{proof}

\begin{cor}\label{hsULKG} Let $\bY \in \bX_w(\cL,G)$, let $\cB$ be a $G$-stable and $\cL$-stable affine formal model in $\cO(\bY)$ and let $\cL' := \cB \otimes_{\cA} \cL$. Then $(\cL', N)$ is a $\cB$-trivialising pair, and there is a commutative diagram of $K$-algebras
\begin{equation}\label{ResULKG}\xymatrix{ \cD(\bX) \rtimes G \ar[rr]\ar[d] &&  \cD(\bY) \rtimes G \ar[d] \\ \hK{U(\cL)} \rtimes_N^{\beta_{\cL|N}} G \ar[rr] && \hK{U(\cL')} \rtimes_N^{\beta_{\cL'|N}} G.}\end{equation}
\end{cor}
\begin{proof} Apply Proposition \ref{UPofAfSub} together with Proposition \ref{hUGfunc}.
\end{proof}

\textbf{We will assume until the end of \ts{\S} \ref{SectionhsULKG} that $\cL$ is smooth}. 

Note that this hypothesis implies that our affinoid variety $\bX$ is also smooth.

\begin{defn}\label{DefnhsULKG}For any $\bY \in \bX_w(\cL,G)$ and any $\cL$-stable and $G$-stable affine formal model $\cB$ on $\cO(\bY)$, we define
\[(\hsULK \rtimes_N G)(\bY) := \hK{U(\cB \otimes_{\cA} \cL)} \rtimes_N G.\]
\end{defn}

\begin{prop} $\hsULK \rtimes_N G$ is a presheaf on $\bX_w(\cL,G)$.
\end{prop}
\begin{proof}
We first check that $(\hsULK \rtimes_N G)(\bY)$ does not depend on the choice of $\cB$. So, let $\cB'$ be another $\cL$-stable and $G$-stable affine formal model in $\cO(\bY)$, and suppose first that $\cB'$ contains $\cB$. Then $\cB' \otimes_{\cA} \cL$ is a $G$-stable $\cB'$-Lie lattice in $\cT(\bY)$ and $G_{\cB \otimes_{\cA} \cL} \leq G_{\cB \otimes_{\cA} \cL}$ by Proposition \ref{UPofAfSub} applied to $\bY$ in place of $\bX$ and $\cB'$ in place of $\cB$. Hence $(\cB'\otimes_{\cA}\cL, N)$ is a $\cB'$-trivialising pair, and Proposition \ref{hUGfunc} induces a $K$-algebra homomorphism
\[ \theta : \hK{U(\cB \otimes_{\cA} \cL)} \rtimes_N G \longrightarrow \hK{U(\cB' \otimes_{\cA} \cL)} \rtimes_N G.\]
Because $\cL$ is smooth, the map $\hK{U(\cB \otimes_{\cA} \cL)} \to \hK{U(\cB' \otimes_{\cA} \cL)}$ is an isomorphism by \cite[Proposition 3.3(b)]{DCapOne}, and it follows from Lemma \ref{RingSGN}(b) that $\theta$ is an isomorphism. In general, $\cB \subset \cB \cdot \cB' \supset \cB'$ are all $\cL$-stable and $G$-stable, so we get isomorphisms
\[\hK{U(\cB \otimes_{\cA} \cL)} \rtimes_N G \congs \hK{U(\cB \cdot \cB' \otimes_{\cA} \cL)} \rtimes_N G \stackrel{\cong}{\longleftarrow} \hK{U(\cB \otimes_{\cA} \cL)} \rtimes_N G\]
as required. Suppose now that $\bZ, \bY \in \bX_w(\cL,G)$ with $\bZ \subseteq \bY$. Choose an $\cL$-stable and $G$-stable affine formal model $\cB$ in $\cO(\bY)$, and an $\cL$-stable and $G$-stable affine formal model $\cC'$ in $\cO(\bZ)$. Let $\cC$ be the product of $\cC'$ with the image of $\cB$ in $\cO(\bZ)$; then $\cC$ is again an $\cL$-stable and $G$-stable affine formal model in $\cO(\bZ)$. In this way we obtain $G$-equivariant $\cR$-algebra homomorphisms $\cA \to \cB \to \cC$ which produce the restriction maps $\cO(\bX) \to \cO(\bY) \to \cO(\bZ)$ after applying $-\otimes_{\cR}K$. Corollary \ref{hsULKG} now produces the restriction map
\[ \mu^\bY_\bZ : (\hsULK \rtimes_NG)(\bY) \to (\hsULK \rtimes_NG)(\bZ)\]
as the bottom arrow in the commutative diagram $(\ref{ResULKG})$. 

Finally, suppose that $\bT \subseteq \bZ \subseteq \bY$ are three objects of $\bX_w(\cL,G)$. It follows from Corollary \ref{hsULKG} that $\mu^\bZ_\bT \circ \mu^\bY_\bZ$ agrees with $\mu^\bY_\bT$ on the dense image of $\cD(\bY) \rtimes G$ in $(\hsULK \rtimes_NG)(\bY)$, so the two maps are equal. Thus $\hsULK \rtimes_N G$ is a presheaf on $\bX_w(\cL,G)$.
\end{proof}

Recall the presheaf $\hsULK$ on the $\cL$-admissible $G$-topology $\bX_w(\cL)$ from Definition \ref{HSULKpresheaf}.

\begin{notn} We define the following presheaves of rings on $\bX_w(\cL,G)$:
\[\cQ := \hsULK \rtimes_NG \qmb{and} \sT := \hsULK_{|\bX_w(\cL,G)}.\]
\end{notn}

\begin{prop}\label{BiModIso} There is an isomorphism 
\[ \sT \otimes_{\sT(\bX)} \cQ(\bX) \congs \cQ\]
of presheaves of $(\sT, \cQ(\bX))$-bimodules on $\bX_w(\cL,G)$.
\end{prop}
\begin{proof} By construction, there is a morphism $F : \sT \to \cQ$ of presheaves of $K$-algebras on $\bX_w(\cL,G)$. Let $\bX' \in \bX_w(\cL,G)$ and abbreviate $T := \sT(\bX), T' := \sT(\bX'), Q := \cQ(\bX)$ and $Q' := \cQ(\bX')$, so that $Q = T \rtimes_N G \qmb{and} Q' = T' \rtimes_N G.$  There is a commutative diagram of $K$-algebras
\[ \xymatrix{ & T \ar[dr]^{f} \ar[dl]_{\lambda} & \\ T' \ar[dr]_{f'} &&  Q \ar[dl]^{\mu} \\ & Q', &}\]
where $f := F(\bX), f' := F(\bX')$ and $\lambda : T \to T'$ and $\mu : Q \to Q'$ are restriction maps in the presheaves $\sT$ and $\cQ$, respectively. The commutativity of this square induces a well-defined map of $(T',Q)$-bimodules
\[ \alpha(\bX') := f' \otimes_T \mu : T' \otimes_T Q \to Q'.\]
Define $\gamma : G \to Q^\times$ and $\gamma' : G \to Q'^\times$ by equation $(\ref{DefnGamma})$ in $\S \ref{SkewGroupRingSect}$. Then
\[\mu (\gamma(g)) = \gamma'(g) \qmb{for all} g \in G.\]
It now follows from Lemma \ref{RingSGN}(b) that $\alpha(\bX')$ carries a left $T'$-module basis for $T' \otimes_T Q$ to a left $T'$-module basis for $Q'$, and is therefore a bijection. It is straightforward to check that $\alpha(\bX')$ is functorial in $\bX' \in \bX_w(\cL,G)$.
\end{proof}

\begin{cor}\label{QSheaf} $\cQ$ is a sheaf on $\bX_w(\cL,G)$ with vanishing higher \v{C}ech cohomology. \end{cor}
\begin{proof} By Proposition \ref{BiModIso} and Lemma \ref{RingSGN}(b), $\cQ$ is isomorphic to a free sheaf of left $\sT$-modules of rank $|G:N|$. Because $\bX_w(\cL,G)$ is a coarser $G$-topology than $\bX_w(\cL)$, we may apply Theorem \ref{TateHSULK}.
\end{proof}

Recall \emph{$\cL$-accessible affinoid subdomains} of $\bX$ from Definition \ref{LaccessDefn}.

\begin{defn} $\bX_{\ac}(\cL,G)$ denotes the set of $G$-stable $\cL$-accessible affinoid subdomains of $\bX$. 
\end{defn}

With our standing convention on $G$-topologies explained in Remark \ref{GTopCovs},  \cite[Lemma 4.8(a)]{DCapOne} and Lemma \ref{XwG} imply that $\bX_{\ac}(\cL,G)$ is a $G$-topology on $\bX$.

\begin{thm} \label{FlatAcc} Assume that $[\cL,\cL] \subseteq \pi\cL$ and $\cL \cdot \cA \subseteq \pi \cA$, and let $\bY$ be a member of $\bX_{\ac}(\cL,G)$. 
\be \item  $\cQ(\bY)$ is a flat right $\cQ(\bX)$-module.
\item If $\{\bY_1,\ldots,\bY_m\}$ is an $\bX_{\ac}(\cL,G)$-covering of $\bY$, then $\oplus_{i=1}^m \cQ(\bY_i)$ is a faithfully flat right $\cQ(\bY)$-module.
\ee\end{thm}
\begin{proof} It follows from Proposition \ref{BiModIso} that there is an isomorphism of left $\sT(\bY)$-modules $\cQ(\bY) \otimes_{\cQ(\bX)} M \cong \sT(\bY) \otimes_{\sT(\bX)} M$ for every left $\cQ(\bX)$-module $M$. But $\sT(\bY)$ is a flat right $\sT(\bX)$-module, and $\oplus_{i=1}^m \sT(\bY_i)$ is a faithfully flat right $\sT(\bY)$-module by Theorem \ref{HSULKflat}.
\end{proof}

Because $\cQ$ is a sheaf of rings on the $G$-topology $\bX_{\ac}(\cL,G)$, we can make the following

\begin{defn} \label{LocQFunc}
The \emph{localisation functor} $\Loc_{\cQ}$ from finitely generated $\cQ(\bX)$-modules to presheaves of $\cQ$-modules on $\bX_{\ac}(\cL,G)$ is given by
\[ \Loc_{\cQ}(M)(\bY) = \cQ(\bY) \underset{\cQ(\bX)}{\otimes} M \qmb{for all} \bY \in \bX_{\ac}(\cL,G).\]
The functor $\Loc_{\sT}$ from finitely generated $\sT(\bX)$-modules to presheaves of $\sT$-modules on $\bX_{\ac}(\cL,G)$ is defined in a similar manner. 
\end{defn}

\begin{lem} Let $\bY \in \bX_{\ac}(\cL,G)$, choose a $G$-stable and $\cL$-stable affine formal model $\cB$ in $\cO(\bY)$, and let $\cL' := \cB \otimes_{\cA} \cL$. 
\be \item $\bX_{\ac}(\cL,G) \quad \cap \quad \bY_w = \bY_{\ac}(\cL',G)$.
\item The restriction of $\cQ$ to $\bX_{\ac}(\cL,G) \hsp \cap \hsp \bY_w$ is isomorphic to $\hK{\sU(\cL')} \rtimes_N G$.
\ee\end{lem}
\begin{proof} (a) This is a straightforward (but long) exercise relying ultimately on some properties of $\cL$-accessible rational subdomains given in \cite[Proposition 4.7]{DCapOne}. 

(b) This follows from \cite[Lemma 4.6]{DCapOne} together with Proposition \ref{hUGfunc}.
\end{proof}

Thus $\bY_{\ac}(\cL',G)$ in fact does not depend on the choice of $\cB$. Following \cite[\S 9.4.3]{BGR} and \cite[Definition 5.2]{DCapOne}, we make the following

\begin{defn} Let $\cU$ be an $\bX_{\ac}(\cL,G)$-covering of $\bX$. We say that a $\cQ$-module $\cM$ is \emph{$\cU$-coherent} if for each $\bY \in \cU$, writing $\cY := \bX_{\ac}(\cL,G) \cap \bY_w$, there is a finitely generated $\cQ(\bY)$-module $M$, and a $\cQ_{|\cY}$-linear isomorphism 
\[\Loc_{\cQ_{|\cY}}(M) \congs \cM_{|\cY}.\]
\end{defn}

$\cU$-coherent $\sT$-modules on $\bX_{\ac}(\cL,G)$ are defined in a similar way.

\begin{lem}\label{LocMLocN} Let $M$ be a finitely generated $\cQ(\bX)$-module, and let $L$ denote its restriction to $\sT(\bX)$. Then $\Loc_{\cQ}(M)$ is isomorphic to $\Loc_{\sT}(L)$ as a presheaf of $\sT$-modules.
\end{lem}
\begin{proof} Since $\cQ(\bX)$ is a finitely generated $\sT(\bX)$-module by construction, it follows that $M$ is also finitely generated as an $\sT(\bX)$-module. By Proposition \ref{BiModIso}, there is an isomorphism $\sT \otimes_{\sT(\bX)} \cQ(\bX) \congs \cQ$ of presheaves of $(\sT, \cQ(\bX))$-bimodules on $\bX_{\ac}(\cL,G)$. Applying the functor $- \otimes_{\cQ(\bX)} M$ yields an isomorphism 
\[ \left(\sT \otimes_{\sT(\bX)} \cQ(\bX)\right) \otimes_{\cQ(\bX)} M \congs \cQ \otimes_{\cQ(\bX)} M\]
of presheaves of $\sT$-modules. Contracting the tensor product on the left hand side gives the required $\sT$-linear isomorphism $\Loc_{\sT}(L) \congs \Loc_{\cQ}(M)$.
\end{proof}

\begin{cor}\label{TateForCohsQmod} Suppose that $[\cL,\cL] \subseteq \pi\cL$ and $\cL \cdot \cA \subseteq \pi \cA$. Then the augmented \v{C}ech complex $C^{\aug}(\cU, \Loc_{\cQ}(M))$ is exact for every $\bX_{\ac}(\cL,G)$-covering $\cU$ and every finitely generated $\cQ(\bX)$-module $M$.
\end{cor}
\begin{proof} The $G$-topology $\bX_{\ac}(\cL,G)$ is weaker than $\bX_{\ac}(\cL)$, so $C^{\aug}(\cU, \Loc_{\sT}(N))$ is exact by Proposition \ref{LocSMacyclic}. Now apply Lemma \ref{LocMLocN}.
\end{proof}

Recall that $\cS$ denotes the restriction of the sheaf $\hK{\sU(\cL)}$ to $\bX_{\ac}(\cL)$. Note that in this notation we have
\[ \sT = \cS_{|\bX_{\ac}(\cL,G)}.\]

\begin{prop}\label{Enhance} Let $\cU$ be an $\bX_{\ac}(\cL,G)$-covering of $\bX$, and let $\cM$ be a $\cU$-coherent $\cQ$-module on $\bX_{\ac}(\cL,G)$. Then there is a $\cU$-coherent $\cS$-module $\cN$ on $\bX_{\ac}(\cL)$ and a $\sT$-linear isomorphism
\[ \cN_{|\bX_{\ac}(\cL,G)} \congs \cM.\]
\end{prop}
\begin{proof} Let $\cU = \{\bY_1,\ldots, \bY_n\}$, fix $1 \leq i,j \leq n$ and write $\bY_{ij} := \bY_i \cap \bY_j$. We have $G$-topologies $\cY_i := \bY_{i,w} \cap \bX_{\ac}(\cL)$ and $\cY_{ij} := \bY_{ij,w} \cap \bX_{\ac}(\cL)$ and sheaves $\cS_i := \cS_{|\cY_i}$ and $\cS_{ij} := \cS_{| \cY_{ij}}$ on them. Since $\cM$ is $\cU$-coherent, $\cM(\bY_i)$ is a finitely generated $\cQ(\bY_i)$-module, and hence also a finitely generated $\cS(\bY_i)$-module. So we can define
\[ \cN_i :=  \Loc_{\cS_i}(\cM(\bY_i)),\]
a coherent sheaf of $\cS_i$-modules on $\cY_i$.  For every $\bV \in \cY_{ij}$, the map
\[\begin{array}{ccc} \cS(\bV) \times \cM(\bY_i)& \longrightarrow &\cS(\bV) \underset{\cS(\bY_{ij})}{\otimes} \cM(\bY_{ij}) \\ 
(a,m) & \mapsto & a \otimes m_{|\bY_{ij}}\end{array}\]
 is $\cS(\bY_i)$-balanced and therefore descends to a well-defined $\cS(\bV)$-linear map
\[\psi^i_j(\bV) : \cS(\bV) \underset{\cS(\bY_i)}{\otimes} \cM(\bY_i) \longrightarrow \cS(\bV) \underset{\cS(\bY_{ij})}{\otimes} \cM(\bY_{ij}).\]
Thus we obtain a morphism of presheaves of $\cS_{ij}$-modules on $\cY_{ij}$
\[ \psi^i_j : \cS_{ij} \underset{\cS(\bY_i)}{\otimes} \cM(\bY_i) \longrightarrow \cS_{ij} \underset{\cS(\bY_{ij})}{\otimes} \cM(\bY_{ij}).\]
Since $\cM$ is $\cU$-coherent, we see that actually $\psi^i_j$ is an isomorphism. In this way, we obtain $\cS_{ij}$-linear isomorphisms
\[ \phi_{ij} := (\psi^j_i)^{-1} \circ \psi^i_j : \cN_{i|\cY_{ij}} \congs \cN_{j|\cY_{ij}}\]
and we omit the verification of the fact that these satisfy the cocycle condition
\[ (\phi_{jk})_{|\cY_{ijk}} \circ (\phi_{ij})_{|\cY_{ijk}}  = (\phi_{ik})_{|\cY_{ijk}} \]
where, of course, $\cY_{ijk} = (\bY_i \cap \bY_j \cap \bY_k)_w \cap \bX_{\ac}(\cL)$. Therefore by \cite[\href{http://stacks.math.columbia.edu/tag/04TP}{Lemma 7.25.4}]{stacks-project}, the sheaves $\cN_i$ glue together to give a sheaf $\cN$ on $\bX_{\ac}(\cL)$. Examining the construction, we see that $\cN$ is given by the formula
\begin{equation}\label{NWformula} \cN(\bW) = \ker \left( \bigoplus_{i=1}^n \cS(\bW \cap \bY_i) \underset{\cS(\bY_i)}{\otimes} \cM(\bY_i) \to \bigoplus_{i<j} \cS(\bW \cap \bY_{ij}) \underset{\cS(\bY_{ij})}{\otimes} \cM(\bY_{ij}) \right) \end{equation}
for any $\bW \in \bX_{\ac}(\cL)$. Using this description, it is straightforward to see that $\cN$ is in fact a sheaf of $\cS$-modules. It also comes together with a canonical $\cS_i$-linear isomorphism
\[ \phi_i : \cN_{|\cY_i} \congs \cN_i\]
which is given on local sections by projection onto the $i$th component. 

Let $\cY_i/G := \cY_i \hsp \cap \hsp \bX_w/G$, let $\sT_i := \cS_{i| \cY_i/G}$, and let $\omega_i := \phi_{i|\cY_i/G}$. Since $\cM$ is a $\cU$-coherent $\sT$-module, we also have the canonical $\sT_i$-linear isomorphism
\[ c_i : \cN_{i|\cY_i/G} \congs \cM_{|\cY_i/G}.\]
Let $\bW \in \cY_{ij} \cap \bX_w/G$ and take an element
\[\xi = (s_r \otimes m_r)_{r=1}^n \in \cN(\bW) \subseteq \bigoplus_{r=1}^n \cS(\bW \cap \bY_r) \underset{\cS(\bY_r)}{\otimes} \cM(\bY_r).\]
Unravelling the notation, we see that for all $i$ and $\xi \in \cN(\bW)$ we have
\[ (c_i \circ \omega_i)\left( \xi \right) = s_i \cdot m_{i| \bW}.\]
The expression for $\cN(\bW)$ given by formula (\ref{NWformula}) now implies that 
\[ (c_i \circ \omega_i)_{|\cY_{ij}} = (c_j \circ \omega_j)_{|\cY_{ij}}\]
for any $i,j$. Thus, these local $\sT_i$-linear isomorphisms 
\[ c_i \circ \omega_i : \cN_{|\cY_i/G} \congs \cM_{|\cY_i/G}\]
glue to give the required $\sT$-linear isomorphism $\cN_{|\bX_{\ac}(\cL,G)} \congs \cM$. \end{proof}

\begin{thm}\label{LevelEqKiehl} Suppose that $[\cL,\cL] \subseteq \pi\cL$ and $\cL \cdot \cA \subseteq \pi \cA$. Let $\cU$ be an $\bX_{\ac}(\cL,G)$-covering which admits an $\cL$-accessible Laurent refinement, and let $\cM$ be a $\cU$-coherent sheaf of $\cQ$-modules on $\bX_{\ac}(\cL,G)$. Then $\cM(\bX)$ is a finitely generated $\cQ(\bX)$-module and the canonical $\cQ$-linear map $\Loc_{\cQ}(\cM(\bX)) \longrightarrow \cM$ is an isomorphism.
\end{thm}
\begin{proof} By Proposition \ref{Enhance}, we can find a $\cU$-coherent $\cS$-module $\cN$ on $\bX_{\ac}(\cL)$ and a $\sT$-linear isomorphism $\varphi : \cN_{|\bX_{\ac}(\cL,G)} \congs \cM.$ Let $\cV$ be an $\cL$-accessible Laurent refinement of $\cU$. Then $\cN$ is also a $\cV$-coherent $\cS$-module, so $\cN(\bX)$ is finitely generated as an $\cS(\bX)$-module by Theorem \ref{OldKiehl}. Since $\cN(\bX)$ is isomorphic to $\cM(\bX)$ as a $\sT(\bX)$-module via $\varphi(\bX)$, we see that the $\sT(\bX)$-module $\cM(\bX)$ is finitely generated. Therefore the $\cQ(\bX)$-module  $\cM(\bX)$ is also finitely generated. 

Finally, the diagram of sheaves of $\sT$-modules on $\bX_{\ac}(\cL,G)$
\[ \xymatrix{ \Loc_{\cQ} (\cM(\bX))_{|\sT} \ar[rr]\ar[d]_\cong && \cM_{|\sT} \\ \Loc_{\sT} (\cN(\bX)) \ar[rr] && \cN_{|\bX_{\ac}(\cL,G)} \ar[u]_{\cong}}\]
is commutative, the left vertical arrow is an isomorphism by Lemma \ref{LocMLocN} and the bottom horizontal arrow is an isomorphism because the canonical map
\[ \Loc_{\cS} (\cN(\bX)) \longrightarrow \cN\]
is an isomorphism by Theorem \ref{OldKiehl}. So the canonical map $\Loc_{\cQ}(\cM(\bX)) \to \cM$ is also an isomorphism.
\end{proof}

\subsection{Theorems of Tate and Kiehl in the equivariant setting}\label{TateFlatKiehl}

In this subsection, we give proofs of Theorem \ref{DGsheaf} and Theorem \ref{DUHcflat}, and complete the proof of Theorem \ref{LocEquiv}.

\begin{lem}\label{CoverStab} Suppose that $\bX$ is affinoid and that $\cL$ is a smooth $\cA$-Lie lattice in $\cT(\bX)$ for some affine formal model $\cA$ in $\cO(\bX)$. Let $\cU$ be a finite affinoid covering of $\bX$. Then there is a compact open subgroup $H$ of $G$ which stabilises $\cA$, $\cL$ and each member of $\cU$.
\end{lem}
\begin{proof} This follows from Lemma \ref{GstableFM}, Lemma \ref{StabLopen} and Proposition \ref{OpenStab}.
\end{proof}

We now relate $\w\cD(-,G)$ and $\cP^A_\bX(M)$ with the sheaves appearing in $\S$ \ref{SectionhsULKG}.

\begin{prop}\label{PassToInfty} Suppose $(\bX,G)$ is small. Choose a $G$-stable affine formal model $\cA$ in $\cO(\bX)$ and a $G$-stable smooth $\cA$-Lie lattice $\cL$ in $\cT(\bX)$. Let $(N_\bullet)$ be a good chain for for $\cL$, and let 
\[\cQ_n := \hK{\sU(\pi^n \cL)} \rtimes_{N_n} G \qmb{for each} n \geq 0,\]
viewed as a sheaf of $K$-Banach algebras on $\bX_w(\cL,G)$. \be \item There is an isomorphism
\[ \w\cD(-,G)_{|\bX_w(\cL,G)} \cong \invlim \cQ_n\]
of presheaves on $\bX_w(\cL,G)$.
\item Let $M$ be a coadmissible $\w\cD(\bX,G)$-module, and define 
\[\cM_n := \Loc_{\cQ_n}(\cQ_n(\bX) \otimes_{\w\cD(\bX,G)} M) \qmb{for each} n \geq 0.\] 
Then there is an isomorphism
\[ M(-,G)_{|\bX_{\ac}(\cL,G)}  \cong \invlim \cM_n\]
of presheaves on $\bX_{\ac}(\cL,G)$.\ee
\end{prop}
\begin{proof} (a)  This follows from Lemma \ref{StdPres}.

(b) Let $M_n :=\cQ_n(\bX) \otimes_{\w\cD(\bX,G)} M$ and let $\bU \in \bX_{\ac}(\cL,G)$. Then \[ \begin{array}{lll} M(\bU,G) &=& \w\cD(\bU,G) \underset{\w\cD(\bX,G)}{\w\otimes} M = \\
&=& \invlim \hsp \cQ_n(\bU) \underset{\w\cD(\bX,G)}{\otimes} M = \\
&\cong& \invlim \hsp \cQ_n(\bU) \underset{\cQ_n(\bX)}{\otimes} \left( \cQ_n(\bX) \underset{\w\cD(\bX,G)}{\otimes} M\right) = \\
&=& \invlim \hsp \cQ_n(\bU) \underset{\cQ_n(\bX)}{\otimes} M_n = \invlim \hsp \cM_n(\bU) \end{array}\]
functorially in $\bU \in \bX_{\ac}(\cL,G)$, and the result follows.
\end{proof}

\begin{proof}[Proof of Theorem \ref{DGsheaf}] Using Lemma \ref{SmallPairsExist}, choose a $\bU$-small subgroup $J$ of $G$. Proposition \ref{LocTrans} gives us an isomorphism 
\[\cP^A_\bX(M)_{|\bU_w} \cong \cP^{\w\cD(\bU,J)}_\bU\left(\w\cD(\bU,J) \w\otimes_{A_J} M\right)\]
of presheaves on $\bU_w$. By replacing $G$ by $J$ and $M$ by $\w\cD(\bU,J) \w\otimes_{A_J} M$, we can therefore assume that $(\bX,G)$ is small and $A = \w\cD(\bX,G)$. Let $\cM = \cP^A_\bX(M)$; it will suffice to show that for any finite admissible covering $\cU$ of $\bX$ by affinoid subdomains, the natural map $M \longrightarrow \check{H}^0(\cU, \cM)$ is an isomorphism.

Choose a $G$-stable affine formal model $\cA$ in $\cO(\bX)$ and a $G$-stable free $\cA$-Lie lattice $\cL$ in $\cT(\bX)$. By Lemma \ref{CoverStab} and Lemma \ref{RestToHonP}, we may assume that $\cA$, $\cL$ and each member of $\cU$ are all $G$-stable. By Corollary \ref{ResCor}, it will therefore be enough to show that the augmented \v{C}ech complex $C^\bullet_{\aug}(\cU,M(-,G))$ is exact. by replacing $\cL$ by $\pi \cL$ if necessary, we may ensure that $[\cL, \cL]\subseteq \pi \cL$ and $\cL \cdot \cA \subseteq \pi \cL$. By the proof of \cite[Proposition 7.6]{DCapOne}, we may further replace $\cL$ by a sufficiently large $\pi$-power multiple and thereby ensure that every member of $\cU$ is an $\cL$-accessible affinoid subdomain of $\bX$. Thus, without loss of generality, $\cU$ is an $\bX_{\ac}(\cL,G)$-covering.  With the notation from Proposition \ref{PassToInfty}, we have isomorphisms
\[ \w\cD(-,G)_{|\bX_{\ac}(\cL,G)} \cong \invlim \cQ_n \qmb{and} M(-,G)_{|\bX_{\ac}(\cL,G)} \cong \invlim \cM_n.\]
The augmented \v{C}ech complex $C^\bullet_{\aug}(\cU, \cM_n)$ is exact by Corollary \ref{TateForCohsQmod}. 
By \cite[Theorem B]{ST}, $\invlim{}^{(j)}\cM_n(\bX)=0 $ and $\invlim{}^{(j)}\cM_n(\bU)=0$ for each $j>0$ and each $\bU \in \cU$. Consider the exact complex of towers of $\cD$-modules $C^\bullet_{\aug}( \cU, (\cM_n) )$. An induction starting with the left-most term shows that $\invlim{}^{(j)}$ is zero on the kernel of every differential in this complex, for all $j>0$. Therefore $\invlim C_{\aug}^\bullet(\cU,\cM_n)$ is exact. But this complex is isomorphic to $C_{\aug}^\bullet(\cU,M(-,G))$.
\end{proof}

\begin{proof}[Proof of Theorem \ref{DUHcflat}] Choose an $H$-stable free $\cA$-Lie lattice $\cL$ in $\cT(\bX)$ for some $H$-stable affine formal model $\cA$ in $\cO(\bX)$. By rescaling $\cL$, we may assume that $\bU$ is $\cL$-accessible, that $[\cL,\cL] \subseteq \pi \cL$ and $\cL \cdot \cA \subseteq \pi \cL$. Choose some good chain $(N_\bullet)$ for $\cL$, and recall the sheaf $\cQ_n := \hK{\sU(\pi^n \cL)} \rtimes_{N_n} H$ on $\bX_{\ac}(\pi^n \cL,H)$ from Proposition \ref{PassToInfty}. By Lemma \ref{StdPres} there are compatible isomorphisms $\w\cD(\bX,H) \cong \invlim \cQ_n(\bX)$ and $\w\cD(\bU,H) \cong \invlim \cQ_n(\bU)$ , which give presentations of $\w\cD(\bX,H)$ and $\w\cD(\bU,H)$ as Fr\'echet-Stein algebras. Now $\cQ_n(\bU)$ is a flat right $\cQ_n(\bX)$-module for each $n\geq 0$ by Theorem \ref{FlatAcc}(a), and $\cQ_n(\bX)$ is a flat right $\w\cD(\bX,H)$-module by Theorem \ref{FrSt} and \cite[Remark 3.2]{ST}, so $\cQ_n(\bU)$ is a flat right $\w\cD(\bX,H)$-module for all $n \geq 0$. Hence $\w\cD(\bU,H)$ is a right c-flat $\w\cD(\bX,H)$-module by \cite[Proposition 7.5(b)]{DCapOne}.  
\end{proof}

We now work towards completing the proof of Theorem \ref{LocEquiv}. We will establish the following result, which can be viewed as being an extension of Kiehl's Theorem for coadmissible $\w\cD$-modules, \cite[Theorem 8.4]{DCapOne}, to the equivariant case.

\begin{thm}\label{Kiehl} Suppose that $(\bX,G)$ is small, and let $\cM$ be a coadmissible $G$-equivariant $\cD$-module on $\bX$. Then $\cM(\bX)$ is a coadmissible $\w\cD(\bX,G)$-module, and there is an isomorphism
\[ \Loc^{\w\cD(\bX,G)}_\bX(\cM(\bX)) \congs \cM\]
of $G$-equivariant locally Fr\'echet $\cD$-modules.
\end{thm}

The proof will require some preparation, and we begin with some topological preliminaries.

\begin{lem}\label{ExtendableModules} Let $A_0$ be a dense $K$-subalgebra of a Fr\'echet-Stein algebra $A$. Let $M$ be a $K$-Fr\'echet space equipped with the structure of an $A_0$-module, and suppose that $\psi : N \to M$ is a continuous $A_0$-linear isomorphism for some coadmissible $A$-module $N$.  Then
\be \item there is a coadmissible $A$-module structure on $M$ with respect to which $\psi$ is $A$-linear,
\item the canonical topology on $M$ induced by its structure as a coadmissible $A$-module coincides with its given Fr\'echet topology,
\item the $A$-module structure on $M$ is unique: it does not depend on $\psi$.
\ee \end{lem}
\begin{proof} (a) Since $\psi$ is a bijection, we can transport the $A$-module structure on $N$ to $M$ via $\psi$: that is, we define
\[ a \bullet m := \psi( a \hsp \psi^{-1}(m)) \qmb{for all} a \in A \qmb{and} m \in M.\]
Then $\psi$ is an $A$-linear bijection by construction, and hence $M$ becomes a coadmissible $A$-module.  

(b) Since $\psi$ is continuous by assumption, its inverse is also continuous by the Open Mapping Theorem \cite[Proposition 8.6]{SchNFA}.  Thus the given Fr\'echet topology on $M$ coincides with the canonical topology on $M$ induced by its structure as a coadmissible $A$-module. 

(c) Suppose $\omega : L \to M$ is another continuous $A_0$-linear isomorphism for some coadmissible $A$-module $L$. Then $\omega^{-1} : M \to L$ is continuous by the Open Mapping Theorem, and hence $\omega^{-1}\circ \psi : N \to L$ is a continuous $A_0$-linear bijection between two coadmissible $A$-modules. Since $A_0$ is dense in $A$, $\omega^{-1}\circ \psi$ is also $A$-linear and it follows that the two $A$-module structures on $M$ transported from $N$ and $L$ respectively, coincide:
\[ \psi(a \hsp \psi^{-1}(m)) = \omega\left(\omega^{-1}\psi (a \hsp \psi^{-1}(m))\right) = \omega\left(a \hsp \omega^{-1}\psi(\psi^{-1}(m))\right) = \omega(a \hsp \omega^{-1}(m))\]\
for all $a \in A$ and $m\in M$.
\end{proof}

\begin{prop}\label{Reconstruct} Suppose that $(\bX, G)$ is small. Let $H$ be an open subgroup of $G$, let $M \in \cC_{\w\cD(\bX,H)}$ and suppose that  $\alpha : \Loc^{\w\cD(\bX,H)}_\bX(M) \congs \cM$ is an isomorphism of $H$-equivariant locally Fr\'echet $\cD$-modules on $\bX$.
\be \item For every $\bU \in \bX_w$, there is a unique coadmissible $\w\cD(\bU,H_\bU)$-module structure on $\cM(\bU)$ such that
 \begin{enumerate}[{(}i{)}]
 \item $\gamma^{H_\bU}(g) \cdot m = g^{\cM}(m)$ for all $g \in H_\bU$ and $m \in \cM(\bU)$,
 \item the topology on $\cM(\bU)$ induced by this coadmissible $\w\cD(\bU,H_\bU)$-module structure coincides with the given $K$-Fr\'echet topology on $\cM(\bU)$,
 \item the $\w\cD(\bU,H_\bU)$-action on $\cM(\bU)$ extends the given $\cD(\bU)$-action on $\cM(\bU)$.
 \end{enumerate}
\item These $\w\cD(\bU,H_\bU)$-module structures on $\cM(\bU)$ do not depend on $\alpha$.
\item There is an isomorphism of $H$-equivariant locally Fr\'echet $\cD$-modules on $\bX$
\[\theta : \Loc^{\w\cD(\bX,H)}_\bX(\cM(\bX)) \congs \cM,\] 
whose restriction to $\bX_w$ is given by
\[ \theta(\bU)(s \hsp \w\otimes \hsp m) = s \cdot (m_{|\bU})\]
for any $\bU \in \bX_w$, $s \in \w\cD(\bU,H_\bU)$ and $m \in \cM(\bX)$.
\ee\end{prop}
\begin{proof} (a) Suppose first that $\bU = \bX$ and write $\cN := \Loc^{\w\cD(\bX,H)}_\bX(M)$. Let $\psi : M \to \cM(\bX)$ be the composite of the canonical map $M \to \cN(\bX)$ and $\alpha(\bX) : \cN(\bX) \to \cM(\bX)$; it is a continuous isomorphism by Theorem \ref{DGsheaf} and the definition of the topology on $\cN(\bX)$. Now $\Gamma(\bX,-)$ is a functor from $H$-equivariant $\cD$-modules on $\bX$ to $\cD(\bX) \rtimes H$-modules by Proposition \ref{EqGlSec}, so the continuous map $\alpha(\bX) : \cN(\bX) \to \cM(\bX)$ is $\cD(\bX) \rtimes H$-linear. Hence $\psi : M \to \cM(\bX)$ is also $\cD(\bX) \rtimes H$-linear.  Since $\cD(\bX) \rtimes H$ is dense in $\w\cD(\bX,H)$ by construction, we may apply Lemma \ref{ExtendableModules} to $\psi : M \to \cM(\bX)$ to deduce that the $\cD(\bX) \rtimes H$-action on $\cM(\bX)$ extends to an action of $\w\cD(\bX,H)$ which satisfies all the required properties.  

The general case follows from this special case in view of (\ref{RestOfLocShvs}).

(b) This follows from Lemma \ref{ExtendableModules}(c).

(c) By part (a), we know that $\cM(\bX)$ is a coadmissible $\w\cD(\bX,H)$-module, and the map $\psi : M \to \cM(\bX)$ given by $\psi(m) = \alpha(\bX)(1 \hsp \w\otimes \hsp m)$ is a $\w\cD(\bX,H)$-linear isomorphism. Write $\cP := \cP^{\w\cD(\bX,H)}_\bX$ and consider the diagram
\[ \xymatrix{ & \cP(M) \ar[dl]_{\cP(\psi)}\ar[dr]^{\alpha_{|\bX_w}} & \\ \cP(\cM(\bX)) \ar[rr]_\theta && \cM_{|\bX_w}}\]
where $\theta(\bU)( s \w\otimes m ) =s \cdot (m_{|\bU})$ for $\bU \in \bX_w$, $s \in \w\cD(\bU,H_\bU)$ and $m \in \cM(\bU)$. We identify $\cP(M)(\bU)$ with $\w\cD(\bU,H_\bU) \underset{\w\cD(\bX,H)}{\w\otimes} M$ for each $\bU \in \bX_w$ and compute
\[ \begin{array}{lll} (\theta \circ \cP(\psi))(\bU)(s \hsp \w\otimes \hsp m) &=& \theta(\bU)( s \hsp \w\otimes \hsp \psi(m)) = s \cdot \alpha(\bX)(1 \hsp \w\otimes \hsp m)_{|\bU} =\\ &=& s \cdot \alpha(\bU)( 1 \hsp \w\otimes \hsp m) = \alpha(\bU) (s \cdot (1 \hsp \w\otimes \hsp m)) = \\ &=& \alpha(\bU)(s \hsp \w\otimes \hsp m).\end{array}\]
We have used that $\alpha(\bU)$ is $\w\cD(\bU,H_\bU)$-linear, and also that $\alpha$ is a morphism of sheaves on $\bX_w$ in this calculation.  Since $\cP(\psi)$ and $\alpha$ are morphisms of sheaves on $\bX_w$ and since $\alpha(\bU)$ is an isomorphism, it follows that $\theta$ is also a morphism of sheaves on $\bX_w$ and that the triangle above is commutative. Since $\psi$ is an isomorphism, it follows that $\theta$ is also an isomorphism. By applying \cite[Theorem 9.1]{DCapOne}, we see that $\theta$ extends to the required isomorphism $\theta: \Loc^{\w\cD(\bX,H)}_\bX(\cM(\bX)) \congs \cM$ of sheaves on $\bX$, which is $H$-equivariant, $\cD$-linear and continuous because $\cP(\psi)$ and $\alpha_{|\bX_w}$ both have these properties.
\end{proof}

Until the end of $\S$ \ref{TateFlatKiehl}, we will assume the following
\begin{setup}\label{KiehlProofSetup} \hspace{2em}
\begin{itemize} 
\item $\bX$ is an \emph{affinoid} variety,
\item $G$ is \emph{compact},
\item $\cA$ is a $G$-stable affine formal model in $\cO(\bX)$,
\item $\cL$ is a $G$-stable free $\cA$-Lie lattice in $\cT(\bX)$.
\item $[\cL,\cL] \subseteq \pi\cL$ and $\cL \cdot \cA \subseteq \pi \cA$,
\item $\cU$ is a finite $\bX_{\ac}(\cL)$-covering of $\bX$,
\item $\cU$ admits an $\cL$-accessible refinement,
\item $H$ is an open normal subgroup of $G$ which stabilises every member of $\cU$,
\item $(N_\bullet)$ is a good chain for $\cL$ in $H$,
\item $\cM$ is a $\cU$-coadmissible $H$-equivariant $\cD$-module on $\bX$.
\end{itemize}
\end{setup}

\begin{notn} Let $n \geq 0$ and $\bY$ be an intersection of members of $\cU$.
\be \item $\cX_n := \bX_{\ac}(\pi^n \cL,H)$, a $G$-topology on $\bX$,
\item $\cY_n := \bY_{\ac}(\pi^n \cL,H) = \cX_n \cap \bY_w$, a $G$-topology on $\bY$.
\item $\cQ_n := \hK{\sU(\pi^n \cL)} \rtimes_{N_n} H$, a sheaf of $K$-Banach algebras on $\cX_n$,
\item $\cQ_\infty := \w\cD(-,H)$, a sheaf of Fr\'echet-Stein algebras on $\bX_w/H$.
\ee\end{notn}
We refer the reader to Corollary \ref{QSheaf} and Theorems \ref{FrSt}/\ref{DGsheaf} for assertions (c) and (d) above. Recall that the $G$-topologies $\cX_n$ become finer as $n$ increases:
\[ \cX_0 \subset \cX_1 \subset \cX_2 \subset \cdots \subset \cX_n \subset \cdots \subset \bX_w/H,\]
and that by Proposition \ref{PassToInfty}(a) there is an isomorphism of sheaves on $\cX_0$
\[\cQ_{\infty|\cX_0} \congs \invlim \cQ_n.\]
Since $\cM$ is $\cU$-coadmissible, it follows from Proposition \ref{BiFunc}(b) and Proposition \ref{Reconstruct} that $\cM_{\bY_w/H}$ is naturally a $\cQ_{\infty|\bY_w/H}$-module, for each $\bY\in \cU$. 

\begin{lem}\label{QinftyAction} There is a unique structure of a $\cQ_\infty$-module on $\cM_{|\bX_w/H}$ which extends the $\cQ_{\infty|\bY_w/H}$-module structure on $\cM_{|\bY_w/H}$ for each $\bY \in \cU$. 
\end{lem}
\begin{proof} It follows from Proposition \ref{Reconstruct}(b) that the action maps 
\[a_\bY : \cQ_{\infty|\bY_w/H} \times \cM_{|\bY_w/H} \to \cM_{\bY_w/H}\qmb{for} \bY \in \cU,\] 
agree on overlaps: $(a_\bY)_{|\bY \cap \bY'} = (a_{\bY'})_{|\bY\cap \bY'}$ for all $\bY,\bY'\in\cU$. Since $\cQ_{\infty|\bY_w/H}$ and $\cM_{|\bY_w/H}$ are sheaves and since each $a_\bY$ is a sheaf morphism, it follows that they patch together uniquely to a sheaf morphism $a : \cQ_{\infty} \times \cM_{|\bX_w/H} \to \cM_{|\bX_w/H}$. It is now straightforward to verify that $a$ defines the structure of a $\cQ_\infty$-module on $\cM_{|\bX_w/H}$.
\end{proof}

\begin{lem}\label{UcoherentMn} Let $n \geq 0$ and let $\bY \in \cU$. 
\be \item There is a presheaf $\cP_n$ on $\cX_n$ of $\cQ_n$-modules given by 
\[ \cP_n(\bZ) = \cQ_n(\bZ) \underset{\cQ_\infty(\bZ)}{\otimes} \cM(\bZ) \qmb{for all} \bZ \in \cX_n.\]
\item The canonical map $\Loc_{\cQ_n}( \cP_n(\bY) ) \to \cP_{n|\cY_n}$ is an isomorphism.
\item The restriction of $\cP_n$ to $\cY_n$ is a sheaf.
\ee\end{lem}
\begin{proof}(a) Note that $\cM(\bZ)$ is a $\cQ_\infty(\bZ)$-module for each $\bZ \in \cX_n \subset \bX_w/H$ by Lemma \ref{QinftyAction}, so the formula defining $\cP_n(\bZ)$ makes sense. Because $\cQ_n$, $\cQ_{\infty|\cX_n}$ and $\cM_{|\cX_n}$ are all well-defined functors on $\cX_n$, the functoriality of tensor product ensures that $\cP_n$ is a presheaf.

(b) Because $\cM$ is $\cU$-coadmissible and $\bY \in \cU$, we have at our disposal the isomorphism $\theta_\bY : \Loc^{\w\cD(\bY,H)}_\bY(\cM(\bY)) \congs \cM_{|\bY}$ from Proposition \ref{Reconstruct}(c). Let $\bZ \in \cY_n$ and consider the following diagram:
\[\xymatrix@=18pt{ \cQ_n(\bZ) \underset{\cQ_n(\bY)}{\otimes} \cP_n(\bY) \ar[r]\ar@{=}[d] & \cP_n(\bZ) \ar@{=}[d] \\ 
\cQ_n(\bZ) \underset{\cQ_n(\bY)}{\otimes}\left(\cQ_n(\bY) \underset{\cQ_\infty(\bY)}{\otimes} \cM(\bY)\right) \ar[d]_{\cong} & \cQ_n(\bZ) \underset{\cQ_\infty(\bZ)}{\otimes}\cM(\bZ) \\
\cQ_n(\bZ) \underset{\cQ_\infty(\bY)}{\otimes} \cM(\bY) \ar@{=}[r] & \cQ_n(\bZ) \underset{\cQ_\infty(\bZ)}{\otimes}  \left(\cQ_\infty(\bZ) \underset{\cQ_\infty(\bY)}{\w\otimes} \cM(\bY) \right).\ar[u]^{\cong}_{1 \otimes \theta_\bY(\bZ)} }\]
If $b \in \cQ_n(\bY)$ and $m \in \cM(\bY)$, we have $(b \otimes m)_{|\bZ} = b_{|\bZ} \otimes m_{|\bZ}$ by definition. It follows that the diagram commutes and hence the top horizontal arrow is an isomorphism.

(c) Since $\cM(\bY)$ is a coadmissible $\cQ_\infty(\bY)$-module by Proposition \ref{Reconstruct}(a), $\cP_n(\bY)$ is a finitely generated $\cQ_n(\bY)$-module. Now apply Corollary \ref{TateForCohsQmod}.
\end{proof}

\begin{defn} $\cM_n$ is the sheafification of the presheaf $\cP_n$ on $\cX_n$. \end{defn}

\begin{cor}\label{UcohQn} Let $n \geq 0$.
\be \item $\cM_n$ is a $\cU$-coherent $\cQ_n$-module on $\cX_n$.
\item $\cM_n(\bX)$ is a finitely generated $\cQ_n(\bX)$-module.
\item The canonical $\cQ_n$-linear morphism
\[\sigma_n : \Loc_{\cQ_n}( \cM_n(\bX) ) \longrightarrow \cM_n\]
is an isomorphism.
\ee\end{cor}
\begin{proof} Part (a) follows immediately from Lemma \ref{UcoherentMn}, and parts (b), (c) follow from part (a) and Theorem \ref{LevelEqKiehl}, because we have assumed that $\cU$ admits an $\cL$-accessible refinement.
\end{proof}

\begin{lem}\label{TausBetweencMns} For each $n \geq 0$ there is a $\cQ_n$-linear isomorphism
\[ \tau_n : \Loc_{\cQ_n}\left( \cQ_n(\bX) \otimes_{\cQ_{n+1}(\bX)} \cM_{n+1}(\bX) \right) \congs \Loc_{\cQ_n}\left(\cM_n(\bX)\right).\]
\end{lem}
\begin{proof} For every $\bZ \in \cX_n$, by the definition of $\cP_n(\bZ)$ there is a canonical functorial isomorphism
\[ \cQ_n(\bZ) \underset{\cQ_{n+1}(\bZ)}{\otimes} \cP_{n+1}(\bZ) \congs \cP_n(\bZ). \]
If $\bZ$ happens to be contained in $\bY$ for some $\bY \in \cU$, then Lemma \ref{UcoherentMn}(c) induces a functorial isomorphism
\[ \cQ_n(\bZ) \underset{\cQ_{n+1}(\bZ)}{\otimes} \cM_{n+1}(\bZ) \congs \cM_n(\bZ)\]
which does not depend on the choice of $\bY$ and appears in the bottom row of the following diagram:
\[ \xymatrix{ 
\cQ_n(\bZ) \underset{\cQ_n(\bX)}{\otimes} \left(\cQ_n(\bX) \underset{\cQ_{n+1}(\bX)}{\otimes} \cM_{n+1}(\bX)\right) \ar@{.>}[r]^(.62){\tau_n(\bZ)}\ar[d]_{\cong} & \cQ_n(\bZ) \underset{\cQ_n(\bX)}{\otimes}\cM_n(\bX) \ar[dd]_{\cong}^{\sigma_n(\bZ)} \\ 
\cQ_n(\bZ) \underset{\cQ_{n+1}(\bZ)}{\otimes} \left(\cQ_{n+1}(\bZ) \underset{\cQ_{n+1}(\bX)}{\otimes} \cM_{n+1}(\bX)\right) \ar[d]_{1 \otimes \sigma_{n+1}(\bZ)} & \\ 
\cQ_n(\bZ) \underset{\cQ_{n+1}(\bZ)}{\otimes} \cM_{n+1}(\bZ) \ar[r]_(.62){\cong} & \cM_n(\bZ) }\]
All solid arrows in this diagram are isomorphisms, so we obtain the isomorphism
\[ \tau_n(\bZ) : \cQ_n(\bZ) \underset{\cQ_n(\bX)}{\otimes} \left(\cQ_n(\bX) \underset{\cQ_{n+1}(\bX)}{\otimes} \cM_{n+1}(\bX)\right)  \congs \cQ_n(\bZ) \underset{\cQ_n(\bX)}{\otimes}\cM_n(\bX) \]
as the top dotted arrow in the diagram. This morphism is functorial in $\bZ$. Now both $ \cQ_n(\bX) \otimes_{\cQ_{n+1}} \cM_{n+1}(\bX)$ and $\cM_n(\bX)$ are finitely generated $\cQ_n(\bX)$-modules by Corollary \ref{UcohQn}(b) and $\cU$ is an $\cX_n$-covering, so by Corollary \ref{TateForCohsQmod} these $\tau_n(\bZ)$'s patch together to give the required $\cQ_n$-linear isomorphism $\tau_n$.
\end{proof}

\begin{cor}\label{MinftyCoadm} The $\cQ_\infty(\bX)$-module $M_\infty := \invlim \cM_n(\bX)$ is coadmissible. 
\end{cor}
\begin{proof} By Lemma \ref{TausBetweencMns}, the maps $\tau_n(\bX)$ induce $\cQ_n(\bX)$-linear isomorphisms
\[ \cQ_n(\bX) \otimes_{\cQ_{n+1}(\bX)} \cM_{n+1}(\bX) \congs \cM_n(\bX) \qmb{for each} n \geq 0.\]
Therefore $M_\infty$ is a coadmissible $\invlim \cQ_n(\bX) = \cQ_\infty(\bX)$-module, by Proposition \ref{PassToInfty}(a).
\end{proof}

If $\bY \in \cU$ then $\cM(\bY)$ is a coadmissible $\cQ_\infty(\bY)$-module by Proposition \ref{Reconstruct}(a). Hence the canonical map
\[ \cM(\bY) \longrightarrow \invlim \cQ_n(\bY) \underset{\cQ_\infty(\bY)}{\otimes} \cM(\bY) \]
is an isomorphism, and we will identify $\cM(\bY)$ with $\invlim \cQ_n(\bY) \underset{\cQ_\infty(\bY)}{\otimes} \cM(\bY)$.

\begin{lem}\label{NuLemma} For each $\bY \in \cU$, there is a $\cQ_\infty(\bX)$-linear map
\[ \nu_\bY : M_\infty \longrightarrow \cM(\bY)\]
such that $\nu_\bY(m)_{|\bY \cap \bY'} = \nu_{\bY'}(m)_{|\bY \cap \bY'}$ for all $m \in M_\infty$ and all $\bY' \in \cU$.
\end{lem}
\begin{proof} Let $m = (m_n)_n \in M_\infty$ where $m_n \in \cM_n(\bX)$, and define $\nu_\bY$ by 
\[\nu_\bY(m) := ( (m_n)_{|\bY} )_n.\]
This is $\cQ_\infty(\bX)$-linear because the restriction maps in $\cM_n$ are $\cQ_n(\bX)$-linear. Now
\[ \nu_\bY(m)_{|\bY \cap \bY',n} = (m_{n|\bY})_{|\bY\cap \bY'} = m_{n |\bY\cap \bY'} = (m_{n|\bY'})_{|\bY\cap \bY'} = \nu_{\bY'}(m)_{|\bY \cap \bY',n}\]
for all $n \geq 0$. Hence $\nu_\bY(m)_{|\bY \cap \bY'} = \nu_{\bY'}(m)_{|\bY \cap \bY'}$ for all $m \in M_\infty$.
\end{proof}

\begin{prop}\label{HisomAlpha} There is an isomorphism
\[ \Loc^{\w\cD(\bX,H)}_\bX(M_\infty) \congs \cM\]
of $H$-equivariant locally Fr\'echet $\cD$-modules on $\bX$.
\end{prop}
\begin{proof} Because $\Loc^{\w\cD(\bX,H)}_\bX(M_\infty)$ and $\cM$ are sheaves, by \cite[Theorem 9.1]{DCapOne} it is sufficient to construct an isomorphism of $H$-equivariant $\cD$-modules on $\bX_w$
\[ \alpha : \cP^{\w\cD(\bX,H)}_\bX(M_\infty) \congs \cM_{|\bX_w}.\]
Let $\bY \in \cU$. Using Lemma \ref{NuLemma}, define
\[g_\bY : \w\cD(\bY,H) \underset{\w\cD(\bX,H)}{\w\otimes} M_\infty \longrightarrow \cM(\bY)\]
by setting $g_\bY(s \w\otimes m) = s \cdot \nu_\bY(m)$. This is a $\w\cD(\bY,H)$-linear map. The diagram
\[\xymatrix{
\cQ_\infty(\bY) \underset{\cQ_\infty(\bX)}{\w\otimes} M_\infty \ar[r]^{g_\bY}\ar@{=}[d] &  \cM(\bY) \ar[d]^{\cong} \\
\invlim \cQ_n(\bY) \underset{\cQ_n(\bX)}{\otimes} \cM_n(\bX) \ar[d]^{\cong}_(0.6){\invlim \sigma_n(\bY)} & \invlim \cQ_n(\bY) \underset{\cQ_\infty(\bY) }{\otimes} \cM(\bY) \ar@{=}[d] \\
\invlim \cM_n(\bY)  & \invlim \cP_n(\bY) \ar[l]^{\cong}
}\]
is commutative by construction. The bottom left vertical arrow is an isomorphism by Corollary \ref{UcohQn}(c), and the bottom horizontal arrow is an isomorphism by Lemma \ref{UcoherentMn}(b). Hence $g_\bY$ is an isomorphism. Now consider the following diagram:
\[ \xymatrix{ 
\cP_\bX^{\w\cD(\bX,H)}(M_\infty)_{|\bY_w} \ar@{.>}[rrrr]^{\alpha_\bY} \ar[d]_{\cong} &&&& \cM_{|\bY_w} \\
\cP^{\w\cD(\bY,H)}_\bY(\w\cD(\bY,H) \underset{\w\cD(\bX,H)}{\w\otimes} M_\infty) \ar[rrrr]_{\cP^{\w\cD(\bY,H)}_\bY(g_\bY)} &&&& \cP_\bY^{\w\cD(\bY,H)}(\cM(\bY)) \ar[u]_{\theta_\bY}^{\cong}
}\]
where the left vertical arrow is the isomorphism given by Proposition \ref{LocTrans}, and the right vertical arrow $\theta_\bY$ is the isomorphism given by Proposition \ref{Reconstruct}(c). The bottom arrow is an isomorphism by Proposition \ref{LocFunctor}, and thus we obtain the $H$-equivariant $\cD$-linear isomorphism
\[ \alpha_\bY : \cP^{\w\cD(\bX,H)}_\bX(M_\infty)_{|\bY_w} \congs \cM_{|\bY_w}\]
which makes the diagram commute. 

Recall from Proposition \ref{Reconstruct}(a) that for any $\bZ \in \bY_w$, $\cM(\bZ)$ is naturally a coadmissible $\cD(\bZ,H_\bZ)$-module. Using Corollary \ref{ResCor} to identify $\cP^{\w\cD(\bX,H)}_\bX(M_\infty)(\bZ)$ with $\w\cD(\bZ,H_\bZ) \underset{\w\cD(\bX,H_\bZ)}{\w\otimes} M_\infty$, it is straightforward to verify that 
\[ \alpha_\bY(\bZ) : \w\cD(\bZ,H_\bZ) \underset{\w\cD(\bX,H_\bZ)}{\w\otimes} M_\infty \longrightarrow \cM(\bZ)\]
is given by
\[ \alpha_\bY(\bZ)(s \w\otimes m) = s \cdot (\nu_\bY(m)_{|\bZ}) \qmb{for all}   s \in \w\cD(\bZ,H_\bZ) \qmb{and} m \in M_\infty.\]
Using Lemma \ref{NuLemma}, we see that the local isomorphisms $\alpha_\bY$ satisfy
\[ (\alpha_\bY)_{|\bY \cap \bY'} = (\alpha_{\bY'})_{|\bY \cap \bY'} \qmb{for any} \bY, \bY' \in \cU.\]
Since $\cM_{|\bX_w}$ is a sheaf by assumption and since $\cP^{\w\cD(\bX,H)}_\bX(M_\infty)$ is a sheaf on $\bX_w$ by Theorem \ref{DGsheaf}, the $\alpha_\bY$'s patch together to give the required isomorphism $\alpha$.
\end{proof}

\begin{proof}[Proof of Theorem \ref{Kiehl}]
Let $\cU$ be an admissible affinoid covering of $\bX$ and let $\cM$ be a $\cU$-coadmissible $G$-equivariant $\cD$-module on $\bX$. Because $(\bX,G)$ is small, we can find a $G$-stable affine formal model $\cA$ in $\cO(\bX)$ and a $G$-stable free $\cA$-Lie lattice $\cL$ in $\cT(\bX)$. 

Because $\bX$ is affinoid, by replacing $\cU$ by a finite refinement and applying Lemma \ref{C/Grefinement} we may assume that $\cU$ is finite. Choose and fix a Laurent refinement $\cV$ of $\cU$, using \cite[Lemmas 8.2.2/2-4]{BGR}. By applying \cite[Proposition 7.6]{DCapOne}, we may replace $\cL$ by a sufficiently large $\pi$-power multiple, and thereby ensure that every member of $\cU$ and of $\cV$ is an $\cL$-accessible affinoid subdomain of $\bX$.  Thus $\cU$ is $\cL$-accessible, and admits an $\cL$-accessible Laurent refinement. By replacing $\cL$ by $\pi \cL$ if necessary, we may assume further that $[\cL, \cL] \subseteq \pi \cL$ and $\cL \cdot \cA \subseteq \pi \cA$. By Lemma \ref{CoverStab}, we can find an open normal subgroup $H$ of $G$ which stabilises $\cA$, $\cL$ and each member of $\cU$. Choose a good chain $(N_\bullet)$ for $\cL$ using Lemma \ref{FastChain}.  Thus, all conditions of Hypothesis \ref{KiehlProofSetup} are satisfied.

Now $M := \cM(\bX)$ is a coadmissible $\w\cD(\bX,H)$-module by Corollary \ref{MinftyCoadm}, so by Proposition \ref{HisomAlpha} and Proposition \ref{Reconstruct}(c), there is an isomorphism
\[ \theta : \Loc_\bX^{\w\cD(\bX,H)} (M)\congs \cM\]
of $H$-equivariant locally Fr\'echet $\cD$-modules on $\bX$, given by 
\[ \theta(\bU)(s \hsp \w\otimes \hsp m) = s \cdot (m_{|\bU})\]
for any $\bU \in \bX_w$, $s \in \w\cD(\bU,H_\bU)$ and $m \in M$. On the other hand, $M$ is a $\cD(\bX) \rtimes G$-module by Proposition \ref{EqGlSec}, and by Proposition \ref{Reconstruct}(a), the actions of $\w\cD(\bU,H)$ and $\cD(\bX)\rtimes G$ are compatible in the following sense:
\[ \gamma^H(h) \cdot m = h^\cM(m) \qmb{for all} h\in H\qmb{and} m\in M.\]
Thus the $\w\cD(\bX,H)$-action on $M$ extends to an action of $\w\cD(\bX,H) \rtimes_H G \cong \w\cD(\bX,G)$, by Corollary \ref{Stacky}. Since the restriction of $M$ back to $\w\cD(\bX,H)$ is coadmissible, $M$ is necessarily coadmissible also as a $\w\cD(\bX,G)$-module. 

Note that by construction, $\Loc_\bX^{\w\cD(\bX,G)}(M) = \Loc_\bX^{\w\cD(\bX,H)}(M)$ as $H$-equivariant locally Fr\'echet $\cD$-modules on $\bX$. We will now verify that the isomorphism $\theta$ is $G$-equivariant. To this end, fix $\bU \in \bX_w$ and $g \in G$, and consider the diagram
\[\xymatrix{ M(\bU,H_\bU) \ar[rrrr]^{\theta(\bU)} \ar[d]_{g^M_{\bU,H_\bU}} &&&& \cM(\bU) \ar[d]^{g^\cM(\bU)} \\ 
M(g\bU, gH_\bU g^{-1}) \ar[rrrr]_{\theta(g\bU)} &&&& \cM(g\bU). }\]
Now $\cM(\bU)$ is a coadmissible $\w\cD(\bU,H_\bU)$-module and $\cM(g\bU)$ is a coadmissible $\w\cD(g\bU, gH_\bU g^{-1})$-module by Proposition \ref{Reconstruct}(a), and it is straightforward to see that $\theta(\bU)$ is $\w\cD(\bU,H_\bU)$-linear, whereas $\theta(g\bU)$ is $\w\cD(g\bU, gH_\bU g^{-1})$-linear. 

Using the isomorphism $\w{g}_{\bU,H_\bU} : \w\cD(\bU,H_\bU) \congs \w\cD(g\bU, gH_\bU g^{-1})$, we can regard $\cM(g\bU)$ as a coadmissible $\w\cD(\bU,H)$-module. Because $\cM$ is a $G$-equivariant $\cD$-module, it follows from Definition \ref{GAmodDefn}(a) that the map $g^{\cM}(\bU) : \cM(\bU) \to \cM(g\bU)$ is $\cD(\bU) \rtimes H_\bU$-linear. Because it is also \emph{continuous} by Definition \ref{FrechDmod}(a), we see that it is actually $\w\cD(\bU,H_\bU)$-linear. Similarly, regarding $M(g\bU, gH_\bU g^{-1})$ as a coadmissible $\w\cD(\bU, H_\bU)$-module via $\w{g}_{\bU,H_\bU}$, the maps $g^M_{\bU,H_\bU}$ and $\theta(g\bU)$ become $\w\cD(\bU,H_\bU)$-linear. Since the image of $M$ in $M(\bU,H_\bU)$ generates a dense $\w\cD(\bU,H_\bU)$-submodule in $M(\bU,H_\bU)$, to show that the diagram commutes it is enough to verify that $g^{\cM}(\bU) \circ \theta(\bU)$ and $\theta(g\bU)\circ g^M_{\bU,H_\bU}$ agree on this image, by Lemma \ref{AutoCts}. But
\[ \begin{array}{lllll}
    [g^\cM(\bU) \circ \theta(\bU)](1 \hsp\w\otimes\hsp m) &=& g^{\cM}(\bU)( m_{|\bU} ) &=& g^{\cM}(\bX)(m)_{|\bU} =\\
    &=& \theta(g\bU)(1 \hsp \w\otimes \hsp g\cdot m) &=& [\theta(g\bU) \circ g^M_{\bU,H_\bU}](1 \hsp \w\otimes \hsp m)
\end{array}\]
for all $m \in M$, and the result follows.\end{proof}

This completes the proof of Theorem \ref{LocEquiv}.

\section{Beilinson-Bernstein localisation theory}

\subsection{Invariant vector fields on affine formal group schemes} We begin by reminding the reader some basic facts from the theory of group schemes, following Demazure and Gabriel, \cite{DG}. Let $R$ be an arbitary commutative ring. Recall that an \emph{$R$-functor} is a set-valued functor on the category of (small, commutative) $R$-algebras, and an \emph{$R$-group-functor} is a group-valued functor on the same category. We will always identify a scheme $X$ over $R$ with its $R$-functor of points $\X = \Sch_R(\Spec(-),X)$. 

Let $\G$ be an $R$-group-functor. By \cite[Chapter II, \S 4, 1.2]{DG}, for every $R$-algebra $S$ we have the group
\[ \zLie(\G)(S) = \ker \left( \G(S[\tau]) \to \G(S) \right),\]
where $S[\tau] := S[T]/\langle T^2\rangle$ is the ring of dual numbers. This construction yields another $R$-group functor $\zLie(\G)$, and $\zLie(-)$ is functorial in $\G$. If $\G$ is actually a group \emph{scheme} over $R$, then it is shown in \cite[Chapter II, \S 4, Proposition 4.5]{DG} that $\zLie(\G)(S)$ is an $S$-Lie algebra for every $R$-algebra $S$. In this case, the \emph{Lie algebra} of $\G$ is defined to be $\Lie(\G) := \zLie(\G)(R)$.

Let $\X$ be an $R$-functor. Its \emph{automorphism group} $\zAut(\X)$ is given by
\[ \zAut(\X)(S) = \Aut_S( \X \otimes_R S)\]
for every $R$-algebra $S$. $\zAut(\X)$ is another $R$-functor, and $\X \otimes_R S$ denotes the base-change of $\X$ to $S$ \cite[Chapter II, \S 1, 2.7]{DG}. Explicitly, $\X \otimes_R S$ is simply the restriction of $\X$ to the category of commutative $S$-algebras, and $\Aut_S( \X \otimes_R S)$ is the group of invertible natural transformations $\X \otimes_R S \stackrel{\cong}{\to} \X \otimes_RS$. 

By \cite[Chapter II, \S 4, Proposition 2.4]{DG}, there is homomorphism
\[ \zLie(\zAut(\X))(R) \to \Der_R( \cO_{\X} ), \quad \phi \mapsto D_\phi\]
to the group of $R$-linear derivations $\Der_R(\cO_{\X})$ of the structure sheaf $\cO_{\X}$, \cite[Chapter I, \S1, 6.1]{DG}, of the $R$-functor $\X$. It is given by the following formula:
\begin{equation}\label{InfActFormula} f( \phi_S(x_{S[\tau]})) = f(x) + \tau D_\phi(\Y)(f), \end{equation}
where $S$ is an $R$-algebra, $\phi \in \zLie(\zAut(\X))(R)$, $\phi_S$ is its image in $\zLie(\zAut(\X))(S)$, $\Y$ is an open subfunctor of $\X$, $f \in \cO(\Y)$, $x \in \X(S)$ and $x_{S[\tau]}$ is its image in $\X(S[\tau])$.

In the case where $\X$ is actually an $R$-scheme, $\Der_R(\cO_{\X})$ is simply the space $\cT(\X)$ of global vector fields on $\X$. 

\begin{defn}\label{InfActionDefn} Let $\G \times \X \to \X$ be an action of the $R$-group scheme $\G$ on the $R$-scheme $\X$, and let $\varphi : \G \to \zAut(\X)$ be the corresponding homomorphism. The \emph{infinitesimal action of $\fr{g} := \Lie(\G)$ on $\X$ associated to $\varphi$} is the $R$-linear map
\[ \varphi' : \fr{g} \to \cT(\X)\]
given by $\varphi'(u) = D_\phi$ where $\phi := \zLie(\varphi)(R)(u^{-1}) \in \zLie(\zAut(\X))(R)$. 
\end{defn}

\begin{rmks} \hspace{2em}
\be\item Note that our notation differs from the one used in \cite[Chapter II, \S 4, 4.4]{DG} because of the minus sign in the exponent of $u$. With this modified notation, \cite[Chapter II, \S 4, Proposition 4.4]{DG} tells us that the infinitesimal action $\varphi' : \fr{g} \to \cT(\X)$ is a \emph{homomorphism} of $R$-Lie algebras, and not an anti-homomorphism.
\item The infinitesimal action is given by the intuitive formula
\[ \varphi'(g)(f)(x) = \frac{ f(g^{-1}x) - f(x) }{\tau}, \qmb{for all} g \in \fr{g},  f \in \cO_{\X} \qmb{and} x \in \G\]
which may explain the phrase ``differentiating the $\G$-action on $\X$".
\ee\end{rmks}

Let $\varphi : \G \to \zAut(\X)$ be an action of the $R$-group scheme $\G$ on the $R$-scheme $\X$. The structure sheaf $\cO_{|\X|}$ on the underlying topological space $|\X|$ of the $R$-functor $\X$ is then naturally $\G(R)$-equivariant in the sense of Definition \ref{DefnEquivSheaf}, and similarly the tangent sheaf $\cT_{|\X|}$ is also $\G(R)$-equivariant. In this way, we obtain a group action of $\G(R)$ on $\cT(\X) = \cT(|\X|)$ by $R$-Lie algebra automorphisms, which is determined by the following property:
\[ (g\cdot v)(g\cdot f) = g \cdot v(f) \]
whenever $g \in \G(R)$, $v \in \cT(\X)$, $\Y$ is an open subfunctor of $\G$ and $f \in \cO(\Y)$. We will call this the \emph{conjugation action} of $\G(R)$ on $\cT(\X)$. On the other hand, recall from \cite[Chapter II, \S 4, 1.2]{DG} that we have the \emph{adjoint representation} 
\begin{equation}\label{AdjRep} \Ad : \G(R) \to \Aut_R( \Lie \G )\end{equation}
given by $\Ad(g)(u) := g \cdot  u := g \hsp u \hsp g^{-1}$ for all $g \in \G(R)$ and $u \in \Lie \G$. Here we abuse notation and identify $\G(R)$ with its image inside $\G(R[\tau])$.

 \begin{lem}\label{EquivariantInfAction} Let $\G \times \X \to \X$ be an action of the $R$-group scheme $\G$ on the $R$-scheme $\X$, and let $\varphi : \G \to \zAut(\X)$ be the corresponding homomorphism. Then the infinitesimal action $\varphi' : \fr{g} \to \cT(\X)$ is equivariant with respect to the adjoint action of $\G(R)$ on $\fr{g}$ and the conjugation action of $\G(R)$ on $\cT(\X)$.
 \end{lem}
 \begin{proof}Let $h \in \G(R)$, $g \in \fr{g}$, $f \in \cO_{|\X|}$ and $x \in \X$. Then 
 \[ (h \cdot \varphi'(g)) (h\cdot f) = h \cdot \varphi'(g)(f)\]
 by the definition of the $\G(R)$-action on $\cT(\X)$ given above, so
\[\begin{array}{lll} \tau (h \cdot \varphi'(g))( h\cdot f)(x) &=& \tau \hsp \left(h \cdot \varphi'(g)(f) \right)(x) = \\ 
&=& \tau \hsp \varphi'(g)(f)(h^{-1}\cdot x) = \\
&=& f(g^{-1}h^{-1}\cdot x) - f(h^{-1}\cdot x) = \\
&=& (h\cdot f)(hg^{-1}h^{-1}\cdot x) - (h\cdot f)(x)= \\
&=& (h\cdot f)((h \cdot g)^{-1} \cdot x) - (h\cdot f)(x)= \\
&=& \tau \hsp \varphi'(h \cdot g)(h\cdot f)(x).\end{array}\]
Because this is true for all $x \in \X$ and because multiplication by $\tau$ on $R[\tau]$ induces a bijection $R \to \tau R$, we deduce that $(h \cdot \varphi'(g))(h\cdot f) = \varphi(h \cdot g)(h\cdot f).$ Replacing $f$ by $h^{-1}\cdot f$, we conclude that $(h \cdot \varphi'(g))(f) = \varphi'(h \cdot g)(f)$ for all $f \in \cO_{|\X|}$. Therefore $h \cdot \varphi'(g) = \varphi'( h \cdot g )$ as required.
 \end{proof}

Suppose now that $\varphi : \G \to \zAut(\X)$ and $\psi : \G \to \zAut(\Y)$ are two actions of the $R$-group scheme $\G$ on the $R$-schemes $\X$ and $\Y$, respectively, and let $\xi : \Y \to \X$ be a $\G$-equivariant map.  Let $\H$ be an affine $R$-group scheme which is flat over $R$, and suppose in addition that $\xi : \Y \to \X$ is a Zariski locally trivial $\H$-torsor in the sense of \cite[\S 4.3]{AW13}. Because $\H$ is flat over $R$, \cite[Lemma 4.3]{AW13} tells us that the pullback of functions map $\xi^\sharp : \cO_{\X} \to (\xi_\ast \cO_{\Y})^{\H}$ is an isomorphism of $\cO_{\X}$-modules. Recall from \cite[\S 4.4]{AW13} that in this situation, we have  the \emph{anchor map} 
\begin{equation}\label{AnchorMapAlpha}  \alpha : (\xi_\ast \cT_{\Y})^{\H} \longrightarrow \cT_{\X}\end{equation}
of the Lie algebroid $(\xi_\ast \cT_{\Y})^{\H}$, which is defined by the rule 
\begin{equation}\label{AnchorMapAlphaDefn} \xi^\sharp\left(\alpha(v)(f)\right) = v(\xi^\sharp(f))\end{equation}
for any $v \in (\xi_\ast \cT_{\Y})^{\H}$ and any $f \in \cO_{\X}$. This map allows us to relate the infinitesimal actions of $\fr{g}$ on $\X$ and on $\Y$ in the following way.

\begin{lem}\label{AlphaPsiVarphi} We have $\alpha \circ \psi' = \varphi'$.
\end{lem}
\begin{proof} Let $\U$ be an affine open subscheme of $\X$ over which $\xi$ is trivialisable. Let $g \in \fr{g}$, $f \in \cO(\U)$ and $x \in \xi^{-1}(\U)$. Using the $\G$-equivariance of $\xi$, we calculate using Definition \ref{InfActionDefn} that
\[ \begin{array}{lllll} \tau \hsp \psi'(g)(\xi^\sharp(f))(x) &=& \xi^\sharp(f)(g^{-1}\cdot x) - \xi^\sharp(f)(x) &=& f(\xi(g^{-1}\cdot x)) - f(\xi(x)) = \\
&=& f(g^{-1}\cdot \xi(x)) - f(\xi(x)) &=& \tau \hsp \varphi'(g)(f)(\xi(x))\end{array}\]
 and therefore that
\[\alpha(\psi'(g))(f)(\xi(x)) = \xi^\sharp(\alpha(\psi'(g))(f))(x)= \psi'(g)(\xi^\sharp(f))(x) = \varphi'(g)(f)(\xi(x)).\]
Because $\xi$ is Zariski locally trivial, we deduce that $\alpha \circ \psi' = \varphi'$ as claimed.
\end{proof}

The group $\G$ acts on itself by left and right translations. The corresponding homomorphisms of $R$-group-functors
\[\gamma : \G \to \zAut(\G) \qmb{and} \delta : \G \to \zAut(\G)\]
are given by the familiar formulas $\gamma(g)(x) = gx$ and $\delta(g)(x) = xg^{-1}$ for every $g,x \in \G$. These actions induce $\G(R)$-actions on $\cO(\G)$ by $R$-algebra automorphisms; see the end of \cite[I.2.7]{Jantzen} for a discussion. Abusing notation, we will denote the corresponding group homomorphisms $\G(R) \to \Aut_{R-\alg} \cO(\G)$ by $\gamma$ and $\delta$, respectively: $\gamma(g)(f)(x) =  f(g^{-1}x)$ and $\delta(g)(f)(x) =  f(xg)$ for each $f \in \cO(\G)$ and $g,x \in \G$.

Assume from now on that our $R$-group scheme $\G$ is \emph{affine}. We will next give an explicit formula for the infinitesimal actions associated $\gamma$ and $\delta$, but first we need to recall some language from the theory of Hopf algebras. The actions $\gamma$ and $\delta$ induce two right $\cO(\G)$-comodule structures, otherwise known as coactions, on $\cO(\G)$. Following \cite[\S I.2.8]{Jantzen}, we denote these coactions by $\Delta_{\gamma}$ and $\Delta_{\delta}$, respectively. Using the sumless Sweedler notation $\Delta(f) = f_1 \otimes f_2$ for $f \in \cO(\G)$ to denote the comultiplication map $\Delta : \cO(\G) \to \cO(\G) \otimes_R \cO(\G)$ of the $R$-Hopf algebra $\cO(G)$, it follows from  \cite[\S I.2.8(5) and I.2.8(6)]{Jantzen} that these coactions are given by
\begin{equation}\label{CoactionDefns} \Delta_{\gamma}(f) = f_2 \otimes \sigma(f_1) \qmb{and} \Delta_{\delta}(f) = f_1 \otimes f_2,\end{equation}
where $\sigma : \cO(\G) \to \cO(\G)$ denotes the antipode of $\cO(\G)$. 

Let $\epsilon : \cO(\G) \to R$ be the counit of the $R$-Hopf algebra $\cO(\G)$. Recall that an \emph{$\epsilon$-derivation} is an $R$-linear map $d : \cO(\G) \to R$ such that $d(ab) = \epsilon(a) d(b) + d(a) \epsilon(b)$ for all $a,b \in \cO(\G)$. For every $u \in \Lie(\G)= \ker \left( \Hom( \cO(\G), R[\tau] ) \to \Hom( \cO(\G), R) \right)$, there is a unique $\epsilon$-derivation $\underline{u} : \cO(\G) \to R$ given by the formula
\begin{equation}\label{EpsDerFormula} u(f) = \epsilon(f) + \tau \underline{u}(f) \qmb{for all} f \in \cO(\G).\end{equation}
It follows from \cite[Theorem 12.1]{Waterhouse} that $u \mapsto \underline{u}$ is a bijection from $\Lie(\G)$ onto the set of all $\epsilon$-derivations.

\begin{lem} \label{InfActionOnG}  Let $u \in \fr{g} = \Lie(\G)$ and $f \in \cO(\G)$. Then
\[ \gamma'(u)(f) = \underline{u}(\sigma(f_1)) f_2 \qmb{and}  \delta'(u)(f) =  \underline{u}( f_2) \hsp f_1.\]
\end{lem}
\begin{proof} We will only deal with the first equation, because the second one is entirely similar. Let $\phi := \zLie(\gamma)(R)(u^{-1})$, let $x \in \G$ and let $f \in \cO(\G)$; then abusing notation and applying (\ref{InfActFormula}) we have
\[ \tau D_\phi(f)(x) = f( \phi(x) ) - f(x) = f( u^{-1} x ) - f(x).\]
Using (\ref{EpsDerFormula}), we can re-write this as follows:
\[\begin{array}{lll} f(u^{-1}x) - f(x) &=& u^{-1}(f_1) x(f_2) -f(x) = (u( \sigma(f_1)) \hsp f_2- f)(x) \\ 
&=& (\epsilon( \sigma(f_1) ) + \tau \underline{u}( \sigma(f_1) )\hsp f_2 - f)(x) = \\
&=& (\tau \underline{u}( \sigma(f_1) ) \hsp f_2)(x) \end{array}\]
because $\epsilon \circ \sigma = \epsilon$ and because $\epsilon(f_1)f_2 = f$. Hence
\[ \gamma'(u)(f) = D_\phi(f) = \underline{u}(\sigma(f_1)) \hsp f_2 \qmb{for all} f \in \cO(\G). \qedhere\]
  \end{proof}

Next, we remind the reader of some more standard definitions.
\begin{defn}\label{LeftRightInvDers} Let $C$ be an $R$-coalgebra, and let $(M, \rho_M), (N,\rho_N)$ be right $C$-comodules.
\be \item A \emph{morphism of right $C$-comodules} is an $R$-linear map $\xi : M \to N$ such that 
\[\rho_N \circ \xi = (\xi \otimes 1)\circ \rho_M.\] 
\item We will denote the endomorphism ring of $(M,\rho_M)$ in the category of right $C$-comodules by $\End(M, \rho_M)$. 
\item The set of \emph{left-invariant $R$-linear derivations} of $\cO(\G)$ is
\[{}^{\G}\cT(\G) := \cT(\G) \cap \End\left(\cO(\G),\Delta_{\gamma}\right).\] 
\item The set of \emph{right-invariant $R$-linear derivations} of $\cO(\G)$ is
\[\cT(\G)^{\G} := \cT(\G) \cap \End\left(\cO(\G),\Delta_{\delta}\right).\]
\ee\end{defn}

\begin{prop}\label{LieTriangle} There is a commutative diagram of $R$-Lie algebras
\[\xymatrix{  & \Lie(\G) \ar[dl]_{\gamma'}\ar[dr]^{\delta'} & \\ \cT(\G)^{\G} \ar[rr]_{\xi \mapsto \sigma \circ \xi \circ \sigma^{-1}} &&  {}^{\G}\cT(\G) }\]
in which all arrows are bijective.
\end{prop}
\begin{proof} By Definition \ref{LeftRightInvDers}(a), an $R$-linear map $\xi : \cO(\G) \to \cO(\G)$ is a morphism of right $\cO(\G)$-comodules for the coaction $\Delta_{\gamma}$ if and only if $\Delta_{\gamma} \circ \xi = (\xi \otimes 1) \circ \Delta_{\gamma}$. Using (\ref{CoactionDefns}), this is equivalent to
\begin{equation} \label{DeltaRhoL}
\xi(f)_2 \otimes \sigma(\xi(f)_1) = \xi(f_2) \otimes \sigma(f_1) \qmb{for all} f \in \cO(\G). \end{equation}
By applying the invertible endomorphism $a \otimes b \mapsto \sigma^{-1}(b) \otimes a$ of $\cO(\G) \otimes_R \cO(\G)$, we see that (\ref{DeltaRhoL}) is equivalent to
 \begin{equation} \label{WaterhouseInvDefn} \xi(f)_1 \otimes \xi(f)_2 = f_1 \otimes \xi(f_2) \qmb{for all} f \in \cO(\G),\end{equation}
 or equivalently, to $\Delta \circ \xi = (1 \otimes \xi) \circ \Delta$.   This agrees with the definition of left-invariant derivations found on \cite[p. 92]{Waterhouse}. Now if $u \in \Lie(\G)$ then for any $f \in \cO(\G)$ we have
 \[(\Delta \circ \delta'(u))(f) = \Delta( f_1 \underline{u}(f_2)) = (f_1 \otimes f_2)\underline{u}(f_3) = f_1 \delta'(u)(f_2) = (1 \otimes \delta'(u))\Delta(f)\]
 by Lemma \ref{InfActionOnG}. Hence $\delta'(u) \in {}^{\G}\cT(\G)$ for all $u \in \Lie(\G)$, and now the fact that $\delta'$ is a well-defined bijection follows from the proof of \cite[Theorem 12.1]{Waterhouse}. 

The antipode map can be viewed as an isomorphism of right $\cO(\G)$-comodules
\[ \sigma : (\cO(\G), \Delta_{\gamma}) \stackrel{\cong}{\longrightarrow} (\cO(\G), \Delta_{\delta}).\]
To see this, let $f \in \cO(\G)$ and apply (\ref{CoactionDefns}) to obtain
\[(\sigma \otimes 1)\Delta_{\gamma}(f) = (\sigma \otimes 1)(f_2 \otimes \sigma(f_1)) = \sigma(f_2) \otimes \sigma(f_1) = \sigma(f)_1 \otimes \sigma(f)_2  = \Delta_{\delta} \sigma(f).\]
We have used here the fact that $\sigma$ is an anti-coalgebra morphism --- see \cite[Proposition 1.5.10(2)]{Mont}. It follows that $\xi \mapsto \sigma \xi \sigma^{-1}$ is an $R$-algebra isomorphism 
\[\End\left(\cO(\G),\Delta_{\gamma}\right) \stackrel{\cong}{\longrightarrow} \End\left(\cO(\G),\Delta_{\delta}\right).\]
Because $\sigma$ is also an anti-algebra morphism by \cite[Proposition 1.5.10(1)]{Mont}, it follows that this isomorphism preserves the subspace of $R$-linear derivations:
\[\sigma \xi \sigma^{-1}(ab) = \sigma\left(\xi(\sigma^{-1}b) \sigma^{-1}a + \sigma^{-1}b  \hsp \xi(\sigma^{-1}a)\right) = a \hsp \sigma\xi\sigma^{-1}(b) + \sigma\xi\sigma^{-1}(a) \hsp b\]
whenever $\xi(ab) = \xi(a)b + a\xi(b)$ for all $a, b \in \cO(\G)$. Hence the bottom arrow of the triangle is a well-defined bijection. Finally, we will show that
\begin{equation} \label{PhiTriangleCommutes} \sigma \delta'(u) \sigma^{-1} = \gamma'(u) \qmb{for all} u \in \Lie(\G).\end{equation}
Using (\ref{CoactionDefns}) and Lemma \ref{InfActionOnG}, we see that for any $u \in \Lie(\G)$ and $f \in \cO(\G)$, 
\[ \begin{array}{lll} \sigma (\gamma'(u) (f)) &=& \underline{u}(\sigma(f_1)) \sigma(f_2) =  (\underline{u} \hsp \overline{\otimes}\hsp 1)( \sigma(f_1) \otimes \sigma(f_2)) = \\
&=&(\underline{u} \hsp \overline{\otimes}\hsp  1) ( \sigma(f)_2 \otimes \sigma(f)_1) =  \underline{u}( \sigma(f)_2) \hsp \sigma(f)_1 = \delta'(u)(\sigma(f)). \end{array}\]
Thus $\sigma \circ \gamma'(u) = \delta'(u) \circ \sigma$ which implies (\ref{PhiTriangleCommutes}). It follows that the triangle in the statement of the Proposition commutes, and also that $\gamma'$ is a bijection. Finally, $\gamma'$ and $\delta'$ are $R$-Lie algebra homomorphisms by \cite[Chapter II, \S 4, Proposition 4.4]{DG}.\end{proof}

Recall \cite[\S I.7.9]{Jantzen} that the group scheme $\G$ is said to be \emph{infinitesimally flat} if $I/I^2$ is a finitely presented and projective $R$-module, where $I = \ker \epsilon$ is the augmentation ideal of $\cO(\G)$. In view of \cite[Corollaire 19.5.4]{EGAIV1}, this is for example the case whenever $\G$ is smooth and of finite type over $\Spec(R)$. 

\begin{prop}\label{InvVectFldsGenerate} The canonical $\cO(\G)$-linear maps
\[ \cO(\G) \otimes_R \cT(\G)^{\G} \longrightarrow \cT(\G) \qmb{and} \cO(\G) \otimes_R {}^{\G}\cT(\G) \longrightarrow \cT(\G)\]
are isomorphisms whenever $\G$ is infinitesimally flat. 
\end{prop}
\begin{proof} Write $A = \cO(\G)$,  let $\pi : A \to I/I^2$ be the $R$-linear map given by $\pi(a) = a - \epsilon(a)1 + I^2$ and note that it is an $\epsilon$-derivation with values in $I/I^2$:
\[\pi(ab) = \epsilon(a) \pi(b) + \pi(a) \epsilon(b) \qmb{for all} a,b \in A.\]
It follows that there is an $A$-linear map $\theta : \Hom_R(I/I^2,A) \to \Der_R(A)$, given by 
\[\theta(f)(a) = a_1 f(\pi(a_2)) \qmb{for all} a\in A.\]
From the proof of \cite[Theorem 11.3]{Waterhouse}, we deduce that $\theta$ is in fact an \emph{isomorphism}, whose inverse is given by the explicit formula
\[\theta^{-1}(d)(a + I^2) = \sigma(a_1)d(a_2) \qmb{for any} a \in I.\]
Let $\Ad(\sigma) : \End_R A \to \End_R A$ be the map $\xi \mapsto \sigma \xi \sigma^{-1}$ and recall from Proposition \ref{LieTriangle} that $\Ad(\sigma)$ maps ${}^{\G}\cT(\G)$ onto $\cT(\G)^{\G}$. Consider the following diagram of $R$-modules:
\[\xymatrix{ A \otimes_R \cT(\G)^{\G} \ar[rr]^\mu  && \Der_R(A) \\
A \otimes_R {}^{\G}\cT(\G) \ar[rr]^\mu \ar[u]^{\sigma \hsp \otimes \hsp \Ad(\sigma)} && \Der_R(A)  \ar[u]_{\Ad(\sigma)}\\
A \otimes_R \Hom_R(I/I^2, R) \ar[u]^{1 \hsp \otimes \hsp \delta'\circ\pi^\ast} \ar[rr]_w && \Hom_R(I/I^2,A) \ar[u]^\cong_\theta }\]
Here, $\mu$ is given by $\mu(a \otimes d)(b) = a \hsp d(b)$, $\pi^\ast : \Hom_R(I/I^2,R) \to \Lie(\G)$ is the isomorphism given by $\underline{\pi^\ast(f)} = f \circ \pi$, and $w$ is given by $w(a \otimes f)(x) = a \hsp f(x)$.  Using Lemma \ref{InfActionOnG}, It is straightforward to verify that this diagram commutes. Now $w$ is an isomorphism because $\G$ is infinitesimally flat and $\delta'$ is an isomorphism by \cite[Theorem 12.1]{Waterhouse}. Since $\Ad(\sigma)$ and $\sigma$ are also isomorphisms, $\mu$ is also an isomorphism.
\end{proof}


We now specialise to the case where our ground ring $R$ is complete with respect to some nested family of ideals $J_i$:
 \[ R \stackrel{\cong}{\longrightarrow} \invlim R / J_i.\]
We will write 
\[R_i := R / J_i \qmb{and}\G_i := \G \times_{\Spec(R)} \Spec(R_i).\]

\begin{lem}\label{LieGhat} Suppose that $\G$ is infinitesimally flat over $R$. 
\be \item For any index $i$, all arrows in the commutative diagram
\[ \xymatrix{ \Lie(\G) \otimes_{R} R_i \ar[rr]^{\varphi_{l,\G} \otimes 1} \ar[d] && \cT(\G)^{\G} \otimes_{R} R_i \ar[d] \\
\Lie(\G_i) \ar[rr]_{\varphi_{l,\G_i}} && \cT(\G_i)^{\G_i} }\]
are isomorphisms.
\item The natural map
\[\Lie(\G) \longrightarrow \invlim \Lie(\G_i)\]
is an isomorphism. 
\ee \end{lem}
\begin{proof} (a) We identify $\Lie(\G)$ with $\Hom_{R}(I/I^2, R)$ where $I = \ker \epsilon$ is the augmentation ideal of $\cO(\G)$. The discussion in \cite[\S I.7.4(1)]{Jantzen} shows that 
\[\Lie(\G \otimes_R S) \cong \Lie(\G) \otimes_{R} S\] 
for any commutative $R$-algebra $S$, using the fact that $\G$ is infinitesimally flat over $R$. Hence the vertical arrow on the left is an isomorphism, and in particular, $\G_i$ is infinitesimally flat over $R_i$. The result now follows from Proposition \ref{LieTriangle}.

(b) The required map fits into a natural commutative triangle
\[ \xymatrix{ \Lie(\G) \ar[rrrr]\ar[drr] &&&&\invlim  \Lie(\G_i) \\ && \invlim \Lie(\G) \otimes_{R} R_i \ar[urr] &&}\]
Because $R$ is complete, every finitely generated projective $R$-module is also complete. Since $\G$ is infinitesimally flat over $R$, $\Lie(\G)$ is a finitely generated projective $R$-module, so the first diagonal map in the diagram is an isomorphism. The second diagonal map is an isomorphism by part (a).
\end{proof}

\begin{defn} \be \item We denote the completion of $\G$ by $\h\G$. This is an affine formal scheme whose coordinate ring
\[ \cA := \cO(\h\G) = \invlim \cO(\G_i)\]
is the completion of the $R$-algebra $\cO(\G)$. 
\item Let $\cT(\h\G) := \Der_{R}(\cA)$ denote the set of $R$-linear derivations of $\cA$. 
\ee \end{defn}

\begin{prop}\label{FormalVectFields} Suppose that $\G$ is infinitesimally flat. Then the natural maps
\[ \beta : \cT(\h\G) \longrightarrow \invlim \cT(\G_i) \qmb{and} \cO(\h\G) \otimes_{R} \cT(\G)^{\G} \longrightarrow \cT(\h\G)\]
are isomorphisms.
\end{prop}
\begin{proof} These maps appear in the commutative diagram
\[\xymatrix{ \cO(\h\G) \otimes_{R} \cT(\G)^{\G} \ar[rrrr] \ar[d] &&&& \cT(\h\G) \ar[dd]^\beta \\ 
\invlim \left(\cO(\h\G) \otimes_{R} \cT(\G)^{\G}\right) \otimes_{R} R_i \ar[d] &&&& \\
\invlim \cO(\G_i) \otimes_{R_i} \cT(\G_i)^{\G_i} \ar[rrrr] &&&& \invlim \cT(\G_i). }\]
Suppose that $\beta(\partial) = 0$ for some $\partial \in \cT(\h\G)$. Then $\beta(\cA) \subseteq J_i \cA$ for all $i$. Because the topology on $\cA$ is separated, we see that $\beta$ is injective. 

Because $\G$ is infinitesimally flat, $\cT(\G)^{\G}$ is a finitely generated projective $R$-module by Proposition \ref{LieTriangle}. Hence $\cO(\h\G) \otimes_{R} \cT(\G)^{\G}$ is a finitely generated projective $\cO(\h\G)$-module, and is therefore complete. Hence the first vertical arrow on the left is an isomorphism. Now Lemma \ref{LieGhat}(a) implies that the second vertical arrow on the left is an isomorphism, and the bottom horizontal map is an isomorphism by Proposition \ref{InvVectFldsGenerate}. Because $\beta$ is injective, it follows that in fact all arrows in this diagram are isomorphisms.
\end{proof}

The coactions 
\[ \Delta_{\gamma} : \cO(\G) \to \cO(\G) \otimes_{R} \cO(\G) \qmb{and} \Delta_{\delta} : \cO(\G) \to \cO(\G) \otimes_{R} \cO(\G)\]
from (\ref{CoactionDefns}) can be  completed to obtain $R$-algebra homomorphisms
\[ \h{\Delta}_{\gamma} : \cA \to \cA \h\otimes_{R} \cA \qmb{and} \h{\Delta}_{\delta} : \cA \to \cA \h\otimes_{R} \cA.\]

\begin{defn}\label{FormalRightInvDers} Let $\cT(\h\G)^{\h\G} := \{\xi \in \cT(\h\G) : \h{\Delta}_{\delta} \circ \xi  = (\xi \h{\otimes} 1) \circ \h{\Delta}_{\delta} \}$ denote the $R$-Lie algebra of \emph{right-invariant} derivations of $\h\G$.
\end{defn}

\begin{prop}\label{FormalInvariantVectFlds} Suppose that $\G$ is infinitesimally flat. The the natural map
\[ \cT(\G)^{\G} \longrightarrow \cT(\h\G)^{\h\G}\]
is an isomorphism.
\end{prop}
\begin{proof} Consider the following diagram:
\[ \xymatrix{ \Lie(\G) \ar[rrrr]\ar[d]_{\varphi_{l,\G}} &&&& \invlim \Lie(\G_i) \ar[d]^{\invlim \varphi_{l,\G_i}} \\ \cT(\G)^{\G} \ar[rrrr]\ar[drr] &&&& \invlim \cT(\G_i)^{\G_i} \\ && \cT(\h\G)^{\h\G} \ar[urr]_\beta && }\]
The vertical arrows are isomorphisms by Proposition \ref{LieTriangle}, and the top horizontal arrow is an isomorphism by Lemma \ref{LieGhat}(b). So the middle horizontal arrow is also an isomorphism. Because the restriction of the map $\beta$ from Proposition \ref{FormalVectFields} to $\cT(\h\G)^{\h\G}$ is injective, the bottom triangle now implies that $\cT(\G)^{\G} \longrightarrow \cT(\h\G)^{\h\G}$ is an isomorphism.
\end{proof}

\begin{lem}\label{AinjectsIntoAhat} Suppose that $\G$ is integral and of finite type over $R$, and that $R$ is a complete valuation ring of height $1$. Then the natural maps
\[ \cO(\G) \longrightarrow \cO(\h\G) \qmb{and} \cT(\G) \longrightarrow \cT(\h\G)\]
are injective.
\end{lem}
\begin{proof} Let $A := \cO(\G$) and fix a non-zero non-unit $\pi \in R$. We have to show that $\bigcap \pi^n A$ is zero. Because of the assumptions imposed on $R$ and because $A$ is a finitely generated $R$-algebra, \cite[Corollaire 1.12.14(i)]{Abbes} implies the ideal $\pi A$ satisfies the \emph{strong Artin-Rees condition} --- see \cite[Definition 1.8.25(a)]{Abbes}. Let $a \in \bigcap \pi^n A$; then by \cite[Lemme 1.8.27]{Abbes} there exists $b \in A$ such that $a (1 - \pi b) = 0$. Note that $1 - \pi b \neq 0$ as otherwise $\epsilon(b)$ would be an inverse to $\pi$ in $R$. Because $\G$ is integral, we deduce that $a = 0$. The second statement is now clear.
\end{proof}

\begin{rmk} The first condition on $\G$  in Lemma \ref{AinjectsIntoAhat} cannot be replaced with the weaker condition that $\G$ is reduced. To see this, let $A$ be the finitely presented $R$-algebra $A := R[v] / \langle v - \pi v^2\rangle$. We may view $A$ as an $R$-subalgebra of the group ring $K[C_2] := K[g] / \langle g^2 - 1 \rangle$ of the cyclic group of order two over $K$, via the $R$-algebra embedding which sends $v$ to $\frac{1 - g}{2 \pi}$. It follows that $A$ is isomorphic to $R \times K$, so $A$ is reduced and flat over $R$, but it is not integral. Now $K[C_2]$ is a $K$-Hopf algebra in a standard way, where the element $g$ is grouplike. We can then calculate that the structure morphisms $\epsilon, \sigma$ and $\Delta$ on $K[C_2]$ satisfy
\[\epsilon(v) = 0, \quad \sigma(v) = v, \qmb{and} \Delta(v) = v\otimes 1 + 1 \otimes v - 2 \pi v \otimes v.\]
It follows that $A$ is an $R$-sub-Hopf algebra of $K[C_2]$. Hence $\G := \Spec(A)$ is a reduced, affine group scheme of finite presentation over $R$. However, because $v = \pi^n v^{2^n} \in \pi^n A$ for any $n \geq 0$, the kernel of the natural map $A \to \h{A}$ is the non-zero ideal $vA$. The group scheme $\G$ above is even infinitesimally flat because $Av^2 \subseteq Av =  A \pi v^2 \subseteq A v^2$. Thus Proposition \ref{FormalInvariantVectFlds} is applicable in this case, even though both arrows appearing in the statement of Lemma \ref{AinjectsIntoAhat} fail to be injective.\end{rmk}

\subsection{The algebra \ts{\wDGG}}\label{wDGGsect} 
We will assume from now on that $R = \cR$ is a complete valuation ring of height one and mixed characteristic $(0,p)$, and work under the following

\begin{setup}\hspace{2em} \label{SigmaGG}
\begin{itemize}
\item $\G$ is an affine group scheme, smooth and of finite type over $\cR$,
\item $\fr{g} := \Lie(\G)$ is the Lie algebra of $\G$,
\item $G$ is a compact $p$-adic Lie group,
\item $\sigma : G \to \G(\cR)$ is a continuous group homomorphism.
\end{itemize}
\end{setup}
We equip $\G(\cR)$ with the congruence subgroup topology from Definition \ref{CongSubTop}. For example, if $\G$ is defined  over the ring of integers $\cO_L$ of some finite extension $L$ of $\Qp$ contained in $K$, then $\sigma$ could be the inclusion of $\G(\cO_L)$ into $\G(\cR)$. 

Because $\G$ is smooth and of finite type, $\cO(\G)$ is a finitely generated $\cR$-algebra which is flat as an $\cR$-module. It follows from Corollary \ref{RGCor} and Theorem \ref{RG346} that $\cO(\G)$ is even finitely presented as an $\cR$-algebra. Passing to the $\pi$-adic completion shows that $\cA := \h{\cO(\G)}$ is an admissible $\cR$-algebra, so $\h{\G} = \Spf \h{\cO(\G)}$ is an affine formal scheme of topologically finite presentation, and its rigid generic fibre $\bG := {\h\G}_{\rig}$ of $\h\G$ is an affinoid rigid analytic variety over $K$. Because $\G$ is infinitesimally flat by \cite[Corollaire 19.5.4]{EGAIV1}, we may and will use Proposition \ref{FormalInvariantVectFlds} to identify $\cT(\G)^{\G}$ with the $\cR$-Lie algebra $\cT(\h\G)^{\h\G}$ of right invariant derivations of $\h\G$.

Next, we have the group homomorphism $\h{\gamma} : \G(\cR) \to \Aut(\h\G, \cO_{\h\G})$ from the proof of Proposition \ref{GpOfPtsActsCtsly}(b), arising from the left-translation action $\gamma$ of $\G$ on itself. Because $\h\G$ is affine, we may identify $\Aut(\h\G, \cO_{\h\G})$ with the group $\cG(\cA)$ of $\cR$-algebra automorphisms of $\cA = \cO(\h\G)$ as in $\S \ref{GActionSection}$, via $(\varphi,\varphi^\sharp) \mapsto \Gamma(\varphi^\sharp)^{-1}$. Thus we obtain the continuous group homomorphism $\rho :G \to \cG(\cA)$ given explicitly by
\[ \rho(g) = \Gamma(\h{\gamma} \sigma(g)^\sharp)^{-1} : \cO(\h\G) \to \cO(\h\G).\] 
Via the natural embedding $\rig : \cG(\cA) \hookrightarrow \Aut_K( \cO(\bG) )$ appearing in the discussion following Lemma \ref{ProdFormModels}, this gives us a continuous action of $G$ on the $K$-affinoid algebra $\cO(\bG)$ by $K$-algebra automorphisms, which we still denote by $\rho$:
\[ \rho : G \to \Aut_K( \cO(\bG) ).\]
By Definition \ref{StarDefn}, we now have at our disposal the algebra
\[ \w\cD(\bG, G)\]
and our aim is to define its subalgebra of right $\bG$-invariants $\wDGG$. However, because we will need to know its structure quite explicitly, we will give a definition which does not require making the notion of ``right $\bG$-invariants'' completely precise. 

Recall that we have at our disposal the action $\Ad \circ \sigma$ of $G$ on $\fr{g}$ by $\cR$-Lie algebra automorphisms. 

\begin{defn} \hspace{2em}
\be\item A \emph{Lie lattice} in $\fr{g}$ is a finitely generated $\cR$-submodule $\cJ$ of $\fr{g}$ which is stable under the Lie bracket on $\fr{g}$ and which contains a $\pi$-power multiple of $\fr{g}$. 
\item The Lie lattice $\cJ$ is \emph{$G$-stable} if it is preserved by $\Ad(\sigma(g))$ for all $g \in G$.
\ee\end{defn}

\begin{lem}\label{ALLie} Let $\cJ$ be a Lie lattice in $\fr{g}$, and let $\cL := \cA \cdot \gamma'(\cJ)$.
\be \item $\cJ$ is free of $\rk \fr{g}$ as an $\cR$-module.
\item $\cL$ is free of rank $\rk \fr{g}$ as an $\cA$-module.
\item $\cL$ is an $\cA$-Lie lattice in $\cT(\bG) = \Der_K \cO(\bG)$.
\item If $\cJ$ is $G$-stable, then so is $\cL$.
\ee\end{lem}
\begin{proof}(a)  Because $\G$ is infinitesimally flat, its Lie algebra $\fr{g}$ is a finitely generated projective $\cR$-module. Since $\cR$ is a local ring, it is actually free of finite rank. Because $\cR$ is a valuation ring, the same is true of any Lie lattice $\cJ$ in $\fr{g}$. 

(b) This follows from part (a) together with  Proposition \ref{FormalVectFields}.

(c) It is straightforward to verify that $\cL = \cA \cdot \gamma'(\cJ)$ is closed under the Lie bracket. It follows from Proposition \ref{FormalVectFields} that $\cT(\h\G) = \Der_{\cR}(\cA) = \cA \cdot \gamma'(\fr{g}).$ Because $\cJ$ contains a $\pi$-power multiple of $\fr{g}$, it follows that $\cL$ contains a $\pi$-power multiple of $\cT(\h\G)$. 

(d) By Lemma \ref{EquivariantInfAction}, the $\cR$-Lie algebra homomorphism $\gamma'$ is equivariant with respect to the adjoint action of $G$ on $\fr{g}$ and the conjugation action $\dot{\rho}$ of $G$ on $\cT(\h\G)$. So $\gamma'(\cJ)$ is stable under $\dot{\rho}$. Now
\[ \dot{\rho}(g) ( a\cdot v ) = \rho(g)(a) \cdot \dot{\rho}(g)(v) \qmb{for all} a \in \cA, v \in \gamma'(\fr{g}), g \in G\]
so $\cL = \cA \cdot \gamma'(\cJ)$ is also stable under $\dot{\rho}$ as required.
\end{proof}

Let $\cJ$ be a $G$-stable Lie lattice in $\fr{g}$, and let $\cL := \cA \cdot \gamma'(\cJ)$. Then $\cL$ is a free $\cA$-module of finite rank by Lemma \ref{ALLie}(b), so $U(\cL)$ is a flat $\cA$-module by \cite[Theorem 3.1]{Rinehart}. In particular it is flat as an $\cR$-module, so the algebra $\cU$ appearing in Lemma \ref{PsiL} is just $\h{U(\cL)}$.  By Definition \ref{BetaDefn}, we now have at our disposal the map 
\[\beta_{\cL} := (\psi_{\cL}^\times)^{-1} \circ \rho : G_{\cL} = \rho^{-1} \left( \exp( p^\epsilon \cL ) \right)  \to \h{U(\cL)}^\times,\]
which is a $G$-equivariant trivialisation of the $G_{\cL}$-action on $\h{U(\cL)}$ by Theorem \ref{Beta}. On the other hand, by functoriality, the $\cR$-Lie algebra homomorphism $\gamma' : \cJ \to \cL$ extends to an $\cR$-algebra homomorphism 
\[ \theta := \h{U(\gamma')} : \h{U(\cJ)} \to \h{U(\cL)}.\]
\begin{prop}\label{ThetaBeta} Let $\cJ$ be a $G$-stable Lie lattice in $\fr{g}$ and let $\cL := \cA \cdot \gamma'(\cJ)$. 
\be \item The map $\theta$ is injective and $G$-equivariant.
\item For every $g \in G_{\cL}$ there is a unique element $v \in \cJ$ such that 
\[\rho(g) = \exp( p^\epsilon \gamma'(v) ) \qmb{and} \beta_{\cL}(g) = \theta(\exp (p^\epsilon \iota(v))).\]
\item $\theta^{-1} \circ \beta_{\cL}$ is a $G$-equivariant trivialisation of the $G$-action on $\h{U(\cJ)}$.
\ee\end{prop}
\begin{proof} (a) The $G$-equivariance of $\theta$ follows from Lemma \ref{EquivariantInfAction}. For the injectivity, it is enough to show that $U(\gamma') : U(\cJ) \to U(\cL)$ is injective, because both of these algebras are flat as $\cR$-modules. By Lemma \ref{ALLie}(a,b) and \cite[Theorem 3.1]{Rinehart}, it is enough to see that $S(\cJ) \to S(\cL)$ is injective. Now by Propositions \ref{LieTriangle} and \ref{FormalVectFields} there is an $\cA$-module isomorphism $\h{1 \otimes \gamma'} : \cA \otimes_{\cR} \cJ \stackrel{\cong}{\longrightarrow} \cL$ which induces an $\cA$-algebra isomorphism $S(\cL) \cong \cA \otimes_{\cR} S(\cJ)$, and the canonical map $S(\cJ) \to \cA \otimes_{\cR} S(\cJ)$ is injective because $\cA$ and $S(\cJ)$ are flat $\cR$-modules.

(b) Let $g \in G_{\cL}$ so that $\rho(g) = \psi_{\cL}( \exp(p^\epsilon \iota(u)) )  = \exp(p^\epsilon u)$ for some $u \in \cL$. Because the left translation action of $\G$ on itself commutes with the right translation action, $\rho(g) : \cA \to \cA$ is a morphism of topological right $\cA$-comodules, in the sense that the following diagram is commutative:
\[\xymatrix{ \cA \ar[rr]^{\rho(g)}\ar[d]_{\h{\Delta_\delta}} && \cA \ar[d]^{\h{\Delta_\delta}} \\ \cA \h\otimes_{\cR} \cA \ar[rr]_{\rho(g) \h\otimes 1} && \cA \h\otimes_{\cR} \cA. }\]
The same is then true for the endomorphism $\log(\rho(g)) = p^\epsilon u$ of $\cA$. It follows that $p^\epsilon u \in \cT(\h\G)$ is a right-invariant derivation of $\h\G$. Since $\cA$ is flat over $\cR$, this implies that $u$ is also right-invariant: $u \in \cT(\h\G)^{\h\G}$. Hence, by Propositions \ref{FormalInvariantVectFlds} and \ref{LieTriangle}, there exists $v \in \fr{g}$ such that $\gamma'(v) = u$. Because $\cL \cong \cA \otimes \gamma'(\cJ)$, in fact $v$ must lie in $\cJ \subset \fr{g}$. Using equation (\ref{ExpIotaLogRho}) we see that
\[ \beta_{\cL}(g) = (\exp \circ \iota \circ \log \circ \rho)(g) = \exp(p^\epsilon \iota(\gamma'(v))) = \theta(\exp(p^\epsilon \iota(v))).\]
The uniqueness of $v$ follows from the injectivity of $\theta$, established in part (a) above.
(c) Let $g \in G_{\cL}$ and let $h \in \exp( p^\epsilon \iota(\cJ) )$ be such that $\beta_{\cL}(g) = \theta(h)$. The $g$-action on $\h{U(\cL)}$ is given by conjugation by $\theta(h)$ by Theorem \ref{Beta}(b), and the map $\theta$ is $G$-equivariant, so $\theta(h a h^{-1}) = g \cdot \theta(a) = \theta(g \cdot a)$ for all $a \in \h{U(\cJ)}$. So the $g$-action on $\h{U(\cJ)}$ is given by conjugation by $h$, by part (a). Finally, $\beta_{\cL}$ is $G$-equivariant by Theorem \ref{Beta}(b) and $\theta$ is $G$-equivariant, so $\theta^{-1} \circ \beta_{\cL}$ is also $G$-equivariant.
\end{proof}

We can now mimic Definition \ref{StarDefn} and make the following
\begin{defn}\label{wDGGDefn}  We define the \emph{completed skew-group algebra}
\[ \wDGG := \invlim\limits \hK{U(\cJ)} \rtimes_N G.\]
The inverse limit is taken over the set $\cK(G)$ all pairs $(\cJ, N)$, where $\cJ$ is a $G$-stable Lie lattice in $\fr{g}$ and $N$ is an open normal subgroup of $G$ contained in $G_{\cA \cdot \gamma'(\cJ)}$. The $G$-equivariant trivialisation of the $N$-action on $\hK{U(\cJ)}$ is understood to be the restriction to $N$ of the map $\theta^{-1} \circ \beta_{\cA \cdot \gamma'(\cJ)}$ from Proposition \ref{ThetaBeta}(c).
\end{defn}

Recall the definition of \emph{good chains} from Definition \ref{GoodChain}. 
\begin{prop}\label{StdPresWDGG} Let $U_n := \hK{U(\pi^n\fr{g})}$, let $\cA = \cO(\h\G)$ and let $(N_\bullet)$ be a good chain for the $\cA$-Lie lattice $\cA \cdot \gamma'(\fr{g})$ in $\cT(\h\G)$. Then there is an isomorphism
\[ \wDGG \cong \invlim U_n \rtimes_{N_n} G.\]
\end{prop}
\begin{proof} Each $\pi^n \fr{g}$ is evidently a $G$-stable Lie lattice in $\fr{g}$, so $(\pi^n \fr{g}, N_n) \in \cK(G)$ for each $n \geq 0$. The proof of Lemma \ref{StdPres} implies that the set of these pairs is in fact cofinal in $\cK(G)$.
\end{proof}

\begin{thm} \label{wDGGFrSt} $\wDGG$ is Fr\'echet-Stein.
\end{thm}
\begin{proof} Let $\cJ := \pi^2 \fr{g}$. This is a $G$-stable Lie lattice in $\fr{g}$ which satisfies $[\cJ,\cJ] \subseteq \pi^2 \cJ$. We may view $\cJ$ as an $(\cR,\cR)$-Lie algebra with the trivial anchor map. Then for each $n \geq 0$, $\pi^n \cJ$ also satisfies all these conditions, so in particular $U_n := \hK{U(\pi^n\cJ)}$ is Noetherian by Corollary \ref{HTF}  and $U_n$ is a flat $U_{n+1}$-module for each $n \geq 0$ by Theorem \ref{HKULflat}. Furthermore, the image of $U_{n+1}$ is dense in $U_n$ because it contains the dense image of $U(\fr{g}_K)$ in $U_n$. Now choose a good chain $(N_\bullet)$ for $\cA \cdot \gamma'(\cJ)$ using Corollary \ref{ChainCap} and apply Proposition \ref{StdPresWDGG} to find an isomorphism
\[ \wDGG \cong \invlim U_n \rtimes_{N_n} G.\]
The same argument appearing in the proof of Theorem \ref{FrSt}, using  \cite[Lemma 2.2]{SchmidtBB}, now completes the proof.  \end{proof}

\begin{lem}\label{NewStacky} Let $H$ be an open normal subgroup of $G$. Then there is a natural isomorphism
\[ \wDGG \cong \w\cD(\bG, H)^{\bG} \rtimes_H G.\]
\end{lem}
\begin{proof} Follow the proof of Corollary \ref{Stacky}. 
\end{proof}

We will now impose the following additional hypotheses on $\G$.
\begin{setup}\hspace{2em}\label{HgsSpace}
\begin{itemize}\item $\B$ is a closed and flat $\cR$-subgroup scheme of $\G$,
\item $\X := \G/\B$ is a scheme, flat and of finite presentation over $\cR$,
\item the canonical map $\xi : \G \to \X$ is a Zariski locally trivial $\B$-torsor.
\end{itemize}\end{setup}

We are interested in the $\pi$-adic completion $\h\X$ of $\X$ which is an admissible formal $\cR$-scheme, and its rigid generic fibre
\[ \bX := \h{\X}_{\rig}.\]
By Proposition \ref{GpOfPtsActsCtsly}, the group $G$ acts continuously on $\bX$. 

\begin{thm}\label{DGGaction} Assume that Hypotheses \ref{SigmaGG} and \ref{HgsSpace} are satisfied. Then $\wDGG$ acts on $\bX$ compatibly with $G$ in the sense of Definition \ref{AactsCompatibly}.
\end{thm}
We begin the proof of this result by fixing the following data:
\begin{itemize}
\item an affinoid subdomain $\bU$ of $\bX$, 
\item an open subgroup $H$ of $G_{\bU}$, and 
\item an $H$-stable affine formal model $\cB$ in $\cO(\bU)$. 
\end{itemize}
Because $\xi : \G \to \X$ is a Zariski locally trivial $\B$-torsor by assumption, we have at our disposal the anchor map $\alpha : (\xi_\ast \cT_{\G})^{\B} \to \cT_{\X}$ from equation (\ref{AnchorMapAlpha}). Let
\[ \alpha_{\bU} : \cT(\G)^{\G} \to \cT(\bU)\]
be the composite of the maps
\[ \cT(\G)^{\G} \hookrightarrow \cT(\G)^{\B} \stackrel{\alpha(\X)}{\longrightarrow} \cT(\X) \to \cT(\h{\X}) \longrightarrow \cT(\bX) \longrightarrow \cT(\bU).\]
In what follows, if $H$ is acting on an affinoid variety $\bY$ we will denote the corresponding action of $H$ on $\cO(\bY)$ by
\[\rho^{\bY} : H \to \Aut_K \cO(\bY).\]
The actions $\rho^{\bG}$ and $\rho^{\bU}$ preserve the affine formal models $\cA = \cO(\h{\G})$ and $\cB$, respectively. On the other hand, we have at our disposal the congruence subgroups $\cG_{p^\epsilon}(\cA)$ and $\cG_{p^\epsilon}(\cB)$ appearing in the proof of Proposition \ref{GactsOnX}.
\begin{lem}\label{AlphaLogRho} Let $N$ be the intersection of the subgroups $(\rho^{\bG})^{-1}\left(\cG_{p^\epsilon}(\cA)\right)$ and $(\rho^{\bU})^{-1}\left(\cG_{p^\epsilon}(\cB)\right)$ of $H$. Then $\alpha_{\bU} \circ \log \circ \rho^{\bG}_{|N} = \log \circ \rho^{\bU}_{|N}.$
\end{lem}
\begin{proof} We fix an affine covering $\{\X_1,\ldots,\X_m\}$ of $\X$ over which $\xi$ is trivialisable, let $\bX_i := \h{\X_i}$ and $\bU_i := \bX_i \cap \bU$ for each $i$, so that $\{\bU_1,\ldots,\bU_m\}$ is an admissible affinoid covering of $\bU$. Because $N \leq \cG_{p^\epsilon}(\cA)$, it acts on each $\cO(\h{\X_i})$ and these actions are trivial modulo $p^\epsilon \cO(\h{\X_i})$. So $\cA_i := \cB \cdot \cO(\h{\X_i})$ is an $N$-stable affine formal model in $\cO(\bU_i)$, and $N$ acts trivially on $\cA_i / p^\epsilon \cA_i$ for each $i$.

Recall from Proposition \ref{ThetaBeta}(b) that the image of $\log \circ \rho^{\bG}_{|N}$ lies in $\cT(\h{\G})^{\h{\G}}$. Let $\Y_i := \mu^{-1}(\X_i)$ and let $\bY_i := \h{\Y_i}_{\rig}$. Let us also write
\[\mu_{\bY} := \log \circ \rho_{|N}^{\bY} : N \to \cT(\bY)\]
whenever this makes sense. We may now consider the diagram
\[ \xymatrix{ \cT(\bU)\ar[dd] & & & & \cT(\h{\G})^{\h{\G}} \ar[llll]_{\alpha_{\bU}}\ar[rrd] & & \\ 
& & N\ar[llu]_{\mu_{\bU}}\ar[rru]^{\mu_{\bG}}\ar[rrrr]_{\mu_{\bY_i}}\ar[lld]_{\mu_{\bU_i}}\ar[rrd]_{\mu_{\bX_i}} & & & & \h{\cT(\Y_i)^{\B}} \ar[lld]^{\h{\alpha(\X_i)}}\\
\cT(\bU_i) & & & & \cT(\bX_i)\ar[llll]. }\]
Now, $\bU_i$ is an affinoid subdomain of both $\bU$ and $\bX_i$, and $\bY_i$ is an affinoid subdomain of $\bG$. Therefore, the three triangles in this diagram that contain an unlabelled arrow commute by Lemma \ref{LogRhoLemma}. The outer pentagon in this diagram commutes because $\alpha_{\bU}$ and $\alpha_{\bX_i}$ both factor through $\cT(\bX)$ by definition. Because the map $\cT(\bU) \to \oplus_i \cT(\bU_i)$ is injective, to show that $\alpha_{\bU} \circ \mu_{\bG} = \mu_{\bU}$ it remains to show that $\h{\alpha(\X_i)} \circ \mu_{\bY_i} = \mu_{\bX_i}$. We will drop the subscript $i$ to aid legibility.

Because $\xi$ is trivialisable over $\X$, by \cite[Lemma 4.3]{AW13} we have the isomorphism $q := \xi^\sharp(\X) : \cO(\X) \to \cO(\Y)^{\B}$. Passing to the $\pi$-adic completion gives us the $N$-equivariant map $\h{q} : \cO(\h{\X}) \longrightarrow \cO(\h{\Y})$, so that $\rho^{\bY}(n) \circ \h{q} = \h{q} \circ \rho^{\bX}(n)$ for all $n \in N$. It follows that
\[  \log \rho^{\bY}(n) \circ \h{q} = \h{q} \circ  \log \rho^{\bX}(n) \qmb{for all} n \in N.\]
Finally, let $n \in N$ and write $v := \log \rho^{\bY}(n) \in \cT(\h{\Y})$ and $a:= \h{\alpha(\X)}$. Then appealing to equations (\ref{AnchorMapAlpha}) and (\ref{AnchorMapAlphaDefn}), we see that
\[ \h{q}(a(v)(f)) = v(\h{q}(f)) = \h{q}( \log \rho^{\bX}(n)(f)) \qmb{for any} f \in \cO(\h{\X}).\] 
Since $\cO(\X)$ and $\cO(\Y)$ are flat as $\cR$-modules, we see that $\h{q}$ is injective, and it follows that $a \circ \log \circ \rho^{\bY} = \log \circ \rho^{\bX}$ as required.
\end{proof}
Let $\varphi : \G \to \zAut(\X)$ denote the action of $\G$ on $\X$, and consider the $\cR$-Lie algebra map $\varphi'_{\bU} : \fr{g} \longrightarrow \cT(\bU)$ that is the composite of the following natural maps:
\[ \varphi'_{\bU} : \fr{g} \stackrel{\varphi'}{\longrightarrow} \cT(\X) \longrightarrow \cT(\h{\X}) \longrightarrow \cT(\bX) \longrightarrow \cT(\bU).\]
It follows from Lemma \ref{EquivariantInfAction} together with the functorialities of $\pi$-adic completion and the rigid generic fibre functor that the map $\varphi'_{\bU}$ is $H$-equivariant. 

\begin{prop}\label{PhiFactors} Let $(\cJ, J)$ be a $\cB$-trivialising pair. Then there is an $H$-stable Lie lattice $\cH$ in $\fr{g}$ and an open normal subgroup $N$ of $H_{\cA \cdot \gamma'(\cH)}$ contained in $J$ such that the map
\[ U(\varphi'_{\bU}) \rtimes H : U(\fr{g}_K) \rtimes H \longrightarrow \hK{U(\cJ)} \rtimes_J H\]
factors through $\hK{U(\cH)} \rtimes_N H$.
\end{prop}
\begin{proof} By Definition \ref{DefnTrivPair}, $\cJ$ is an $H$-stable $\cB$-Lie lattice in $\cT(\bU)$ and $J$ is an open normal subgroup of $H$ contained in $H_{\cJ}$. Because $\fr{g}$ has finite rank as an $\cR$-module and because $\varphi'_{\bU}$ is $H$-equivariant, we can find a $\pi$-power multiple $\cH$ of $\fr{g}$ contained \footnote{If $\varphi'$ was known to be injective we could simply take $\cH$ to be this preimage. But without this hypothesis this preimage contains a $K$-line and therefore cannot be a finitely generated $\cR$-module.} in the preimage $(\varphi'_{\bU})^{-1}(\cJ)$ of $\cJ$ in $\fr{g}$. Then $\cH$ is an $H$-stable Lie lattice in $\fr{g}$, and we have a $K$-algebra map $\h{\varphi'_{\bU,K}} : \hK{U(\cH)} \to \hK{U(\cJ)}$. Let $\cA := \cO(\h{\G})$, $\cL := \cA  \cdot \gamma'(\cH)$, $H_0 :=  (\rho^{\bG})^{-1}\left(\cG_{p^\epsilon}(\cA)\right) \hsp \cap  \hsp (\rho^{\bU})^{-1}\left(\cG_{p^\epsilon}(\cB)\right)$ and consider the subgroup $N := H_{\cL} \cap J \cap H_0$ of $H$. It is open and normal in $H$, is contained in $J$ and satisfies
\begin{equation}\label{ALR} \alpha_{\bU} \circ \log \circ \rho^{\bG}_{|N} = \log \circ \rho^{\bU}_{|N}\end{equation}
by Lemma \ref{AlphaLogRho}. Now, $\theta^{-1} \circ \beta_{\cL} : N \to \hK{U(\cH)}^\times$ is an $H$-equivariant trivialisation of the $N$-action on $\hK{U(\cH)}$  by Proposition \ref{ThetaBeta}(c), and $\beta_{\cJ}  : J \to \hK{U(\cJ)}^\times$ is an $H$-equivariant trivialisation of the $J$-action on $\hK{U(\cJ)}$ by Theorem \ref{Beta}(b). We aim to apply Lemma \ref{CPfunc} to the ring homomorphism $\h{\varphi'_{\bU,K}}$, the identity map $1 : H \to H$ and the normal subgroups $N$ and $J$ of $H$. The first condition holds because $N \leq J$ and the second condition holds because $\varphi'_{\bU}$ is $H$-equivariant. To satisfy the third condition, we have to show that
\begin{equation}\label{VarphiThetaBeta} \h{\varphi'_{\bU,K}}^\times \circ \theta^{-1} \circ \beta_{\cL} = \beta_{\cJ}.\end{equation}
Let $n \in N$ so that $\beta_{\cL}(n) = \exp \hsp \iota(u)$ where $u := \log \rho^{\bG}(n) \in \cL \subset \cT(\bG)$. Here we abuse notation and use $\iota$ to denote the inclusions of $\cH$, $\cL$ and $\cJ$ into $\hK{U(\cH)}$, $\hK{U(\cL)}$ and $\hK{U(\cJ)}$, respectively.  By Proposition \ref{ThetaBeta}(b), we know that $u = \theta(w) = \gamma'(w)$ for some $w \in \cH$. Hence
\[ \theta^{-1} (\beta_{\cL}(n)) = \theta^{-1}(\exp \hsp \iota(\theta(w))) = \exp \hsp \iota(w).\]
Similarly, $\beta_{\cJ}(n) = \exp \iota(v)$ where $v := \log \rho^{\bU}(n) \in \cJ \subset \cT(\bU)$. Now
\[ \varphi'_{\bU}(w) = \alpha_{\bU}( \gamma'(w) ) = \alpha_{\bU}( u ) = \alpha_{\bU} (\log \rho^{\bG}(n)) = \log \rho^{\bU}(n) = v\]
by Lemma \ref{AlphaPsiVarphi} and equation (\ref{ALR}). Because $\h{\varphi'_{\bU,K}} \circ \iota = \iota \circ \varphi'_{\bU}$, we see that
\[ (\h{\varphi'_{\bU,K}}^\times \circ \theta^{-1} \circ \beta_{\cL})(n) = \h{\varphi'_{\bU,K}}^\times\left( \exp \hsp \iota(w) \right) = \exp\hsp \iota( \varphi'_{\bU}(w)) = \exp \hsp \iota(v) = \beta_{\cJ}(n)\]
for any $n \in N$, which proves equation (\ref{VarphiThetaBeta}). Now we may apply Lemma \ref{CPfunc} to obtain a $K$-algebra homomorphism
\[ \varphi'_{\bU} \rtimes 1 : \hK{U(\cH)} \rtimes_N H \longrightarrow \hK{U(\cJ)} \rtimes_J H\]
which extends $\h{\varphi'_{\bU,K}} : \hK{U(\cH)}  \longrightarrow \hK{U(\cJ)}$ and $1 : H \to H$.
\end{proof}

\begin{proof}[Proof of Theorem \ref{DGGaction}]
Let $A := \wDGG$. There is a canonical group homomorphism $\eta : G \to A^\times$. For every compact open subgroup $H$ of $G$, we define $A_H := \w\cD(\bG,H)^{\bG}$. We see from Definition \ref{wDGGDefn} that this is a $K$-subalgebra of $A$, which is Fr\'echet-Stein by Theorem \ref{wDGGFrSt}.

Now fix a compact open subgroup $H$ of $G$ and $\bU \in \bX_w/H$. Fix an $H$-stable affine formal model $\cB$ in $\cO(\bU)$ and let $(\cJ, J)$ be a $\cB$-trivialising pair. Then Definition \ref{wDGGDefn} and Proposition \ref{PhiFactors} give a $K$-algebra homomorphisms
\[ \w\cD(\bG, H)^{\bG} \longrightarrow \hK{U(\cH)} \rtimes_N H \longrightarrow \hK{U(\cJ)} \rtimes_J H\]
where $\cH \leq (\varphi'_{\bU})^{-1}(\cJ)$ is some $H$-stable Lie lattice in $\fr{g}$ and $N$ is the open normal subgroup $H_{\cA \cdot \gamma'(\cH)} \cap J \cap H_0$ of $H$. Here $H_0$ is some open normal subgroup of $H$ which does not depend on $(\cJ,J)$.  Choosing a different $H$-stable Lie lattice $\cH'$ with these properties will yield the same composite homomorphism $\w\cD(\bG,H)^{\bG} \to \hK{U(\cJ)} \rtimes_J H$ because both candidates will factor through $\hK{U(\cH'')} \rtimes_{N''} H$ for $\cH'' = \cH \cap \cH'$ and $N'' = N\cap H_{\cA \cdot \gamma'(\cH')}$. If $(\cJ',J')$ is another $\cB$-trivialising pair such that $\cJ' \leq \cJ$ and $J' \leq J$, then we can choose the corresponding pair $(\cH', N')$ to satisfy $\cH' \leq \cH$ and $N' \leq N$. Thus we obtain the commutative diagram
\[ \xymatrix{   \w\cD(\bG, H)^{\bG} \ar[r]\ar[dr] & \hK{U(\cH)} \rtimes_N H \ar[r] & \hK{U(\cJ)} \rtimes_J H \\ 
& \hK{U(\cH')} \rtimes_{N'} H \ar[r]\ar[u] & \hK{U(\cJ')} \rtimes_{J'} H. \ar[u] }
\]
We may now pass to the inverse limit over all $\cB$-trivialising pairs $(\cJ,J)$ to obtain the continuous $K$-algebra homomorphism
\[ \varphi^H(\bU) : A_H = \w\cD(\bG,H)^{\bG} \longrightarrow \w\cD(\bU,H).\]
required by Definition \ref{AactsCompatibly}. Let $\bV$ be an $H$-stable affinoid subdomain of $\bU$ and consider the diagram
\[\xymatrix{  \cD(\bU)\rtimes H \ar[rrrr]\ar[dd] &&&& \w\cD(\bU,H) \ar[dd]^{\tau^{\bU}_{\bV}}\\ 
& U(\fr{g}_K)\rtimes H\ar[dl]\ar[ul]\ar[rr] && \w\cD(\bG,H)^{\bG}  \ar[ur]^{\varphi^H(\bU)}\ar[dr]_{\varphi^H(\bV)}  & \\
 \cD(\bV) \rtimes H \ar[rrrr] &&&& \w\cD(\bV,H). }\]
The two trapezia commute by definition of $\varphi^H(\bU)$ and $\varphi^H(\bV)$. The outer rectangle commutes by definition of the restriction map $\tau^{\bU}_{\bV} : \w\cD(\bU,H) \to \w\cD(\bV,H)$ given in Lemma \ref{wDGpresheaf}, and the triangle on the left commutes because $\varphi'_{\bU}(v)_{|\bV} = \varphi'_{\bV}(v)$ for any $v \in \fr{g}$. Hence $\tau^{\bU}_{\bV} \circ \varphi^H(\bU)$ agrees with $\varphi^H(\bV)$ on the image of $U(\fr{g}_K) \rtimes H$ in $\w\cD(\bG,H)^{\bG}$. Because this image is dense and these two maps are continuous, it follows that they are equal. Hence $\varphi^H : A_H \to \w\cD(-,H)$ is a morphism of presheaves of $K$-Fr\'echet algebras on $\bX_w/H$.

We must now check that axioms (a)-(d) of Definition \ref{AactsCompatibly} are verified for these data. This verification is quite similar to the one carried out in the proof of Proposition \ref{XaffCompat}; it ultimately relies on Lemma \ref{EquivariantInfAction} and on a modified version of Proposition \ref{IHIGcofinal}. We leave the details to the reader. \end{proof}
\subsection{ \ts{\hnK{\cD^\lambda}}-affinity of the flag variety}\label{hDaffFlagVar}
In \cite{AW13}, we constructed a sheaf $\hnK{\cD^\lambda}$ on the flag variety $\G/\B$ over $\cR$ in the case where the ground ring $\cR$ was assumed to be \emph{Noetherian}. We proved that $\G/\B$ was $\hnK{\cD^\lambda}$-acyclic whenever $\lambda$ was $\rho$-dominant, and that it was $\hnK{\cD^\lambda}$-acyclic whenever $\lambda$ was $\rho$-regular. We also related the global sections of this sheaf to affinoid enveloping algebras $\hnK{U(\fr{g})}$ under the hypothesis that the prime $p$ is \emph{very good}. In this subsection, we will revisit \cite[\S 4,5,6]{AW13} in order to \emph{remove} this assumption on $p$ and remove the Noetherian assumption on $\cR$ from all of these statements.

We begin by recalling the main Beilinson-Bernstein construction from \cite[\S 4]{AW13}. Let $\G$ be a connected, simply connected, split semisimple, affine algebraic group scheme over an arbitrary commutative base ring $R$, for now. Let $\B$ be a closed and flat Borel $R$-subgroup scheme of $\G$, let $\U$ be its unipotent radical and let $\H := \B/\U$ be the abstract Cartan group. Choose a Cartan subgroup $\T$ of $\G$ complementary to $\U$ in $\B$, and let $i : \T \to \H$ be the natural isomorphism induced by the inclusion of $\T$ into $\B$. Let $\fr{g},\fr{b},\fr{n}$, $\fr{h}$ and $\fr{t}$ be the corresponding $R$-Lie algebras, and let $i : \fr{t} \tocong \fr{h}$ be the isomorphism induced by $i : \T \tocong \H$. The adjoint action of $\T$ on $\fr{g}$ induces a root space decomposition $\fr{g} = \fr{n} \oplus \fr{t} \oplus \fr{n}^+$. The adjoint action of $\G$ on $U(\fr{g})$ is by ring automorphisms, and we have the invariant subalgebra $U(\fr{g})^{\G}$. We call the composite of the natural inclusion of $\Uf{g}^{\G} \hookrightarrow \Uf{g}$ with the projection $\Uf{g} \twoheadrightarrow \Uf{t}$ with kernel $\fr{n} \hsp \Uf{g} + \Uf{g}\hsp \fr{n}^+$ the \emph{Harish-Chandra homomorphism} $\phi : \Uf{g}^{\G} \longrightarrow \Uf{t}$. Let $\cB=\G/\B$ be the \emph{flag variety} and $\Ex{\cB} = \G/\U$ the \emph{base affine space} of $\G$. The groups $\G$ and $\H$ act on $\Ex{\cB}$ by left and right translations respectively, and these actions commute. Applying Definition \ref{InfActionDefn}, we obtain the infinitesimal actions $\fr{g} \to \cT_{\Ex{\cB}}$ and $\fr{h} \to \cT_{\Ex{\cB}}$.  The natural projection $\xi : \Ex{\cB} \to \cB$ is a Zariski locally trivial $\H$-torsor, and $\Ex{\cT} := (\xi_\ast \cT_{\Ex{\cB}})^{\H}$ is a Lie algebroid on $\cB$.  Because the actions of $\G$ and $\H$ commute with the action of $\H$ on $\Ex{\cB}$, the infinitesimal actions of $\fr{g}$ and $\fr{h}$ on $\Ex{\cB}$ therefore descend to give homomorphisms of $\cR$-Lie algebras $\varphi : \fr{g} \to \Ex{\cT} \qmb{and} j : \fr{h} \to \Ex{\cT}$. The algebra $U(\Ex{\cT})$ is isomorphic to the \emph{relative enveloping algebra} $\Ex{\cD} := \xi_\ast U(\cT_{\Ex{\cB}})^{\H}$ from \cite[\S 4.6]{AW13}. By \cite[Lemma 4.10]{AW13}, these maps fit together into the following commutative diagram
\begin{equation}\label{VarphiJIPhi} \xymatrix{ \Uf{g}^{\G} \ar[d]\ar[r]^{\phi} & \Uf{t} \ar[d]^{j \circ i} \\
\Uf{g} \ar[r]_{U(\varphi)} & U(\Ex{\cT}).
}\end{equation}
Put another way, the restriction of $U(\varphi)$ to $\Uf{g}^{\G}$ is equal to $j \circ i \circ \phi$. 

Next, we specialise to the case where $R = \cR$ is a complete valuation ring of height one and mixed characteristic $(0,p)$ and recall some definitions from \cite[\S 3]{AW13}. Let $A$ be a positively $\Z$-filtered $\cR$-algebra with $F_0A$ an $\cR$-subalgebra of $A$. We say that $A$ is a \emph{deformable $\cR$-algebra} if $\gr A$ is a flat $\cR$-module. A \emph{morphism} of deformable $\cR$-algebras is an $\cR$-linear filtered ring homomorphism. The \emph{$n$-th deformation} of $A$ is the $\cR$-subalgebra
\[A_n := \sum_{j\geq 0} \pi^{jn} F_jA \subseteq A.\]
of $A$. It becomes a deformable $\cR$-algebra when we equip $A_n$ with the subspace filtration arising from the given filtration on $A$. By \cite[Lemma 3.5]{AW13}, multiplication by $\pi^{jn}$ on graded pieces of degree $j$ extends to a natural isomorphism of graded $\cR$-algebras $\gr A \stackrel{\cong}{\longrightarrow} \gr A_n$. The assignment $A \mapsto A_n$ is functorial in $A$. For example, whenever $\fr{g}$ is an $\cR$-Lie algebra which is free as an $\cR$-module, the PBW-theorem implies that its enveloping algebra $\Uf{g}$ is deformable, and $\Uf{g}_n \cong U(\pi^n \fr{g})$. Evidently these concepts only depend on the ideal $\pi \cR$ in $\cR$. 

Applying the deformation functor to diagram $(\ref{VarphiJIPhi})$ produces
\begin{equation}\label{VPn} \xymatrix{ (\Uf{g}^{\G})_n \ar[d]\ar[r]^{\phi_n} & \Uf{t}_n \ar[d]^{j_n \circ i_n} \\
\Uf{g}_n \ar[r]_{U(\varphi)_n} & \Ex{\cD}_n.
}\end{equation}
Let $\lambda : \pi^n \fr{h} \to \cR$ be a linear form. It extends to an $\cR$-algebra homomorphism $U(\fr{h})_n \tocong U(\pi^n \fr{h}) \to \cR$ and gives $\cR$ the structure of a $U(\fr{h})_n$-module which we denote by $\cR_\lambda$. Recall from \cite[Definition 6.4]{AW13} the sheaf
\[ \dnl := \Ex{\cD}_n \underset{U(\fr{h})_n}{\otimes} \cR_\lambda.\]
Let $\cS$ be the basis for $\cB$ consisting of open affine subschemes of $\cB$ that trivialise the $\H$-torsor $\xi : \Ex{\cB} \to \cB$. It contains each Weyl-group translate of the big cell in $\cB$.
\begin{lem}\label{DNL} Let $\Y \in \mathcal{S}$. 
\be\item $\left(\dnl\right)_{|\Y}$ is isomorphic to $(\cD_n)_{|\Y}$ as a sheaf of filtered $\cR$-algebras.
\item $\dnl$ is a sheaf of deformable $\cR$-algebras.
\item There is an isomorphism of graded $\cR$-algebras $\gr \dnl \cong \Sym_\cO \cT$ on $\cB$.
\item $\dnl$ is a quasi-coherent $\cO$-module.
\ee\end{lem}
\begin{proof} In the case where $\cR$ is Noetherian, parts (a,b,c) are \cite[Lemma 6.4(a,b,c)]{AW13}. The proof of \cite[Lemma 6.4]{AW13} in fact does not use the Noetherian assumption. By part (c), $\dnl$ is the direct limit of its subsheaves $F_i \dnl$, which are finite iterated extensions of the $\cO$-modules $\Sym^j_{\cO} \cT$. Therefore $\dnl$ is $\cO$-quasi-coherent by \cite[\href{http://stacks.math.columbia.edu/tag/01LA}{\S 25.24, (9), (6) and (4)}]{stacks-project}.
\end{proof}

\begin{defn} $\h{\dnl} := \invlim \dnl / \pi^a \dnl$ and $\hdnlK := \h{\dnl} \underset{\cR}{\otimes} K.$
\end{defn}

As in \cite{AW13}, we regard these objects as sheaves on the $\cR$-scheme $\cB$, supported only on the special fibre of $\cB$. As such, they are also naturally sheaves of $\cO_{\h{\cB}}$-modules on the $\pi$-adic completion $\h{\cB}$ of $\cB$, but this structure will only come into play later.

Our first task will be to compute the derived global sections of $\hdnlK$ in our general setting. We begin by computing higher cohomology locally.
\begin{lem}\label{CohDNLonS} $H^j(\Y, \h{\dnl}) = H^j(\Y, \hdnlK) = 0$ for all $j > 0$ whenever $\Y \in \cS$.
\end{lem}
\begin{proof} We will apply \cite[Chapter 0, Proposition 13.3.1]{EGAIII} to the inverse system $(\cF_a)_{a \in \N}$ of abelian sheaves on $\Y$, where $\cF_a$ is the restriction of $\dnl / \pi^a \dnl$ to $\Y$. Take the basis $\mathfrak{B}$ for $\Y$ be the set of affine open subschemes of $\Y$; then $H^j(\U, \cF_a) = 0$ for all $j > 0$ and all $\U \in \cB$ by \cite[Proposition 1.4.1]{EGAIII} because each $\cF_a$ is $\cO$-quasi-coherent by Lemma \ref{DNL}(d). This implies condition (ii) holds. Now $\dnl(\U) \cong U(\pi^n \cT(\U))$ for any $\U \in \mathfrak{B}$ by Lemma \ref{DNL}(a). Since this is a flat $\cR$-module, we have the short exact sequence $0 \to \cF_1 \to \cF_{a+1} \to \cF_a \to 0$ for each $a \geq 0$. This gives condition (iii), and also implies that the maps $\Gamma(\Y, \cF_{a+1}) \to \Gamma(\Y, \cF_a)$ are surjective for all $a \in \N$. Thus $(\Gamma(\Y, \cF_a))_{a \in \N}$ satisfies the Mittag-Leffler condition which yields condition (i). We conclude that $H^j(\Y, \h{\dnl}) \tocong \invlim H^j(\Y, \cF_a) = 0$. 

\vspace{-.37cm}For the second part, note that $\hdnlK$ is the inductive limit of the system $(\h{\dnl})_{a \geq 0}$ where each transition map is multiplication by $\pi$. Because $\Y$ is an affine scheme,  \cite[\href{http://stacks.math.columbia.edu/tag/01FE}{Lemma 20.20.1}]{stacks-project} implies that $H^j(\Y, \hdnlK) = \lim\limits_{\longrightarrow} H^j(\Y,\h{\dnl})$ for all $j \geq 0$. \end{proof}

Let $\fr{g}_K := \fr{g} \otimes_{\cR} K$. We write $Z(\fr{g}_K)$ to denote the centre of $U(\fr{g}_K)$.

\begin{lem}\label{UgGn} $(U(\fr{g})^{\G})_n \otimes_{\cR} K = Z(\fr{g}_K)$.
\end{lem} 
\begin{proof} By \cite[I.2.10(3)]{Jantzen} we have $(U(\fr{g})^{\G})_n \otimes_{\cR} K = U(\fr{g})^{\G} \otimes_{\cR} K = U(\fr{g}_K)^{\G_K}$. Because $K$ is a field of characteristic zero, it follows from \cite[I.7.10(1), II.1.9(4), II.1.12(1) and I.7.16]{Jantzen} that $U(\fr{g}_K)^{\G_K}$ equals the algebra of $\fr{g}_K$-invariants in $U(\fr{g}_K)$ under the adjoint representation. But this is just $Z(\fr{g}_K)$.
\end{proof}
The linear form $\lambda : \pi^n \fr{h} \to \cR$ extends to a $\cR$-algebra homomorphism $\lambda \circ \phi : (U(\fr{g})^{\G})_n \to \cR$, and hence to a $K$-algebra homomorphism $\lambda \circ \phi : Z(\fr{g}_K) \to K$. This gives $K$ the structure of a $Z(\fr{g}_K)$-module, which we denote by $K_\lambda$.

\begin{thm} \hspace{2em} \label{GlobSecDnlK}
\be\item There is an isomorphism $\hK{U(\pi^n\fr{g})} \underset{Z(\fr{g}_K)}{\otimes} K_\lambda \tocong \Gamma(\cB, \hdnlK).$
\item $\hdnlK$ has vanishing higher cohomology.
\ee\end{thm}
\begin{proof}
Let $\cU$ be the finite affine covering of $\cB$ given by the Weyl-translates of the big cell in $\cB$. Because every finite intersection of members of $\cU$ lies in $\cS$, the \v{C}ech-to-derived functor spectral sequence \cite[\href{http://stacks.math.columbia.edu/tag/03AZ}{Lemma 21.11.6}]{stacks-project} together with Lemma \ref{CohDNLonS} implies that $H^j(\cB, \hdnlK) = \check{H}^j(\cU, \hdnlK) = 0$ for all $j \geq 0$.

Consider the \v{C}ech complex $C^\bullet := C^\bullet(\cU, \dnl)$. The map $U(\varphi_n) : U(\fr{g})_n \to \Ex{\cD}_n$ induces an $\cR$-algebra homomorphism
\[ \psi : U(\fr{g})_n \to C^0, \quad x \mapsto U(\varphi)_n(x) \otimes 1 \in \Ex{\cD}_n \underset{U(\fr{h})_n}{\otimes} \cR_\lambda = \dnl.\]
Let $\fr{m}_\lambda$ be the kernel of the $K$-algebra homomorphism $\lambda \circ \phi : Z(\fr{g}_K) \to K$. Choose a finite set of generators $X \subset \fr{m}_\lambda$ for the ideal $U(\fr{g}_K)\fr{m}_\lambda$ in the Noetherian ring $U(\fr{g}_K)$. By Lemma \ref{UgGn}, we may assume that in fact $X \subset (\Uf{g}^{\G})_n \cap \fr{m}_\lambda$. Diagram (\ref{VPn}) implies that $\psi(X) = 0$, so we have the augmented complex
\[ D^\bullet := [0 \to U_n \cdot X \to U_n \stackrel{\psi}{\longrightarrow} C^0 \to C^1 \to \cdots \to C^d \to 0]\]
where $d = |\cU| - 1$ and $U_n := U(\fr{g})_n$. On the one hand, because $\dnl \otimes_{\cR} K$ is isomorphic to the usual sheaf of twisted differential operators $\cD^\lambda_K$ on the generic fibre $\cB_K$ of $\cB$, applying the functor $-\otimes_{\cR} K$ produces the complex
\[ 0 \to U(\fr{g}_K) \cdot \fr{m}_\lambda \to U(\fr{g}_K)  \to C^\bullet(\cU, \cD^\lambda_K) \]
which is acyclic by classical results in characteristic zero --- see \cite[Lemme 2.3]{BB} and \cite[Th\'eor\`eme 3.2(iv)]{SprBour}. It follows that the complex $D^\bullet$ has $\pi$-torsion cohomology. On the other hand, because $\dnl$ is $\cO$-quasi-coherent by Lemma \ref{DNL}(d), we have by \cite[Proposition 1.4.1]{EGAIII} that $H^j(C^\bullet) = H^j(\cU, \dnl) \cong H^j(\cB, \dnl)$ for all $j \geq 0$. \emph{Provided $\cR$ is Noetherian}, this is a finitely generated $U_n$-module by \cite[Proposition 5.15(b)]{AW13}. Because $H^{-1}(D) = \ker \psi / U_n \cdot X$  is a finitely generated $U_n$-module being a subquotient of the Noetherian $\cR$-algebra $U_n$, we conclude that the complex $D^\bullet$ has \emph{bounded} $\pi$-torsion cohomology whenever $\cR$ is Noetherian.

Let $h_1,\ldots, h_l \in \fr{h}$ be the simple coroots corresponding to the simple roots in $\fr{h}^\ast_K$ given by the adjoint action of $\H$ on $\fr{g}/\fr{b}$. Note that because $\G$ is assumed to be simply connected, the Lie algebra $\fr{h}$ is generated as an $\cR$-module by the $h_i$ by \cite[II.1.6, II.1.11]{Jantzen}. Using the elementary Lemma \ref{SmallDVR} below, we can find a complete discrete valuation subring $\cR'$ of $\cR$ containing the finite set $\{\lambda(h_1),\ldots,\lambda(h_l),\pi\}$; then $\lambda$ is defined over $\cR'$ in the sense the restriction $\lambda'$ of $\lambda$ to $\pi^n \cR'$ takes values in $\cR'$. Let $\G', \fr{g}', K'$, $\fr{m}_{\lambda'}$ and $X'$ be the corresponding objects defined over $\cR'$ instead of $\cR$. The isomorphism $U(\fr{g}'_{K'}) \otimes_{K'} K \tocong U(\fr{g}_K)$ carries $\fr{m}_{\lambda'}$ into $\fr{m}_\lambda$, and in fact $\fr{m}_\lambda = \fr{m}_{\lambda'} \cdot K$. It follows that $X'$ generates $U(\fr{g}_K)\fr{m}_\lambda$ as a left ideal so we may assume that $X = X'$.
Because  $\pi \in \cR'$ by construction, we can also form the sheaves $\cD_n^{\lambda'}$ and $\h{\cD_{n,K}^{\lambda'}}$ on the flag variety $\cB'$ defined over $\cR'$, which has the covering $\cU'$ by the Weyl-translates of the big cell. 

Consider the \v{C}ech complex of $\cD_n^{\lambda'}$ with respect to the covering $\cU'$.  Forming the augmented complex $D^{'\bullet}$ as above, there is an evident map $D^{'\bullet} \otimes_{\cR'} \cR \to D^\bullet$ which is in fact an isomorphism because $X' = X$. We may now use the flatness of $\cR$ as an $\cR'$-module to conclude that $H^j(D^\bullet) \cong H^j(D^{'\bullet}) \otimes_{\cR'} \cR$. Because $\cR'$ is Noetherian, we see that $D^\bullet$ has \emph{bounded} $\pi$-torsion cohomology, in general. Since $D^\bullet$ consists of torsionfree $\cR$-modules by construction, we conclude from \cite[Lemma 3.6]{DCapOne} that $\h{D^\bullet}_K$ is acyclic. This implies that 
\[\Gamma(\cB, \hdnlK) = H^0(\h{C^\bullet}_K) = \frac{\h{U_{n,K}} }{ \h{U_{n,K}} \cdot X} \cong \hK{U(\pi^n\fr{g})} \underset{Z(\fr{g}_K)}{\otimes} K_\lambda\]
and also that $H^j(\cB, \hdnlK)  = \check{H}^j(\cU, \hdnlK) = 0$ for all $j > 0$. 
\end{proof}

\begin{lem}\label{SmallDVR} Let $X$ be a finite subset of $\cR$. Then there is a complete discrete valuation subring $\cR'$ of $\cR$ containing $X$.
\end{lem}
\begin{proof} The subring $S$ of $\cR$ generated by $X$ is Noetherian, being a homomorphic image of a polynomial ring over $\Z$ in finitely many variables. On the other hand it is a domain. Let $\fr{m}$ be the maximal ideal of $\cR$, so that $\fr{m} \cap S$ is a prime ideal of $S$. Because $\cR$ is a local ring, the localisation $T$ of $S$ at $\fr{m} \cap S$ is contained in $\cR$. It is a local Noetherian ring with maximal ideal $J$, say; because $p \in \fr{m} \cap S$ and because $T$ is a ring of characteristic zero, we see that $J$ cannot be zero. Let $F$ be a finite generating set for $J$ as an ideal in $T$. Because $\cR$ is a valuation ring, we see that $J\cdot \cR = F \cdot \cR$ is principal. Choose a non-zero element $\pi' \in F$ such that $J \cdot \cR = \pi' \cR$. Now $(J/\pi'T) \otimes_T \cR = 0$ implies that $(J / \pi'T) \otimes_T \cR/\fr{m} = 0$. Since $\cR/\fr{m}$ is a field extension of $T/J$, it is faithfully flat. Hence $(J / \pi'T) \otimes_T (T/J) = 0$ which forces $J = \pi'T$ by Nakayama's Lemma, since $J$ is finitely generated. Because $\cap J^n = 0$ by Krull's Intersection Theorem, we see that $T$ is a discrete valuation ring with maximal ideal $\pi'T$. Finally, since $\pi'$ is a non-zero non-unit in $\cR$, $\cR$ is $\pi'$-adically complete, which implies that the $\pi'$-adic completion $\cR'$ of $T$ embeds into $\cR$. This $\cR'$ is the required complete discrete valuation ring such that $\cR' \supseteq T \supseteq S \supseteq X$.
\end{proof}

Our next task will be to exhibit a well-behaved countable family of generators for the category of coherent $\hdnlK$-modules on $\cB$. This will extend \cite[Theorem 6.3]{AW13} to our non-Noetherian setting. We begin with the following generalisation of Serre's Theorem.

\begin{lem}\label{GenSerre} Let $Z$ be a closed subscheme of $\P^N_{S(\fr{g})}$ and let $\cF$ be a coherent $\cO_Z$-module. Then for sufficiently large $n \in \N$, the Serre twist $\cF(n)$ of $\cF$ is $\Gamma$-acyclic and generated by finitely many global sections.
\end{lem}
\begin{proof} We may assume $Z = \P^N_{S(\fr{g})}$. By \cite[Chapter 0, Corollary 9.2.7 and Definition 8.5.17]{FujiKato}, $\cR$ is $\pi \cR$-adically universally adhesive. Because it is also $\pi$-torsion-free, \cite[Chapter 0, Theorem 8.5.25]{FujiKato} implies that $\cR$ is universally coherent. In particular, the polynomial algebra $B := S(\fr{g})$ is also universally coherent. In this situation, \cite[Chapter I, Proposition 8.2.2]{FujiKato} tells us that for every coherent $\cO_{\P^N_{B}}$-module $\cF$, $\cF(n)$ is $\Gamma$-acyclic for sufficiently large $n$ and $H^q(\P^N_B, \cF)$ is a finitely presented $B$-module for all $q \in \N$. On the other hand, \cite[Corollaire 2.7.9]{EGAII} tells us that $\cF(n)$ is generated by finitely many global sections for sufficiently large $n$. 
\end{proof}

The scheme $\cB$ is projective over $\cR$ by \cite[II.1.8]{Jantzen}. Fix an embedding ${\iota : \cB \hookrightarrow \P^N_\cR}$ into some projective space over $\cR$ and let $\sL := \iota^\ast \cO(1)$ be the corresponding very ample invertible sheaf on $\cB$. For any $\cO_\cB$-module $\cM$ and any $s \in \Z$, we let $\cM(s) :=\cM \otimes_{\cO_{\cB}}  \sL^{\otimes s}$ denote the Serre twist of $\cM$. Let $\cA$ be one of the sheaves of rings $\dnl$, $\h{\dnl}$ or $\hdnlK$ and define $\cA^{(s)} := \sL^{\otimes s} \otimes_{\cO_\cB} \cA \otimes_{\cO_\cB} \sL^{\otimes (-s)}$. As an $\cO_\cB$-module this sheaf is isomorphic to $\cA$, but it is also naturally a sheaf of rings which is \emph{not} isomorphic to $\cA$. For every $\cA$-module $\cN$, the twisted sheaf $(s)\cN: = \sL^{\otimes s} \otimes_{\cO_\cB} \cN$ is naturally an $\cA^{(s)}$-module by contracting tensor products. We retain the notation $\cA(s)$ to mean the left $\cA$-module $\cA\otimes_{\cO_{\cB}}\sL^{\otimes s}$ with $\cA$ acting on the left factor.

\begin{lem}\label{GenThmPrep} Write $\overline{\cR} := \cR / \pi \cR$ and $\overline{\cB} := \cB \otimes_{\cR} \overline{\cR}$, and suppose that $n \geq 1$.
\be \item $\dnl / \pi \dnl \cong \Sym_{\cO_{\overline{\cB}}} (\cT_{\overline{\cB}})$ as sheaves of graded $\cR / \pi \cR$-algebras.
\item Let $\cM$ be a coherent $\Sym_{\cO_{\overline{\cB}}} (\cT_{\overline{\cB}})$-module. Then for all sufficiently large integers $s$, $(s)\cM$ is $\Gamma$-acyclic and generated by finitely many global sections as a $\Sym_{\cO_{\overline{\cB}}} (\cT_{\overline{\cB}})$-module.
\item $\h{\dnl}$ is a coherent sheaf.
\item Every coherent $\h{\dnl}$-module is $\pi$-adically complete.
 \ee\end{lem}
\begin{proof} (a) Recall from Lemma \ref{DNL}(c) that we have an isomorphism $\gr \dnl \cong \Sym_{\cO} \cT$ of graded $\cR$-algebras on $\cB$. Because each graded piece of $\gr \dnl$ is flat as an $\cR$-module,  $\gr (\dnl / \pi \dnl) \cong (\gr \dnl) \otimes_{\cR} \overline{\cR} \cong (\Sym_{\cO} \cT) \otimes_{\cR} \overline{\cR} \cong \Sym_{\cO_{\overline{\cB}}} \cT_{\overline{\cB}}$.

(b) Let $\beta : T^\ast \cB \hookrightarrow \cB \times \fr{g}^\ast$ be the closed embedding provided by the infinitesimal action of $\fr{g}$ on $\cB$, and let $p : \cB \times \fr{g}^\ast \to \cB$ be the projection onto the first factor. Because $\Sym_{\cO_{\overline{\cB}}} (\cT_{\overline{\cB}}) = p_\ast \beta_\ast \cO_{T^\ast \overline{\cB}}$, there is a coherent $\cO_{\cB \times \fr{g}^\ast}$-module $\cN$, killed by $\pi$, such that $\cM = p_\ast \cN$.  

Let $B := S(\fr{g})$ and let $\iota_B : \cB \times \fr{g}^\ast = \cB \otimes_{\cR} B \hookrightarrow \P^N \otimes_{\cR} B = \P^N_B$ be the base change of $\iota$ to $B$. It is still a closed embedding, and $\iota_B^\ast \cO_{\P^N_B}(1) \cong  p^\ast \mathscr{L}$ is a very ample line bundle on $\cB \times \fr{g}^\ast$, so we may define $\cN(s) := \cN \otimes_{\cO_{\cB \times \fr{g}^\ast} } (p^\ast \mathscr{L})^{\otimes s}$ for each $s \in \Z$. In this situation \cite[\S 0.5.4.10]{EGAI} implies that $p_\ast(\cN \otimes_{\cO_{\cB \times \fr{g}^\ast}}p^\ast\mathscr{L})\cong p_\ast\cN \otimes_{\cO_\cB} \mathscr{L}$, so
\[p_\ast(\cN(s))\cong \mathcal{M}(s)\]
for all $s \in \Z$. Because $p$ is an affine morphism, \cite[Corollaire III.1.3.3]{EGAIII} tells us that $H^i(\cB, \cM(s)) \cong H^i(\cB,p_\ast(\cN(s))) \cong H^i(\cB\times \fr{g}^\ast,\cN(s))$ for all $i$ and $s$. So $\cM(s)$ is $\Gamma$-acyclic whenever $\cN(s)$. Finally, if $\cF$ is a coherent $\cO_{\cB \times \fr{g}^\ast}$-module which is generated by finitely many global sections, then $p_\ast \cF$ is a coherent $\cO_{\cB} \otimes_{\cR} B$-module generated as such by finitely many global sections. We can now apply Lemma \ref{GenSerre}.

(c) Berthelot's \cite[Proposition 3.1.1]{Berth} has the following weaker form: if $\mathscr{D}$ is a sheaf of rings on a topological space $X$ and if $\cS$ is a basis for $X$ such that $\mathscr{D}(U)$ is a left coherent ring for each $U  \in \cS$, then $\mathscr{D}$ is left coherent. So it is enough to see that $\h{\dnl}(\Y)$ is left coherent for each $\Y \in \cS$. By Lemma \ref{DNL}(a), there is an isomorphism $\h{\dnl}(\Y) \cong \h{\cD_n}(\Y)$, and this $\cR$-algebra satisfies Hypothesis \ref{NCadmRalg} because $n \geq 1$. Now Theorem \ref{KeyRGLemma} implies that every finitely generated left ideal in $\h{\cD_n}(\Y)$ is finitely presented so $\h{\cD_n}(\Y)$ is left coherent.

(d) Let $\Y \in \cS$ and let $D = \cD_n(\Y)$. We will first show that the functor 
\[M \mapsto M^\Delta := \invlim \cD_n \otimes_D (M / \pi^a M)\] 

\vspace{-0.3cm} \noindent is exact on finitely presented $\h{D}$-modules.  Let $0 \to A \to B \to C \to 0$ be an exact sequence of finitely presented $\h{D}$-modules. For every $D$-module $M$, the isomorphism $\cD_n \otimes_D M \cong \cO_{\Y} \otimes_{\cO(\Y)} M$ of sheaves of left $\cO_{\Y}$-modules shows that the functor $\cD_n \otimes_D -$ is exact on $D$-modules. Therefore the sequence of towers of $\cD_n$-modules
\[0 \to \left[ \cD_n \otimes_{D} \frac{A + \pi^a B}{\pi^aB} \right]_a \to \left[ \cD_n \otimes_{D} \frac{B}{\pi^aB} \right]_a \to \left[ \cD_n \otimes_{D} \frac{C}{\pi^aC}\right]_a \to 0\]
is exact. The maps in the left-most non-zero tower are surjective, so it satisfies the Mittag-Leffler condition. Taking inverse limits gives a short exact sequence
\[0 \to \invlim  \cD_n \otimes_{D} \frac{A + \pi^a B}{\pi^aB} \to B^\Delta \to C^\Delta \to 0.\]
We saw in the proof of part (c) that $\h{D}$ satisfies Hypothesis \ref{NCadmRalg}. Because $B$ is a finitely generated $\h{D}$-module, it follows from Theorem \ref{KeyRGLemma} that we can find an integer $n_0$ such that $\pi^a A \subseteq A \cap \pi^a B \subseteq \pi^{n - n_0}A$ for all $n \geq n_0$. Therefore the natural map $A^\Delta = \invlim \cD_n \otimes_D \frac{A}{\pi^aA}\longrightarrow \invlim \cD_n \otimes_D \frac{A + \pi^a B}{\pi^aB}$ is an isomorphism, and the functor $M \mapsto M^\Delta$ is exact as claimed. 

We will next show that every finitely presented $\h{\cD_n}$-module $\cN$ on $\Y$ is $\pi$-adically complete. So $\cN$ is the cokernel of a $\h{\cD_n}$-linear map $u$ between two free $\h{\cD_n}$-modules of finite rank. Let $\h{D} := \h{\cD_n}(\Y)$; then $\h{D}^\Delta = \h{\cD_n}$, and it follows from the above paragraph that $\cN \cong N^\Delta$ where $N$ is the finitely presented $\h{D}$-module $\coker \Gamma(\Y,u)$. Now, the sequence $N^\Delta \stackrel{\pi^a}{\rightarrow} N^\Delta \to (N / \pi^a N)^\Delta \to 0$ is exact, so 
\[ \cN / \pi^a \cN \cong N^\Delta / \pi^a N^\Delta \cong (N / \pi^a N)^\Delta \cong \cD_n \otimes_{\cD_n(\Y)} N / \pi^a N\]
for any $a \geq 1$. Hence $\cN \cong N^\Delta = \invlim \cD_n \otimes_{\cD_n(\Y)} N/\pi^a N \cong \invlim \cN / \pi^a \cN$.

Finally, the definition of coherent modules from \cite[Chapitre 0, $\S$ 5.3.1]{EGAInew} together with the fact that $\cB$ admits a basis of quasi-compact open subsets in $\cS$ implies that we can find an $\cS$-covering of $\cB$ on which our coherent $\h{\dnl}$-module $\cM$ is finitely presented. The restriction of the canonical map $\cM \to \invlim \cM/\pi^a \cM$ to each member of this covering is an isomorphism by the above, and therefore the map is itself an isomorphism.
\end{proof}

\begin{thm}\label{TwistsGenerate}
The sheaves $\{\hdnlK(s) : s \in \Z\}$ generate the category of coherent $\hdnlK$-modules whenever $n \geq 1$.
\end{thm}
\begin{proof} Write $\cA := \h{\dnl}$ to aid legibility. Let $\cM$ be a coherent $\cA_K$-module. By the proof of \cite[Lemma 3.4.3]{Berth}, we can find a coherent $\cA$-submodule $\cN$ of $\cM$ such that $\cM \cong \cN \otimes_{\cR} K$. Using Lemma \ref{GenThmPrep}(a,b), we can find an integer $s$ such that $(s)(\cN / \pi \cN)$ is $\Gamma$-acyclic and generated as an $\cA^{(s)}$-module by its global sections. Because $\cM$ has no $\pi$-torsion, the sequence $0 \to \cN / \pi \cN \stackrel{\pi^i}{\longrightarrow} \cN / \pi^{i+1} \cN\to \cN/\pi^i \cN \to 0$ is exact for all $i \geq 0$. Since also $H^1(\cB, (s)(\cN / \pi \cN)) = 0$, twisting this sequence by $\sL^{\otimes s}$ on the left and taking cohomology shows that the arrow $\Gamma(\cB, \cK / \pi^{i+1} \cK) \to \Gamma(\cB, \cK / \pi^i \cK)$ is surjective for all $i \geq 0$, where $\cK := (s)\cN$. Because $\cK$ is a coherent $\cA^{(s)}$-module, it is $\pi$-adically complete by Lemma \ref{GenThmPrep}(d). Thus we find an $\cA^{(s)}$-linear map $\theta : \left(\cA^{(s)}\right)^a \to \cK$ which is surjective modulo $\pi \cK$. So $\cC := \coker(\theta)$ satisfies $\cC = \pi \cC$.  Because $\cA$ is a coherent $\cA$-module by Lemma \ref{GenThmPrep}(c), we see that $\cC$ is a coherent $\cA^{(s)}$-module, so $\cC$ is $\pi$-adically complete by Lemma \ref{GenThmPrep}(d). Therefore $\cC = 0$ and $\theta$ is surjective. Twisting back by $\sL^{\otimes -s}$ on the left, we find a surjective map $((-s)\cA^{(s)})^a \twoheadrightarrow \cN$ of left $\cA$-modules. But $(-s)\cA^{(s)}\cong \cA(-s)$ as left $\cA$-modules, so we obtain a surjective map $(\cA(-s))^a \twoheadrightarrow \cN$ of $\cA$-modules, and after inverting $\pi$ a surjective map $(\cA_K(-s))^a  \twoheadrightarrow \cM$ of $\cA_K$-modules, as required.
\end{proof}

Let $h_1,\ldots, h_l \in \fr{h}$ be the simple coroots corresponding to the simple roots in $\fr{h}^\ast_K$ given by the adjoint action of $\H$ on $\fr{g}/\fr{b}$, let $\omega_1, \ldots, \omega_l \in \fr{h}^\ast_K$ be the dual basis to $\{h_1,\ldots,h_l\}$ and recall that $\rho := \omega_1 + \ldots + \omega_l$. 

\begin{defn}\label{DomRegDefn} Let $\lambda \in \fr{h}_K^\ast = \Hom_{\cR}(\fr{h}, K)$.
\be 
\item $\lambda$ is said to be \emph{integral} if $\lambda(h_i) \in \Z$ for all $i$. 
\item $\lambda$ is \emph{dominant} if $\lambda(h)\notin \{-1,-2,-3,\cdots \}$ for any positive coroot $h\in \fr{h}$. 
\item $\lambda$ is \emph{$\rho$-dominant} if $\lambda + \rho$ is dominant.
\item $\lambda$ is \emph{regular} if its stabiliser under the action of the Weyl group is trivial.
\item $\lambda$ is \emph{$\rho$-regular} is $\lambda + \rho$ is regular.
\ee\end{defn}

\begin{thm}\label{hdnlKgenerates} Suppose that $n \geq 1$ and that $\lambda$ is $\rho$-dominant and $\rho$-regular. Then $\hdnlK$ generates the category of coherent $\hdnlK$-modules.
\end{thm}
\begin{proof} By Theorem \ref{TwistsGenerate}, it will be enough to show that for each $s \in \Z$ there is a surjection $\hdnlK^a \twoheadrightarrow \hdnlK(s)$ for some $a \in \N$. Because $\cD^\lambda_K(s)$ is a coherent $\cD^\lambda$-module and because $\lambda$ is $\rho$-dominant and $\rho$-regular, we can find a surjection $(\cD^\lambda_K)^a \to \cD^\lambda_K(s)$ for some $a \in \N$, by paragraph (iv) of the proof of \cite[Th\'eor\`eme Principal]{BB}. By clearing denominators, we can find a $\cD^\lambda_n$-module homomorphism $\theta : (\cD^\lambda_n)^a \to \cD^\lambda_n(s)$ with $\pi$-torsion cokernel $\cC$ say. Because $\cD^\lambda_n(s)$ is locally on $\cS$ a free $\cD^\lambda_n$-module of rank $1$ and because $\cB$ is quasi-compact, we see that in fact $\cC$ is bounded $\pi$-torsion. 

If $0 \to A \to B \to C \to 0$ is an exact sequence of $\cR$-modules, then $0 \to \h{A} \to \h{B} \to \h{C} \to 0$ is again exact whenever $C$ is either $\pi$-torsion-free, or bounded $\pi$-torsion. It follows from this that the sequence $\h{ \cD^\lambda_n (\Y)}^a \to \h{ \cD^\lambda_n(s)(\Y)} \to \h{\cC(\Y)} \to 0$ is exact for any $\Y \in \cS$. Therefore $\h{\theta}_K : \hdnlK^a \to \hdnlK(s)$ is surjective.
\end{proof}

Recall from \cite[Definition 5.1]{AW13} that a sheaf of rings $\mathscr{D}$ on a topological space $X$ is said to be \emph{coherently $\mathscr{D}$-affine} if for every coherent $\mathscr{D}$-module $\cM$, $\cM$ is generated by its global sections, $\cM$ is $\Gamma$-acyclic, and $\cM(X)$ is a coherent $\mathscr{D}(X)$-module. 

\begin{thm}\label{hdnlKaffine} Suppose that $n \geq 1$ and that $\lambda$ is $\rho$-dominant and $\rho$-regular. Then $\cB$ is coherently $\hdnlK$-affine.
\end{thm}
\begin{proof} Let $\cM$ be a coherent $\hdnlK$-module. It follows from Lemma \ref{GenThmPrep}(c) that $\hdnlK$ is also a coherent $\hdnlK$-module. Using Theorem \ref{hdnlKgenerates}, we can find a resolution $\cdots \to \cP_2 \stackrel{d_2}{\to} \cP_1 \stackrel{d_1}{\to} \cP_0 \to \cM \to 0$ of $\cM$, where each $\cP_i$ is a finite free $\hdnlK$-module of finite rank. A dimension shifting argument together using  Theorem \ref{GlobSecDnlK}(b) now implies that $H^j(\cB, \cM) = H^{j+r}(\cB, \im d_r)$ for any $j \geq 1$ and $r \geq 1$. This vanishes whenever $r \geq \dim \cB$ by \cite[Theorem III.2.7]{Hart}, so $\cM$ is $\Gamma$-acyclic. Hence $\im d_r$ is also $\Gamma$-acyclic for each $r \geq 1$, and it follows that $\cM$ is generated by its global sections. We also obtain the exact sequence $\cdots \to \Gamma(\cB, \cP_1) \to \Gamma(\cB, \cP_0) \to \Gamma(\cB, \cM) \to 0$ of $\Gamma(\cB,\hdnlK)$-modules. Because $\Gamma(\cB, \hdnlK)$ is Noetherian by Theorem \ref{GlobSecDnlK}(a) and Corollary \ref{HTF}, it follows that $\Gamma(\cB,\cM)$ is a coherent $\Gamma(\cB,\hdnlK)$-module.
\end{proof}

\begin{cor}\label{BBLevelwise} Suppose that $n \geq 1$  and that $\lambda$ is $\rho$-dominant and $\rho$-regular. Then $\Gamma(\cB,-)$ is an exact equivalence between the category of coherent $\hdnlK$-modules and the category of finitely generated $\h{U(\pi^n \fr{g})_K}$-modules killed by $\fr{m}_\lambda$.
\end{cor}
\begin{proof} This follows from Theorems \ref{hdnlKaffine},  \ref{GlobSecDnlK}(a) and \cite[Proposition 5.1]{AW13}.
\end{proof}

\subsection{The localisation functor is essentially surjective}
We continue with the notation and hypotheses of $\S \ref{hDaffFlagVar}$. Thus, $\G$ is a connected, simply connected, split semisimple affine algebraic group scheme over $\cR$ and $\B$ is a closed and flat Borel $\cR$-subgroup scheme of $\G$. Note that the canonical map $\xi : \G \to \X := \G/\B$ is then a Zariski locally trivial $\B$-torsor by \cite[Chapter II, 1.10(2)]{Jantzen}, so Hypothesis \ref{HgsSpace} is satisfied for the flag variety $\X$. 

Recall from Definition \ref{CongSubTop} that $\G_{\pi^n}(\cR)$ denotes the congruence kernel
\[ \G_{\pi^n}(\cR) = \ker (\G(\cR) \to \G(\cR/\pi^n\cR)). \]
By definition of the topology on $\G(\cR)$, these form a descending chain of open normal subgroups in $\G(\cR)$. We will assume that we are given a continuous group homomorphism from a compact $p$-adic Lie group $G$ to the first congruence kernel
\[ \sigma : G \to G_\pi(\cR).\]
Certainly then Hypothesis \ref{SigmaGG} is satisfied.  By Theorem \ref{DGGaction}, we know that the $K$-algebra $A := \wDGG$ acts on the rigid analytic flag variety $\bX := \h{\X}_{\rig}$ compatibly with $G$. Hence, by Proposition \ref{LocFunCoadm}, we have at our disposal the localisation functor 
\[ \Loc^A_\bX : \cC_A \longrightarrow \cC_{\bX/G}.\]
For each $K$-linear character $\lambda : \fr{h}_K \to K$, we have a $K$-algebra homomorphism $\lambda \circ \phi : Z(\fr{g}_K) \to K$, and we will denote its kernel by $\fr{m}_\lambda$. We are concerned here only with the case where $\lambda$ is the zero map.  

\begin{thm}\label{BBEssSurj} The localisation functor $\Loc^A_\bX : \cC_A \longrightarrow \cC_{\bX/G}$ is essentially surjective on objects.
\end{thm}
	
We now start working towards proving Theorem \ref{BBRGamma} below, which states that  $\Gamma(\bX, \Loc^A_{\bX}(M)) \cong M / \fr{m}_0 \cdot M$ for every coadmissible $A$-module $M$. Our first task is to extend the sheaves appearing in Definition \ref{DefnhsULKG} to certain affinoid subdomains of the rigid analytic flag variety, which is  itself not affinoid. Because our ultimate goal is to relate this construction with the material in $\S$\ref{hDaffFlagVar}, we will work on the (special fibre of) the $\cR$-scheme $\X$ for simplicity.

Recall from \S \ref{wDGGsect} that $\cA = \cO(\h\G)$ and that $\fr{g} = \Lie(\G)$. For each $n \geq 1$, we define $G_n := \sigma^{-1}(\G_{\pi^n}(\cR))$ and note that
\[ G = G_1 \geq G_2 \geq \cdots\]
is a descending chain of open normal subgroups of $G$ since $\sigma$ is continuous. 

\begin{lem}\label{CongSubsTriv} Let $\Y \in \cS$, write $\cY = \h{\Y}$, and $\cT(\cY) := \Der_{\cR}(\cO(\cY))$. Then
\be 
\item $\cO(\cY)$, $\cT(\cY)$ and $\cY_{\rig}$ are $G$-stable.
\ee
Suppose further that the integer $n$ is large enough so that $\pi^n \in p^\epsilon\cR$. Then 
\begin{enumerate}[{(}a{)}]
\setcounter{enumi}{1}
\item $\rho^{\bG}(G_n) \leq \cG_{p^\epsilon}(\cA)$ and $\rho^{\cY_{\rig}}(G_n) \leq \cG_{p^\epsilon}(\cO(\cY)),$ and
\item $G_n \leq G_{\cA \cdot \gamma'(\pi^n\fr{g})} \cap G_{\pi^n \cT(\cY)}$.
\ee
\end{lem}
\begin{proof} (a) We saw in the proof of Lemma \ref{AutSheaf} that the open affine subscheme $\cY$ of $\h\X$ is $\cG_{\pi}(\h\X)$-stable. Because $\sigma(G) \subset \G_\pi(\cR)$ by assumption, the image of $G$ in $\cG(\h\X)$ lands in $\cG_{\pi}(\h\X)$, so $\cY$ is $G$-stable. Therefore $\cO(\h{\Y})$ and the affinoid subdomain $\cY_{\rig}$ of $\bX$ are also $G$-stable. The $\cO(\cY)$-Lie lattice $\cT(\cY)$ in $\cT(\cY_{\rig})$ is $G$-stable by Lemma \ref{StabLopen}(a).

(b) The first inclusion follows directly from the assumption $\pi^n \in p^\epsilon \cR$. Now, $\xi^{-1}(\Y)$ is an open affine subscheme of $\G$ equipped with a $\B$-action from the right, and $\xi^{-1}(\Y)$ is non-canonically isomorphic to $\Y \times \B$ because $\Y \in \cS$. Therefore $\cO(\xi^{-1}(\Y))$ is an $\cO(\B)$-comodule-algebra, and the corresponding subring of invariants is precisely $\cO(\Y)$. These statements hold over a general base ring, and in particular, over $\cR / p^\epsilon \cR$. Now $G_n$ acts trivially on $\cO(\xi^{-1}(\Y)) \otimes_{\cR} \cR / p^\epsilon \cR$, and this action respects the $\cO(\B) \otimes_{\cR} \cR / p^\epsilon \cR$-comodule structure. It follows that $G_n$ also acts trivially on $\cO(\Y) \otimes_{\cR} \cR / p^\epsilon \cR$ as required.

(c) Let $g \in G_n$. Then $\log \rho^{\bG}(g) = \pi^n u$ for some right-invariant $\cR$-linear derivation $u :  \cA \to \cA$. Now $u = \gamma'(v)$ for some $v \in \fr{g}$ by Propositions \ref{FormalInvariantVectFlds} and \ref{LieTriangle}. So $\log \rho^{\bG}(g) \in \cA \cdot \gamma'(\pi^n \fr{g})$ and hence $g \in G_{\cA \cdot \gamma'(\pi^n \fr{g})}$ as required. Now by Lemma \ref{AlphaLogRho} (which is applicable in view of part (b)), and Lemma \ref{AlphaPsiVarphi} we have
\[\log \rho^{\cY_{\rig}}(g) = \alpha_{\cY_{\rig}}(\log \rho^{\bG}(g)) = \alpha_{\cY_{\rig}}(\pi^n\gamma'(v)) = \varphi_{\cY_{\rig}}(\pi^n v).\]
Therefore $\log \rho^{\cY_{\rig}}(g) \in \pi^n \cT(\cY)$ and $g \in G_{\pi^n \cT(\cY)}$.
\end{proof}
By Lemma \ref{CongSubsTriv}(c), $(\pi^n \cT(\cY), G_n)$ is an $\cO(\cY)$-trivialising pair in $\cO(\cY_{\rig})$, so we may form the crossed product $\widehat{U(\pi^n \cT(\cY) )_K} \rtimes_{G_n} G$ whenever $\Y \in \cS$.
\begin{cor}\label{QnDefnCor} For each integer $n$ such that $\pi^n \in p^\epsilon \cR$, there is a unique sheaf of $K$-Banach algebras
\[ \hdnK \rtimes_{G_n} G \]
on $\X$ whose value on $\Y \in \cS$ is $\widehat{U(\pi^n \cT(\cY) )_K} \rtimes_{G_n} G$, where $\cY = \h\Y$. It is a finitely generated free $\hdnK$-module of rank $[G : G_n]$.
\end{cor}
\begin{proof} It follows from Lemma \ref{CongSubsTriv} and Proposition \ref{hUGfunc} that this construction is functorial in $\Y \in \cS$ and therefore defines a presheaf $\hdnK \rtimes_{G_n} G$ of $K$-Banach algebras on $\cS$. Since $\widehat{U(\pi^n \cT(\cY) )_K} = \hdnK(\Y)$, it follows from Lemma \ref{RingSGN}(b) that $\hdnK \rtimes_{G_n} G$ is finitely generated and free of rank $[G:G_n]$ as a presheaf of $\hdnK$-modules. But $\hdnK$ is a sheaf on $\cS$, so $\hdnK \rtimes_{G_n} G$ is also a sheaf on $\cS$. Since $\cS$ is a base for the Zariski topology on $\X$, $\hdnK \rtimes_{G_n} G$ extends uniquely to a sheaf of $K$-Banach algebras on $\X$, which we will also denote by $\hdnK \rtimes_{G_n} G$.
\end{proof}
In order to aid legibility, we will also denote this sheaf on $\X$ by
\[\cQ_n := \hdnK \rtimes_{G_n} G.\]

\begin{thm}\label{XcohQnAff} $\X$ is coherently $\cQ_n$-affine.
\end{thm}
\begin{proof} It follows from Definition \ref{DomRegDefn} that $\lambda = 0$ is $\rho$-dominant and $\rho$-regular. Therefore $\X = \G / \B = \cB$ is coherently $\hdnK$-affine by Theorem \ref{hdnlKaffine}. Now $\cQ_n$ is a free $\hdnK$-module of finite rank by Corollary \ref{QnDefnCor}. It follows in a straightforward manner from \cite[Definition 5.1]{AW13} that $\X$ is also coherently $\cQ_n$-affine.
\end{proof}

We assume henceforth that $\pi^n \in p^\epsilon \cR$. It follows from Lemma \ref{CongSubsTriv} and Proposition \ref{StdPresWDGG} that the algebra $A = \w\cD(\bG,G)^{\bG}$ admits a presentation of the form
\[ A \cong \invlim A_n, \qmb{where} A_n := \hK{U(\pi^n\fr{g})} \rtimes_{G_n} G,\]
and by Proposition \ref{PhiFactors},  there is a $K$-algebra homomorphism $A_n \longrightarrow \cQ_n(\Y)$ for each $\Y \in \cS$. 

Given a coadmissible $A$-module $M$, we may on the one hand regard it as a coadmissible $A$-module via restriction, then form the finitely generated $A_n$-module $M_n := A_n \underset{A}{\otimes}{} M$ and then form the sheaf of $\cQ_n$-modules on $\X$
\[\cM_n := \cQ_n \otimes_{A_n}M_n = \cQ_n \underset{A}{\otimes}{} M.\]
On the other hand, we have the coadmissible $G$-equivariant $\cD$-module $\cM := \Loc^A_{\bX}(M)$ on $\bX$. These constructions are related in the following way.

\begin{lem}\label{PstarM} For each coadmissible $A$-module $M$ and each $\Y \in \cS$ there is an $A$-linear isomorphism
\[ \cM(\h\Y_{\rig}) \cong \invlim \cM_n(\Y)\]
which is functorial in $\Y$.
\end{lem}
\begin{proof} This is just a matter of decoding the definitions. Because $(\h\Y_{\rig}, G)$ is small, by Definition \ref{DefnOfLocM} and Corollary \ref{ResCor} we have
\[ \cM(\h\Y_{\rig}) = \cP^A_{\bX}(M)(\h\Y_{\rig}) \cong M(\h\Y_{\rig}, G) = \w\cD(\h\Y_{\rig}, G) \underset{A}{\w\otimes}{} M.\]
Now $\w\cD(\h\Y_{\rig}, G) = \invlim \cQ_n(\Y)$ and $A = \invlim A_n$ so 
\[ \w\cD(\h\Y_{\rig}, G) \hsp \underset{A}{\w\otimes}{} \hsp M = \invlim \cQ_n(\Y) \otimes_{A} M = \invlim \cM_n(\Y).\]
These isomorphisms are functorial in $\Y$.
\end{proof}

\begin{lem}\label{TwoSided}  The algebra $Z(\fr{g}_K)$ remains central in $A$, so $J \cdot A$ is a two-sided ideal of $A$ whenever $J$ is an ideal of $Z(\fr{g}_K)$. 
\end{lem}
\begin{proof} The $G$-action on $U(\fr{g}_K)$ relevant to Definition \ref{wDGGDefn} is given by the adjoint action of $\G(\cR)$ on $U(\fr{g}_K)$. Hence $G$ fixes $Z(\fr{g}_K)$ pointwise under this action, so $Z(\fr{g}_K)$ is central in the skew-group algebra $U(\fr{g}_K) \rtimes G$. The result follows because $U(\fr{g}_K) \rtimes G$ is dense in $A = \wDGG$ by construction.\end{proof}

\begin{thm}\label{BBRGamma} Let $M \in \cC_A$. Then $\Gamma(\bX, \Loc^A_{\bX}(M)) \cong M / \fr{m}_0 \cdot M$ as $A$-modules.
\end{thm}
\begin{proof} Choose an $\cS$-covering $\cU := \{\Y_1,\ldots,\Y_m\}$ of $\X$ and let $\bY_i := \h\Y_{i,\rig}$. Then $\h{\cU}_{\rig} := \{\bY_1,\ldots,\bY_m\}$ is an $\bX_w(\cT)$-covering of the rigid analytic flag variety $\bX$. Letting $\cM := \Loc^A_{\bX}(M)$, Lemma \ref{PstarM} then implies that
\[ \Gamma(\bX, \cM) = \check{H}^0(\h{\cU}_{\rig}, \cM) \cong \check{H}^0(\cU, \invlim \cM_n) \cong \invlim \check{H}^0(\cU, \cM_n) = \invlim \Gamma(\X, \cM_n).\]
Now, by imitating the proof of Lemma \ref{LocMLocN} we see that the $\cQ_n$-module $\cM_n$ is isomorphic to $\hdnK \underset{\hK{U(\pi^n\fr{g})}}{\otimes}{}  M_n$ as a $\hdnK$-module. 
Because the zero weight is $\rho$-dominant, we may apply Corollary \ref{BBLevelwise} and Theorem \ref{GlobSecDnlK}(a) to deduce that
\[\Gamma(\X, \cM_n) \cong  \hdnK(\X) \underset{\hK{U(\pi^n\fr{g})}}{\otimes}{} M_n \cong M_n / \fr{m}_0 M_n\]
as $\hdnK(\X)$-modules for all $n$, and is it straightforward to see that the composite isomorphism $\Gamma(\X,\cM_n) \cong M_n / \fr{m}_0 M_n$ is in fact $A_n$-linear. Since $\fr{m}_0 A$ is a two-sided ideal in $A$ by Lemma \ref{TwoSided}, $\fr{m}_0 M$ is an $A$-submodule of $M$. By choosing a finite generating set for the $Z(\fr{g}_K)$-ideal $\fr{m}_0$, we have an exact sequence $M^t \longrightarrow M \longrightarrow M / \fr{m}_0 M \to 0$ of $A$-modules, and because $A$ is Fr\'echet-Stein by Theorem \ref{wDGGFrSt}, it follows from \cite[Corollary 3.4(ii)]{ST} that $M / \fr{m}_0M$ is a coadmissible $A$-module. Hence
\[ M / \fr{m}_0 M \cong \invlim A_n \otimes_A (M / \fr{m}_0 M) \cong \invlim M_n / \fr{m}_0 M_n.\]
Putting everything together, we find an $A$-linear isomorphism
\[M / \fr{m}_0 M \cong  \invlim M_n / \fr{m}_0 M_n \cong  \invlim \Gamma(\X, \cM_n) \cong \Gamma(\bX, \cM). \qedhere\]
\end{proof}

We now start working towards proving Theorem \ref{BBEssSurj}. The proof is very similar to that of Proposition \ref{HisomAlpha} but we give the details for the convenience of the reader.

\begin{notn}\hspace{2em}\label{RigidScov}
\begin{itemize}
\item $\cU := \{\h\Y_{\rig} : \Y \in \cS\}$, a subset of $\bX_w(\cT)$.
\item $\cQ_\infty(\Y) := \w\cD(\h\Y_{\rig}, G)$ for each $\Y \in \cS$.
\item $\cM$ is a coadmissible $G$-equivariant $\cD$-module on $\bX$.
\item The integer $n$ satisfies $\pi^n \in p^\epsilon \cR$.
\end{itemize}
\end{notn}

Let $\bY \in \cU$. Then $\cM_{|\bY} \in \cC_{\bY/ G}$ for all $i$ by Proposition \ref{FrechRestr}, so by Theorem \ref{LocEquiv}, $\cM(\bY)$ is a coadmissible $\w\cD(\bY, G)$-module  and there is an isomorphism $\Loc^{\w\cD(\bY, G)}_{\bY}(\cM(\bY)) \stackrel{\cong}{\longrightarrow} \cM_{|\bY}$ of $G$-equivariant locally Fr\'echet $\cD$-modules on $\bY$.

\begin{lem} \label{MnQnX} 
There is a coherent sheaf $\cM_n$ of $\cQ_n$-modules on $\X$, whose values on $\Y \in \cS$ are given by
\[ \cM_n(\Y) = \cQ_n(\Y) \underset{\cQ_\infty(\Y)}{\otimes} \cM(\h\Y_{\rig}).\]
\end{lem}
\begin{proof} The formula on the right hand side of the displayed equation is evidently functorial in $\Y \in \cS$ and thus defines a presheaf $\cM_n$ of $\cQ_n$-modules on $\cS$. Let $\U,\Y \in \cS$ be fixed with $\U \subseteq \Y$. By the remarks made above, there is an isomorphism $\cM(\h\U_{\rig}) \cong \cQ_\infty(\U) \underset{\cQ_\infty(\Y)}{\w\otimes} \cM(\h\Y_{\rig})$. Therefore
\begin{equation} \label{MnCohPf} \begin{array}{lll} \cM_n(\U) &=&  \cQ_n(\U) \underset{\cQ_\infty(\U)}{\otimes} \cM(\h\U_{\rig})= \\ &\cong& \cQ_n(\U) \underset{\cQ_\infty(\U)}{\otimes} \left( \cQ_\infty(\U) \underset{\cQ_\infty(\Y)}{\w\otimes} \cM(\h\Y_{\rig}) \right)= \\ 
&=& \cQ_n(\U) \underset{\cQ_\infty(\Y)}{\otimes} \cM(\h\Y_{\rig})= \\ &=& \cQ_n(\U) \underset{\cQ_n(\Y)}{\otimes} \cM_n(\Y).\end{array}\end{equation}
Now $\h\U_{\rig}$ is a $G$-stable affinoid subdomain of $\bY := \h\Y_{\rig}$ by Lemma \ref{CongSubsTriv}. Because $\U$ is an affine Zariski open subscheme of $\Y$, it is a finite union of basic open subsets, and it follows directly from Definition \ref{LaccessDefn} that $\h\U_{\rig}$ is an $\cL := \pi^n \cT(\h\Y)$-accessible affinoid subdomain of $\h\Y_{\rig}$. Recall the sheaf $\cQ := \hK{\sU(\cL)} \rtimes_{G_n} G$ on $\bY_{\ac}(\cL, G)$ from Definition \ref{LocQFunc}. Now, formula (\ref{MnCohPf}) shows that the presheaf $\cN :=\Loc_{\cQ}( \cM_n(\Y) )$ on $\bY_{\ac}(\cL, G)$ is related to the restriction of $\cM_n$ to $\Y$ via
\[ \cN(\h\U_{\rig}) \cong \cM_n(\U) \qmb{for all} \U \in \cS \qmb{such that} \U \subseteq \Y.\]
It now follows from Corollary \ref{TateForCohsQmod} that $\cM_n$ is a sheaf on $\cS$. Because $\cS$ is a basis for $\X$, $\cM_n$ extends to a sheaf of $\cQ_n$-modules on $\X$, which is coherent in view of formula (\ref{MnCohPf}).
\end{proof}

\begin{lem}\label{NewTaus} There is a $\cQ_n$-linear isomorphism
\[ \tau_n : \cQ_n \underset{\cQ_n(\X)}{\otimes}{} \left( \cQ_n(\X) \underset{\cQ_{n+1}(\X)}{\otimes}{} \cM_{n+1}(\X) \right) \congs \cQ_n \underset{\cQ_n(\X)}{\otimes}{} \cM_n(\X).\]
\end{lem}
\begin{proof} We omit the details which are very similar to the proof of Lemma \ref{TausBetweencMns}.
\end{proof}

\begin{cor}\label{NewMinftyCoadm} The $\cQ_\infty(\X)$-module $M_\infty := \invlim \cM_n(\X)$ is coadmissible. 
\end{cor}
\begin{proof} By Lemma \ref{NewTaus} and Theorem \ref{XcohQnAff}, the maps $\tau_n(\X)$ induce $\cQ_n(\X)$-linear isomorphisms $\cQ_n(\X) \otimes_{\cQ_{n+1}(\X)} \cM_{n+1}(\X) \congs \cM_n(\X)$ for each $n$. Therefore $M_\infty$ is a coadmissible $\invlim \cQ_n(\X) = \cQ_\infty(\X)$-module.
\end{proof}

If $\Y \in \cS$ and $\bY := \h\Y_{\rig}$, then $\cM(\bY)$ is a coadmissible $\cQ_\infty(\Y) \cong \w\cD(\h\Y_{\rig}, G)$-module by Proposition \ref{Reconstruct}(a). Hence the canonical map
\[ \cM(\bY) \longrightarrow \invlim \cQ_n(\Y) \underset{\cQ_\infty(\Y)}{\otimes} \cM(\bY) \]
is an isomorphism, and we will identify $\cM(\bY)$ with $\invlim \cQ_n(\Y) \underset{\cQ_\infty(\Y)}{\otimes} \cM(\bY)$.

\begin{lem}\label{NewNuLemma} For each $\bY \in \cU$, there is a $\cQ_\infty(\X)$-linear map
\[ \nu_\bY : M_\infty \longrightarrow \cM(\bY)\]
such that $\nu_\bY(m)_{|\bY \cap \bY'} = \nu_{\bY'}(m)_{|\bY \cap \bY'}$ for all $m \in M_\infty$ and all $\bY' \in \cU$.
\end{lem}
\begin{proof} Let $m = (m_n)_n \in M_\infty$ where $m_n \in \cM_n(\X)$, and define $\nu_\bY$ by 
\[\nu_\bY(m) := ( (m_n)_{|\bY} )_n.\]
This is $\cQ_\infty(\X)$-linear because the restriction maps in $\cM_n$ are $\cM_n(\X)$-linear. Now
\[ \nu_\bY(m)_{|\bY \cap \bY',n} = (m_{n|\bY})_{|\bY\cap \bY'} = m_{n |\bY\cap \bY'} = (m_{n|\bY'})_{|\bY\cap \bY'} = \nu_{\bY'}(m)_{|\bY \cap \bY',n}\]
for all $n \geq 0$. Hence $\nu_\bY(m)_{|\bY \cap \bY'} = \nu_{\bY'}(m)_{|\bY \cap \bY'}$ for all $m \in M_\infty$.
\end{proof}

\begin{proof}[Proof of Theorem \ref{BBEssSurj}] Let $\cM$ be a coadmissible $G$-equivariant $\cD$-module on $\bX$. We have constructed a coadmissible $\cQ_\infty(\X)$-module $M_\infty$ above in Corollary \ref{NewMinftyCoadm}. We will next construct an isomorphism of ${G}$-equivariant $\cD$-modules on $\cU$
\[ \alpha : \cP^{A}_\bX(M_\infty)_{|\cU} \congs \cM_{|\cU}.\]
Let $\Y \in \cS$ so that $\bY := \h\Y_{\rig}  \in \cU$. Using Lemma \ref{NewNuLemma}, define
\[g_\bY : \w\cD(\bY,{G}) \hsp \underset{A}{\w\otimes} \hsp M_\infty \longrightarrow \cM(\bY)\]
by setting $g_\bY(s \w\otimes m) = s \cdot \nu_\bY(m)$. This is a $\w\cD(\bY,{G})$-linear map. The diagram
\[\xymatrix{
\cQ_\infty(\Y) \underset{\cQ_\infty(\X)}{\w\otimes} M_\infty \ar[r]^{g_\bY}\ar@{=}[d] &  \cM(\bY) \ar[dd]^{\cong} \\
\invlim \cQ_n(\Y) \underset{\cQ_n(\X)}{\otimes} \cM_n(\X) \ar[d]^{\cong} &  \\
\invlim \cM_n(\Y)  & \invlim \cQ_n(\Y) \underset{\cQ_\infty(\Y) }{\otimes} \cM(\bY)  \ar[l]^{\cong}
}\]
is commutative by construction. Now $\cM_n$ is a coherent $\cQ_n$-module by Lemma \ref{MnQnX}, so the bottom left vertical arrow is an isomorphism by Theorem \ref{XcohQnAff}. The  bottom horizontal arrow is an isomorphism by the definition of $\cM_n$ given in the proof of Lemma \ref{MnQnX}, and the vertical arrow on the right is an isomorphism by the remarks made just before Lemma \ref{NewNuLemma}. It follows that $g_\bY$ is an isomorphism. 

Now consider the following diagram:
\[ \xymatrix{ 
\cP_\bX^{A}(M_\infty)_{|\bY_w} \ar@{.>}[rrrr]^{\alpha_\bY} \ar[d]_{\cong} &&&& \cM_{|\bY_w} \\
\cP^{\w\cD(\bY,{G})}_\bY(\w\cD(\bY,{G}) \underset{A}{\w\otimes} M_\infty) \ar[rrrr]_{\cP^{\w\cD(\bY,{G})}_\bY(g_\bY)} &&&& \cP_\bY^{\w\cD(\bY,{G})}(\cM(\bY)). \ar[u]_{\theta_\bY}^{\cong}
}\]
Here the left vertical arrow is the isomorphism given by Proposition \ref{LocTrans}, and the right vertical arrow $\theta_\bY$ is the isomorphism given by Proposition \ref{Reconstruct}(c). Because $g_\bY$ is an isomorphism, the bottom arrow is an isomorphism because $\cP_\bY^{\w\cD(\bY,{G})}$ is a functor by Proposition \ref{LocFunctor}. Thus we obtain the ${G}$-equivariant $\cD$-linear isomorphism
\[ \alpha_\bY : \cP^{A}_\bX(M_\infty)_{|\bY_w} \congs \cM_{|\bY_w}\]
which makes the diagram commute. 

By Proposition \ref{Reconstruct}(a), $\cM(\bU)$ is naturally a coadmissible $\w\cD(\bU,G_{\bU})$-module for any $\bU \in \bY_w$. Identifying $\cP^{A}_\bX(M_\infty)(\bU)$ with $\w\cD(\bU,G_{\bU}) \underset{\w\cD(\bX,G_{\bU})}{\w\otimes} M_\infty$ using Corollary \ref{ResCor}, it is straightforward to verify that the map
\[ \alpha_\bY(\bU) : \cP^{A}_\bX(M_\infty)(\bU) \longrightarrow \cM(\bU)\]
is given by
\begin{equation}\label{BBalphaDefn} \alpha_\bY(\bU)(s \hsp \w\otimes \hsp m) = s \cdot (\nu_\bY(m)_{|\bU}) \qmb{for all}   s \in \w\cD(\bU,G_{\bU}) \qmb{and} m \in M_\infty.\end{equation}
Using Lemma \ref{NewNuLemma}, we see that the local isomorphisms $\alpha_\bY$ satisfy
\[ (\alpha_\bY)_{|\bY \cap \bY'} = (\alpha_{\bY'})_{|\bY \cap \bY'} \qmb{for any} \bY, \bY' \in \cU.\]
Since $\Loc^{A}_\bX(M_\infty)$ and $\cM$ are sheaves on $\bX$ and $\cU$ contains an admissible covering of $\bX$, the $\alpha_\bY$'s patch together to the required isomorphism $\alpha : \Loc^A_{\bX}(M_\infty) \to \cM$.
\end{proof}

\section{Extensions to general \ts{p}-adic Lie groups}

\subsection{The associative algebra \ts{F(G)}}\label{FGformalism}
Let $G$ be a group, and let $k$ be a commutative ring. We fix a set $\cC$ of subgroups of $G$ which is closed under finite intersections and conjugation. It may help to keep in mind the basic example where $G$ is a $p$-adic Lie group and $\cC$ is the set of its compact open subgroups.

We will view $\cC$ as a partially ordered set, ordered by inclusion, and hence as a category whose $\Hom$-sets have at most one member. Let $\kAlg$ denote the category of associative unital $k$-algebras, and let 
\[ F : \cC \to \kAlg\]
be a covariant functor. We will write $F(a)$ for the image of $a \in F(J)$ in $F(H)$ under the canonical map $F(J) \to F(H)$ whenever $J \leq H$ are members of $\cC$. A basic example of such a functor is afforded by the \emph{group ring functor}
\[ k[-] : \cC \to \kAlg\]
which sends $H \in \cC$ to the subalgebra $k[H]$ of $k[G]$. Suppose further that we are given a morphism
\[ \iota : k[-] \to F\]
of functors $\cC \to \kAlg$ as above.  Whenever $J \leq H$ are members of $\cC$, the natural square in $\kAlg$ 
\[ \xymatrix{ k[J] \ar[d]_{\iota_J}\ar[r]& k[H] \ar[d]^{\iota_H} \\ F(J) \ar[r] & F(H) }\]
turns $F(H)$ into a $F(J)$-$k[H]$-bimodule, and there is a natural map
\[ \begin{array}{ccccc} t_{J,H} &:& F(J) \otimes_{k[J]} k[H] &\longrightarrow& F(H) \\ & & a \otimes b &\mapsto & F(a) \iota_H(b) \end{array}\]
of $F(J)$-$k[H]$-bimodules. There is also a unique $F(J)$-$k[G]$-bimodule map
\[ \begin{array}{ccccc} \alpha_{J,H} &:& F(J) \otimes_{k[J]} k[G] &\longrightarrow &F(H) \otimes_{k[H]} k[G] \\ & & a \otimes b &\mapsto & F(a) \otimes b \end{array}\]
which makes the following diagram commute:
\[\xymatrix{ F(J) \otimes_{k[J]} k[G] \ar[rr]^{\alpha_{J,H}}\ar[dr]_{\cong} & & F(H) \otimes_{k[H]} k[G]. \\ & \left( F(J) \otimes_{k[J]} k[H]\right) \otimes_{k[H]} k[G] \ar[ur]_{t_{J,H} \otimes 1}  }\]
The functoriality of $F$ makes it clear that 
\[\alpha_{J,H} \circ \alpha_{L,J} = \alpha_{L,H}\]
whenever $L \leq J \leq H$ are members of $\cC$. In this way, $H \mapsto F(H) \otimes_{k[H]} k[G]$ becomes a covariant functor from $\cC$ to the category of right $k[G]$-modules.

\begin{defn} \hspace{2em}
\be \item Let $V$ denote the right $k[G]$-module 
\[V := \lim\limits_{\stackrel{\longrightarrow}{H \in \cC}} F(H) \otimes_{k[H]} k[G].\]

\item For every $H \in \cC$, let $\tau_H : F(H) \otimes_{k[H]} k[G] \longrightarrow V$ denote the structure morphism of the direct limit.
\item We will work inside the $k$-algebra $\cE := \End(V_{|k[G]})$.
\ee\end{defn}

\begin{defn}\label{Sheafy} We say that $(F, \iota)$ is \emph{sheafy} if $t_{J,H}$ is an isomorphism whenever $J \leq H$ are objects of $\cC$.
\end{defn}
\begin{rmk}If $(F, \iota)$ is sheafy, the maps $\alpha_{J,H}$ are all isomorphisms, and it follows that the structure morphisms $\tau_H : F(H) \otimes_{k[H]} k[G] \longrightarrow V$ are also isomorphisms of right $k[G]$-modules, for any $H \in \cC$. 
\end{rmk}
Suppose that $(F,\iota)$ is sheafy. By transport of structure, $V$ becomes naturally an $F(H)$-$k[G]$-bimodule for every $H \in \cC$, if we define
\[ a \cdot \tau_H(b\otimes c) := \tau_H(ab \otimes c) \qmb{for all} a,b \in F(H)\qmb{and} c\in k[G].\]
In this way, we obtain a family of $k$-algebra homomorphisms 
\[\rho_H : F(H) \longrightarrow \cE\]
given by $\rho_H(a)(v) = a \cdot v$ for all $a \in F(H)$ and $v \in V$.  These homomorphisms are naturally compatible in the following sense.

\begin{lem}\label{RhoJRhoHCompat} Suppose that $(F,\iota)$ is sheafy and let $J \leq H$ be in $\cC$. Then 
\[ \rho_J(a) = \rho_H(F(a)) \qmb{for all} a \in F(J).\]
\end{lem}
\begin{proof} Let $b \in F(J)$ and $c \in k[G]$ so that $\tau_J(b \otimes c) = \tau_H( \alpha_{J,H}( b \otimes c) ) = \tau_H( F(b) \otimes c).$ Then $a \cdot \tau_J(b \otimes c)= \tau_J(ab \otimes c) = \tau_H(F(ab) \otimes c) =  \tau_H(F(a)F(b) \otimes c) = F(a) \cdot \tau_H(F(b) \otimes c) = F(a) \cdot \tau_J(b \otimes c)$ so that $\rho_J(a)(v) = a \cdot v = F(a) \cdot v = \rho_H(F(a))(v)$ for all $v = \tau_J(b \otimes c) \in V$. 
\end{proof}

\begin{defn}\label{GequivFiota} We say that $(F,\iota)$ is \emph{$G$-equivariant} if for all $g \in G$ and $H \in \cC$, there is given a $k$-algebra homomorphism 
\[\varphi_g(H) : F(H) \to F(gHg^{-1})\]
such that
\be\item $\varphi_g(hHh^{-1}) \circ \varphi_h(H) = \varphi_{gh}(H)$ for all $g,h \in G$ and $H \in \cC$, 
\item the diagram
\[\xymatrix{ F(J) \ar[rr]^{\varphi_g(J)}\ar[d] && F(gJg^{-1}) \ar[d] \\ F(H) \ar[rr]_{\varphi_g(H)} && F(gHg^{-1}) }\]
commutes for all $J \leq H$ in $\cC$ and all $g \in G$, and
\item $\iota_{gJg^{-1}}( gag^{-1} ) = \varphi_g(J) (\iota_J(a))$ for all $g\in G$, $J \in \cC$ and $a \in k[J]$.
\ee\end{defn}

We will abuse notation and write ${}^ga$ to mean $\varphi_g(H)(a)$ for any $g \in G$, $H \in \cC$ and $a \in F(H)$. Thus, $(F,\iota)$ is $G$ equivariant if and only if
\[ {}^g ({}^h a) = {}^{gh}a, \quad F({}^ga) = {}^gF(a) \qmb{and} \iota_{{}^gJ}({}^ga) = {}^g\iota_J(a)\]
for all $g, h \in G$, $J \leq H$ in $\cC$ and $a \in F(J)$. We will also abbreviate $gHg^{-1}$ to ${}^gH$.

\begin{rmk} We could enhance the category $\cC$ by adding an arrow $H \longrightarrow gHg^{-1}$ for each $H \in \cC$ and $g\in G$. Formally, we could define 
\[\Hom_{\cC}(J,H) := \{ g \in G : g^{-1} J g \leq H\}\]
and define composition using the multiplication law in $G$. With this terminology, we see that $(F,\iota)$ is $G$-equivariant if and only if $\iota : k[-] \to F$ is a morphism of functors from $\cC$ to $\kAlg$. However, we will not need to use this formalism.
\end{rmk}

\begin{lem} Suppose that $(F,\iota)$ is sheafy and $G$-equivariant.  Then there is a well-defined left $G$-action on $V$ given by
\[g \cdot \tau_H(a \otimes b) = \tau_{{}^gH}\left({}^ga \otimes gb\right)\]
for all $g \in G$, $H \in \cC$, $a \in F(H)$ and $b \in k[G]$. This action commutes with the right $G$-action on $V$.
\end{lem}
\begin{proof} We first show that the map $Q : (a,b) \mapsto \tau_{{}^gH}({}^ga\otimes gb)$ is $k[H]$-balanced. But if $x \in k[H]$ then ${}^gx \in k[{}^gH]$, so using Definition \ref{GequivFiota}(c) we see that
\[{}^g(a\iota_H(x)) \otimes gb = {}^ga\hsp {}^g\iota_H(x) \otimes gb = {}^ga \hsp \iota_{{}^gH}({}^gx) \otimes gb = {}^ga \otimes {}^gx \hsp gb = {}^ga \otimes gxb.\]
Hence $Q(ax,b) = Q(a,xb)$ as claimed. Next, we must show that 
\[g \cdot \tau_H(v) = g \cdot \tau_J(w) \qmb{whenever} \tau_H(v) = \tau_J(w).\]
Because $\cC$ is stable under finite intersections, we may assume that $J \leq H$. Since $\tau_H$ and $\tau_J$ are isomorphisms, we see that $\alpha_{J,H}(w) = v$. Without loss of generality, we can assume that $w = a \otimes b$ for some $a \in F(J)$ and $b \in k[G]$. We must then have $v = \alpha_{J,H}(w) = F(a) \otimes b$. Now,
\[\begin{array}{lll} g \cdot \tau_H(v) &=& g \cdot \tau_H(F(a) \otimes b) = \\
&=& \tau_{{}^gH}({}^gF(a) \otimes gb) = \\
&=& \tau_{{}^gH}(F({}^ga) \otimes gb) = \\
&=& \tau_{{}^gH}(\alpha_{J,H}({}^ga \otimes gb)) = \\
&=& \tau_{{}^gJ}({}^ga \otimes gb)= g \cdot \tau_J(w),\end{array}\]
where we have used Definition \ref{GequivFiota}(b) on the third line. In view of Definition \ref{GequivFiota}(a), we now have a well-defined left $G$-action on $V$, and the last assertion is clear.
\end{proof}

\begin{cor}\label{MapFromGtoE} If $(F,\iota)$ is sheafy and $G$-equivariant, then there is a well-defined $k$-algebra homomorphism
\[ \eta : k[G] \longrightarrow \cE\]
given by $\eta(g)(v) = g\cdot v$ for all $g \in G$ and $v\in V$.
\end{cor}

\begin{defn} We define $F(G)$ to be the $k$-subalgebra of $\cE$ generated by the image of $\eta$, and the images of $\rho_H$ as $H$ runs over all members of $\cC$.
\end{defn}

The images of $\eta$ and the $\rho_H$ inside $\cE$ are related in the following manner.

\begin{lem}\label{GconjImRho} Suppose that $(F, \iota)$ is sheafy and $G$-equivariant. Then
\[ \eta(g) \rho_H(a) \eta(g)^{-1} = \rho_{{}^gH}( {}^ga) \qmb{for all}  H \in \cC, a \in F(H)
\qmb{and} g \in G.\]
\end{lem}
\begin{proof}  Let $b \in F(H)$ and $c \in k[G]$. Then 
\[\begin{array}{lllll}
(\eta(g) \circ \rho_H(a) \circ \tau_H)(b \otimes c) &=& g \cdot \tau_H(ab \otimes c) &=& \tau_{{}^gH}({}^g(ab) \otimes gc) = \\
&=& \tau_{{}^gH}({}^ga {}^g b \otimes gc) &=& (\rho_{{}^gH}({}^ga) \circ \eta(g) \circ \tau_H )(b \otimes c). \end{array}\]
Hence $\eta(g) \circ \rho_H(a) \circ \tau_H = \rho_{{}^gH}({}^ga)\circ \eta(g) \circ \tau_H$ and the result follows because $\tau_H$ is an isomorphism.\end{proof}

We will need the following stronger version of Definition \ref{GequivFiota}(c).
\begin{defn}\label{GoodFIota} $(F, \iota)$ is \emph{good} if $(F,\iota)$ is sheafy, $G$-equivariant, and 
\[{}^g a  = \iota_H(g) \hsp a \hsp \iota_H(g)^{-1}\]
whenever $H \in \cC, g \in H$ and $a \in F(H)$.
\end{defn}

\begin{lem}\label{SigmaRhoIota} Suppose that $(F,\iota)$ is good. Then 
\[ \eta_{|k[H]} = \rho_H \circ \iota_H\]
for all $H \in \cC$.
\end{lem}
\begin{proof} Let $g \in H$, $a \in F(H)$ and $b \in k[G]$. Then
\[\begin{array}{lllll} g \cdot \tau_H(a \otimes b) &=& \tau_H({}^g a \otimes gb) &=& \tau_H\left( \iota_H(g) a \iota_H(g)^{-1} \otimes gb\right) = \\
&=&\tau_H(\iota_H(g) a \otimes g^{-1} gb) &=& \rho_H(\iota_H(g))( \tau_H(a \otimes b))\end{array}\]
so $\eta(g) \circ \tau_H = \rho_H(\iota_H(g)) \circ \tau_H$. The result follows because $\tau_H$ is an isomorphism.
\end{proof}

\begin{cor} Suppose that $(F,\iota)$ is good. Then for every $H \in \cC$ there is a well-defined morphism of $F(H)$-$k[G]$-bimodules
\[ \rho_H \otimes \eta : F(H) \otimes_{k[H]} k[G] \longrightarrow \cE.\]
\end{cor}
\begin{proof} This follows immediately from Lemma \ref{SigmaRhoIota}.\end{proof}

\begin{defn} Let $\cA_H$ denote the image of $\rho_H \otimes \eta$ inside $\cE$.
\end{defn}

\begin{thm}\label{MainFG} Suppose that $(F,\iota)$ is good, and let $H \in \cC$.
\be \item $\cA_H = \cA_J$ for all $J \in \cC$.
\item $\cA_H$ is an associative $k$-subalgebra of $\cE$.
\item The map $\rho_H \otimes \eta : F(H) \otimes_{k[H]} k[G] \longrightarrow \cA_H$ is bijective.
\ee\end{thm}
\begin{proof} (a) Since $\cC$ is closed under finite intersections, we may assume that $J \leq H$. But now the diagram
\[ \xymatrix{ F(J) \otimes_{k[J]} k[G] \ar[rr]^{\alpha_{J,H}} \ar[dr]_{\rho_J \otimes \eta} & & F(H)\otimes_{k[H]} k[G] \ar[dl]^{\rho_H \otimes \eta} \\ & \cE & }\]
is commutative, because $\rho_J(a) = \rho_H(F(a))$ by Lemma \ref{RhoJRhoHCompat}. Thus $\cA_J = \cA_H$ because $\alpha_{J,H}$ is an isomorphism.
 
(b) Since $\rho_H$ and $\eta$ are ring homomorphisms, it is enough to show that 
\[ \eta(g) \rho_H(a) \in \cA_H\]
for any $g \in G$ and $a\in F(H)$.  However, Lemma \ref{GconjImRho} implies that
\[\eta(g) \rho_H(a) = \rho_{{}^gH}({}^ga)\eta(g) \]
and this lies in $\cA_{{}^gH}$. Now use part (a).

(c) The map $\rho_H \otimes \eta$ is surjective by the definition of $\cA_H$, so suppose that $\xi = \sum_{i=1}^n a_i \otimes g_i$ lies in the kernel, where we may assume that $g_i \in G$ for all $i$. Then on the one hand,
\[ (\rho_H \otimes \eta)(\xi)(\tau_H(1 \otimes 1)) = 0,\]
and on the other hand,
\[ \begin{array}{lll} (\rho_H \otimes \eta)(\xi)(\tau_H(1 \otimes 1)) &=& \sum_{i=1}^n \rho_H(a_i) \eta(g_i)( \tau_H(1 \otimes 1)) = \\
&=& \sum_{i=1}^n \rho_H(a_i) \tau_{{}^{g_i}H}(1 \otimes g_i) = \\
&=& \sum_{i=1}^n \rho_H(a_i) \tau_H(1 \otimes g_i) = \\
&=& \sum_{i=1}^n \tau_H(a_i \otimes g_i) = \tau_H(\xi).
\end{array}\]
Since $\tau_H$ is an isomorphism, we conclude that $\xi = 0$ as required.
\end{proof}

\begin{cor}\label{FGproduct} Suppose that $(F,\iota)$ is good, and fix $H \in \cC$. 
\be \item There is an associative $k$-bilinear product
$\star$ on $F(H)\otimes_{k[H]}k[G]$, uniquely determined by the rule that
\[ (a \otimes g) \star (b \otimes h) = \sum_i a \hsp {}^gb'_i \otimes gh'_i\]
where $a,b \in F(H)$, $g,h \in G$ and
\[ \tau_{H^g}^{-1} \tau_H(b \otimes h) = \sum_i b'_i \otimes h'_i\]
for some $b'_i \in F(H^g)$ and $h'_i \in G$.
\item The map $\rho_H \otimes \eta : F(H) \otimes_{k[H]} k[G] \longrightarrow F(G)$ is a $k$-algebra isomorphism.
\ee \end{cor}
\begin{proof} (a) We can transport the $k$-algebra structure on $\cA_H$ to $F(H) \otimes_{k[H]}k[G]$ using the bijection $\rho_H \otimes \eta$, in view of Theorem \ref{MainFG}. 

Let $x = a \otimes g, y = b \otimes h$ and $y' = \sum_i b_i' \otimes h_i' = \tau^{-1}_{H^g}\tau_H(y)$ as above. Let $J := H \cap H^g$, and let $y'' \in F(J) \otimes_{k[J]}k[G]$ be such that 
\[ y = \alpha_{J,H}(y'') \qmb{and} y' = \alpha_{J,H^g}(y'').\]
We saw in the proof of Theorem \ref{MainFG}(a) that
\[ (\rho_H \otimes \eta) \circ \alpha_{J,H} = \rho_J \otimes \eta.\]
Hence
\[\begin{array}{lll} (\rho_H \otimes \eta)(y) &=& (\rho_H \otimes \eta)(y) = (\rho_H \otimes \eta)(\alpha_{J,H}(y'')) = \\
&=& (\rho_J \otimes \eta)(y'') = (\rho_{H^g} \otimes \eta)(y'). \end{array}\]
Hence, using Lemma \ref{GconjImRho} we see that
\[ \begin{array}{lll} (\rho_H \otimes\eta)(x) \circ(\rho_H \otimes\eta)(y) &=& \sum_i \rho_H(a)\eta(g)\rho_{H^g}(b'_i)\eta(h'_i) = \\
&=&\sum_i \rho_H(a) \rho_H({}^gb'_i) \eta(g)\eta(h'_i) = \\
&=&\sum_i \rho_H(a \hsp {}^gb'_i)\eta(gh'_i) = \\
&=&(\rho_H \otimes\eta)\left(\sum_i a \hsp {}^gb'_i \otimes gh'_i\right) = (\rho_H \otimes\eta)(x \star y).\end{array}\]
We are done because $\rho_H \otimes \eta$ is a bijection and $\cA_H$ is an associative $k$-algebra.

(b) It remains to observe that $F(G) = \cA_H$ for any $H \in \cC$.\end{proof}

This construction was known to Emerton. A concise explanation of the particular example where $F(H)$ is the Iwasawa algebra $\Lambda(H)$ can be found at \cite[p. 6]{Kohlhaase2017}.

\begin{ex}\label{DNKgood} Let $L$ be a finite extension of $\Qp$, let $G$ be a locally $L$-analytic group and let $\cC$ be the set of compact open subgroups of $G$. For each $N \in \cC$, let $\iota_N : K[N] \hookrightarrow D(N,K)$ be the inclusion of the group algebra of $N$ into the algebra $D(N,K)$ of locally $L$-analytic distributions on $N$. Then $(D(-,K), \iota)$ is good and the algebra $D(G,K)$ given by Corollary \ref{FGproduct} agrees the algebra of locally $L$-analytic distributions on $G$.
\end{ex}

\subsection{The algebra \ts{\wUg{G}}}
We will assume from now on that $R = \cR$ is a complete valuation ring of height one and mixed characteristic $(0,p)$, and work under the following

\begin{setup}\hspace{2em} \label{EqUgCapNotn}
\begin{itemize}
\item $\G$ is an affine algebraic group of finite type over $K$,
\item $\fr{g} := \Lie(\G)$ is its Lie algebra,
\item $G$ is a $p$-adic Lie group, 
\item $\sigma : G \to \G(K)$ is a continuous group homomorphism.
\end{itemize}
\end{setup}
Here we equip $\G(K)$ with the topology where a fundamental system of neighbourhoods of the identity element is given by the subsets $\bY(K)$ of $\G(K)$, as $\bY$ ranges over the affinoid subdomains of the rigid analytification $\bG := \G^{\rig}$ of $\G$ containing the identity. Recall the adjoint action $\Ad$ of $\G(K)$ on $\fr{g}$ by $K$-Lie algebra automorphisms. 

\begin{defn}\label{LieLatticeForAlgGps} \hspace{2em}
\be\item A \emph{Lie lattice} in $\fr{g}$ is a finitely generated $\cR$-submodule $\cL$ of $\fr{g}$ which spans $\fr{g}$ as a $K$-vector space and which is closed under the Lie bracket.
\item Let $H$ be an open subgroup of $G$. We say that the Lie lattice $\cL$ is \emph{$H$-stable} if it is preserved by $\Ad(\sigma(g))$ for all $g \in H$.
\ee\end{defn}

\begin{lem}\label{LieLatticesExist} Let $H$ be a compact open subgroup of $G$. Then $\fr{g}$ has at least one $H$-stable Lie lattice.
\end{lem}
\begin{proof} Fix a basis $\{x_1,\ldots,x_d\}$ for the $K$-vector space $\fr{g}$, and let $\cL$ be the $\cR$-submodule spanned by this basis. By replacing each $x_i$ by a sufficiently large $\pi$-power multiple, we may assume that $\cL$ is a Lie lattice. We can find an affinoid subgroup $\bH$ of $\GL_{d,K}^{\rig}$ such that $\bH(K) = \GL_d(\cR)$ is the stabiliser of $\cL$ in $\GL_d(K)$. The analytification $\Ad^{\rig}$ of the adjoint representation $\Ad : \G \to \GL_d$ is continuous, so we can find an affinoid subdomain $\bY$ of $\bG$ containing the identity such that $\Ad^{\rig}(\bY) \subseteq \bH$. Therefore the subset $\sigma^{-1} \bY(K)$ of $G$ is open, and as $H$ is an open profinite subgroup of $G$ we can find an open subgroup $J$ of $H$ such that $\sigma(J) \subset \bY(K)$. This $J$ then stabilises the Lie lattice $\cL$ under $\Ad \circ \sigma$, and $J$ has finite index in $H$ because $H$ is compact. So the $H$-orbit of $\cL$ is finite, and the intersection $\bigcap_{h \in H} \Ad(\sigma(h))(\cL)$ of members of this orbit is an $H$-stable Lie lattice in $\fr{g}$.
\end{proof}

Now let $H$ be an open subgroup of $G$ and let $\cL$ be an $H$-stable Lie lattice in $\fr{g}$. The Campbell-Hausdorff-Baker formula converges on the $\cR$-Lie algebra $p^\epsilon \cL$ and defines a group operation $\ast$ on this set. The adjoint representation $\ad : \cL \to \Der_{\cR}(\cL)$ exponentiates to a group homomorphism 
\[ \exp \circ \ad : (p^\epsilon \cL, \ast) \to \Aut_{\cR-\Lie}(\cL).\]
\begin{defn}\label{HJAdj} Let $H$ be an open subgroup of $G$ and let $\cL$ be an $H$-stable Lie lattice in $\fr{g}$. We define $H_{\cL}$ to be the pullback of the two maps 
\[H \stackrel{\Ad \circ \sigma}{\longrightarrow} \Aut_{\cR-\Lie}(\cL) \stackrel{\exp \circ \ad}{\longleftarrow} (p^\epsilon \cL, \ast).\] 
\end{defn}
Thus we have the commutative diagram of groups and group homomorphisms
\begin{equation}\label{DeltaLdiag} \xymatrix{ H_{\cL} \ar[rr]^{\delta_{\cL}}\ar[d] && (p^\epsilon \cL, \ast) \ar[d]^{\exp\circ \ad} \\ H \ar[rr]_{\Ad \circ\sigma} && \Aut_{\cR-\Lie}(\cL). }\end{equation}
Letting $\iota : \cL \to U(\cL)$ denote the canonical inclusion of $\cL$ into its enveloping algebra, the group $(p^\epsilon \cL ,\ast)$ is also isomorphic to the subgroup $\exp( p^\epsilon \iota(\cL) )$ of the group of units of the $\pi$-adic completion $\h{U(\cL)}$ of $U(\cL)$.  

For convenience and simplicity, we will work under the following

\begin{setup} \label{ZgZero} The centre $Z(\fr{g}) = \ker \ad$ of $\fr{g}$ is zero.
\end{setup}

Under this hypothesis, the map $\exp \circ \ad : p^\epsilon \cL \to \Aut_{\cR-\Lie}(\cL)$ is injective, and we see that $H_{\cL}$ is a \emph{subgroup} of $H$.

\begin{prop}\label{HJgroup} Suppose $Z(\fr{g})$ is zero, let $H$ be an open subgroup of $G$ and let $\cL$ be an $H$-stable Lie lattice in $\fr{g}$. Then 
\be \item $H_{\cL}$ is a normal subgroup of $H$,
\item the map $\exp \circ \iota \circ \delta_{\cL} : H_\cL \to \hK{U(\cL)}^\times$ is an $H$-equivariant trivialisation of the $H_{\cL}$-action on $\hK{U(\cL)}$.
\ee\end{prop}
\begin{proof} (a) Let $h \in H_{\cL}$ and let $g \in H$. Letting $u := \delta_{\cL}(h) \in p^\epsilon \cL$, we have $\Ad(\sigma(h)) = \exp(\ad(u))$. Now the adjoint representation of $\G$ preserves the Lie bracket on $\fr{g}$:
\[ \Ad(g) ([u,v]) = [\Ad(g)(u), \Ad(g)(v)] \qmb{for all} g\in \G(K), u,v \in \fr{g}\] 
and it follows that $\Ad(g) \circ \ad(u) \circ \Ad(g)^{-1} = \ad( \Ad(g)(u) )$ for all $g\in \G(K)$ and $u \in \fr{g}$. Therefore, if $v := \Ad(\sigma(g))(u) = g \cdot u$ then
\[ \begin{array}{lll} \Ad\sigma ( g h g^{-1} ) &=& \Ad(\sigma(g)) \circ \exp(\ad(u)) \circ \Ad(\sigma(g))^{-1} = \\
&=& \exp( \Ad(\sigma(g)) \circ \ad(u) \circ \Ad(\sigma(g))^{-1} ) = \exp(v).\end{array}\]
Because $u \in p^\epsilon \cL$ and $\cL$ is $H$-stable, we see that $v \in p^\epsilon \cL$. So $g h g^{-1} \in H_{\cL}$ and $\delta_{\cL}(ghg^{-1}) = g \cdot \delta_{\cL}(h)$.

(b) Let $h \in H_{\cL}$ and let $u = \delta_{\cL}(h) \in p^\epsilon \cL$, so that $\Ad(\sigma(h)) = \exp(\ad(u))$. Let $v \in \cL$. Then $e^{\iota u}\cdot (\iota v) \cdot e^{\iota u} = e^{\ad \iota u}(\iota v) = \iota (e^{\ad(u)}(v))$ by applying \cite[Exercise 6.12]{DDMS} to the $K$-Banach algebra $\hK{U(\cL})$, and using the fact that $\iota$ is a $\pi$-adically continuous $\cR$-Lie algebra homomorphism. Because $e^{\ad(u)}(v) = \Ad(\sigma(h))(v) = h\cdot v$, we see that $e^{\iota u} \cdot (\iota v) \cdot e^{\iota u} = \iota( h\cdot v) =h \cdot \iota(v)$ for all $v \in \cL$. Because $\cL$ generates $\hK{U(\cL)}$ as a topological $K$-algebra, we deduce that the $h$-action on $\hK{U(\cL)}$ is given by conjugation by $\exp(\iota(\delta_{\cL}(h)))$.  Thus $\exp \circ \iota \circ \delta_{\cL} : H_{\cL} \to \hK{U(\cL)}^\times$ is a trivialisation of the $H_{\cL}$-action on $\hK{U(\cL)}$. It is $H$-equivariant because the map $\delta_{\cL}$ satisfies $\delta_{\cL}(ghg^{-1}) = g \cdot \delta_{\cL}(h)$ as we saw above.
\end{proof}

We can now mimic Definition \ref{StarDefn} and make the following
\begin{defn}\label{wUgHDefn}  Suppose that $Z(\fr{g}) = 0$ and that $H$ is a compact open subgroup of $G$.
\be \item Let $\cJ(H)$ denote the set of all pairs $(\cL, N)$, where $\cL$ is an $H$-stable Lie lattice in $\fr{g}$ and $N$ is an open subgroup of $H_{\cL}$ which is normal in $H$. 
\item For each $(\cL,N) \in \cJ(H)$, the restriction to $N$ of the map $\exp \circ \iota \circ \delta_{\cL}$ from Proposition \ref{HJgroup}(b) is an $H$-equivariant trivialisation of the $N$-action on $\hK{U(\cL)}$, and we use this trivialisation to define the algebra
\[\hK{U(\cL)} \rtimes_N H.\]
\item  $\wUg{H} := \underset{(\cL,H)\in\cJ(H)}{\invlim\limits}{} \hK{U(\cL)} \rtimes_N H$.

\ee\end{defn}

\begin{rmk} We will see shortly that $\wUg{H}$ is closely related to the algebra $\w\cD(\bG, H)^{\bG}$ which appeared in $\S \ref{wDGGsect}$. However it is important to note that $\wUg{H}$ only depends on the $K$-group $\G$ and not on any particular affinoid subdomain of $\G^{\rig}$ such as $\bG$.
\end{rmk}

\begin{thm}\label{UgHFrSt} $\wUg{H}$ is Fr\'echet-Stein for every compact subgroup $H$ of $G$.
\end{thm}
\begin{proof} $\fr{g}$ admits at least one $H$-stable Lie lattice, by Proposition \ref{LieLatticesExist}. Now the proof of Theorem \ref{wDGGFrSt} works with obvious modifications.
\end{proof}

We will now use the formalism of $\S \ref{FGformalism}$ to glue the algebras $\wUg{H}$ together, as $H$ varies over \emph{all} compact open subgroups of $G$.  

\begin{lem}\label{CofinalTwo} Let $J \leq H$ be compact open subgroups of $G$. Then $\cJ(H) \cap \cJ(J)$ is cofinal in $\cJ(H)$ and in $\cJ(J)$.
\end{lem}
\begin{proof} It follows from Definition \ref{HJAdj} that $J_{\cL} = H_{\cL} \cap J$ for any $\cR$-Lie lattice $\cL$ in $\fr{g}$. Now the proof of Proposition \ref{IHIGcofinal} works with obvious modifications.
\end{proof}

\begin{prop}\label{wUgGood}  Let $\cC$ denote the set of compact open subgroups of $G$.
\be \item $H \mapsto \wUg{H}$ is a functor from $\cC$ to $K$-algebras.
\item There is a natural transformation $\iota : K[-] \longrightarrow \wUg{-}$.
\item The pair $(\wUg{-}, \iota)$ is good in the sense of Definition \ref{GoodFIota}.
\ee\end{prop}
\begin{proof} Only the third statement requires proof. Fix compact open subgroups $J \leq H$ of $G$, and let $(\cL, N) \in \cJ(H) \cap \cJ(J)$. Because $\hK{U(\cL)} \rtimes_N J$ is a crossed product of $\hK{U(\cL)}$ with $J/N$ by Lemma \ref{RingSGN}(b), the canonical map 
\begin{equation}\label{ULNJH} \left(\hK{U(\cL)} \rtimes_N J\right) \underset{K[J]}{\otimes}{} K[H] \longrightarrow \hK{U(\cL)} \rtimes_N H\end{equation}
is a bijection. Now consider the following commutative diagram:
\[\xymatrix{  \wUg{J} \underset{K[J]}{\otimes}{} K[H] \ar[d]\ar[rr]^{t_{J,H}} && \wUg{H} \ar[d] \\ \invlim \left( (\hK{U(\cL)} \rtimes_N J ) \underset{K[J]}{\otimes}{}K[H] \right) \ar[rr] && \invlim \hK{U(\cL)} \rtimes_N H. }\]
Here, the bottom horizontal arrow is the inverse limit of the maps in (\ref{ULNJH}), taken over all possible $(\cL,N) \in \cJ(H) \cap \cJ(J)$; therefore it is a bijection. Since $J$ has finite index in $H$ and inverse limits commute with finite direct sums, using Lemma \ref{CofinalTwo} we see that the left vertical arrow is a bijection. The right vertical arrow is also a bijection, again by Lemma \ref{CofinalTwo}. So the top horizontal arrow $t_{J,H}$ is an isomorphism, so that  $(\wUg{-}, \iota)$ is sheafy in the sense of Definition \ref{Sheafy}.

Now fix $g \in G$. Then there is a bijection $\cJ(H) \to \cJ(gHg^{-1})$ which sends $(\cL,N)$ to $(\cL', gNg^{-1})$ where $\cL' := \Ad(\sigma(g))(\cL)$. The $\cR$-Lie algebra isomorphism $\Ad(\sigma(g)) : \cL \to \cL'$ extends to a $K$-Banach algebra isomorphism $\hK{U(\cL)} \cong \hK{U(\cL')}$, which we will also denote by $\Ad(\sigma(g))$. Letting $\Ad_g: H \to g H g^{-1}$ denote conjugation by $g$, it follows from Lemma \ref{CPfunc} that there is an isomorphism 
\[ \Ad(\sigma(g)) \rtimes \Ad_g : \hK{U(\cL)} \rtimes_N H \longrightarrow \hK{U(\cL')} \rtimes_{gNg^{-1}} gHg^{-1}.\]
Passing to the limit over $(\cL, N) \in \cJ(H)$ induces a $K$-algebra isomorphism
\[\varphi_g(H) : \wUg{H} \to \wUg{g H g^{-1}}.\]
It is now straightforward to verify that these maps satisfy Definition \ref{GequivFiota}, as well as Definition \ref{GoodFIota}.
\end{proof}

Because of Proposition \ref{wUgGood} we may now apply Corollary \ref{FGproduct} and conclude that the tensor product 
\[ \wUg{H} \underset{K[H]}{\otimes} K[G]\]
carries the structure of an associative $K$-algebra, which contains both $\wUg{H}$ and $K[G]$ as $K$-subalgebras whenever $H$ is a compact open subgroup of $G$.

\begin{defn} \label{wUgDefn} We define $\wUg{G}$ to be the $K$-algebra $\wUg{H} \underset{K[H]}{\otimes} K[G]$, for any choice of compact open subgroup $H$ of $G$.
\end{defn}

It follows from Theorem \ref{MainFG} that this algebra is independent of $H$ in the sense that it is isomorphic to the subalgebra of 
\[\End\left( \lim\limits_{\stackrel{\longrightarrow}{H \in \cC}} \wUg{H} \underset{K[H]}{\otimes} K[G]\right) _{K[G]}\]
generated by the left-action of $K[G]$ and the left-action of $\wUg{J}$ as $J$ varies over \emph{all} compact open subgroups $J$ of $G$.

\subsection{Continuous actions on analytifications of algebraic varieties}
In Proposition \ref{GpOfPtsActsCtsly} we saw that if $\G$ is an $\cR$-group scheme which acts on a flat $\cR$-scheme $\X$ of finite presentation and $\sigma : G \to \G(\cR)$ is a continuous group homomorphism from some topological group $G$, then $G$ acts continuously on the rigid generic fibre $\bX := \cX_{\rig}$ of $\cX$. This was our first example of continuous group actions. We will now give a second example.

In preparation for this, we have the following useful result which tells us in particular that it suffices to verify continuity on an affinoid covering.

\begin{prop}\label{CtyIsLocal} Let $\{\bX_i\}$ be an admissible covering of a rigid analytic space $\bX$ by qcqs admissible open subsets. Let $G$ be a topological group and let $\rho : G \to \Aut(\bX, \cO_{\bX})$ be a group homomorphism. Suppose that for each $i$ we have
\be\item $G_{\bX_i}$ is open in $G$, and
\item $G_{\bX_i} \to \Aut(\bX_i, \cO_{\bX_i})$ is continuous. 
\ee 
Then $G$ acts continuously on $\bX$.
\end{prop}
\begin{proof} Let $\bU$ be a qcqs admissible open subset of $\bX$, and choose an admissible affinoid covering $\{\bU_{ij}\}_j$ of each $\bU \cap \bX_i$. Since $\bU$ is quasi-compact, we can find a finite subset $\{\bV_1,\ldots,\bV_n\} \subset \{ \bU_{ij}\}$ which forms an admissible covering of $\bU$. For each $k = 1,\ldots, n$, choose indices $i_k$ such that $\bV_k \subset \bX_{i_k} =: \bY_k$. 

Now $G_{\bY_k}$ acts continuously on $\bY_k$ by assumption (b) and Proposition \ref{OpenStab}. Hence $(G_{\bY_k})_{\bV_k}$ is open in $G_{\bY_k}$ and the map $(G_{\bY_k})_{\bV_k} \to \Aut(\bV_k, \cO_{\bV_k})$ is continuous. Since $G_{\bY_k}$ is open in $G$ by assumption (a) and since $(G_{\bY_k})_{\bV_k} = G_{\bY_k} \cap G_{\bV_k}$ is contained in $G_{\bV_k}$ we see that $G_{\bV_k}$ is open in $G$. Hence $G_{\bU}$ is open in $G$ because it contains the open subgroup $\bigcap_{k=1}^n G_{\bV_k}$. 

Because $\bU$ is qcqs, by \cite[Lemma 4.4]{BL1}, we can find a formal model $\cU$ for $U$ together with an open covering $\{\cV_1,\ldots,\cV_n\}$ such that $\cV_{k,\rig} = V_k$ for each $k$. By Theorem \ref{AutRigidTop}, it will suffice to show that the map $G_{\bU} \to \Aut(\bU, \cO_{\bU})$ is continuous with respect to the $\cT_{\cU}$-topology. Write $\cV_{kl} := \cV_k \cap \cV_l$ and $V_{kl} := V_k \cap V_l$ for each $k,l$, and define $H := \bigcap_{k=1}^n G_{\bV_k} \cap G_{\bY_k}$, an open subgroup of $G_{\bU}$ by the above. Note that each $V_k$ and each $V_{kl}$ is then $H$-stable. Let $\sigma : H \to  \Aut(\bU, \cO_{\bU})$, $\sigma_k: H \to \Aut(\bV_k, \cO_{\bV_k})$ and $\sigma_{kl} : H \to \Aut(\bV_{kl}, \cO_{\bV_{kl}})$ denote the actions of $H$ on $U, V_k$ and $V_{kl}$ respectively. It follows from Lemma \ref{CtsActLem}(a) that $H$ acts continuously on each $V_k$ and $V_{kl}$, so the group homomorphisms $\sigma_k$ and $\sigma_{kl}$ are continuous. Now for each $r \geq 0$, define
\[ J_r := \bigcap_{k=1}^n \sigma_k^{-1} \cG_{\pi^r}(\cV_k) \hsp \cap \hsp \bigcap_{k < l} \sigma^{-1}_{kl} \cG_{\pi^r}(\cV_{kl}), \]
an open subgroup of $H$. Since $\cG_{\pi^r}$ is a sheaf on $\cU$ by Lemma \ref{AutSheaf}, we see that $J_r \leq \sigma^{-1} \cG_{\pi^r}(\cU)$. Hence $\sigma^{-1} \cG_{\pi^r}(\cU)$ is open in $H$ (and therefore also in $G_{\bU}$) for each $r \geq 0$. So $G_{\bU} \to \Aut(\bU, \cO_{\bU})$ is continuous as required.
\end{proof}

Let $\G$ be an affine algebraic group of finite type over $K$. By \cite[Proposition 5.4/4]{BoschFRG},  $\G$ admits a \emph{rigid analytification} $\bG := \G^{\rig}$. Because rigid analytification preserves fibre products by \cite[Satz 1.8]{Koepf}, the underlying set of $\bG$ forms a group. Suppose also that $a : \G \times \X \to \X$ is an action of $\G$ on a scheme $\X$, locally of finite type over $K$. Applying the rigid analytification functor, we obtain an action $a^{\rig} : \bG \times \bX \to \bX$ of the rigid analytic group $\bG$ on $\bX := \X^{\rig}$. Let $g \cdot x = a^{\rig}(g,x)$ for $g \in \bG$ and $x \in \bX$, and let $T_n := K \langle y_1,\ldots, y_n\rangle$ denote the Tate algebra.

\begin{prop}\label{RigCtsAction}  Let $\Y$ be an affine open subscheme of $\X$, let $K[y_1,\ldots, y_n] / \fr{a} \cong \cO(\Y)$ be a presentation of $\cO(\Y)$ and let $\bU$ be the affinoid subdomain $\Sp (T_n / \fr{a} T_n)$ of $\bY := \Y^{\rig} \subset \bX$. Then there is an affinoid subdomain $\bT$ of $\bG$ containing $1$ such that $\bT \cdot \bU \subseteq \bU.$
\end{prop}
\begin{proof} By replacing $\G$ by the connected component of the identity, we may assume $\G$ to be connected. Therefore it is irreducible by \cite[Theorem 6.6]{Waterhouse}. Let $\W$ be the fibre product of the inclusion $\Y \hookrightarrow \X$ and the action map $a : \G \times \Y \to \X$:
\[ \xymatrix{ \G \times \Y \ar[rr]^{a} && \X \\ \W \ar[u] \ar[rr] && \Y .\ar[u] }\]
Then $\W$ is an open subscheme of the affine scheme $\G \times \Y$ containing the closed subscheme $\{1\} \times \Y$. Write $\W = \cup \W_j$ as a union of basic open sets. Since this closed subscheme is isomorphic to $\Y$ and is therefore irreducible, we see that it is already contained in some $\W_j$. Write $\W_j = D(f)$ for some $f \in \cO(\G \times \Y)$ and note that restriction $r(f)$ of $f$ to $\{1\} \times \Y$ must be a unit. The augmentation map $\epsilon : \cO(\G) \to K$ is split by the inclusion of the ground field $K$ into $\cO(\G)$, and it follows that the restriction map $r : \cO(\G \times \Y) \to \cO(\{1\} \times \Y)$ is a split surjection with right-inverse $s$, say. By replacing $f$ by $f / s(f)$, we may assume that $r(f) = 1$. Choose a generating set $\{c_1,\ldots, c_d\}$ for the ideal $\ker \epsilon$ in $\cO(\G)$; then its image in $\cO(\G \times \Y)$ generates the ideal $\ker r$ and we deduce that $f - 1$ lies in $\sum_{i=1}^d \cO(\G \times \Y) c_i = \ker r$. 

Now, consider the pullback of functions map $a^\sharp : \cO(\Y) \to \cO(\W_j) = \cO(\G \times \Y)_f$. The restriction of the action map $\W_j \to \Y$ to $\{1\} \times \Y$ is the projection onto the second component, so $a^\sharp(y_j) - y_j \in \sum_{i=1}^d \cO(\W_j) c_i$ for each $j=1,\ldots, n$. Because also $f - 1 \in \sum_{i=1}^d \cO(\G \times \Y) c_i$ and the $y_i$'s generate $\cO(\Y)$, we may write
\[ a^\sharp(y_j) = \frac{ y_j + \sum_{\alpha\in\N^n} F^{(j)}_\alpha y^\alpha } { 1 + \sum_{\alpha \in \N^n} G^{(j)}_\alpha y^\alpha } \in \cO(\G \times \Y)_f\]
where $F^{(j)}_\alpha, G^{(j)}_\alpha \in \ker \epsilon = \sum_{i=1}^d \cO(\G) c_i$ and only \emph{finitely many} of these are non-zero. Because the $c_i$'s generate $\cO(\G)$ as a $K$-algebra, we can choose $N \in \N$ sufficiently large so that each $F^{(j)}_\alpha / \pi$  and $G^{(j)}_\alpha / \pi$ lies in the $\cR$-subalgebra of $\cO(\G)$ generated by $c_1 / \pi^N , \ldots, c_d / \pi^N$. If we now take $\bT$ to be the affinoid subdomain of $\bG$ defined by
\[ \bT := \cO(\bG)\langle \frac{c_1}{\pi^N}, \ldots, \frac{c_d}{\pi^N} \rangle,\]
then we see that $F^{(j)}_\alpha, G^{(j)}_\alpha \in \pi \cO(\bT)^\circ$ for each $j=1,\ldots,n$ and $\alpha \in \N^n$. Then $f \in \cO(\bT \times \bU)^\times$ so the restriction map $\cO(\G \times \Y) \to \cO(\bT \times \bU)$ factors through $\cO(\W_j) = \cO(\G \times \Y)_f$. Because the elements $a^\sharp(y_j)$ now all lie in $\cO(\bT \times \bU)$, we see that the action map $a^{\sharp,\rig} : \cO(\bY) \to \cO(\W_j^{\rig})$ extends to an action map $\cO(\bU) \to \cO(\bT \times \bU)$.  Now by definition, $\bU$ is the subset of $\bY$ defined by the inequalities $\{ |y_1| \leq 1 ,\ldots, |y_n| \leq 1\}$; if $t \in \bT$ and $u \in \bU$ then $y_i(u) \in \cR$ for each $i$ and $F^{(j)}_\alpha(t) \in \pi \cR$ and $G^{(j)}_\alpha(t) \in \pi \cR$ for each $j, \alpha$ by construction, so 
\[y_j(t\cdot u) = a^\sharp(y_j)(t,u) = \frac{ y_j(u) + \sum_{\alpha\in\N^n} F^{(j)}_\alpha(t) y(u)^\alpha } { 1 + \sum_{\alpha \in \N^n} G^{(j)}_\alpha(t) y(u)^\alpha } \in \cR\]
for each $j=1,\ldots, n$. We conclude that $\bT \cdot \bU \subseteq \bU$ as required. \end{proof}

\begin{rmk} The analogue of Proposition \ref{RigCtsAction} does not work in the Zariski topology: there may not be any open neighbourhood $\U$ of the identity element in $\G$ such that $\U \cdot \Y \subset \Y$. For example, let $\G$ be the additive group $\G = \G_a = \Spec K[c]$ acting on the projective line $\X = \P^1$ via the rule $c \cdot y = y / (cy + 1)$, where $y$ is a local coordinate on $\X$. Then if $\Y = \Spec K[y]$ is the affine open subset of $\X$ where $y$ is regular, we find that $\W = \Spec K[c,y]_{cy + 1}$ cannot contain an open subset of the form $\U \times \Y$ for some open neighbourhood $\U$ of the identity element in $\G$. To see this, write $\U = \G \backslash V(I)$ for some ideal $I$ in $K[c]$ and suppose for a contradiction that $\U \times \Y \subset \W$. Then $V(cy+1) = \G \times \Y \backslash \W$ is contained in the closed subset $V(I)$ of $\G \times \Y$. Therefore $I$ is contained in $\langle cy + 1\rangle$ by the Nullstellensatz. Since $\langle cy + 1 \rangle \cap K[c]$ is zero, it follows that $\U$ is the empty set and does not contain the identity element.
\end{rmk}

We can now give our second example of continuous group actions.

\begin{thm}\label{2ndCtsAction} Let $\G$ be an affine algebraic group of finite type over $K$, acting on a scheme $\X$ which is locally of finite type over $K$. Then $\G(K)$ acts continuously on the rigid analytification $\bX := \X^{\rig}$.
\end{thm}
\begin{proof} By \cite[Lemma 5.4/2]{BoschFRG}, the canonical map $\bG(K) \to \G(K)$ is a bijection, and we will equip $\G(K)$ with the topology where a fundamental system of neighbourhoods of the identity element is given by the subsets $\bT(K)$ of $\G(K)$, as $\bT$ ranges over the affinoid subdomains of $\bG$ containing the identity.  The action $\G \times \X \to \X$ induces a homomorphism of group functors $\G \to \zAut(\X)$, which in turn induces a group homomorphism $\G(K) \to \Aut(\X, \cO_{\X})$. On the other hand, the functoriality of rigid analytification gives a group homomorphism $\Aut(\X, \cO_{\X}) \to \Aut(\bX, \cO_{\bX})$ and therefore an action $\G(K) \to \Aut(\bX, \cO_{\bX})$ of $\G(K)$ on $\bX$.  Of course the action of $\G(K)$ on the underlying set of $\bX$ agrees with the restriction of the action of $\bG$ on $\bX$ to its subgroup of $K$-rational points $\G(K) = \bG(K) \subset \bG$.

Let $\{\Y_j\}$ be an affine covering of $\X$ with each $\Y_j$ of finite presentation, and for each index $j$ fix a presentation $K[y_1^{(j)}, \ldots, y_{n_j}^{(j)}] / \fr{a}_j \cong \cO(\Y_j)$. Then for each $N \in \N$,
\[\bU_{j,N} := \Sp K \left\langle \pi^N y_1^{(j)}, \ldots, \pi^N y_{n_j}^{(j)} \right\rangle / \langle \fr{a}_j \rangle\]
is an affinoid subdomain of $\bY_j := \Y_j^{\rig} \subset \bX$ and $\{\bU_{j,N}\}$ forms an admissible affinoid covering of $\bX$. Fix the index $j$ and the natural number $N$, write $\bU := \bU_{j,N}$ and drop the subscripts and superscripts $j$ from the notation. By Proposition \ref{CtyIsLocal}, it will be enough to check that the stabiliser $\G(K)_\bU$ of $\bU$ in $\G(K)$ is open in $\G(K)$, and that the map $\G(K)_{\bU} \to \Aut(\bU, \cO_{\bU})$ is continuous. Now, because $K[\pi^N y_1, \ldots, \pi^N y_n] / \langle \fr{a} \rangle$ is still a presentation of $\cO(\Y)$, Proposition \ref{RigCtsAction} gives an affinoid subdomain $\bT$ of $\bG$ containing $1$ such that $\bT \cdot \bU \subseteq \bU$. We also see from the proof of this Proposition that the image of the resulting map $\bT(K) \to \End_K \cO(\bU)$ consists of invertible elements that preserve the closure of the $\cR$-subalgebra of $\cO(\bU)$ generated by $y_1,\ldots, y_n$. It follows that $\bT(K)$ is contained in $\G(K)_{\bU}$ and that the map $\bT(K) \to \Aut(\bU, \cO_{\bU})$ is continuous.
\end{proof}

\subsection{The Localisation Theorem for \ts{\wUg{G}}} \label{LocUgG}
Let $\G_0$ be a connected, simply connected, split semisimple affine algebraic group scheme over $\cR$ and let $\G := \G_0 \otimes_{\cR} K$ be its generic fibre. Let $G$ be a $p$-adic Lie group and let $\sigma : G \to \G(K)$ be a continuous group homomorphism.

Because $\G$ is semisimple, the centre of its Lie algebra $\fr{g}$ is zero. Thus we see that Hypotheses \ref{EqUgCapNotn} and \ref{ZgZero} is satisfied, so we have at our disposal the completed skew-group algebra $\wUg{G}$ from Definition \ref{wUgDefn}. On the other hand, whenever $G_0$ is a compact open subgroup of $\sigma^{-1} \cG_0(\cR)$,  Hypothesis \ref{SigmaGG} is satisfied for the group $G_0$ and thus we have at our disposal the completed skew-group algebra $\w\cD(\bG_0, G_0)^{\bG_0}$ from Definition \ref{wDGGDefn}.  We will now relate the two constructions.

We continue with the notation from $\S \ref{wDGGsect}$, so that $\bG_0 := (\h{\G_0})_{\rig}$ is an affinoid subdomain of the rigid analytic group $\bG := \G^{\rig}$ and $\cA$ denotes the $\cR$-algebra $\cO(\h\G_0)$.

\begin{lem}\label{RelatingTrivs} Let $H$ be an open subgroup of $G_0$, let $\cJ$ be an $H$-stable Lie lattice in $\fr{g}$ contained in $\fr{g}_0$, and let $\cL = \cA \cdot \gamma'(\cJ)$. Then
\be \item $H_{\cL} \leq H_{\cJ}$, and
\item the restriction of $\exp \circ \iota \circ \delta_{\cJ}$ to $H_{\cL}$ equals $\theta^{-1} \circ \beta_{\cL}$.
\ee
\end{lem}
\begin{proof} Let $h \in H_{\cL}$. By Proposition \ref{ThetaBeta}(b), there is a unique $v \in \cJ$ such that $\rho(h) = \exp(p^\epsilon \gamma'(v))$ and $\beta_{\cL}(h) = \theta(\exp(p^\epsilon \iota(v)))$. Now if $w \in \cJ$, then
\[\begin{array}{lll} \gamma'(\Ad(\sigma(h))(w)) &=& \rho(h) \circ \gamma'(w) \circ \rho(h)^{-1} = \\
&=& \exp( p^\epsilon \gamma'(v)) \circ \gamma'(w) \circ \exp(-p^\epsilon \gamma'(v)) = \\
&=& \exp( p^\epsilon \ad(\gamma'(v)) )(\gamma'(w)) = \\
&=& \gamma'\left( \exp(p^\epsilon \ad(v))(w) \right). \end{array}\]
Here we have used Lemma \ref{EquivariantInfAction} on the first line, \cite[Exercise 6.12]{DDMS} on the third line and the fact that $\gamma'$ is an $\cR$-Lie algebra homomorphism on the last line. Since $\gamma'$ is injective by Proposition \ref{LieTriangle}, we deduce that $\Ad(\sigma(h)) = \exp(p^\epsilon \ad(v))$. Applying Definition \ref{HJAdj}, we see that  $h \in H_{\cJ}$ and $\delta_{\cJ}(h) = p^\epsilon v$. Therefore
\[\exp(\iota(\delta_{\cJ}(h))) = \exp(\iota(p^\epsilon v)) = \theta^{-1}(\beta_{\cL}(h))\]
as required.
\end{proof}

\begin{prop}\label{wUgH=wdGH}  There is a continuous $K$-algebra isomorphism 
\[\kappa_H : \wUg{H} \congs \w\cD(\bG_0, H)^{\bG_0}\]
for every compact open subgroup $H$ of $\sigma^{-1}\G_0(\cR)$, such that the diagram
\[ \xymatrix{ \wUg{J} \ar[rr]^{\kappa_J} \ar[d] && \w\cD(\bG_0, J)^{\bG_0}\ar[d] \\ \wUg{H} \ar[rr]_{\kappa_H}  && \w\cD(\bG_0, H)^{\bG_0} } \]
commutes whenever $J \leq H$ are compact open subgroups of $G_0$.
\end{prop}
\begin{proof} Let $(\cJ,N) \in \cK(H)$, where $\cK(H)$ is the indexing set appearing in Definition \ref{wDGGDefn}. Thus, $\cJ$ is an $H$-stable Lie lattice in $\fr{g}$ contained in $\fr{g}_0$ and $N$ is an open normal subgroup of $H$ contained in $H_{\cA \cdot \gamma'(\cJ)}$. Hence $N \leq H_{\cJ}$ by Lemma \ref{RelatingTrivs}(a), so $(\cJ, N) \in \cJ(H)$ in the notation of Definition \ref{wUgHDefn}. Thus, $\cK(H) \subseteq \cJ(H)$.

Now if $(\cL, N)$ is some other member of $\cJ(H)$, then because $\fr{g}$ is finite dimensional over $K$, $\pi^n \cL \subseteq \fr{g}_0$ for some $n \geq 0$. Therefore
\[\cK(H) \ni (\pi^n \cL, N \cap H_{\cA \cdot \gamma'(\pi^n\cL)}) \leq (\cL, N)\]
which means that $\cK(H)$ is cofinal in $\cJ(H)$. Finally, for any $(\cJ,N) \in \cK(H)$, the two possible trivialisations of the $N$-action on $\hK{U(\cJ)}$ appearing in Definitions \ref{wDGGDefn} and \ref{wUgHDefn} agree by Lemma \ref{RelatingTrivs}(b). The result now follows easily.
\end{proof}

Now we fix a closed and flat Borel $\cR$-subgroup scheme $\B_0$ of $\G_0$, and we set $\B := \B_0 \otimes_{\cR} K$. Let $\X_0 := \G_0 / \B_0$, $\X := \G / B$, $\bX_0 := (\h\X_0)_{\rig}$ and $\bX := \X^{\rig}$. 

\begin{prop}\label{X0andXagree} The canonical map $\bX_0 \to \bX$ is an isomorphism.
\end{prop}
\begin{proof} The flag variety $\X_0$ can be covered in the Zariski topology by Weyl-translates of the big cell, by \cite[\S II.1.10(1)]{Jantzen}. Since the big cell is isomorphic to an affine space over $\cR$, it is flat over $\cR$. Hence $\X_0$ is flat over $\cR$. On the other hand, $\X_0$ is projective over $\cR$ by \cite[\S II.1.8]{Jantzen}. The result follows. \end{proof}

We will henceforth identify $\bX_0$ with $\bX$.

\begin{thm}\label{wUgAction} $\wUg{G}$ acts on $\bX$ compatibly with $G$. \end{thm}
\begin{proof} By Theorem \ref{2ndCtsAction}, $\G(K)$ acts continuously on $\bX$. Since $\sigma : G \to \G(K)$ is continuous, it follows from Definition \ref{CtsAct} that $G$ also acts continuously on $\bX$. By the construction of $\wUg{G}$ --- see Corollary \ref{MapFromGtoE} --- there is a canonical group homomorphism $\eta : G \to \wUg{G}^\times$, and by Theorem \ref{UgHFrSt}, the subalgebra $\wUg{H}$ of $\wUg{G}$ is Fr\'echet-Stein whenever $H$ is a compact open subgroup of $G$. 

Fix a compact open subgroup $G_0$ of $\sigma^{-1} \cG_0(\cR)$. Then Hypotheses \ref{SigmaGG} and \ref{HgsSpace} are satisfied by the affine $\cR$-group scheme $\G_0$ and the continuous homomorphism $\sigma : G_0 \to \G_0(\cR)$. Therefore $\w\cD(\bG_0,G_0)^{\bG_0}$ acts compatibly on $\bX_0$ by Theorem \ref{DGGaction}. By Proposition \ref{wUgH=wdGH}, there is a continuous $K$-algebra isomorphism $\kappa_{G_0} : \wUg{G_0} \stackrel{\cong}{\longrightarrow} \w\cD(\bG_0,G_0)^{\bG_0}$ which sends $\wUg{N}$ onto $\w\cD(\bG_0,N)^{\bG_0}$ for every compact open subgroup $N$ of $G_0$. Hence $\wUg{G_0}$ acts on $\bX_0$ compatibly with $G_0$. Since $\bX = \bX_0$ by Proposition \ref{X0andXagree}, it follows that $\wUg{G_0}$ acts on $\bX$ compatibly with $G_0$.  Thus we have a continuous homomorphism $\varphi^N : \wUg{N} \longrightarrow \w\cD(-,N)$ on $\bX_w/N$ whenever $N$ is an open subgroup of $G_0$.

Fix a compact open subgroup $H$ of $G$ and an $H$-stable affinoid subdomain $\bU$ of $\bX$. Then $H \cap G_0$ is open in $H$. For any open normal subgroup $N$ of $H$ contained in $G_0$, we have isomorphisms
\[\wUg{H} \cong \wUg{N} \rtimes_N H \qmb{and} \w\cD(\bU, H) \cong \w\cD(\bU,N)\rtimes_NH,\]
the first by Proposition \ref{wUgGood}(c) and the second by Corollary \ref{Stacky}. Using Lemma \ref{CPfunc}, we may therefore unambiguously define 
\[ \varphi^H(\bU) : \wUg{H} \to \w\cD(\bU,H)\]
to be $\varphi^N(\bU) \rtimes_N 1_H$ for any choice of open normal subgroup $N$ of $H$ contained in $G_0$. It is now straightforward to verify that these maps commute with restrictions in $\bX_w/H$, and that conditions (a)-(d) of Definition \ref{AactsCompatibly} are satisfied.
\end{proof}

Next, we record a general topological fact about coadmissible $G$-equivariant $\cD$-modules.

\begin{lem}\label{EnhanceFrechTop} Let $\bY$ be a smooth rigid analytic variety with a continuous action of a $p$-adic Lie group $H$, let $\bZ$ be a quasi-compact open subset of $\bY$ and let $\cM \in \cC_{\bY /H}$.
\be \item The $K$-vector space $\cM(\bZ)$ carries a canonical $K$-Fr\'echet topology such that the restriction maps $\cM(\bZ) \to \cM(\bU)$ are continuous for any $\bU \in \bZ_w$.
\item For any $f : \cM \to \cN$ in $\cC_{\bY/H}$, the map $f(\bZ) : \cM(\bZ) \to \cN(\bZ)$ is continuous.
\ee\end{lem}
\begin{proof} (a) Because $\bZ$ is quasi-compact, we can choose a finite $\bZ_w(\cT)$-covering $\cV$ of $\bZ$. Now, it follows from Theorem \ref{LocEquiv} that the restriction maps $\cM(\bU) \to \cM(\bV)$ are continuous for any $\bU \in \bZ_w(\cT)$ and $\bV \in \bU_w$. Because $\cV$ is finite and because $\cM(\bU)$ is a $K$-Fr\'echet space for each $\bU \in \cV$ by Definition \ref{FrechDmod}, we see that $\check{H}^0(\cV, \cM)$ is the kernel of a continuous map between two $K$-Fr\'echet spaces. Using the isomorphism $\cM(\bZ) \stackrel{\cong}{\longrightarrow} \check{H}^0(\cV, \cM)$ we can therefore equip $\cM(\bZ)$ with a $K$-Fr\'echet topology.  It is easy to check that this topology is independent of the choice of $\cV$. 

(b) This follows immediately from the construction of the topologies on $\cM(\bZ)$ and $\cN(\bZ)$ in part (a) together with the fact that $f(\bU) : \cM(\bU) \to \cN(\bU)$ is continuous for each $\bU \in \bX_w(\cT)$ by Definition \ref{FrechDmod}(b).
\end{proof}

\begin{prop}\label{Reconstruct2} Let $H$ be a compact open subgroup of $G$, let $M \in \cC_{\wUg{H}}$ and suppose that  $\alpha : \Loc^{\w\cD(\bX,H)}_\bX(M) \congs \cM$ is an isomorphism of $H$-equivariant locally Fr\'echet $\cD$-modules on $\bX$. 
\be \item There is a unique coadmissible $\wUg{H}$-module structure on $\cM(\bX)$ such that
 \begin{enumerate}[{(}i{)}]
 \item $\eta(g) \cdot m = g^{\cM}(m)$ for all $g \in H$ and $m \in \cM(\bX)$,
 \item the topology on $\cM(\bX)$ induced by this coadmissible $\wUg{H}$-module structure coincides with the canonical $K$-Fr\'echet topology given by Lemma \ref{EnhanceFrechTop}(a),
 \item the $\wUg{H}$-action on $\cM(\bX)$ extends the natural $\cD(\bX) \rtimes H$-action on $\cM(\bX)$ given by Proposition \ref{EqGlSec}.
 \end{enumerate}
\item This $\wUg{H}$-module structure on $\cM(\bX)$ does not depend on $\alpha$.
\item There is an isomorphism of $H$-equivariant locally Fr\'echet $\cD$-modules on $\bX$
\[\theta : \Loc^{\wUg{H}}_\bX(\cM(\bX)) \congs \cM,\] 
whose restriction to $\bX_w$ is given by
\[ \theta(\bU)(s \hsp \w\otimes \hsp m) = s \cdot (m_{|\bU})\]
for any $\bU \in \bX_w$, $s \in \w\cD(\bU,H_\bU)$ and $m \in \cM(\bX)$.
\ee\end{prop}
\begin{proof} (a) Because the rigid analytic flag variety $\bX$ is quasi-compact, the space of global sections $\cM(\bX)$ carries a $K$-Fr\'echet space topology by Lemma \ref{EnhanceFrechTop}(a). On the other hand, $\cM(\bX)$ is a $\cD(\bX) \rtimes H$-module by Proposition \ref{EqGlSec}, and hence also a $U(\fr{g}) \rtimes H$-module.

Write $\cN := \Loc^{\wUg{H}}_\bX(M)$ and let $\psi : M \to \cM(\bX)$ be the composite of the canonical map $M \to \cN(\bX)$ and $\alpha(\bX) : \cN(\bX) \to \cM(\bX)$. Because the maps $M \to \cM(\bU)$ are continuous for each $\bU \in \bX_w(\cT)$, we see that $\psi$ is continuous. It is also $U(\fr{g}) \rtimes H$-linear by construction. 

Because $U(\fr{g}) \rtimes H$ is dense in $\wUg{H}$, we may apply Lemma \ref{ExtendableModules} to the map $\psi : M \to \cM(\bX)$ to deduce that the $U(\fr{g}) \rtimes H$-action on $\cM(\bX)$ extends to an action of $\wUg{H}$ which satisfies all the required properties.  

(b) This now follows from Lemma \ref{ExtendableModules}(c).

(c) After replacing $\w\cD(\bX,H)$ with $\wUg{H}$, the proof of Lemma \ref{Reconstruct}(c) carries over word for word. \end{proof}

\begin{thm}\label{BBEssSurjFunc} Let $\cM$ be a coadmissible $G$-equivariant $\cD$-module on $\bX$. Then $\cM(\bX)$ is a coadmissible $\wUg{G}$-module, and there is an isomorphism
\[ \Loc^{\wUg{G}}_\bX(\cM(\bX)) \congs \cM\]
of $G$-equivariant locally Fr\'echet $\cD$-modules on $\bX$.
\end{thm}
\begin{proof} Note first of all that $\wUg{G}$ acts on $\bX$ compatibly with $G$ by Theorem \ref{wUgAction}, so the localisation functor $\Loc^{\wUg{G}}_\bX$ is defined --- see Definition \ref{DefnOfLocM} --- and it sends coadmissible $\wUg{G}$-modules to coadmissible $G$-equivariant $\cD$-modules on $\bX$ by Proposition \ref{LocFunCoadm}. 

Fix a compact open subgroup $H$ of $\sigma^{-1}\G_{0,\pi}(\cR)$. Then $\cM \in \cC_{\bX/H}$ by Proposition \ref{FrechRestr}(a), so there is a coadmissible $\w\cD(\bG_0,H)^{\bG_0}$-module $M_\infty$ and an isomorphism $\Loc^{\w\cD(\bG_0,H)^{\bG_0}}_{\bX}(M_\infty) \stackrel{\cong}{\longrightarrow} \cM$ in $\cC_{\bX/H}$ by Theorem \ref{BBEssSurj}. Using the isomorphism $\kappa_H : \wUg{H} \to \w\cD(\bG_0,H)^{\bG_0}$ from Proposition \ref{wUgH=wdGH}, we find a coadmissible $\wUg{H}$-module $M$ and an isomorphism $\Loc^{\wUg{H}}_{\bX}(M) \stackrel{\cong}{\longrightarrow} \cM$ in $\cC_{\bX/H}$. Now, by Proposition \ref{Reconstruct2}(a), the $U(\fr{g}) \rtimes H$-module structure on $\cM(\bX)$ and the canonical topology on $\cM(\bX$) given by Lemma \ref{EnhanceFrechTop}(a) extend to a coadmissible $\wUg{H}$-module structure. Moreover, by Proposition \ref{Reconstruct2}(c), there is an isomorphism 
\[\theta : \Loc^{\wUg{H}}_\bX(\cM(\bX)) \congs \cM\]
in $\cC_{\bX/H}$ given by the explicit formula
\[ \theta(\bU)(s \hsp \w\otimes \hsp m) = s \cdot (m_{|\bU})\]
for any $\bU \in \bX_w$, $s \in \w\cD(\bU,H_\bU)$ and $m \in \cM(\bX)$.

By Proposition \ref{EqGlSec}, $\cM(\bX)$ is a $\cD(\bX) \rtimes G$-module  and therefore a $U(\fr{g})\rtimes G$-module. The restriction of this $U(\fr{g}) \rtimes G$-module back to $U(\fr{g})\rtimes H$ agrees with the restriction of the $\wUg{H}$-module to $U(\fr{g})\rtimes H$. Therefore the $\wUg{H}$ and $K[G]$-actions on $\cM(\bX)$ extend to an action of $\wUg{G}$. Since the restriction of $\cM(\bX)$ back to $\wUg{H}$ is coadmissible, $\cM(\bX)$ is coadmissible also as a $\wUg{G}$-module. Note that by construction, $\Loc_\bX^{\wUg{G}}(M) = \Loc_\bX^{\wUg{H}}(M)$ as $H$-equivariant locally Fr\'echet $\cD$-modules on $\bX$. Finally, the exact same argument used in the proof of Theorem \ref{Kiehl} at the end of $\S \ref{LevelwiseSect}$ shows that $\theta$ is in fact $G$-equivariant. 
\end{proof}

We can finally state and prove our main result.
\begin{thm}\label{BBEquiv} Let $\G$ be a connected, simply connected, split semisimple affine algebraic group scheme over $K$, and let $\bX$ be the rigid analytification of the flag variety of $\G$.  Let $G$ be a $p$-adic Lie group and let $\sigma : G \to \G(K)$ be a continuous group homomorphism.  Then the localisation functor 
\[\Loc^{\wUg{G}}_{\bX} : \{M \in \cC_{\wUg{G}} : \fr{m}_0 \cdot M = 0\} \to \cC_{\bX/G}\]
is an equivalence of categories.
\end{thm}
\begin{proof} By the classification of connected split semisimple algebraic groups, it is known \cite[\S II.1]{Jantzen} that $\G$ extends to an affine algebraic group scheme $\G_0$ over $\cR$ satisfying the conditions given at the beginning of $\S$\ref{LocUgG}, so that $\G = \G_0 \otimes_{\cR} K$ is the generic fibre of $\G_0$. Now in view of Theorem \ref{BBEssSurjFunc}, it remains to show that the restriction of $\Loc^{\wUg{G}}_\bX$ to the full subcategory of $\cC_{\wUg{G}}$ consisting of objects killed by $\fr{m}_0$ is fully faithful. Let $M, N \in \cC_{\wUg{G}}$ be killed by $\fr{m}_0$, and write $\Loc := \Loc^{\wUg{G}}_\bX$, $\cM := \Loc(M)$ and $\cN := \Loc(N)$. For any $\wUg{G}$-linear morphism $f : M \to N$ we have the commutative diagram
\[ \xymatrix{  M  \ar[rrr]^f\ar[d] &&& N  \ar[d]\\ 
\cM(\bX) \ar[rrr]_{\Loc(f)(\bX)} &&& \cN(\bX)
}\]
where the vertical arrows are bijective by Theorem \ref{BBRGamma}. It follows immediately that $\Loc$ is faithful.

Suppose next that $\alpha : \cM \to \cN$ is a morphism in $\cC_{\bX/G}$. Then because the rigid analytic flag variety $\bX$ is quasi-compact, $\alpha(\bX) : \cM(\bX) \to \cN(\bX)$ is continuous. On the other hand, it is follows from Proposition \ref{EqGlSec} that it is $U(\fr{g})\rtimes G$-linear. Since $U(\fr{g}) \rtimes G$ is dense $\wUg{G}$ and since $\cM(\bX)$ and $\cN(\bX)$ are coadmissible $\wUg{G}$-modules by Theorem \ref{BBEssSurjFunc}, it follows that $\alpha(\bX)$ is $\wUg{G}$-linear. Now let $f : M \to N$ be defined by the commutative diagram
\[ \xymatrix{  M  \ar[rrr]^f\ar[d]_{\cong} &&& N  \ar[d]^{\cong}\\ 
\cM(\bX) \ar[rrr]_{\alpha(\bX)} &&& \cN(\bX).
}\]
Then $f$ is $\wUg{G}$-linear, and we claim that $\Loc(f) = \alpha$. To see this, let $\bU \in \bX_w(\cT)$ be arbitrary and choose a compact open subgroup $H$ of $G$ such that $(\bU,H)$ is small, using Lemma \ref{SmallPairsExist}. By construction, $\alpha(\bU)$ and $\Loc(f)(\bU)$ agree on the image of $M$ in $\cM(\bU)$. Since both maps are $\w\cD(\bU,H)$-linear and since this image generates $\cM(\bU)$ as a $\w\cD(\bU,H)$-module, we deduce that $\alpha(\bU) = \Loc(f)(\bU)$. Since $\cM$ and $\cN$ are sheaves and since $\bX_w(\cT)$ is a basis for $\bX$, it follows that $\alpha = \Loc(f)$ as claimed.
\end{proof}

\subsection{Connection to locally analytic distribution algebras}\label{ConnLADist}
Our goal in this subsection is to establish the following

\begin{thm}\label{LADistFormula}  Let $L$ be a finite extension of $\Q_p$ contained in $K$ and let $\G$ be an affine algebraic group of finite type over $L$ with Lie algebra $\fr{g}$, such that $Z(\fr{g}) = 0$. Let $G$ be an open subgroup of $\G(L)$ and let $D(G,K)$ denote the $K$-algebra of locally $L$-analytic distributions on $G$. Then there is a continuous $K$-algebra isomorphism
\[ \eta_G : D(G,K) \tocong \w{U}(\fr{g}_K, G).\]
where $\fr{g}_K := \fr{g} \otimes_L K$.
\end{thm}

We fix the field $L$, the algebraic group $\G$, its Lie algebra $\fr{g}$ and the open subgroup $G$ of $\G(L)$ for the remainder of $\S \ref{ConnLADist}$. 

Recall from e.g. \cite[Remark 2.2.5(ii)]{OrlikStrauch} that a uniform pro-$p$ group $N$ is said to be \emph{$L$-uniform} if the $\Qp$-Lie algebra $\Qp \otimes_{\Zp} L_N$ is in fact an $L$-Lie algebra, and if the $\Zp$-Lie algebra $L_N$ is an $\cO_L$-submodule of $\Qp \otimes_{\Zp} L_N$.

\begin{lem}\label{LuniformsExist} Every compact locally $L$-analytic group has at least one open normal $L$-uniform subgroup $N$.
\end{lem}
\begin{proof} This is \cite[Lemma 2.2.4]{OrlikStrauch}.
\end{proof}

We begin the proof of Theorem \ref{LADistFormula} by explaining how to remove the restriction $p \neq 2$ from the material in \cite[\S 5]{ArdWad2014}; recall from $\S\ref{GActionSection}$ that $\epsilon := 1$ if $p$ is odd and $\epsilon := 2$ if $p = 2$.  Let $N$ be an $L$-uniform pro-$p$ group; we will denote its dimension as a locally $L$-analytic group by $\dim_L N$. On the one hand, its $\Zp$-Lie algebra $L_N$ is powerful, and on the other hand, it is an $\cO_L$-module, by definition. We define the \emph{$\cR$-Lie algebra associated with $N$} to be
\[\cL_N := \cR \underset{\cO_L}{\otimes} \frac{1}{p^\epsilon}L_N.\]
Note that $L_N$ is a $\Zp$-module of rank $\dim_L N \cdot \dim_{\Qp} L$, and an $\cO_L$-module of rank $\dim_L N$. Thus, $\cL_N$ is an $\cR$-module of rank $\dim_L N$. Because $L_N$ is powerful, $\cL_N$ is also an $\cR$-Lie algebra. Recall from \cite[\S 4]{ST} that the algebra $D(N,K)$ of $K$-valued locally $L$-analytic distributions admits $K$-Banach space completions $D_r(G,K)$ for each real number $1/p \leq r < 1$.	We have the following version of the famous Lazard isomorphism.

\begin{thm}\label{LazardIso} Let $N$ be an $L$-uniform pro-$p$ group and let $\cL_N := \cR \otimes_{\cO_L} \frac{1}{p^\epsilon}L_N$. Then there is a natural isomorphism of $K$-Banach algebras
\[ \nu_N : D_{1/p}(N,K) \stackrel{\cong}{\longrightarrow} \hK{U(\cL_N)}.\]
\end{thm}
\begin{proof} When $p$ is odd, this is precisely \cite[Lemma 5.2]{ArdWad2014}. We indicate how to extend the argument given there to the case where $p$ is also allowed to be $2$.

Suppose first that $L = \Qp$. Let $|\cdot| : K \to \mathbb{R}$ be the norm which induces the topology on $K$, normalised by $|p| = 1/p$. Because $N$ is uniform, we can regard $N$ as a $p$-saturated group in the sense of Lazard, by setting 
\[\omega(g) := \sup\{m \in \N : g \in N^{p^m}\} + \epsilon \qmb{for each} g \in N.\]
Fix a minimal topological generating set $\{g_1,\ldots,g_d\}$ for $N$, where $d := \dim_{\Qp} N$, and write $b_i := g_i-1$ as usual. Then $\{g_1,\ldots, g_d\}$ is an ordered basis for $N$ such that $\omega(g_i) = \epsilon$ for all $i = 1,\ldots, d$. Now, because $L = \Qp$, it follows from \cite[\S 4, p.160]{ST} that $D_{1/p}(N,K)$ consists of all formal power series $\lambda = \sum_{\alpha \in \mathbb{N}^d} \lambda_\alpha \mathbf{b}^\alpha$ in $b_1,\ldots, b_d$ with coefficients $\lambda_\alpha \in K$ such that
\[ ||\lambda||_{1/p} := \sup\limits_{\alpha \in \mathbb{N}^d} |\lambda_\alpha| (1/p)^{\epsilon |\alpha|} \]
is finite. Thus, $D_{1/p}(N,K) = K \underset{\Qp}{\h\otimes}  \Lambda_{\Qp}(N,\omega)$ where $\Lambda_{\Qp}(N,\omega)$ is the Banach $\Qp$-algebra completion of the group ring $\Qp[N]$ introduced in \cite[\S 30, p209]{SchLieGps}. This algebra contains the elements $u_i := \log(g_i) / p^\epsilon$ for $i=1,\ldots, d$, and it follows from \cite[Lemma 7.12]{DDMS} that there is an isomorphism $\log$ between the $\Zp$-Lie algebra $L_N = (N, +, [,])$ and the $\Zp$-submodule of $\Lambda_{\Qp}(N,\omega)$ generated by $\{p^\epsilon u_1,\ldots, p^\epsilon u_d\}$  which sends $g_i$ to $p^\epsilon u_i$ for each $i$. Letting $v_i := 1 \otimes g_i / p^\epsilon \in \cL_N$ for $i=1,\ldots, d$, this isomorphism extends to a $K$-Banach algebra homomorphism
\[ \hK{U(\cL_N)} \longrightarrow K \underset{\Qp}{\h\otimes} \Lambda_{\Qp}(N,\omega)\]
which sends $v_i$ to $u_i$ for each $i=1,\ldots,d$. Because $L = \Qp$, $\{v_1,\ldots,v_d\}$ is a basis for $\cL_N$ as an $\cR$-module, so  $\hK{U(\cL_N)}$ consists of power series $\sum_{\alpha\in\mathbb{N}^d} \mu_\alpha \mathbf{v}^\alpha$ where $\mu_\alpha \to 0$ as $|\alpha|\to\infty$. On the other hand, it follows from \cite[Theorem 31.5(ii)]{SchLieGps} that every element of $K \underset{Qp}{\h\otimes} \Lambda_{\Qp}(N, \omega)$ can be written uniquely in the form $\sum_{\alpha \in \mathbb{N}^d} \mu_\alpha \mathbf{u}^\alpha$ with coefficients $\mu_\alpha \in K$ such that $|\mu_\alpha| \to 0$ as $|\alpha| \to \infty$. So the displayed map is an isometric isomorphism and we can take $\nu_N$ to be its inverse. 

To deal with the case of an arbitrary finite extension $L$ of $\Qp$, follow the last part of the proof of \cite[Lemma 5.2]{ArdWad2014} using \cite[Lemma 5.1]{Schmidt} and \cite[\S 3.2.3(iii)]{Berth}.
\end{proof}

\begin{cor}\label{LazardIsomNpn} Let $N$ be an $L$-uniform pro-$p$ group and let $\cL_N := \cR \otimes_{\cO_L} \frac{1}{p^\epsilon}L_N$. Then there is a natural isomorphism of $K$-Banach algebras
\[ \nu_{N^{p^n}} : D_{1/p}(N^{p^n},K) \stackrel{\cong}{\longrightarrow} \hK{U(p^n \cL_N)}\]
for each $n \geq 0$.
\end{cor}
\begin{proof} This follows immediately from Theorem \ref{LazardIso} because $N^{p^n}$ is again $L$-uniform with $\cL_{N^{p^n}} = p^n \cL_N$, for each $n \geq 0$.
\end{proof}

We now fix an $L$-uniform open subgroup $N$ of $G$ until the end of the proof of Corollary \ref{DNKUGN}.

\begin{prop}\label{AlgAnalLie} There is an $N$-equivariant isomorphism of $L$-Lie algebras
\[ \theta : L \otimes_{\cO_L} L_N \tocong  \fr{g}\]
where $N$ acts on $\fr{g}$ via the adjoint action of $\G(L)$ on $\fr{g}$, and by conjugation on $L_N$.
\end{prop}
\begin{proof} Note that the group $\G(L)$ is locally $L$-analytic. Because $\G$ is smooth over $L$, it follows from the Jacobian criterion for smoothness and a version of the implicit function theorem (see, e.g. \cite[5.8.10, 5.8.11]{BourVarAn}) that the canonical map $\fr{g} = \Lie(\G) \to \Lie(\G(L))$ is a bijection.

Since $N$ is open $G$ which is open in $\G(L)$, the Lie algebra $\Lie(N)$ of $N$ as a locally $L$-analytic group coincides with $\Lie(\G(L))$. On the other hand, by \cite[Theorem 9.11(iii)]{DDMS} there is an isomorphism $\Lie(N) \tocong L \otimes_{\cO_L} L_N$. It is straightforward to check that the composite map 
\[\fr{g} \tocong \Lie(\G(L)) = \Lie(N) \tocong L \otimes_{\cO_L} L_N\]
respects Lie brackets and the given $N$-actions, and we take $\theta$ to be the inverse of this map.
\end{proof}

We also use the letter $\theta$ to denote the induced isomorphism 
\[ K \otimes_{\cR} \cL_N = K \otimes_L(L\otimes_{\cO_L} L_N) \tocong K \otimes_L \fr{g} = \fr{g}_K\]
With this notation, we see that for each $n \geq 0$, $p^n \theta(\cL_N)$ is a Lie lattice in $\fr{g}_K$ in the sense of Definition \ref{LieLatticeForAlgGps}(a).

\begin{lem}\label{ComparingTrivs} Write $\cJ := \theta(\cL_N)$. Then for every $n \geq 0$ there is a commutative diagram of groups
\[ \xymatrix{ 
N^{p^n} \ar@{.>}[rrrr] \ar[ddd] \ar[dr]^{\delta_n} &&&& N_{p^n \cJ} \ar[dl]_{\delta_{p^n\cJ}}\ar[ddd] \\
& (p^{n+\epsilon}\cL_N, \ast) \ar[d]_{\exp \circ\ad} \ar[rr]^{\theta} && (p^{n+\epsilon} \cJ, \ast) \ar[d]^{\exp \circ\ad}  & \\
& \Aut_{\cR-\Lie}(p^n\cL_N) \ar[rr]_{\theta\hsp  \circ \hsp - \hsp \circ \hsp \theta^{-1}} & & \Aut_{\cR-\Lie}(p^n\cJ)  & \\
N \ar[ur]_{\Ad_N} \ar@{=}[rrrr] &&&& N.\ar[ul]^{\Ad \circ \sigma}}\]
\end{lem}
\begin{proof} We begin by explaining the various objects and solid arrows in this diagram. The group $(p^{n+\epsilon}\cL, \ast)$ is obtained from the powerful $\cR$-Lie algebra $p^{n+\epsilon} \cL$ by means of the Campbell-Baker-Hausdorff formula, and the map $\delta_n$ sends $x \in N^{p^n} = p^n L_N$ to $1 \otimes x \in \cR \otimes_{\cO_L} p^n L_N = p^{n+\epsilon}\cL_N$.  $\Ad_N$ denotes the conjugation action of $N$ on $p^n \cL_N$, $\sigma$ denotes the inclusion of $N$ into $\G(L)$ and $\Ad : \G(L) \to \Aut_{\cR-\Lie}(p^n \cJ)$ denotes the base change to $K$ of the adjoint representation $\Ad : \G(L) \to \Aut(\fr{g})$ from $(\ref{AdjRep})$. The trapezium on the right hand side is an instance of diagram (\ref{DeltaLdiag}). 

Now, $\theta : p^n \cL_N \to p^n\cJ$ is an isomorphism of $\cR$-Lie algebras by Proposition \ref{AlgAnalLie}. It follows that the middle square commutes. The trapezium on the right commutes by definition of the group $N_{p^n\cJ}$. The bottom trapezium commutes because $\theta : K \otimes_{\cR}\cL_N\to \fr{g}_K$ is $N$-equivariant by Proposition \ref{AlgAnalLie}. Finally, the commutativity of the trapezium on the left can be deduced from \cite[Exercise 6.12]{DDMS}. Therefore the restriction of $\Ad \circ \sigma$ to $N^{p^n}$ agrees with $\exp \circ \ad \circ \theta \circ \delta_n$, so the universal property of the pullback $N_{p^n\cJ}$ gives the horizontal dotted arrow at the top of the diagram. \end{proof}

\begin{cor} With the notation of Lemma \ref{ComparingTrivs}, we have $N^{p^n} \leq N_{p^n \cJ}$.\end{cor}
\begin{proof} The dotted arrow is injective because both $\delta_n$ and $\theta$ are injective.\end{proof}

The subgroup $N^{p^n}$ is normal in $N$, and the conjugation action of $N$ on $N^{p^n}$ extends by functoriality to an action of $N$ on $D_{1/p}(N^{p^n}, K)$ by $K$-Banach algebra automorphisms. The inclusion of $N^{p^n}$ into $D_{1/p}(N^{p^n},K)$ is an $N$-equivariant trivialisation of this action in the sense of Definitions \ref{Triv} and \ref{DefnTriv}(b), so using Definition \ref{DefnTriv}(a) we may form the crossed product algebra
\[ D_{1/p}(N^{p^n},K) \rtimes_{N^{p^n}} N.\]

\begin{prop}\label{PuttingNonTop} Let $\cJ:= \theta(\cL_N)$. Then for each $n \geq 0$, the isomorphism $\nu_{N^{p^n}}$ from Corollary \ref{LazardIsomNpn} extends to an isomorphism of $K$-Banach algebras
\[ \theta  \nu_{N^{p^n}} \rtimes_{N^{p^n}} 1_N : D_{1/p}(N^{p^n},K) \rtimes_{N^{p^n}} N \tocong \hK{U(p^n \cJ)} \rtimes_{N^{p^n}}^{e^{\iota \delta_{p^n\cJ}}} N.\]
\end{prop}
\begin{proof} The isomorphism $\theta$ induces an isomorphism $\theta : \hK{U(p^n\cL_N)} \tocong \hK{U(p^n\cJ)}$ of $K$-Banach algebras, by functoriality. Let $\iota : p^n \cJ \to \hK{U(p^n \cJ)}$ denote the natural inclusion. Examining the construction of $\nu_{N^{p^n}}$ given in Proposition \ref{LazardIso}, we see that the composite map $\theta \circ \nu_{N^{p^n}} : D_{1/p}(N^{p^n},K) \tocong \hK{U(p^n \cJ)}$ is given on group elements as follows:
\[ \theta \nu_{N^{p^n}}(x) = (\exp \circ \iota \circ \theta)(x) \qmb{for all} x \in N^{p^n}.\]
Because $\theta$ is $N$-equivariant by Proposition \ref{AlgAnalLie}, we see that
\[ \theta \nu_{N^{p^n}}(gxg^{-1}) = g \cdot \theta\nu_{N^{p^n}}(x) \qmb{for all} x\in N^{p^n} \qmb{and} g \in N.\]
For any fixed $g \in N$, both sides of this equation are continuous functions in $x$. Since $K[N^{p^n}]$ is dense in $D_{1/p}(N^{p^n},K)$, it follows that
\[ \theta \nu_{N^{p^n}}(gxg^{-1}) = g \cdot \theta\nu_{N^{p^n}}(x) \qmb{for all} x\in D_{1/p}(N^{p^n},K) \qmb{and} g \in N.\]
On the other hand, by Lemma \ref{ComparingTrivs}, we have $\delta_{p^n\cJ}(x) = \theta(x)$ if we identify $x \in N^{p^n}$ with its image $\delta_n(x) = 1 \otimes x$ in $p^{n+\epsilon}\cL_N$. Hence
\[\theta \nu_{N^{p^n}}(x) = (\exp \circ \iota)(\theta(x)) = (\exp \circ \iota)(\delta_{p^n\cJ}(x)).\]
The last two displayed equations are exactly the conditions needed to apply Lemma \ref{CPfunc}, which shows the existence of the required map
\[\theta  \nu_{N^{p^n}} \rtimes_{N^{p^n}} 1_N : D_{1/p}(N^{p^n},K) \rtimes_{N^{p^n}} N \tocong \hK{U(p^n \cJ)} \rtimes_{N^{p^n}}^{e^{\iota \delta_{p^n\cJ}}} N.\]
This map has to be an isomorphism of $K$-Banach algebras in view of Proposition \ref{AlgAnalLie} and Proposition \ref{LazardIso}.
\end{proof}

\begin{cor}\label{ExtendedLazIsom} There is an isomorphism of $K$-Fr\'echet algebras
\[ \invlim D_{1/p}(N^{p^n},K) \rtimes_{N^{p^n}} N \tocong \w{U}(\fr{g}_K, N).\]
\end{cor}
\begin{proof} Apply Proposition \ref{PuttingNonTop} and Definition \ref{wUgHDefn}(c). \end{proof}

In to relate the inverse limit appearing on the left hand side of the isomorphism in Corollary \ref{ExtendedLazIsom} to locally analytic distribution algebras, we will need the following elementary result on the $p$-adic valuations of binomial coefficients.

\begin{lem}\label{PadicBinomVals} Let $n$ and $j$ be integers such that $n \geq 1$ and $1 \leq j \leq p^n - 1$. Then
\[p^n v_p \left( \binom{p^n}{j} \right) \geq \epsilon j,\]
with equality if and only if $p=2$ and $j = 2^{n-1}$.
\end{lem}
\begin{proof} Suppose first that $p > 2$. Then the binomial coefficient is always divisible by $p$ so $p^n v_p\left( \binom{p^n}{j} \right) \geq p^n > j = \epsilon j$. Now suppose that $p = 2$ so that $\epsilon = 2$. By \cite[Chapter II, \S 8.1, Lemma 1]{BourLGLA}, we know that $v_2(j!) = j - s(j)$ where $s(j)$ is the sum of the $2$-adic digits of $j$. Therefore
\[v_2 \left( \binom{2^n}{j} \right) = 2^n - 1 - (j - s(j)) - ((2^n-j) - s(2^n-j)) = s(j) + s(2^n-j) - 1.\]
Now $s(j) \geq 1$, with equality if and only if $j$ is a power of $2$, and the only way of writing $2^n$ as a sum of two other powers of $2$ is $2^n = 2^{n-1} + 2^{n-1}$. Hence $s(j) + s(2^n-j)$ is at least $3$ unless $j = 2^{n-1}$ when $s(j) + s(2^n-j) = 2$. So when $j \neq 2^{n-1}$ we have
\[ 2^n v_2 \left(\binom{2^n}{j}\right) \geq 2^n(3 - 1) = 2^{n+1} > 2j = \epsilon j,\]
and when $j = 2^{n-1}$, the left hand side is equal to $2^n$ which is precisely $\epsilon j$.
\end{proof}

Next, we extend the results of Schmidt \cite[\S 6]{Schmidt}, \cite[\S 5]{SchmidtBB} to include the case $p = 2$.

\begin{thm}\label{SchmidtCP} For each $n \geq 0$, there is a continuous $K$-algebra isomorphism
\[ D_{1/p}(N^{p^n},K) \rtimes_{N^{p^n}} N \tocong D_{r_n}(N,K)\]
where $r_n := \sqrt[p^n]{1/p}$. 
\end{thm}
\begin{proof} As in the proof of Theorem \ref{LazardIso}, the statement reduces to the case where $L = \Qp$. As the statement is trivially true when $n = 0$, we may assume that $n \geq 1$. Let $\{g_1,\ldots, g_d\}$ be a minimal topological generating set for $N$, let $b_i := g_i - 1 \in K[N]$ and $c_i := g_i^{p^n} - 1 \in K[N^{p^n}]$ for each $i=1,\ldots, d$. 

Recall the notion of \emph{degree functions} from \cite[\S 2.2]{AW13}. By the definitions given in \cite[\S 4]{ST} and the discussion in \cite[\S 10.7]{AW13} (extended naturally to include the case $p = 2$ as well as the case of an arbitrary $p$-adic coefficient field $K$), the algebra $D_{r_n}(N,K) = K \underset{\Qp}{\h\otimes} D_{r_n}(N, \Qp)$ is the completion of the rational Iwasawa algebra $KN$ with respect to the degree function $\deg^{(0)}_{p^n}$ given by
\[ \deg^{(0)}_{p^n} \left( \sum_{\alpha \in \N^d} \lambda_\alpha \mathbf{b}^\alpha \right) = \min \{ p^n v_p(\lambda_\alpha) + \epsilon |\alpha| : \alpha \in \N^d \}\]
where $v_p$ denotes the valuation on $K$ normalised by $v_p(p) = 1$. Similarly, the algebra $D_{1/p}(N^{p^n},K)$ is the completion of the $KN^{p^n}$ with respect to the degree function $p^n\deg^{(n)}_1$ given by
\[ p^n\deg^{(n)}_1 \left( \sum_{\alpha \in \N^d} \lambda_\alpha \mathbf{c}^\alpha \right) = p^n\min \{ v_p(\lambda_\alpha) + \epsilon |\alpha| : \alpha \in \N^d\}.\]
Consider the inclusion $\phi:KN^{p^n} \hookrightarrow KN$. It sends $c_i$ to $(1+b_i)^{p^n} - 1$, and 
\begin{equation}\label{DegDeg} \begin{array}{lll} \deg^{(0)}_{p^n}(\phi(c_i)) &=& \deg^{(0)}_{p^n} \left( b_i^{p^n} + \sum_{j = 1}^{p^n-1} \binom{p^n}{j} b_i^{p^n - j} \right) \\
&=& \min \{p^n v_p\left( \binom{p^n}{j} \right) + \epsilon \cdot (p^n - j) : 0 \leq j \leq p^n-1 \} \\
&\geq & p^n \epsilon \\
&=& p^n \deg^{(n)}_1(c_i)
\end{array}\end{equation}
by Lemma \ref{PadicBinomVals}. It follows that $KN^{p^n} \hookrightarrow KN$ is a morphism of filtered $K$-modules with respect to these degree functions, in the sense of \cite[I.2.1.4.1]{Laz1965}:
\[\begin{array}{lll} \deg^{(0)}_{p^n} \phi \left( \sum_{\alpha \in \N^d} \lambda_\alpha \mathbf{c}^\alpha \right)  &\geq&  \min\{ \deg^{(0)}_{p^n}(\lambda_\alpha) + \deg^{(0)}_{p^n}(\phi(\mathbf{c}^\alpha)) : \alpha \in \N^d\} \\
&\geq& \min\{ p^n \cdot v_p (\lambda_\alpha) + \sum_{i=1}^d \alpha_i \cdot \deg^{(0)}_{p^n}(\phi(c_i)) : \alpha \in \N^d\} \\
&\geq& \min\{ p^n \cdot v_p (\lambda_\alpha) + p^n \epsilon \cdot |\alpha| : \alpha \in \N^d\} \\
&=&p^n \deg^{(n)}_1 \left( \sum_{\alpha \in \N^d} \lambda_\alpha \mathbf{c}^\alpha \right) .\end{array}\]

Now, it follows from \cite[Th\'eor\`eme III.2.3.3]{Laz1965} and the discussion in \cite[\S 4]{ST} that
\[ \gr D_{1/p}(N^{p^n}, K) = U(\gr N^{p^n}) \otimes_{\gr \Zp} \gr K.\]
Since $n > 0$, the Lie algebra $\gr N^{p^n}$ is commutative, so the algebra on the right hand side is a polynomial ring over $k := \cR / p \cR$ in $d + 1$ variables. More precisely:
\[\gr D_{1/p}(N^{p^n},K) = (\gr K)[y_1,\ldots,y_d] = k[s,s^{-1}, y_1,\ldots, y_d]\]
where $y_i$ is the principal symbol of $c_i$ for each $i$ and $s$ is the principal symbol of $p$. Note that because we are working with $p^n \deg^{(n)}_1$ as opposed to $\deg^{(n)}_1$, the degrees of the  graded generators of this algebra are as follows:
\[ \deg(s) = p^n \qmb{and} \deg (y_i) = \epsilon p^n \qmb{for each} i.\]
Next, because $r_n > 1/p$ as $n > 0$, it follows from \cite[Theorem 4.5(i)]{ST} that 
\[ \gr D_{r_n}(N,K) = (\gr K)[x_1,\ldots,x_d] = k[s,s^{-1},x_1,\ldots,x_d]\]
where $x_i$ is the principal symbol of $b_i$ for each $i$. Note that we have
\[ \deg(s) = p^n \qmb{and} \deg (x_i) = \epsilon \qmb{for each} i.\]
Now, the associated graded morphism
\[ \gr \phi : \gr D_{1/p}(N^{p^n},K) \longrightarrow \gr D_{r_n}(N,K)\]
sends $s$ to $s$, and by Lemma \ref{PadicBinomVals} and the calculation performed in (\ref{DegDeg}), we have
\begin{equation}\label{GrPhi} (\gr \phi)(y_i) = \left\{ \begin{array}{lll} x_i^{p^n} &\qmb{if} &p > 2 \\ x_i^{2^n} + s \cdot x_i^{2^{n-1}} &\qmb{if} & p = 2 \end{array} \right.\end{equation}
for each $i$. Consider the restriction of $\gr \phi$ to $k[s,y_1,\ldots,y_d]$. The above formulas show that this restriction sends $y_i^{p^n}$ to $x_i^{p^n}$ modulo $s$ for each $i$, so this restriction is injective modulo $s$. It follows easily that $\gr \phi$ is also injective, and hence
\[ \deg^{(0)}_{p^n}(\phi(\lambda)) = p^n \deg^{(n)}_1(\lambda) \qmb{for all} \lambda \in D_{1/p}(N^{p^n},K).\]
Thus the inclusion $\phi : KN^{p^n} \to KN$ is an isometry, so it extends to an inclusion
\[ \hat{\phi} : D_{1/p}(N^{p^n},K) \hookrightarrow D_{r_n}(N,K),\]
which extends to a continuous $K$-algebra homomorphism
\[ \Phi : D_{1/p}(N^{p^n},K) \rtimes_{N^{p^n}} N \longrightarrow D_{r_n}(N,K).\]
Let $\N_p := \{a \in \N : a<p\}$. Using (\ref{GrPhi}) we also see that $\gr D_{r_n}(N,K)$ is a finitely generated free $D_{1/p}(N^{p^n},K)$-module of rank $p^d$ via $\gr \phi$, with basis $\{ \mathbf{x}^\alpha : \alpha \in \N_p^d \}$. As $\gr \mathbf{b}^\alpha = \mathbf{x}^\alpha$ for each $\alpha \in \N^d$,  \cite[Th\'eor\`eme I.2.3.17]{Laz1965} implies that 
$D_{r_n}(N,K)$ is a finitely generated free $D_{1/p}(N^{p^n},K)$-module with basis $\{ \mathbf{b}^\alpha : \alpha \in \N_p^d \}$. Using \cite[Lemma 7.8]{DDMS}, we see that
\[ D_{r_n}(N,K) = \bigoplus_{\alpha \in \N_p^d} D_{1/p}(N^{p^n},K) \cdot \mathbf{g}^\alpha.\]
Using Lemma \ref{RingSGN}(b) we conclude that $\Phi$ is a $K$-algebra isomorphism.
\end{proof}

\begin{cor}\label{DNKUGN} There is an isomorphism of $K$-Fr\'echet algebras
\[\eta_N : D(N,K) \tocong \w{U}(\fr{g}_K, N).\]
\end{cor}
\begin{proof} Since $r_n \to 1$ as $n \to \infty$, the discussion at \cite[\S 4, p.162]{ST} shows that $D(N,K) = \invlim D_{r_n}(N,K)$. Now compose the continuous isomorphisms 
\[D(N,K) = \invlim D_{r_n}(N,K) \tocong \invlim D_{1/p}(N^{p^n},K) \rtimes_{N^{p^n}} N \tocong \w{U}(\fr{g}_K, N)\]
given by Theorem \ref{SchmidtCP} and Corollary \ref{ExtendedLazIsom}.
\end{proof}

\begin{proof}[Proof of Theorem \ref{LADistFormula}]
Let $\cC$ denote the set of open $L$-uniform subgroups of $G$. It is non-empty by Lemma \ref{LuniformsExist}, and it is straightforward to see that it is stable under conjugation in $G$ and finite intersections. For any $H \leq N$ in $\cC$, consider the following diagram:
\[ \xymatrix{ D(H,K) \ar[rrrr]^{\eta_H}\ar[d] &&&& \w{U}(\fr{g}_K, H)\ar[d] \\ D(N,K) \ar[rrrr]_{\eta_N} &&&& \w{U}(\fr{g}_K, N). }\]
Each arrow in this diagram is a continuous $K$-algebra homomorphism. For each $h \in H$, the images of $\eta_H(h)$ and $\eta_N(h)$ in $\w{U}(\fr{g}_K,N)$ coincide. Since $K[H]$ is dense in $D(H,K)$, it follows that the diagram commutes. 

Let $\iota_N$ denote the inclusions of $K[N]$ into $\w{U}(\fr{g}_K, N)$ and $D(N,K)$. The pair $(\w{U}(\fr{g}_K,-), \iota)$ is good in the sense of Definition \ref{GoodFIota} by Proposition \ref{wUgGood}, and the pair $(D(-,K),\iota)$ is good  by Example \ref{DNKgood}. The result now follows from Corollary \ref{DNKUGN}.
\end{proof}

\begin{proof}[Proof of Theorem \ref{EqEq}]
The ideal $\fr{m}_0$ from the proof of Theorem \ref{GlobSecDnlK} consists of the central elements of $U(\fr{g}_K)$ that annihilate the trivial representation of $\fr{g}_K$; now it follows from  \cite[Theorem 6.3]{ST} that the category appearing on the left hand side in Theorem \ref{EqEq} is anti-equivalent to the category of coadmissible $D(G,K)$-modules killed by $\fr{m}_0$. We can finally apply Theorem \ref{LADistFormula} and Theorem \ref{BBEquiv}.
\end{proof}

\emph{Acknowledgements.} I would like to thank Michael Harris for the invitation to a very inspiring workshop in C\'ordoba, Argentina and mentioning the words `$G$-equivariant $\mathcal{D}$-modules' whilst there. I am also grateful to Ahmed Abbes, Oren Ben-Bassat, Joseph Bernstein, Thomas Bitoun, Florian Herzig, Kobi Kremnizer, Arthur-C\'esar Le Bras, Vytas Pa\v{s}k\={u}nas, Peter Schneider, Otmar Venjakob and Simon Wadsley for their interest in this work.

\bibliographystyle{plain}
\bibliography{../references}
\end{document}